\documentclass{emsprocart}

\usepackage{srcltx} 
\usepackage{pictex}

\contact[ringel@math.uni-bielefeld.de]{Claus Michael Ringel\\ Fakult\"at f\"ur Mathematik, 
Universit\"at Bielefeld\\ PO Box 100 131\\ D-33501 Bielefeld, Germany}

\newtheorem{theorem}{Theorem}[section]

\newtheorem{proposition}[theorem]{Proposition}

\newtheorem*{coro}{Corollary} 
\newtheorem*{lem}{Lemma} 
\newtheorem*{prop}{Proposition} 

\theoremstyle{definition}

\newcommand{\End}{\operatorname{End}}

\newcommand{\mo}{\operatorname{mod}}

\newcommand{\rad}{\operatorname{rad}}
\newcommand{\soc}{\operatorname{soc}}

\newcommand{\nn}{\operatorname{nn}}
\newcommand{\topp}{\operatorname{top}}
\newcommand{\ssize}{\scriptstyle }

\title[The minimal representation-infinite algebras which are special biserial]{The Minimal Representation-Infinite Algebras which are Special Biserial}

\author[Claus Michael Ringel]{Claus Michael Ringel}

\begin{document}

\begin{abstract} Let $k$ be a field. 
A finite dimensional $k$-algebra  is said to be minimal repre\-sentation-infinite provided
it is representation-infinite and all its proper factor algebras are 
representation-finite. Our aim is to classify the special
biserial algebras which are minimal representation-infinite. The second part describes
the corresponding module categories.
\end{abstract}

\begin{classification}
Primary 16G20, 16G60. Secondary 16D90, 16G70.
\end{classification}

\begin{keywords}
Minimal representation-infinite algebras, special biserial algebras. Quiver. 
Auslander-Reiten quiver. Auslander-Reiten quilt. Sectional paths. Irreducible maps.
Gorenstein algebras. Semigroup algebras. 
\end{keywords}

\maketitle

\section{Introduction}

The study of
minimal representation-infinite $k$-algebras with $k$ an algebraically closed field was one of the central
themes of the representation theory around 1984 with contributions by Bautista, Gabriel, Roiter,
Salmeron, Bongartz, Fischbacher and many others. Recent
investigations of Bongartz \cite{B2} provide a new impetus for analyzing the module category of such an algebra and
seem to yield  a basis for a classification of these algebras. Here is a short summery of this development.
First of all, there are algebras 
with a non-distributive ideal lattice, such algebras have been studied already 1957 by Jans \cite{J}. 
Second, there are  algebras  with a good
universal cover $\widetilde \Lambda$ and such that $\widetilde \Lambda$ 
has a convex subcategory which is a tame concealed algebra of
type $\widetilde {\mathbb D}_n, \widetilde {\mathbb E}_6,  
\widetilde {\mathbb E}_7$ or $ \widetilde {\mathbb E}_8;$  these were the algebras 
which have been discussed by Bautista, Gabriel, Roiter and Salmeron in \cite{BGRS}  (we say that the universal cover is 
{\it good}
provided it is a Galois cover with free Galois group and is interval-finite).
As Bongartz now has shown, the remaining minimal representation-infinite
algebras also have a good cover $\widetilde \Lambda$, but all finite convex subcategories of $\widetilde \Lambda$
are representation-finite. These are the algebras which will be discussed here.
We will show that such an algebra is special biserial and we will provide a full
classification of the special biserial algebras which are minimal representation-infinite.
	              \medskip

Let us recall the definition: 
A finite dimensional $k$-algebra is said to be {\it special biserial} (see \cite{SW})
provided it is Morita equivalent to the path algebra of a quiver $Q$
with relations with the following properties:
 
(1) Any vertex of $Q$ is endpoint of at most two arrows, and 
also starting point of at most two arrows.

(2) If two different arrows $\gamma$ and $\delta$
start in the endpoint of the arrow $\alpha$, then at least one of the paths
$\gamma\alpha, \delta\alpha$ is a relation.

(2$'$) If two different arrows $\alpha$ and $\beta$
end in the starting point of the arrow $\gamma$, then at least one of the paths
$\gamma\alpha, \gamma\beta$ is a relation.\par

Note that the composition of an arrow $\alpha$ with endpoint $a$
and an arrow 
$\gamma$ with starting point $a$ is here denoted by  $\gamma\alpha$,
one should visualize the situation as follows:
$$
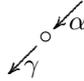
\hbox{\beginpicture
\setcoordinatesystem units <0.6cm,.6cm>
\multiput{} at 0 0.2  2 1.8 /
\put{$\circ$} at 1 1
\arrow <1.5mm>  [0.25,0.75 ] from 0.8 0.8 to 0.2 0.2

\arrow <1.5mm>  [0.25,0.75 ] from 1.8 1.8 to 1.2 1.2

\put{$\alpha$} at 1.7 1.3
\put{$\gamma$} at 0.7 0.3
\endpicture}
$$

The definition of a special biserial algebra looks quite technical,
but actually there are a lot of natural examples of algebras which turn
out to be of this kind. Note that a special biserial algebra is
hereditary if and only if it is Morita equivalent to the path algebra of a quiver of type 
$\mathbb A_n$ or $\widetilde{\mathbb A}_n$, where the cyclic
orientation of $\widetilde {\mathbb A}_n$ has to be excluded in order to get a finite dimensional
algebra.
	
Special biserial algebras were first studied by Gelfand and Ponomarev \cite{GP}:
they have provided the methods in order to classify all the
indecomposable representations of such an algebra. This classification shows that special biserial
algebras are always tame (see also \cite{WW, DS}) and usually they are of non-polynomial growth. 
For the structure
of the Auslander-Reiten quiver of a special biserial algebra we refer
to \cite{BR}. 
The aim of the present paper is to describe the special biserial
algebras which are minimal representation-infinite and to exhibit the
corresponding module categories.

We can assume that the defining relations of the special
biserial algebras to be considered are monomials (since otherwise
we will obtain an indecomposable module which is both projective and
injective, but minimal representation-infinite algebras do not have
indecomposable modules which are both projective and injective).

If $\Lambda$ is a finite dimensional algebra, a simple module $S$ is said to be
a {\it  node} provided $S$ is neither projective nor injective, and
such that $S$ does not occur as a composition factor of a module of the
form $\rad M / \soc M$, where $M$ is indecomposable and not simple.
If $\Lambda$
is given by a quiver with relations, then the simple module $S(a)$ corresponding
to a vertex $a$ is a  node 
if and only if $a$ is neither a sink nor a source
and given an arrow $\alpha$ which ends in $a$ and an arrow $\gamma$ which
starts in $a$, then $\gamma\alpha$ is a relation. There is a well-known
procedure \cite{M} to resolve  nodes: For any algebra $\Lambda$, there is an
algebra $\nn(\Lambda)$ without  nodes such that $\Lambda$ and $\nn(\Lambda)$ 
are stably equivalent:
in case $\Lambda$ is given by a quiver with relations, one just replaces any
vertex $a$ with $S(a)$ a  node by two vertices $a_+, a_-$ such
that $a_+$ is a sink and $a_-$ a source.
   
A vertex $a$ of a quiver will be called an {\it $n$-vertex} provided 
$a$ has $n$ neighbors (this means that there are $n_1$ arrows ending in $a$
and $n_2$ arrows starting in $a$ and $n = n_1+n_2$; observe that in this way, 
the loops at $a$ are counted twice).

\begin{theorem}
 Assume the $k$-algebra $\Lambda$ is special biserial and minimal 
representation-infinite. Then any $4$-vertex of the quiver of $\Lambda$  is a node.
\end{theorem}

If we want to classify algebras which are minimal 
representation-infinite, it is sufficient to deal with algebras without a node,
since an algebra $\Lambda$ is minimal representation-infinite if and only if the
node-free algebra $\nn(\Lambda)$ is minimal
representation-infinite, see section 7.

A finite dimensional hereditary algebra of type $\widetilde {\mathbb A}$ will be said to
be a {\it cycle algebra}. The main task of the paper will be to define two
classes of finite dimensional algebras, the so-called  barbell algebras and wind wheel algebras, 
see sections 5 and 6, respectively. These algebras
are obtained from cycle algebras by a construction which we call barification (see section 4) 
and adding, if necessary, suitable zero relations. 

\begin{theorem}
The special biserial algebras which are 
minimal representation-infinite and have no nodes are the cycle algebras,
 the barbell algebras with non-serial bars
and the wind wheel algebras. 
\end{theorem}
    
\begin{theorem}A minimal representation-infinite algebra is special biserial if and only if
its universal cover $C$ is good and any finite convex subcategory of $C$ is representation-finite.
\end{theorem}

The first part of these notes is devoted to a proof of theorems 1.1, 1.2 and 1.3. 
    \medskip

The second part provides information on the module categories of the minimal
representation-infinite special biserial algebras. As we have mentioned already,  all special biserial 
algebras are tame. Dealing with the minimal representation-infinite ones,
we encounter both algebras of non-polynomial
growth (namely the  barbell algebras) as well as domestic ones
(the  hereditary algebras of type  $\widetilde {\mathbb A}$ as well as all the wind wheel algebras),
note that the domestic ones all are  even $1$-domestic. 
Here, an algebra is said to be {\it $n$-domestic} in case there are precisely $n$
primitive 1-parameter family of indecomposable modules (and additional ``isolated''
indecomposables). 

Not much is known about domestic algebras $\Lambda$ in general, 
not even about 1-domestic algebras! The wind wheel algebras $W$ provide new
examples of $1$-domestic algebras such that the Auslander-Reiten quiver has  an arbitrary 
finite number  of non-regular  components (12.5)
as well as having non-regular components with
arbitrary ramification (12.6). Let us stress that the examples which we present
all have Loewy length $3$. 

Also, we will describe in detail the corresponding
Auslander-Reiten quilt $\Gamma$ of a wind wheel, it is
obtained from the set of Auslander-Reiten components which contain
string modules by inserting suitable infinite dimensional algebraically compact indecomposable modules.
We will see that $\Gamma$  
is a connected orientable 
surface with boundary, its Euler characteristic is $\chi(\Gamma) = -t,$ 
where $t$ is the number of bars.

A further property of the wind wheels $W$ with $t$ bars seems to be of interest: Let $M$ be
a primitive homogeneous and absolutely indecomposable $W$-module, and $\underline E$  
the factor ring of $\End(M)$ modulo the ideal of endomorphisms with semisimple
image, then $\underline E$  is of dimension $t+1$, thus arbitrarily large. 

In general, we will show:  
Let $\Lambda$ be a $k$-algebra which is minimal representation-infinite and special biserial.
Then any complete sectional path is a mono ray, an epi coray or the concatenation of an
epi coray with a mono ray. This implies in particular the following: If $X,Y,Z$ are
indecomposable $\Lambda$-modules with an irreducible monomorphism $X\to Y$
and an irreducible epimorphism $Y \to Z$, then $X = \tau Z.$
    \bigskip

Whereas we present in Part I full proofs for the main results, the discussion in Part II
is less complete, several of the (sometimes tedious) combinatorial verifications are left
to the reader. 
   \bigskip

{\bf Acknowledgment.} The classification of the minimal representation-infinite special
biserial algebras was first announced at the Trondheim conference 2007 
and then presented in lectures at several places. The author is indebted to various
mathematicians for helpful comments. At the ICRA workshop Tokyo 2010, the author
gave a sequence of lectures dealing with minimal representation-infinite algebras
in general. The following text written for the workshop proceedings 
restricts the attention again to the special biserial algebras.
	          \bigskip\bigskip\bigskip

\centerline{\large\bf Part I. The algebras}

\section{Preliminaries: Words }
     
Given a quiver $Q$ with vertex set $Q_0$ and arrow set $Q_1$
and a set $\rho$ of monomial relations (monomial relations are paths of length at least 2),
 we consider (usually finite) words 
using as letters the arrows of the quivers and formal inverses 
of these arrows, the set of such words will be denoted by $\Omega(Q,\rho)$
(and just by $\Omega(Q)$ if no relations are given). In case the algebra $A$
is given by the quiver $Q$ with relations $\rho$, we also dare to write $\Omega(A)$
instead of $\Omega(Q,\rho)$ (but this is an abuse of notation).

Here is the proper definition:
Let $\overline Q$ be the quiver obtained from $Q$ by adding formal 
inverses of the arrows (given an arrow $\alpha$ with starting point 
$s(\alpha)$ and terminal point
$t(\alpha),$ we denote by $\alpha^{-1}$ a formal inverse of
$\alpha,$ with starting point $s(\alpha^{-1}) = t(\alpha)$ and terminal
point $t(\alpha^{-1}) = s(\alpha);$ 
given such a formal inverse $l = \alpha^{-1},$ one writes $l^{-1} = \alpha$).
We consider paths in the quiver $\overline Q$, those of length $n\ge 1$ are of the form 
$$
 w = l_1l_2\cdots l_n \quad\text{with}\quad s(l_i) = t(l_{i+1})\quad  \text{for all}\quad 1 \le i < n
$$
(one may consider $w$ just as the sequence $(l_1,l_2,\dots,l_n)$, but
it will be convenient, to delete the brackets and the colons).
In addition, there are the paths of length zero corresponding to the vertices.
By definition, the {\it inverse\/} 
of $w = l_1\dots l_n$ is $w^{-1} = l_n^{-1}\dots l_1^{-1}$;
a subword of $w$ is of the form 
$l_il_{i+1}\dots l_{j-1}l_j$ (with $1 \le i \le j \le n$ or else a vertex which
is starting or terminal point for some $l_i$.
The elements of $\Omega(Q,\rho)$ are the paths $w = l_1\dots l_n$
in $\overline Q$ which satisfy the following conditions:
                                                 
(W1) We have $l_i^{-1} \neq l_{i+1},$ for all $1 \le i < n.$

(W2) No subword of $w$ or its inverse belongs to $\rho$.
     \medskip

The elements of $\Omega(Q,\rho)$ will be called {\it words} for
the quiver $Q$ with the relations $\rho$, the arrows and their formal
inverses will be called the {\it letters}.
A word $w = l_1\dots l_n$   is said to be {\it direct} provided all the 
letters $l_i$ are arrows, and {\it inverse} provided $w^{-1}$ is direct. 
A word which is either direct of inverse is said
to be {\it serial.} Two letters $l,l'$ will be said to have the same direction, if both
are direct or if both are inverse letters. 

We say that a word $w$ is {\it without repetition} provided no letter
appears twice in $w$. Given a word without repetition and a letter $l$, 
then both $l$ and $l^{-1}$ 
may appear in $w$; in this case we will say that the edge $l^{\pm} = \{l,l^{-1}\}$ 
occurs twice in $w$. 

Given a word $v = l_1\cdots l_s$ of length $s\ge 1$, we will write $v_1 = l_1$
and $v_\omega = l_s$. Given words $v,w$ such that 
the starting point of $v_\omega$ is the endpoint of $w_1$ and
$v_\omega,\ (w_1)^{-1}$ are different, but have the same direction, then we
will say that the pair $(v,w)$ is {\it attracting.} Note that for an attracting
pair $(v,w)$, the composition $vw$ is a word again.

Finally, recall that a word $w$ is called {\it cyclic\/} provided it contains both
direct and inverse letters and such that also $w^2 = ww$ is a word. 
A cyclic word $w$ is said to be {\it primitive\/} provided it is not of the form $v^t$
with $t\ge 2.$
     \medskip

An infinite sequence $l_1l_2 \cdots$ using our letters will be called  a $\mathbb N$-word
provided all the finite subsequences $l_1l_2\cdots l_n$ are words. 
Similarly, a double infinite sequence $\cdots l_{-1}l_0l_1 \cdots$ is said to be a 
$\mathbb Z$-word
provided all the finite subsequences $f_{-n}\cdots l_{-1}l_0l_1\cdots l_n$ are words.
	     \medskip

This report deals mainly with quivers $Q$  with a set $\rho$ of monomial relations 
which yield a special biserial algebra $A$. In this case, the finite dimensional 
$A$-modules are easy to construct and to characterize, this classification goes back
to Gelfand and Ponomarev \cite{GP}. There are two
kinds of indecomposable modules, the string modules and the band modules.
Starting with any word $w\in \Omega(Q,\rho)$ of length $n$, there is an
indecomposable module $M(w)$ of length $n+1$, called a string module.
In addition, there
are one-parameter families of indecomposable $A$-modules which are
constructed starting with a primitive cyclic word $w$ as well as a
finite dimensional vector space $V$ with an automorphism $\phi$ such that
the pair $(V,\phi)$ is indecomposable; the modules $M(w,\phi)$ are called the
band modules. If $V$ is one-dimensional and $\phi$ is the multiplication by
$\lambda\in k\setminus\{0\}$, then we write $M(w,\lambda)$ instead of $M(w,\lambda)$.
For an outline of these constructions we refer to  \cite{Ralgcom}.

\section{The cycle algebras.}
     
We first describe the hereditary algebras of type $\widetilde {\mathbb A}$. They also will be
used in order to construct the barbell as well as the wind wheel algebras.
     
We start with a function $\epsilon:\{1,\dots,n\} \to \{1,-1\}$; if necessary,
we call such a function an {\it orientation sequence} 
of length $n$.
In order to specify $\epsilon$, we usually will write
just the sequence $\epsilon(1)\epsilon(2)\cdots\epsilon(n)$, or the corresponding
sequence of signs $+$ and $-$. 
Note that this means that we consider $\epsilon$ as a word of length $n$ in the
letters $+$ and $-$. This interpretation explains also the following conventions:
Assume there is given an orientation sequence $\epsilon$ of length $n$. We say that
$\epsilon$ starts with $\epsilon(1)$ and ends with $\epsilon(n).$ We write
$\epsilon^{-1}$ for the function with $\epsilon^{-1}(i) = -\epsilon(n+1-i)$ for 
$1 \le i \le n$. Given a further orientation sequence $\epsilon'$ say of length $n'$,
let $\epsilon\epsilon'$ be the orientation sequence of length $n+n'$ with
$\epsilon\epsilon'(i) = \epsilon(i)$ for $1\le i \le n$ and 
$\epsilon\epsilon'(i) = \epsilon'(i-n)$ for $n+1 \le i \le n+n'$.

The orientation sequences which we are interested in will be obtained by starting
with a word $w= l_1\cdots l_s\in\Omega(Q,\rho)$, where $Q$ is a quiver with
monomial relations $\rho$ and looking at $\epsilon(w)$ defined by
$\epsilon(w)(i) = 1$ if $l_i$ is a direct letter and $\epsilon(w)(i) = -1$
otherwise.

Let $\epsilon$ be an orientation sequence.
We attach to $\epsilon$ the hereditary algebra
$H(\epsilon)$ with the following quiver: its vertices are $a_1,a_2,\dots,a_n = a_0$,
and there is an arrow $\alpha_i:a_i \to a_{i-1}$ in case $\epsilon(i) = 1$ and
$\alpha_i:a_{i-1} \to a_i$ in case $\epsilon(i) = -1.$ 
The algebra $H(\epsilon)$ is finite dimensional if and only if $\epsilon$ is not 
constant. The algebras $H(\epsilon)$ with $\epsilon)$ not constant, will be called
the {\it cycle algebras}.

For example, if $\epsilon = (+ + - + - +),$ then $H(\epsilon)$ is the path algebra 
$$
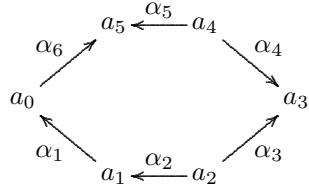
\hbox{\beginpicture
\setcoordinatesystem units <1.2cm,1cm>
\put{$a_1$} at 1 0
\put{$a_2$} at 2 0
\put{$a_3$} at 3 1
\put{$a_4$} at 2 2
\put{$a_5$} at 1 2
\put{$a_0$} at 0 1
\arrow <1.5mm>  [0.25,0.75 ] from 0.8 0.2 to 0.2 0.8
\arrow <1.5mm>  [0.25,0.75 ] from 1.8 0    to 1.2 0
\arrow <1.5mm>  [0.25,0.75 ] from 2.2 0.2 to 2.8 0.8
\arrow <1.5mm>  [0.25,0.75 ] from 2.2 1.8 to 2.8 1.2
\arrow <1.5mm>  [0.25,0.75 ] from 1.8 2    to 1.2 2
\arrow <1.5mm>  [0.25,0.75 ] from 0.2 1.2 to 0.8 1.8
\put{$\alpha_1$} at 0.3 0.3
\put{$\alpha_2$} at 1.5 0.2
\put{$\alpha_3$} at 2.7 0.3
\put{$\alpha_4$} at 2.7 1.7
\put{$\alpha_5$} at 1.5 2.2
\put{$\alpha_6$} at 0.3 1.7
\endpicture}
$$
Always $\alpha_1^{\epsilon(1)}\alpha_2^{\epsilon(2)}\cdots \alpha_n^{\epsilon(n)}$
is a primitive cyclic word.

\section{Barification.}
     
Starting from a hereditary algebra of type $\widetilde {\mathbb A}$, the further
algebras will be obtained by identifying some subquivers and adding zero relations.
The essential part of the construction will be described now.
    
Let $Q$ be a quiver with relations. Let $a_1,\dots,a_t, a'_1,\dots, a'_t$
be pairwise different $2$-vertices such that $a_i, a_{i+1}$ as well
es $a'_i, a'_{i+1}$ are neighbors, for all $1 \le i < t.$ 
Thus there are letters $l_i, l'_i$ for $0 \le i \le t$ such that
$l_{i-1}, l'_{i-1}$ end in $a_{i}$, or $a'_{i}$ respectively, 
and $l_{i}, l'_{i}$ start in $a_{i}$, or $a'_{i}$ respectively, for
$1 \le i \le t.$ We assume that the letters $l_i,l'_i$
have the same direction, for any $1 \le i < t$, wheres $l_0, l'_0$
have different direction, and also $l_t, l'_t$ have different direction.
We assume in addition that the letters $l_i, l'_i$ for $1 \le i < t$ are
not involved in any relation.

Let $v = l_1\cdots l_{t-1}$ and $v' = l'_1\cdots l'_{t-1}$ The {\it
barification}  of $v$ and $v'$ is defined as follows: We identify
the vertex $a_i$ with $a'_i$ for $1 \le i \le t$, and label the new vertex again $a_i$;
also, we identify the arrow $\alpha_i$
between $a_i,a_{i+1}$ with the arrow between $a'_i,a'_{i+1}$ and label it again $\alpha_i$.
We add as new relation the compositions $l_0(l'_0)^{-1}$ as well as
$(l'_t)^{-1}l_t$. If necessary, we will denote the new quiver with relations by
$Q(v,v').$ The subquiver of $Q(v,v')$ given by the identified vertices and arrows is called
a {\it bar} (at least if $t \ge 2$).

If $t = 1$, then we just identify two $2$-vertices of $Q$ in order to form
a $4$-vertex. If $t \ge 2$, then we identify sequences of $2$-vertices
and obtain from the identification of $a_1$ with $a'_1$ a $3$-vertex,
then several $2$-vertices, and finally as the identification of $a_t$ with
$a_t$ again a $3$-vertex. 
    
Note that in case we start with a quiver $Q$ which is special biserial, 
the new quiver $Q(v,v')$ with relations
again will be special biserial. 
     
Here is a schematic example with $t = 5$.
We indicate the relevant parts $v = (a_1 \leftarrow a_2
\leftarrow a_3 \rightarrow a_4 \leftarrow a_5)$ and 
$v' = (a'_1 \leftarrow a'_2
\leftarrow a'_3 \rightarrow a'_4 \leftarrow a'_5)$, but we do not
specify what happens further (we just draw a box) 
$$
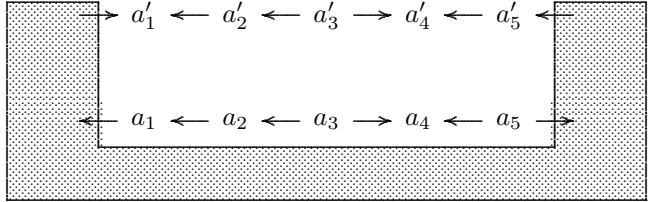
\hbox{\beginpicture
\setcoordinatesystem units <1.2cm,.35cm>
\multiput{} at 0 0  10 4 /
\put{$a'_1$} at 3 4
\put{$a'_2$} at 4 4
\put{$a'_3$} at 5 4
\put{$a'_4$} at 6 4
\put{$a'_5$} at 7 4
\put{$a_1$} at 3 0
\put{$a_2$} at 4 0
\put{$a_3$} at 5 0
\put{$a_4$} at 6 0
\put{$a_5$} at 7 0

\arrow <1.5mm> [0.25,0.75] from 2.3 4 to 2.7 4
\arrow <1.5mm> [0.25,0.75] from 3.7 4 to 3.3 4
\arrow <1.5mm> [0.25,0.75] from 4.7 4 to 4.3 4
\arrow <1.5mm> [0.25,0.75] from 5.3 4 to 5.7 4
\arrow <1.5mm> [0.25,0.75] from 6.7 4 to 6.3 4
\arrow <1.5mm> [0.25,0.75] from 7.7 4 to 7.3 4
\arrow <1.5mm> [0.25,0.75] from 2.7 0 to 2.3 0
\arrow <1.5mm> [0.25,0.75] from 3.7 0 to 3.3 0
\arrow <1.5mm> [0.25,0.75] from 4.7 0 to 4.3 0
\arrow <1.5mm> [0.25,0.75] from 5.3 0 to 5.7 0
\arrow <1.5mm> [0.25,0.75] from 6.7 0 to 6.3 0
\arrow <1.5mm> [0.25,0.75] from 7.3 0 to 7.7 0

\put{} at 0 -3
\setlinear
\plot 2.5 4.5  1.5 4.5  1.5 -3  8.5 -3  8.5 4.5  7.5 4.5  7.5 -1  2.5 -1  2.5 4.5 / 

\setshadegrid span <.4mm>
\vshade 1.5 -3 4.5 <,z,,> 2.5 -3 4.5 <z,z,,> 2.55  -3 -1  <z,z,,>
7.45  -3 -1 <z,z,,> 7.5  -3 4.5 <z,,,> 8.5  -3 4.5 /

\endpicture}
$$
The barification yields a quiver of the following form:
$$
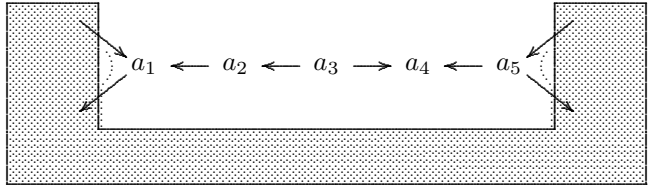
\hbox{\beginpicture
\setcoordinatesystem units <1.2cm,1.2cm>
\multiput{} at 0 -1.3  10 1 /
\put{$a_1$} at 3 0
\put{$a_2$} at 4 0
\put{$a_3$} at 5 0
\put{$a_4$} at 6 0
\put{$a_5$} at 7 0

\arrow <1.5mm> [0.25,0.75] from 2.3 0.5 to 2.8 0.1
\arrow <1.5mm> [0.25,0.75] from 2.8 -.1 to 2.3 -.5
\arrow <1.5mm> [0.25,0.75] from 3.7 0 to 3.3 0
\arrow <1.5mm> [0.25,0.75] from 4.7 0 to 4.3 0
\arrow <1.5mm> [0.25,0.75] from 5.3 0 to 5.7 0
\arrow <1.5mm> [0.25,0.75] from 6.7 0 to 6.3 0
\arrow <1.5mm> [0.25,0.75] from 7.7 0.5 to 7.2 0.1
\arrow <1.5mm> [0.25,0.75] from 7.2 -.1 to 7.7 -.5
\setsolid

\setdots <.7mm>
\setquadratic
\plot 7.45 -0.2  7.35 0 7.45 0.2 /
\plot 2.55 -0.2  2.65 0 2.55 0.2 /

\setlinear
\setsolid
\plot 2.5 0.7  1.5 0.7  1.5 -1.3  8.5 -1.3  8.5 0.7  7.5 0.7  7.5 -.7  2.5 -.7  2.5 0.7 / 

\setshadegrid span <.4mm>
\vshade 1.5 -1.3 0.7 <,z,,> 2.5 -1.3 0.7 <z,z,,> 2.55  -1.3 -.7  <z,z,,>
7.45  -1.3 -.7 <z,z,,> 7.5  -1.3 0.7 <z,,,> 8.5  -1.3 0.7 /
\endpicture}
$$
(the box is not changed).

\section{The barbell algebras.}
     
Definition: Consider  orientation sequences $\epsilon, \eta,\epsilon'$ and assume that
both $\epsilon$ and $\epsilon'$ start and end with $+$. We start with the
hereditary algebra $H(\epsilon\eta\epsilon'\eta^{-1})$, and construct the 
barification using the two copies of $\eta$, this will be the {\it barbell algebra}
$B(\epsilon,\eta,\epsilon')$. The subquiver given by (the identified copies of) $\eta$
will be called its {\it bar}. 
     \medskip

{\bf Example 1:} 
Start with $\epsilon = \eta = \epsilon' = (+).$ Then $H(\epsilon\eta\epsilon\eta^{-1})$ has the following
shape:
$$
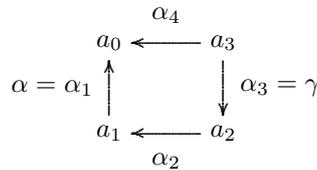
\hbox{\beginpicture
\setcoordinatesystem units <1.5cm,1.2cm>
\put{$a_0$} at 0 1
\put{$a_1$} at 0 0
\put{$a_2$} at 1 0
\put{$a_3$} at 1 1
\arrow <1.5mm> [0.25,0.75] from 0 0.2 to 0 0.8
\arrow <1.5mm> [0.25,0.75] from 0.8 0 to 0.2 0
\arrow <1.5mm> [0.25,0.75] from 1 0.8 to 1 0.2
\arrow <1.5mm> [0.25,0.75] from 0.8 1 to 0.2 1
\put{$\alpha= \alpha_1$}  at -0.5 0.5
\put{$\alpha_2$}          at  0.5 -0.3
\put{$\alpha_3 = \gamma$} at  1.5 0.5
\put{$\alpha_4$}          at  0.5 1.3
\endpicture}
$$
After the identification of $\alpha_2$ and $\alpha_4$, we write $\beta$ for the
identified arrow:
$$
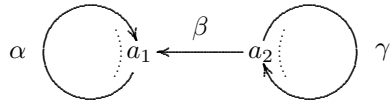
\hbox{\beginpicture
\setcoordinatesystem units <0.8cm,0.8cm>
\put{} at -1 0.8
\put{} at  3 -0.8
\put{$a_1$} at 0 -0.04
\put{$a_2$} at 2 -0.04
\arrow <1.5mm> [0.25,0.75] from 1.7 0 to 0.3 0
\circulararc 320 degrees from -0.05 0.2 center at -0.8 0 
\arrow <1.5mm> [0.25,0.75] from -0.09 0.3 to -0.05 0.2
\circulararc -330 degrees from  2.05 0.2 center at  2.8 0 
\arrow <1.5mm> [0.25,0.75] from 2.09 -0.3 to 2.05 -0.2
\put{$\alpha$} at -2 0
\put{$\beta$} at 1 0.4
\put{$\gamma$} at 4 0
\setdots <.7mm>
\setquadratic
\setdots <.7mm>
\setquadratic
\plot -0.4 -0.4 -0.3 0 -0.4 0.4 /
\plot 2.4 -0.4 2.3 0 2.4 0.4 /

\endpicture}
$$
Here, the bar is just one arrow (namely $\beta$), thus serial.

\begin{proposition}
 A barbell algebra is minimal representation-infinite if and only if the bar
is not serial.
\end{proposition}

In case the bar $\eta$ is direct, there are arrows $\alpha, \gamma$
such that $\alpha \eta \gamma$ is a word for the barbell algebra. If we add
this word as a relation, we obtain an algebra which still is
representation-infinite (it is a wind wheel algebras as discussed in the next section, 
thus  $1$-domestic).  In  example 1, the bar was serial, thus this barbell algebra
was not minimal representation-infinite. Here is an example of a barbell
with non-serial bar:
     \medskip

{\bf Example 2.}
We start with 
$$
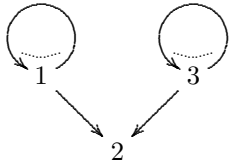
\hbox{\beginpicture
\setcoordinatesystem units <1cm,1cm>
\put{} at 0 2
\put{$1$} at 0 1
\put{$3$} at 2 1
\put{$2$} at 1 0
\arrow <1.5mm> [0.25,0.75] from 0.2 0.8 to 0.8 0.2
\arrow <1.5mm> [0.25,0.75] from 1.8 0.8 to 1.2 0.2
\circulararc 310 degrees from 0.2 1.1 center at 0 1.5 
\circulararc 310 degrees from 2.2 1.1 center at 2 1.5 
\arrow <1.5mm> [0.25,0.75] from -0.25 1.14 to -0.2 1.1
\arrow <1.5mm> [0.25,0.75] from 1.75 1.14 to 1.8 1.1
\setdots <.5mm>
\setquadratic
\plot -0.25 1.3 0 1.25 .25 1.3 /
\plot 1.75 1.3 2 1.25 2.25 1.3 /
\endpicture}
$$
In order to construct this algebra, we can start with 
$\epsilon = (+),$ $\eta = (-+), \epsilon' = (+).$ Then  $H(\epsilon\eta\epsilon\eta^{-1})$
has the following shape:
$$
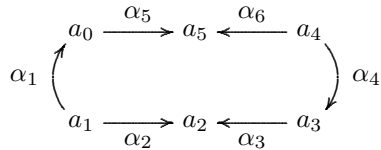
\hbox{\beginpicture
\setcoordinatesystem units <1.5cm,1.2cm>
\put{$a_0$} at 0 1
\put{$a_1$} at 0 0
\put{$a_2$} at 1 0
\put{$a_3$} at 2 0 
\put{$a_4$} at 2 1
\put{$a_5$} at 1 1
\arrow <1.5mm> [0.25,0.75] from -0.175 0.8 to -0.15 0.85
\arrow <1.5mm> [0.25,0.75] from 2.175 0.2 to 2.15 0.15

\arrow <1.5mm> [0.25,0.75] from 0.2 0 to 0.8 0
\arrow <1.5mm> [0.25,0.75] from 1.8 0 to 1.2 0

\arrow <1.5mm> [0.25,0.75] from 0.2 1 to 0.8 1
\arrow <1.5mm> [0.25,0.75] from 1.8 1 to 1.2 1
\circulararc -80 degrees from -0.15 0.15 center at .2 0.5 
\circulararc 80 degrees from  2.15 0.15 center at 1.8 0.5 

\put{$\alpha_1$}  at -0.5 0.5
\put{$\alpha_2$}  at  0.5 -0.2
\put{$\alpha_3$}  at  1.5 -0.2
\put{$\alpha_4$}  at  2.5 0.5
\put{$\alpha_5$}  at  0.5 1.2
\put{$\alpha_6$}  at  1.5 1.2

\endpicture}
$$
Here the bar (given by the arrows $1 \rightarrow 2 \leftarrow 3$) is not serial, thus
the algebra is minimal representation-infinite. 
     \medskip

\section{The wind wheel algebras.}
    
A {\it wind wheel algebra} $W$ 
is given by a cyclic word $w$ without repetition which is of the form
$$
 w = u_1v_1\cdots u_{2t}v_{2t}
$$
with words $u_i, v_i$ of length at least 1 and such that there is a 
(necessarily fixed point free) involution $\sigma$ on the set $\{1,2,\dots 2t\}$ 
with the following properties:

(WW1)  The words $v_i$ are serial and $v_i = v_{\sigma(i)}^{-1}$.

(WW2) The edges appearing in some $u_i$ occur only once in $w$, those
   occurring in some $v_i$ occur twice in $w$ (namely in $v_i$ and in $v_{\sigma(i)}$).

(WW3) The pairs $(v_i,u_{i+1})$ are attracting, the pairs $(u_i,v_i)$ are
   not attracting (here, $u_{2t+1} = u_1)$.

Note that the factorization into the subwords $u_i,v_i$ and the permutation
$\sigma$ are uniquely determined by $w$, thus we can write $W = W(w)$.

The algebra $W(w)$ is obtained from the $\widetilde {\mathbb A}$-algebra $H(\epsilon(w))$
by identifying the path $v_i$ with $v_{\sigma(i)}^{-1}$ (''barification''), for all
$i,$ and using additional zero relations as follows:
Let $v_i$ be direct, and $v_j = v_i^{-1}$ (thus $j = \sigma(i)$). Then the barification
relations are
$$
   u_{i,\omega}u_{j+1,1}, \quad (u_{i+1,1})^{-1}(u_{j.\omega})^{-1}
$$
(recall that for any path $u_i$, we denote by $u_{i,1}$ its first letter, by $u_i,\omega$ the last one).
And we take in addition also the paths
$$
  u_{i,\omega}v_i(u_{j,\omega})^{-1}
$$
as relations. Thus, there are $2t$ monomial relations of length $2$ as well as $t$ long relations
(of length at least $3$). 
      \bigskip

As in the case of a barbell algebra, a subquiver given by 
(the identified copies of) some $v_i$ will be called a {\it bar}. 
     \bigskip

Our example 1 yields the wind wheel algebra for the following word
$$
 \alpha\beta\gamma^{-1}\beta^{-1} \quad\text{with}\quad u_1 = \alpha,\ v_1 = \beta,\
 u_2 = \gamma^{-1},\ v_2 = \beta^{-1}, \quad\text{and}\quad \sigma = (1,2),
$$
its quiver with relations is as follows:
$$
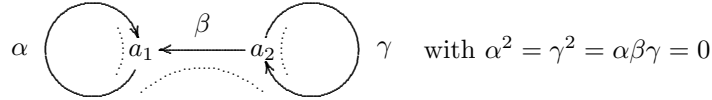
\hbox{\beginpicture
\setcoordinatesystem units <0.8cm,0.8cm>
\put{} at -1 0.8
\put{} at  3 -0.8
\put{$a_1$} at 0 -0.04
\put{$a_2$} at 2 -0.04
\arrow <1.5mm> [0.25,0.75] from 1.7 0 to 0.3 0
\circulararc 320 degrees from -0.05 0.2 center at -0.8 0 
\arrow <1.5mm> [0.25,0.75] from -0.09 0.3 to -0.05 0.2

\circulararc -330 degrees from  2.05 0.2 center at  2.8 0 

\arrow <1.5mm> [0.25,0.75] from 2.09 -0.3 to 2.05 -0.2
\put{$\alpha$} at -2 0
\put{$\beta$} at 1 0.4
\put{$\gamma$} at 4 0
\setdots <.7mm>
\setquadratic
\setdots <.7mm>
\setquadratic
\plot -0.4 -0.4 -0.3 0 -0.4 0.4 /
\plot 2.4 -0.4 2.3 0 2.4 0.4 /
\plot 0 -.7  1 -.3  2 -.7 /
\put{with $\alpha^2 = \gamma^2 = \alpha\beta\gamma = 0$} at  7 0
\endpicture}
$$
	\medskip

Further examples are presented at the end of part I. 
	\medskip

There is a a canonical map $\eta:\Omega(H(\epsilon(w))) \to \Omega(W(w))$ defined
as follows: write $w = l_1\cdots l_n$  with letters $l_i$ for the quiver of $W(w)$, then
$l_i$ may be considered as an arrow of the quiver of $H(\epsilon(w))$. We set
$\eta(l_i) = l_i$ and extend this multiplicatively. 

Given a bar $v$ of $W(w)$, there are uniquely defined letters $l_1, l_2$ such that
both $(l_1,v)$ and $(v,l_2)$ are attracting pairs. 
We call $\overline v = l_1vl_2$ the {\it closure} of the bar $v$. 

\begin{proposition}
A word in $\Omega(W(w))$ does not belong to the image of $\eta$ if and only if
it contains the closure of a bar as a subword. 
\end{proposition}

\begin{proposition}  The wind wheel algebra $W = W(w)$ is 
domestic with only one primitive cyclic word, namely $w$. If $t$ is the number of bars, then
there are precisely
$t$ non-periodic (but biperiodic) $\mathbb Z$-words: Write $w = w_1vw_2v^{-1}$, where $v$
is a bar. Then 
$$
 {}^\infty(w^{-1})w_2^{-1}v^{-1}w_1^{-1}vw_2v^{-1}w^\infty
$$
is such a $\mathbb Z$-word.
\end{proposition}

The $\mathbb Z$-word ${}^\infty(w^{-1})w_2^{-1}v^{-1}w_1^{-1}vw_2v^{-1}w^\infty$
determines uniquely  the central part $\overline v = l_1vl_2$ where
$l_1 = (w_1^{-1})_\omega$ and $l_2 = (w_2)_1$, thus it determines the bar $v$;
Indeed, this word is the only $\mathbb Z$-word which contains a subword of the form    
$l_1vl_2$ such that both $(l_1,v)$ and $(v,l_2)$ are attracting pairs.

\begin{proposition} Let $W = W(w)$ be a wind wheel algebra with $t$ bars. Let $\lambda\in k\setminus\{0\}$.
Then  the endomorphism ring of $M = M(w,\lambda)$ is a radical square zero algebra with radical dimension 
$t$, and the only endomorphism 
of $M$ with semisimple image is the zero endomorphism. 
\end{proposition}
 
Thus we see that the wind wheels provide examples of $1$-domestic algebras
$\Lambda$  with a primitive homogeneous
and absolutely indecomposable $\Lambda$-module $M$ such that 
the factor ring of $\End(M)$ modulo the ideal of endomorphisms with semisimple
image is of arbitrarily large dimension. 

\begin{proof} Any bar $b$ provides an endomorphism of $M$ with image $M(b)$, 
these endomorphisms form a basis of the radical of $\End(M).$
\end{proof}

\section{ Proof of theorem 1.1}
     	           
We want to present the proof of Theorem 1.1. 

\subsection{Resolving a node} The process of resolving a node $a$ can be visualized as follows:
 $$
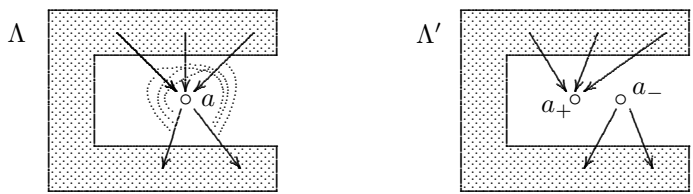
\hbox{\beginpicture
\setcoordinatesystem units <.6cm,.6cm>
\put{\beginpicture
\multiput{} at 0 0  5 4 /
\plot 0 0  0 4   5 4  5 3  1 3  1 1  5 1  5 0  0 0 /
\arrow <1.5mm> [0.25,0.75] from 1.5 3.5 to 2.8 2.2
\arrow <1.5mm> [0.25,0.75] from 1.5 3.5 to 2.8 2.2
\arrow <1.5mm> [0.25,0.75] from 3 3.5 to 3 2.2
\arrow <1.5mm> [0.25,0.75] from 4.5 3.5 to 3.2 2.2
\put{$\circ$} at 3 2
\arrow <1.5mm> [0.25,0.75] from 2.9 1.8 to 2.5 0.5
\arrow <1.5mm> [0.25,0.75] from 3.2 1.8 to 4.2 0.5
\put{$a$} at 3.5 2
\put{$\Lambda$} at -0.7 3.5
\setquadratic
\setdots <.5mm>
\plot 2.3  2.5  2.2 2  2.7  1.3 /
\plot 2.9  2.7  2.4 2.1  2.75  1.5 /
\plot 3.5  2.7  2.6 2.2  2.8  1.7 /
\plot 2.4  2.7   3.7 2.5        3.6 1.5 /
\plot 3.1  2.8   3.9 2.4        3.6 1.35 /
\plot 3.9  2.8   4.1 2.2        3.7 1.2 /
\setlinear
\setshadegrid span <0.5mm>
\hshade 0  0 5  <,,z,z>  0.99  0 5 <,,z,z>  1 0 1 <,,z,z> 3 0 1 <,,z,z> 3.01 0 5 <,,z,z>
  4 0 5 /  
\endpicture} at 0 0

\put{\beginpicture
\multiput{} at 0 0  5 4 /
\plot 0 0  0 4   5 4  5 3  1 3  1 1  5 1  5 0  0 0 /
\arrow <1.5mm> [0.25,0.75] from 1.5 3.5 to 2.3 2.2
\arrow <1.5mm> [0.25,0.75] from 3 3.5 to 2.5 2.2
\arrow <1.5mm> [0.25,0.75] from 4.5 3.5 to 2.7 2.2
\put{$\circ$} at 2.5 2
\put{$\circ$} at 3.5 2
\arrow <1.5mm> [0.25,0.75] from 3.4 1.8 to 2.7 0.5
\arrow <1.5mm> [0.25,0.75] from 3.7 1.8 to 4.2 0.5
\put{$a_+$} at 2.1 1.8
\put{$a_-$} at 4.1 2.2
\put{$\Lambda'$} at -0.7 3.5
\setshadegrid span <0.5mm>
\hshade 0  0 5  <,,z,z>  0.99  0 5 <,,z,z>  1 0 1 <,,z,z> 3 0 1 <,,z,z> 3.01 0 5 <,,z,z>
  4 0 5 /  
\endpicture} at 9 0
\endpicture}
$$
We replace the vertex $a$ by two vertices labeled $a_+$ and $a_-$ such that $a_+$ becomes a sink, 
$a_-$ a source: 
all arrows of $\Lambda$ will be kept, however,
if an arrow of $\Lambda$ ends in $a$, then in $\Lambda'$ it ends in $a_+$, whereas
if an arrow of $\Lambda$ starts in $a$, then in $\Lambda'$ it starts in $a_-$. 
Since all the paths
$\gamma\alpha$ with $\alpha$ ending in $a$ (and therefore $\gamma$ starting in $a$)  are relations for
$\Lambda$, there is a minimal set of relations consisting of these paths as well as of a set
$\rho'$ of relations which do not pass through $a$ (a relation is a linear combination of paths
and we say that the relation passes through $a$ provided at least one of the paths contains
a subpath $\gamma\alpha$  with $\alpha$ ending in $a$). It is the set $\rho'$ which is
used as set of relations for $\Lambda'$.  

The important feature of this construction is the following: 
There is a canonical functor $\mo \Lambda \to \mo \Lambda'$ which yields a bijection
between the indecomposable $\Lambda$-modules and the indecomposable $\Lambda'$-modules different from the
simple $\Lambda'$-module $S(a_-).$)
       
There is the  following quite obvious assertion:

\begin{lem}  Assume $\Lambda$ is a finite dimensional algebra
with a node $a$. Let $\Lambda'$ be obtained from $\Lambda$ by resolving the 
node. Then $\Lambda$ is minimal representation-infinite if and only
if $\Lambda'$ is minimal representation-infinite.
\end{lem}
   
Thus, if we want to classify algebras which are minimal 
representation-infinite, it is sufficient to deal with algebras without a node.
			                                    
\subsection{Cyclic words}
Recall that a word $w$ which is neither direct nor inverse is called cyclic,
provided $w^2$ is a word.
	     
\begin{lem}
 Let $w = \alpha u \alpha v$ be a cyclic word with $\alpha$ an arrow.
Then at least one of the words $\alpha u$, $\alpha v$ is a cyclic word.
\end{lem}    

\begin{proof}
First, assume that neither $u$ nor $v$ is direct, write $u = u_1u_2u_3$
and $v = v_1v_2v_3$ with $u_1,u_3, v_1, v_3$ all being direct and of maximal
possible length. Since $w$ is a cyclic word, noth $v_3\alpha u_1$ and 
$u_3\alpha v_1$ are words. Assume that $\alpha u$ is not a cyclic word, then
there is a zero relation which is a subword of $u_3\alpha u_1$. Since no
subword of $v_3\alpha u_1$ is a zero relation, we conclude that $u_3 = u_3'v_3$
for some word $u_3.$ With $u_3\alpha v_1 = u_3'v_3\alpha v_1$ also
$v_3\alpha v_1$ is a word. This implies that $\alpha v$ is a cyclic word.

Now assume $u$ is direct. Since $w$ is not direct, we know that $v$ cannot be
direct. As above, write $v = v_1v_2v_3$ with $v_1, v_3$ direct and of maximal
possible length. Since $w$ is a cyclic word, there is no zero relation
which is a subword of $v_3\alpha u \alpha v_1$. But the direct word 
$v_3\alpha v_1$ is a subword of $v_3\alpha u \alpha v_1$, thus we see that
there is no subword of $v_3\alpha v_1$ is a zero relation, thus $\alpha v$
is a cyclic word. This completes the proof.
\end{proof}
   
As a consequence, we see: if $w$ is a cyclic word of minimal length, then
any arrow can occur in $w$ at most once as a direct letter, and at most once as an inverse letter. In particular, the length of $w$ is bounded 
by $2a$, where $a$ is the number of arrows. 
(A typical example of a cyclic word of minimal length which contains both
an arrow as well as its inverse is given by example 1.)

\subsection{The $4$-vertices}
Let $a$ be a $4$-vertex, with arrows $\alpha,\beta$ ending in $a$
and arrows $\gamma, \delta$ starting in $a$ such that the
words $\gamma\beta$ and $\delta\alpha$ are relations. 

Let $w$ be a cyclic word of smallest possible length.
Assume $w$ contains $\gamma\alpha$ as a subword. 

Up to rotation, we can assume that
$w$ starts with $\gamma\alpha$. Assume $w$ also contains $\beta^{-1},$
say $w = \gamma\alpha u \beta^{-1} v$, for some words $u,v$. Then
$\alpha u \beta^{-1}$ is a cyclic word of shorter length, 
a contradiction.
Similarly, if $w = \gamma\alpha u \delta^{-1} v$, for some words $u, v$,
then $\gamma \delta^{-1}v$ is a cyclic word of shorter length.

Now assume that $w$ contains $\delta\beta$ as a subword. 
Then $w = \gamma u \beta v$ where $u,v$ are words such that $u$ starts
with $\alpha$ and ends with $\delta$ (it may be that $u = \alpha u'\delta$,
or else $\alpha = u = \delta$). Then we consider $w' = \gamma u^{-1} \beta v$.
This is again a cyclic word, and it contains neither $\delta\beta$
or its inverse as a subword. Namely, according to Lemma 1, $\delta\beta$
was contained just once as a subword of $w$, and this composition has been
destroyed when we built $w'$. Also no new composition has been created.

Finally, we have to consider the case that $w$ does not contain $\delta\beta$.
Since $w$ has to contain $\beta$ and $\delta$, it must
contain $\alpha^{-1}\beta$ and $\delta\gamma^{-1}.$
By Lemma 1, $w$ contains $\gamma\alpha$ only once, and it cannot contain
$\alpha^{-1}\gamma^{-1}$, since otherwise we apply the previous 
considerations to $w^{-1}$. Let $w = \gamma\alpha u \delta\gamma^{-1} u'$,
then the subword $\alpha u \delta$ does not contain $\gamma\alpha$ or
its inverse. Similarly, write $w = \gamma\alpha v'\alpha^{-1}\beta v$, then
$\beta v \gamma$ is a subword of $w^2$ and does not 
contain $\gamma\alpha$ or its inverse. We form the word
$\alpha u \delta (\beta v \gamma)^{-1} = 
\alpha u \delta \gamma^{-1} v^{-1} \beta^{-1}$. This is a cyclic word
which does not contain $\gamma\alpha$ or its inverse. This shows that the
factor algebra with the added relations $\gamma\alpha$ and
$\delta\beta$ is still representation-infinite.
This completes the proof.

\section{Proof of Theorem 1.2}
    
We can assume that we deal with a special biserial algebra  $\Lambda$ with no 4-vertex. 
Let $w$ be a cyclic word of minimal length, thus no letter occurs twice. 
Given an arrow $\alpha$, we will say that the edge 
$\alpha^\pm$ occurs once in $w$ if precisely
one of the letters $\alpha, \alpha^{-1}$ occurs in $w$, and that it occurs twice if
both occur; these are the only possibilities, since we can assume that any arrow or its inverse occurs in $w$. 

We can assume that all the vertices of $Q$ are 2-vertices or 3-vertices (note that the support
of a cyclic word cannot contain a $1$-vertex). In case all the vertices are 2-vertices,
then we deal with a hereditary algebra of type $\widetilde {\mathbb A}$. Thus we can assume that there
is at least one $3$-vertex.

If $a$ is a 3-vertex, there
can be only one zero-relation of length two passing through $a$, since otherwise the
word $w$ could not pass through $a$. Thus, we deal with the following local situation
$$
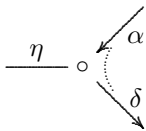
\hbox{\beginpicture
\setcoordinatesystem units <1cm,1cm>
\multiput{} at 0 0  2 2 /
\put{$\circ$} at 1 1
\plot 0 1  0.8  1 /
\arrow <1.5mm> [0.25,0.75] from 1.2 0.8 to 1.8 0.2
\arrow <1.5mm> [0.25,0.75] from 1.8 1.8 to 1.2 1.2
\put{$\alpha$} at 1.7 1.4
\put{$\delta$} at 1.7 0.6
\put{$\eta$} at 0.4 1.2
\setquadratic
\setdots <.5mm>
\plot 1.4 0.7  1.3 1  1.4 1.3 /
\endpicture}
$$
{\it Then $\eta^\pm$ occurs twice in $w$ whereas $\alpha^{\pm}$ and 
$\delta^\pm$ occur just once.} Proof: Since $\alpha^\pm$ as well as $\delta^\pm$ 
both have to occur at least once,
this yields two different subwords of $w$ involving $\eta^\pm$. But if $\alpha^\pm$
or $\delta^\pm$ would occur twice, we would obtain at least three letters 
of the form $\eta$ and $\eta^{-1}$, impossible.

In this way, we see that there are edges which occur once, as well as edges which occur
twice. 

This shows clearly the structure of the word $w$. Up to rotation, we can assume that
$w$ starts in a 3-vertex, and that the inverse of the first letter does not occur in $w$.
Of course, then the last letter yields an edge which occurs twice. We cut $w$ into pieces
$$
 w = u_1v_1\cdots u_mv_m,
$$ 
such that any $u_i$ uses edges which occur only once, whereas
the $v_i$ use edges which occur twice. We obtain an involution $\sigma$ on the
set $\{1,2,\dots,m\}$ with $v_{\sigma(i)} = v_i^{-1}$ (note that in case $l$ is a letter which 
occurs in some $v_i$, and $l^{-1}$ occurs in $v_j$, then necessarily $v_j = v_i^{-1}$, thus
we define $\sigma(i) = j).$ {\it The involution $\sigma$ has no fixed point.} Namely,
$v^{-1} \neq v$ for any word $v$: in case $v$ has odd length, just consider the middle
letter, it cannot be both direct and inverse; in case $v$ has even length, say
$v= l_1\cdots l_{2t}$ with letters $l_i$, then $l_{t+1} = l_t^{-1}$, which is excluded.
This shows that $m = 2n$ is even.
     
Let us look at the behavior of the compositions $u_iv_i$ and $v_iu_{i+1}$ and 
compare this with the compositions $u_{\sigma(i)}v_{\sigma(i)}$ and 
$v_{\sigma(i)}u_{\sigma(i)+1}$, taking into account that 
$v_{\sigma(i)}= v_i^{-1}$. The composition $u_iv_i$ occurs at the same 3-vertex as
$v_{\sigma(i)}u_{\sigma(i)+1}$, thus precisely one of the pairs $(u_i,v_i)$ and 
$(v_{\sigma(i)}, u_{\sigma(i)+1})$ is attracting.
Similarly, precisely one of the pairs $(v_i,u_{i+1})$ and $(u_{\sigma(i)},v_{\sigma(i)})$
is attracting.

   \medskip

Claim: {\it If one of the words $v_i$ is not serial, then $n=1.$}
Let us assume that $n \ge 2$ and that one of the words $v_i$ is not serial, 
say $v_i = xy$ such that the pair ($x,y)$ is attracting.

Two different cases have to be considered. The first case is the following:
$$
\hbox{\beginpicture
\setcoordinatesystem units <.7cm,1cm>
\plot 0 0  12 0 /
\plot 1 0.1 1 -0.1 /
\plot 3 0.1 3 -0.1 /
\plot 4 0.1 4 -0.1 /
\plot 6 0.1 6 -0.1 /
\plot 7 0.1 7 -0.1 /
\plot 9 0.1 9 -0.1 /
\plot 10 0.1 10 -0.1 /
\plot 12 0.1 12 -0.1 /
\put{$v$} at 2 0.3
\put{$v^{-1}$} at 5 0.37
\put{$xy$} at 8 0.3
\put{$y^{-1}x^{-1}$} at 11 0.37
\put{$w_1$} at 0.5 0.3
\put{$w_2$} at 3.5 0.3
\put{$w_3$} at 6.5 0.3
\put{$w_4$} at 9.5 0.3
\endpicture}
$$
where all the $w_i$ start and end with edges which occur only once.

We can assume that the pair $(w_1,v)$ is not attracting.  
Otherwise, $(v^{-1},w_3)$ is not attracting, and we replace the given word $w$ 
first by $w^{-1}$ and then by a rotated
one in order to obtain a similar situation.

We claim that we can replace $w_1$ at the beginning by the word by 
$w_3^{-1}$, thus dealing with 
$$
 w_3^{-1}vw_2v^{-1}w_3xyw_4y^{-1}x^{-1}.
$$
This is a word. Namely, since the pair $(w_1,v)$ is not attracting, the pair
$(v^{-1},w_3)$ is attracting, thus the same is true for the inverse.
Also, it is a cyclic word, since the pair $(w_3x,y)$, thus also 
$(y^{-1},x^{-1}w_3)$ is attracting. 

It remains to observe that this new cyclic word uses less
arrows: all the edges of $w_1$ which have multiplicity 1 in $w$ 
have disappeared. This contradicts that $Q$
is minimal representation-infinite. 
   \bigskip

Let us now discuss the second case, where $xy$ lies in between $v$ and $v^{-1}$. 
$$
\hbox{\beginpicture
\setcoordinatesystem units <.7cm,1cm>
\plot 0 0  12 0 /
\plot 1 0.1 1 -0.1 /
\plot 3 0.1 3 -0.1 /
\plot 4 0.1 4 -0.1 /
\plot 6 0.1 6 -0.1 /
\plot 7 0.1 7 -0.1 /
\plot 9 0.1 9 -0.1 /
\plot 10 0.1 10 -0.1 /
\plot 12 0.1 12 -0.1 /
\put{$v$} at 2 0.3
\put{$xy$} at 5 0.3
\put{$v^{-1}$} at 8 0.37
\put{$y^{-1}x^{-1}$} at 11 0.37
\put{$w_1$} at 0.5 0.3
\put{$w_2$} at 3.5 0.3
\put{$w_3$} at 6.5 0.3
\put{$w_4$} at 9.5 0.3
\endpicture}
$$
where again all the $w_i$ start and end with edges which occur only once.
Again, we can assume that the pair $(w_1,v)$ is not attracting (otherwise, the
pair $(v^{-1},w_4)$ is not attracting and we invert and rotate $w$).

This time, we claim that
$$
 vw_2x|yw_4^{-1}
$$
is a cyclic word. 
The first part is a subword of $w$, the inverse of the second part is also a
subword of $w$. Thus, both words $vw_2x$ and $yw_4^{-1}$ exist and they can be composed,
since $(x,y)$ is an attracting pair. Also, it is a cyclic word, since $(w_4^{-1},v)$
is an attracting pair:  By assumption, $(w_1,v)$ is not attracting, thus
$(v^{-1},w_4)$ is an attracting pair, and therefore the same is true for its inverse
$(w_4^{-1},v).$
	
The case $n=1$ with $v = v_1$ non-serial yields a barbell.
    \bigskip

Thus, we now assume that all the words $v_i$ are serial. 

Claim: 
($*$) 
{\it If $(u_1,v_1)$ is not attracting, then also $(u_{\sigma(1)},v_{\sigma(1)})$ is not attracting, but $(v_1,u_2)$ is attracting.}

Let $s = \sigma(1),$ and assume that $(u_1,v_1)$ is not attracting, but $(u_s,v_s)$ is attracting. 
Let $w_2 = u_2v_2\cdots v_{s-1}u_s$ and $w' = u_{s+1}v_{s+1}\cdots
u_mv_m$, thus $w = u_1v_1w_2v_s w'$ and $v_s = v_1^{-1}$. We claim that 
$$
 w'' = u_1v_1w_2^{-1}v_sw'
$$
is a cyclic word. Of course, $u_1v_1$ is a subword of $w$.
Also, $w_2^{-1}v_s = w_2^{-1}v_1^{-1} = (v_1w_2)^{-1}$
is the inverse of a subword of $w$. Since we assume that $(w_2,v_s) = (w_2,v_1^{-1})$ is 
attracting, also the inverse $(v_1,w_2^{-1})$ is attracting, and thus also
$(u_1v_1,w_2^{-1}v_s)$ is attracting. 
Since $(u_1,v_1)$ is not attracting, the pair $(v_s,u_{s+1}) = (v_1^{-1},u_{s+1})$ is attracting. This shows that 
the concatenation of $v_s$ with $w'$ or even with $w'u_1v_1$ provides no problem.
Therefore we see that $w''$ is a cyclic word.

Note that $w$ has the serial word $z = u_{1,\omega}v_1u_{2,1}$ as a subword, whereas this is
no longer a subword of $w''$. This means that we can use $z$ as an additional relation
and still have the cyclic word $w''$. 
This contradicts the assumption that we deal with a minimal representation-infinite
algebra.
This contraction shows that $(u_{\sigma(1)},v_{\sigma(1)})$ is not attracting, 
but then $(v_1,u_2)$ has to be attracting.
 
Claim: {\it If $(u_1,v_1)$ is not attracting, then also $(u_2,v_2)$ is not attracting.}
Thus, let us assume that $(u_2,v_2)$ is attracting, whereas $(u_1,v_1)$ is not.
We already have seen in ($*$) that $(u_{\sigma(1)},v_{\sigma(1)})$ is not attracting, thus
$\sigma(1) \neq 2.$
	               
We know from ($*$) that $(v_1,u_2)$ is attracting.
By assumption, also $(u_2,v_2)$ is attracting. 
 
Consider the case that $\sigma(2) < \sigma(1).$
$$
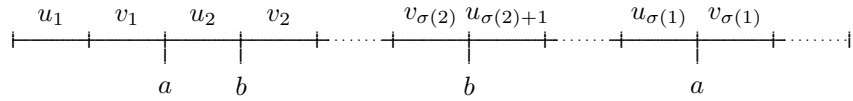
\hbox{\beginpicture
\setcoordinatesystem units <.5cm,1cm>
\plot 0 0  8.3 0 /
\plot 9.7 0  14.3 0 /
\plot 15.7 0  20.3 0 /
\plot 0 0.1 0 -0.1 /
\plot 2 0.1 2 -0.1 /
\plot 4 0.1 4 -0.3 /
\plot 6 0.1 6 -0.3 /
\plot 8 0.1 8 -0.1 /
\plot 10 0.1 10 -0.1 /
\plot 12 0.1 12 -0.3 /
\plot 14 0.1 14 -0.1 /
\plot 16 0.1 16 -0.1 /
\plot 18 0.1 18 -0.3 /
\plot 20 0.1 20 -0.1 /
\plot 22 0.1 22 -0.1 /

\setdots <1mm>
\plot 8.3 0 9.7 0 /
\plot 14.3 0  15.7 0 /
\plot 20.3 0  22 0 /
\put{$u_1$} at 1 0.3
\put{$v_1$} at 3 0.3
\put{$u_2$} at 5 0.3
\put{$v_2$} at 7 0.3
\put{$v_{\sigma(2)}$} at 11 0.3
\put{$u_{\sigma(2)+1}$} at 13 0.3
\put{$u_{\sigma(1)}$} at 17 0.3
\put{$v_{\sigma(1)}$} at 19 0.3 
\put{} at 22 0
\put{$a$} at 4 -0.63
\put{$b$} at 6 -.6
\put{$b$} at 12 -.6
\put{$a$} at 18 -.63
\endpicture}
$$
Thus we may replace
$u_{\sigma(2)+1}\cdots u_{\sigma(1)}$ by $u_2^{-1}$ and obtain
the cyclic word
$$
  u_1v_1\cdots v_{\sigma(2)}|u_2^{-1}|v_{\sigma(1)}\cdots v_m.
$$
But here we use less arrows: all the arrows in $u_{\sigma(2)+1}$ have disappeared.
    
Finally, we have to deal with the case that
$\sigma(1) < \sigma(2).$
$$
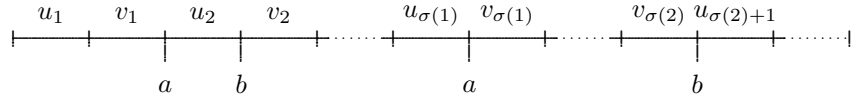
\hbox{\beginpicture
\setcoordinatesystem units <.5cm,1cm>
\plot 0 0  8.3 0 /
\plot 9.7 0  14.3 0 /
\plot 15.7 0  20.3 0 /
\plot 0 0.1 0 -0.1 /
\plot 2 0.1 2 -0.1 /
\plot 4 0.1 4 -0.3 /
\plot 6 0.1 6 -0.3 /
\plot 8 0.1 8 -0.1 /
\plot 10 0.1 10 -0.1 /
\plot 12 0.1 12 -0.3 /
\plot 14 0.1 14 -0.1 /
\plot 16 0.1 16 -0.1 /
\plot 18 0.1 18 -0.3 /
\plot 20 0.1 20 -0.1 /
\plot 22 0.1 22 -0.1 /

\setdots <1mm>
\plot 8.3 0 9.7 0 /
\plot 14.3 0  15.7 0 /
\plot 20.3 0  22 0 /
\put{$u_1$} at 1 0.3
\put{$v_1$} at 3 0.3
\put{$u_2$} at 5 0.3
\put{$v_2$} at 7 0.3
\put{$u_{\sigma(1)}$} at 11 0.3
\put{$v_{\sigma(1)}$} at 13 0.3
\put{$v_{\sigma(2)}$} at 17 0.3
\put{$u_{\sigma(2)+1}$} at 19 0.3 
\put{} at 22 0
\put{$a$} at 4 -0.63
\put{$b$} at 6 -.6
\put{$a$} at 12 -.63
\put{$b$} at 18 -.6
\endpicture}
$$
and take the cyclic word 
$$
 v_{\sigma(1)}\cdots v_{\sigma(2)}u_2^{-1}.
$$
This exists, since the pairs $(v_1,u_2) = (v_{\sigma(1)}^{-1},u_2)$ and
$(u_2,v_2) = (u_2,v_{\sigma(2)}^{-1})$ both are attracting. 
In this case, we lose for example all the arrows which occur in $u_1$.
   
Altogether we see that $w$ defines a wind wheel algebra.

\section{Proof of Theorem 1.3}

Let $\Lambda$ be special biserial, let $Q$ be its quiver. 
If $\Lambda$ is minimal representation-infinite,
then the ideal of relations is generated by a set $\rho$ of zero relations, thus the universal cover 
$\Lambda$ of $\Lambda$
is given by the universal cover $\widetilde Q$ of the quiver $Q$ and the set $\widetilde \rho$ of  lifted relations.
It follows that $\widetilde \Lambda$ is special biserial with a quiver which is a tree, and the indecomposable
$\widetilde \Lambda$-modules are string modules with support a quiver of type $\mathbb A_n$. It follows that
any finite convex subcategory is representation-finite.

Conversely, assume now that $\Lambda$ is a finite dimensional basic $k$-algebra which is
 minimal representation-infinite, that
the universal cover $\widetilde \Lambda$ of $\Lambda$ 
is good and that any finite convex subcategory of $\widetilde\Lambda$ is representation-finite.
We want to show that $\Lambda$ is special biserial. 

Let $Q$ be the quiver of $\Lambda$ and $\widetilde Q$ the quiver of $\widetilde \Lambda$. 
We denote by $\pi: \mo \widetilde\Lambda \to \mo \Lambda$  the covering functor (often called push-down
functor, or forgetful functor).

For any finite dimensional $k$-algebra $A$, let $s(A)$ be the number of (isomorphism classes of)
simple $A$-modules. 

Note that $\widetilde\Lambda$ is not of bounded representation type, since otherwise also $\Lambda$
would be of bounded, thus finite,  representation type. Thus we see that there are indecomposable
representations $M$ of $\widetilde\Lambda$ of arbitrarily large length. Given such a representation $M$ say of
length $m$, its support algebra $C$  is a representation-directed algebra, thus $s(C) \ge \frac m6$. 
This shows that there are indecomposable
representations $M$ of $\Lambda$ whose support has with arbitrarily large  cardinality. 

The sincere representation-directed algebras with large support have been classified
by Bongartz (see \cite{B1} or \cite{R1099}). Such an algebra $C$ has a convex subalgebra
$B$ which is given by a quiver of type $\mathbb A_n$ (without any relation), such that 
$s(B) \ge \frac13 s(C).$ Thus we see: $\widetilde\Lambda$ has convex subcategories $B$ which are given
by a quiver of type $\mathbb A_n$ without relations, where $n$ is arbitrarily large.

Let $p(\Lambda)$ the Loewy length of $\Lambda$. Choose a convex subcategory $B$ of $\widetilde\Lambda$ which 
is of the form $\mathbb A_n$ and without relations, where $n > 4p(\Lambda)s(\Lambda)$ and let $M$
the the (unique) since representation of $B$. Note that the indecomposable $B$-modules are string modules,
they are given by words using as letters the arrows of the quiver of $B$. Since $n > 4l(\Lambda)s(\Lambda)$,
it follows easily that there is a simple $\Lambda$-module $S$
such that $[\soc\pi(M:S)] \ge 3$.
But this implies that there is an indecomposable $B$-module $N$ which is given by a word of the
form $w= l_1l_1\cdots l_t$, with arrows $ l_1$ and $ l_t^{-1}$ such that $\pi(l_1)$ and
$\pi(l_t^{-1})$ are different arrows of $Q$ and  end in the same vertex of $Q$ (namely the support vertex
of $S$).

We denote by $A$ the support algebra of $N$, it is given by a quiver of type $\mathbb A_{t+1}$ without any relation,
and $\pi(N)$ has a one-parameter family of simple submodules $U$ such that the $\Lambda$-modules
$\pi(N)/U$ are indecomposable and pairwise non-isomorphic. Since $\Lambda$ is minimal representation-infinite,
it follows that $\pi(N)$ is faithful, thus we obtain all the vertices and all the arrows of $Q$ by applying
$\pi$ to the vertices and arrows of $A$, respectively. 

Since we now know that the vertices and the arrows of $Q$ are images under $\pi$ of vertices and arrows in
the support of $A$, it follows that the support of $A$ contains a fundamental domain for the action of
the Galois group on $\widetilde Q$. In particular, we see that we can compose  the word $w$ with a word
$w'$ such that $\pi(w) = \pi(w')$ and so on. 

Now assume there is a vertex $x$ of $Q$ with 3 arrows ending in $x$, say $\alpha, \beta, \gamma$. 
Looking at the universal cover, we obtain arrows $\tilde \alpha, \tilde\beta,\tilde \gamma$ ending in the
same vertex $\tilde x.$ But the arrow $\tilde\alpha$ is the first letter of a word $w'= l'_1l'_2l'_3$ of length $3$ 
which yields a string module for $\widetilde\Lambda$. Similarly,  the arrow $\tilde\beta$ is the first letter of a word 
$w'' = l''_1l''_2l''_3$ of length $3$ 
which also  yields a string module for $\widetilde\Lambda$. But the union of the support of the words $w', w''$ and
$\gamma$  is a subquiver (without any relation) of $\widetilde Q$ of type $\widetilde{\mathbb E}_7$.
This contradicts the assumption that any finite subcategory of $\widetilde \Lambda$ is representation-finite.

The dual argument shows that any vertex of $Q$ is starting point of at most two arrows.

Now assume that the vertex $x$ of $Q$ is endpoint of the arrows $\alpha \neq \beta$ and starting point
of the arrow $\gamma$ and that neither $\gamma\alpha$ nor $\gamma\beta$ is a zero relation.
We looking again at the universal cover, and obtain arrows $\tilde \alpha, \tilde\beta$ ending in a vertex $\tilde x$,
as well as $\tilde \gamma$ starting in the vertex $\tilde x$. Since $\gamma\alpha$ is not a zero relation, we see
that $\tilde\gamma\tilde\alpha$ must be a subword of $w$ or $w^{-1}$, in particular we can prolong it
to a word $w' = l'_1l'_2l'_3l'_4$ of length $4$ so that we have a corresponding string module $M(w')$ 
Similarly, we prolong $\tilde\gamma\tilde\beta$  to a word $w'' = l''_1l''_2l''_3l''_4$ of length $4$ with string module $M(w'')$.
We consider  the union of the support of the words $w', w''$; it
is a subquiver (without any relation) of $\widetilde Q$, again of type $\widetilde{\mathbb E}_7$, thus again we obtain
a contradiction to  the assumption that any finite subcategory of $\widetilde \Lambda$ is representation-finite.

\section{ Further examples.}
   
First, we present a second example of a  barbell algebras with non-serial bar.
       	          \medskip

{\bf Example 3.} 
Start with $\epsilon = (+),$ $\eta = (+-), \epsilon' = (+-+).$ Then $H(\epsilon\eta\epsilon'\eta^{-1})$ 
has the following shape:
$$
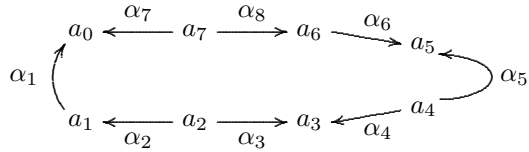
\hbox{\beginpicture
\setcoordinatesystem units <1.5cm,1.2cm>
\put{$a_0$} at 0 1
\put{$a_1$} at 0 0
\put{$a_2$} at 1 0
\put{$a_3$} at 2 0 
\put{$a_4$} at 3 0.15
\put{$a_5$} at 3 0.85
\put{$a_6$} at 2 1
\put{$a_7$} at 1 1
\arrow <1.5mm> [0.25,0.75] from -0.175 0.8 to -0.15 0.85

\arrow <1.5mm> [0.25,0.75] from 0.8 0 to 0.2 0
\arrow <1.5mm> [0.25,0.75] from 1.2 0 to 1.8 0

\arrow <1.5mm> [0.25,0.75] from 2.8 0.13 to 2.2 0
\arrow <1.5mm> [0.25,0.75] from 2.2 1 to 2.8 0.87
\arrow <1.5mm> [0.25,0.75] from 0.8 1 to 0.2 1
\arrow <1.5mm> [0.25,0.75] from 1.2 1 to 1.8 1
\arrow <1.5mm> [0.25,0.75] from 3.2 0.75 to 3.15 0.75
\circulararc -80 degrees from -0.15 0.15 center at .2 0.5 

\put{$\alpha_1$}  at -0.5 0.5
\put{$\alpha_2$}  at  0.5 -0.2
\put{$\alpha_3$}  at  1.5 -0.2
\put{$\alpha_4$}  at  2.6 -0.1
\put{$\alpha_5$}  at  3.8 0.5
\put{$\alpha_6$}  at  2.6 1.1
\put{$\alpha_7$}  at  0.5 1.2
\put{$\alpha_8$}  at  1.5 1.2
\ellipticalarc axes ratio 2.5:1 170 degrees from 3.15 0.25 center at 3.1 0.5 

\endpicture}
$$
and $B(\epsilon,\eta,\epsilon')$ will be
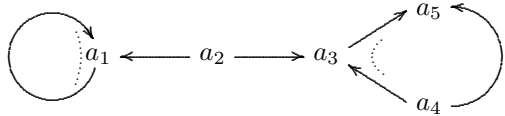
$$
\hbox{\beginpicture
\setcoordinatesystem units <1.5cm,1.2cm>
\put{$a_1$} at 0 0
\put{$a_2$} at 1 0
\put{$a_3$} at 2 0 
\put{$a_4$} at 2.9 -.55
\put{$a_5$} at 2.9 0.5

\arrow <1.5mm> [0.25,0.75] from 0.8 0 to 0.2 0
\arrow <1.5mm> [0.25,0.75] from 1.2 0 to 1.8 0

\arrow <1.5mm> [0.25,0.75] from 2.7 -0.5 to 2.2 -0.1
\arrow <1.5mm> [0.25,0.75] from 2.2 0.1 to 2.7 0.5
\ellipticalarc axes ratio 1:1 180 degrees from 3.1 -0.55 center at 3.1 0
\arrow <1.5mm> [0.25,0.75] from 3.15 0.545 to 3.1 0.55

\circulararc 320 degrees from -0.05 0.2 center at -0.4 0 
\arrow <1.5mm> [0.25,0.75] from -0.095 0.3 to -0.05 0.2
\multiput{} at -1 0.5  4 -0.5 /
\setdots <.7mm>
\setquadratic
\plot -0.2 -0.3 -0.15 0 -0.2 0.3 /
\plot  2.5 -0.2  2.4 0 2.5 0.2 /
\endpicture}
$$	
Again, the bar (given by the arrows $a_1 \leftarrow a_2 \rightarrow a_3)$ is not serial.
       \medskip

The next examples are wind wheel algebras.
    \medskip

{\bf Example 4.} The wind wheel algebra for the word $w$ 
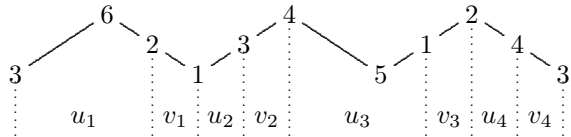
$$
\hbox{\beginpicture
\setcoordinatesystem units <.6cm,.4cm>
\put{$3$} at -0 0
\put{$4$} at -1 1 
\put{$2$} at -2 2
\put{$1$} at -3 1
\put{$5$} at -4 0 
\put{$4$} at -6 2
\put{$3$} at -7 1
\put{$1$} at -8 0
\put{$2$} at -9 1 
\put{$6$} at -10 2 
\put{$3$} at -12 0
\plot -0.3 0.3  -0.7 0.7 /
\plot -1.3 1.3  -1.7 1.7 /
\plot -2.3 1.7  -2.7 1.3 /
\plot -3.3 0.7  -3.7 0.3 /

\plot -4.3 0.3  -5.7 1.7 /
\plot -6.3 1.7  -6.7 1.3 /
\plot -7.3 0.7  -7.7 0.3 /

\plot -8.3 0.3  -8.7 0.7 /
\plot -9.3 1.3  -9.7 1.7 /
\plot -10.3 1.7  -11.7 0.3 /
\put{$v_4$} at  -.5 -1.5
\put{$u_4$} at  -1.5 -1.5  
\put{$v_3$} at  -2.5 -1.5
\put{$u_3$} at  -4.5  -1.5
\put{$v_2$} at  -6.5 -1.5
\put{$u_2$} at  -7.5 -1.5
\put{$v_1$} at  -8.5 -1.5
\put{$u_1$} at  -10.5 -1.5
\setdots <1mm>
\plot  -0 -2   -0 0 /
\plot  -1 -2   -1 1 /
\plot  -2 -2   -2 2 /
\plot  -3 -2   -3 1 /
\plot  -6 -2   -6 2 /
\plot  -7 -2   -7 1 /
\plot  -8 -2   -8 0 /
\plot  -9 -2   -9 1 /
\plot  -12 -2   -12 0 /

\endpicture}
$$
The permutation is $\sigma = (13)(24).$ 
The short zero relations are

\begin{align}
u_{1,\omega}u_{3+1,1}& = 6 - 2 - 4 \nonumber
\\
u_{2,\omega}u_{4+1,1}& = 1 - 3 - 6 \nonumber
\\
u_{3,\omega}u_{1+1,1}& = 5 - 1 - 3 \nonumber
\\
u_{4,\omega}u_{2+1,1}& = 2 - 4 - 5  \nonumber
\end{align}

The long zero relations are

\begin{align}
u_{1,\omega}v_1u_{3,\omega to }^{-1} &=  6 - 2 - 1 - 5 \nonumber
\\
u_{2,\omega}v_2u_{4,\omega to }^{-1} &=  1 - 3 - 4 - 2 \nonumber
\\
u_{3,\omega}v_3u_{1,\omega to }^{-1} &= (u_{1,\omega}v_1u_{3,\omega to }^{-1})^{-1} \nonumber
\\
 u_{4,\omega}v_4u_{2,\omega to }^{-1} &= (u_{2,\omega}v_2u_{4,\omega to }^{-1})^{-1} \nonumber
\end{align}

The corresponding wind wheel algebra $W = W(w)$ is:
$$
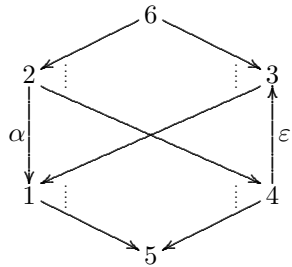
\hbox{\beginpicture
\setcoordinatesystem units <.8cm,.8cm>
\put{$1$} at 0 1
\put{$2$} at 0 3
\put{$3$} at 4 3
\put{$4$} at 4 1
\put{$5$} at 2 0
\put{$6$} at 2 4

\put{$\alpha$} at -0.2 2
\put{$\varepsilon$} at 4.2 2
\arrow <1.5mm> [0.25,0.75] from 0 2.8 to 0 1.2
\arrow <1.5mm> [0.25,0.75] from 4 1.2 to 4 2.8

\arrow <1.5mm> [0.25,0.75] from 0.2 0.9 to 1.8 0.1
\arrow <1.5mm> [0.25,0.75] from 3.8 0.9 to 2.2 0.1

\arrow <1.5mm> [0.25,0.75] from 1.8 3.9 to 0.2 3.1
\arrow <1.5mm> [0.25,0.75] from 2.2 3.9 to 3.8 3.1

\arrow <1.5mm> [0.25,0.75] from 3.8 2.8 to 0.2 1.2
\arrow <1.5mm> [0.25,0.75] from 0.2 2.8 to 3.8 1.2
\setdots<.7mm>
\plot 3.4 2.8  3.4 3.2 /
\plot 3.4 0.8  3.4 1.2 /
\plot 0.6 2.8  0.6 3.2 /
\plot 0.6 0.8  0.6 1.2 /

\endpicture}
$$
with the further relations $6 \to 2 \to 1 \to 5$ and $2 \to 4 \to 3 \to 1.$ There are two 
bars, they are given by the arrows $\alpha$ and $\epsilon$.
      \bigskip

In order to construct $W$, one starts with the following orientation sequence:
$$
\hbox{\beginpicture
\setcoordinatesystem units <.8cm,.5cm>
\put{$+\quad -$} at 1 0
\put{$-$} at 2.5 0
\put{$+$} at 3.5 0
\put{$+$} at 4.5 0
\put{$-\quad +$} at 6 0
\put{$+$} at 7.5 0
\put{$-$} at 8.5 0
\put{$-$} at 9.5 0

\put{$\epsilon_1$} at 1 -1
\put{$\eta_1$} at 2.5 -1
\put{$\epsilon_2$} at 3.5 -1
\put{$\eta_2$} at 4.5 -1
\put{$\epsilon_3$} at 6 -1
\put{$\eta_3$} at 7.5 -1
\put{$\epsilon_4$} at 8.5 -1
\put{$\eta_4$} at 9.5 -1
\setdots <1mm>
\plot 0 -1.5  0 .5 /
\plot 2 -1.5  2 .5 /
\plot 3 -1.5  3 .5 /
\plot 4 -1.5  4 .5 /
\plot 5 -1.5  5 .5 /
\plot 7 -1.5  7 .5 /
\plot 8 -1.5  8 .5 /
\plot 9 -1.5  9 .5 /
\plot 10 -1.5  10 .5 /
\put{} at 0 -0.7
\endpicture}
$$
and constructs the hereditary algebra $H(\epsilon_1\eta_1\cdots\epsilon_4\eta_4)$:
$$
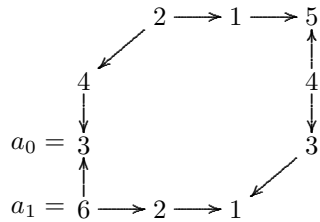
\hbox{\beginpicture
\setcoordinatesystem units <1cm,.85cm>
\put{$a_0 =$} at -0.6 1
\put{$a_1 =$} at -0.6 0
\put{$3$} at 0 1
\put{$6$} at 0 0
\put{$2$} at 1 0 
\put{$1$} at 2 0 
\put{$3$} at 3 1
\put{$4$} at 3 2
\put{$5$} at 3 3
\put{$1$} at 2 3
\put{$2$} at 1 3
\put{$4$} at 0 2
\arrow <1.5mm> [0.25,0.75] from 0.2 0 to 0.8 0
\arrow <1.5mm> [0.25,0.75] from 1.2 0 to 1.8 0
\arrow <1.5mm> [0.25,0.75] from 2.8 0.8 to 2.2 0.2
\arrow <1.5mm> [0.25,0.75] from 3 1.8 to 3 1.2
\arrow <1.5mm> [0.25,0.75] from 3 2.2 to 3 2.8
\arrow <1.5mm> [0.25,0.75] from 2.2 3 to 2.8 3
\arrow <1.5mm> [0.25,0.75] from 1.2 3 to 1.8 3
\arrow <1.5mm> [0.25,0.75] from 0.8 2.8 to 0.2 2.2
\arrow <1.5mm> [0.25,0.75] from 0 1.8 to 0 1.2
\arrow <1.5mm> [0.25,0.75] from 0 0.2 to 0 0.8
\endpicture}
$$

Thus, we start with a quiver which can be drawn either as a zigzag (with arrows pointing
downwards), where the left end and the right end have to be identified, or else
as a proper cycle:
$$
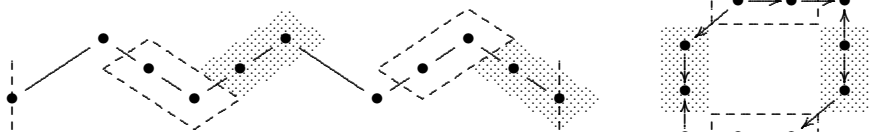
\hbox{\beginpicture
\setcoordinatesystem units <.7cm,.6cm>
\put{\beginpicture
\multiput{$\bullet$} at 0 1  0 0  1 0   2 0   3 1  3 2  3 3  2 3   1 3   0 2 /
\arrow <1.5mm> [0.25,0.75] from 0.2 0 to 0.8 0
\arrow <1.5mm> [0.25,0.75] from 1.2 0 to 1.8 0
\arrow <1.5mm> [0.25,0.75] from 2.8 0.8 to 2.2 0.2
\arrow <1.5mm> [0.25,0.75] from 3 1.8 to 3 1.2
\arrow <1.5mm> [0.25,0.75] from 3 2.2 to 3 2.8
\arrow <1.5mm> [0.25,0.75] from 2.2 3 to 2.8 3
\arrow <1.5mm> [0.25,0.75] from 1.2 3 to 1.8 3
\arrow <1.5mm> [0.25,0.75] from 0.8 2.8 to 0.2 2.2
\arrow <1.5mm> [0.25,0.75] from 0 1.8 to 0 1.2
\arrow <1.5mm> [0.25,0.75] from 0 0.2 to 0 0.8
\setdashes <1mm>
\plot 0.5 -0.5  0.5 0.5  2.5 0.5  2.5 -0.5  0.5 -0.5 /
\plot 0.5  2.5  0.5 3.5  2.5 3.5  2.5  2.5  0.5  2.5 /
\setshadegrid span <0.5mm>

\vshade -.5 0.5 2.5  .5 0.5 2.5 /
\vshade 2.5 0.5 2.5  3.5 0.5 2.5 /
\endpicture} at 9 0
\put{\beginpicture
\setcoordinatesystem units <.6cm,.4cm>
\multiput{$\bullet$} at  -0 0  -1 1   -2 2  -3 1  -4 0  -6 2  -7 1  -8 0  -9 1   -10 2  -12 0 /
\plot -0.3 0.3  -0.7 0.7 /
\plot -1.3 1.3  -1.7 1.7 /
\plot -2.3 1.7  -2.7 1.3 /
\plot -3.3 0.7  -3.7 0.3 /

\plot -4.3 0.3  -5.7 1.7 /
\plot -6.3 1.7  -6.7 1.3 /
\plot -7.3 0.7  -7.7 0.3 /

\plot -8.3 0.3  -8.7 0.7 /
\plot -9.3 1.3  -9.7 1.7 /
\plot -10.3 1.7  -11.7 0.3 /
\setdashes <1mm>
\plot -10 1  -9 2  -7 0  -8 -1  -10 1 /
\plot -1 2  -2 3  -4 1  -3 0  -1 2 /

\setshadegrid span <0.5mm>
\vshade -8 1 1 <z,z,,> -7 0 2 <z,z,,> -6 1 3 <z,z,,> -5 2 2 / 
\vshade -2 1 1  <z,z,,> -1 0 2 <z,z,,> 0 -1 1 <z,z,,> 1 0 0 /
\plot -12 -1 -12 1.5 /
\plot 0 -1 0 1.5 /

\endpicture} at 0 0
\endpicture}
$$
and we barify on the one hand  the two subquivers which are enclosed in rectangular boxes,
on the other hand also the two subquivers with shaded background.

In both cases, the barification yields an identification of a projective serial module of length 2
with an injective serial module of length 2. 
     \medskip

{\bf Example 5.} We start with the following orientation sequence: 
$$
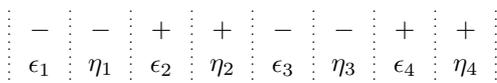
\hbox{\beginpicture
\setcoordinatesystem units <.8cm,.5cm>
\put{$-$} at 1.5 0
\put{$-$} at 2.5 0
\put{$+$} at 3.5 0
\put{$+$} at 4.5 0
\put{$-$} at 5.5 0
\put{$-$} at 6.5 0
\put{$+$} at 7.5 0
\put{$+$} at 8.5 0

\put{$\epsilon_1$} at 1.5 -1
\put{$\eta_1$} at 2.5 -1
\put{$\epsilon_2$} at 3.5 -1
\put{$\eta_2$} at 4.5 -1
\put{$\epsilon_3$} at 5.5 -1
\put{$\eta_3$} at 6.5 -1
\put{$\epsilon_4$} at 7.5 -1
\put{$\eta_4$} at 8.5 -1
\setdots <1mm>
\plot 1 -1.5  1 .5 /
\plot 2 -1.5  2 .5 /
\plot 3 -1.5  3 .5 /
\plot 4 -1.5  4 .5 /
\plot 5 -1.5  5 .5 /
\plot 6 -1.5  6 .5 /
\plot 7 -1.5  7 .5 /
\plot 8 -1.5  8 .5 /
\plot 9 -1.5  9 .5 /
\put{} at 0 -0.7
\endpicture}
$$
and construct the hereditary algebra $H(\epsilon_1\eta_1\cdots\epsilon_4\eta_4)$:
$$
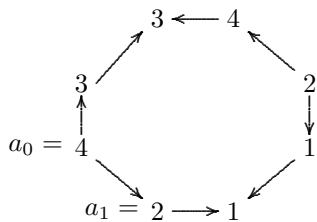
\hbox{\beginpicture
\setcoordinatesystem units <1cm,.85cm>
\put{$a_0 =$} at -0.6 1
\put{$a_1 =$} at 0.4 0
\put{$4$} at 0 1
\put{$2$} at 1 0 
\put{$1$} at 2 0 
\put{$1$} at 3 1
\put{$2$} at 3 2
\put{$4$} at 2 3
\put{$3$} at 1 3
\put{$3$} at 0 2
\arrow <1.5mm> [0.25,0.75] from 0.2 0.8 to 0.8 0.2
\arrow <1.5mm> [0.25,0.75] from 1.2 0 to 1.8 0
\arrow <1.5mm> [0.25,0.75] from 2.8 .8 to 2.2 0.2
\arrow <1.5mm> [0.25,0.75] from 3 1.8 to 3 1.2
\arrow <1.5mm> [0.25,0.75] from 2.8 2.2 to 2.2 2.8
\arrow <1.5mm> [0.25,0.75] from 1.8 3 to 1.2 3
\arrow <1.5mm> [0.25,0.75] from 0.2 2. to 0.8 2.8
\arrow <1.5mm> [0.25,0.75] from 0 1.2 to 0 1.8
\endpicture}
$$
The quiver which we obtain is
$$
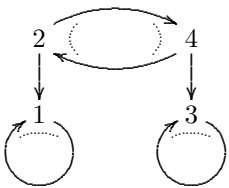
\hbox{\beginpicture
\setcoordinatesystem units <1cm,1cm>
\put{} at 0 1.5
\put{$1$} at 0 0
\put{$2$} at 0 1
\put{$4$} at 2 1
\put{$3$} at 2 0
\arrow <1.5mm> [0.25,0.75] from 0 0.8 to 0 0.2
\arrow <1.5mm> [0.25,0.75] from 2 0.8 to 2 0.2
\circulararc -310 degrees from 0.2 -.1 center at 0 -.5 
\circulararc -310 degrees from 2.2 -.1 center at 2 -.5 
\arrow <1.5mm> [0.25,0.75] from -0.25 -.14 to -0.2 -.1
\arrow <1.5mm> [0.25,0.75] from 1.75 -.14 to 1.8 -.1

\setquadratic
\plot 0.2 1.2  1 1.4  1.8 1.2 /
\plot 0.2 0.8  1 0.6  1.8 0.8 /

\arrow <1.5mm> [0.25,0.75] from 1.75 1.22 to 1.8 1.2
\arrow <1.5mm> [0.25,0.75] from 0.25 0.78 to 0.2 0.8

\setdots <.5mm>
\setquadratic
\plot -0.25 -.3 0 -.25  .25 -.3 /
\plot  1.75 -.3 2 -.25 2.25 -.3 /
\plot 0.5 0.8  0.4 1  0.5 1.2 /
\plot 1.5 0.8  1.6 1  1.5 1.2 /
\put{} at 0 -1 
\endpicture}
$$
with the additional relations: $4 \to 2 \to 1 \to 1$ and $2 \to 4 \to 3 \to 3.$
The bars are given by the arrows $2 \to 1$ and $4 \to 3.$
    \medskip

{\bf Example 6.}
As in example 4, consider again
$$
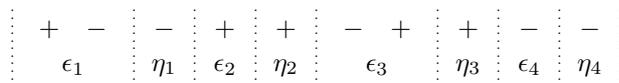
\hbox{\beginpicture
\setcoordinatesystem units <.8cm,.5cm>
\put{$+\quad -$} at 1 0
\put{$-$} at 2.5 0
\put{$+$} at 3.5 0
\put{$+$} at 4.5 0
\put{$-\quad +$} at 6 0
\put{$+$} at 7.5 0
\put{$-$} at 8.5 0
\put{$-$} at 9.5 0

\put{$\epsilon_1$} at 1 -1
\put{$\eta_1$} at 2.5 -1
\put{$\epsilon_2$} at 3.5 -1
\put{$\eta_2$} at 4.5 -1
\put{$\epsilon_3$} at 6 -1
\put{$\eta_3$} at 7.5 -1
\put{$\epsilon_4$} at 8.5 -1
\put{$\eta_4$} at 9.5 -1
\setdots <1mm>
\plot 0 -1.5  0 .5 /
\plot 2 -1.5  2 .5 /
\plot 3 -1.5  3 .5 /
\plot 4 -1.5  4 .5 /
\plot 5 -1.5  5 .5 /
\plot 7 -1.5  7 .5 /
\plot 8 -1.5  8 .5 /
\plot 9 -1.5  9 .5 /
\plot 10 -1.5  10 .5 /
\put{} at 0 -0.7
\endpicture}
$$
But now we take $\sigma = (12)(34)$
$$
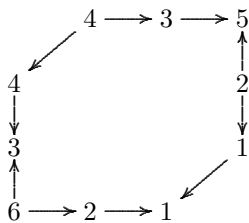
\hbox{\beginpicture
\setcoordinatesystem units <1cm,.85cm>
\put{$3$} at 0 1
\put{$6$} at 0 0
\put{$2$} at 1 0 
\put{$1$} at 2 0 
\put{$1$} at 3 1
\put{$2$} at 3 2
\put{$5$} at 3 3
\put{$3$} at 2 3
\put{$4$} at 1 3
\put{$4$} at 0 2
\arrow <1.5mm> [0.25,0.75] from 0.2 0 to 0.8 0
\arrow <1.5mm> [0.25,0.75] from 1.2 0 to 1.8 0
\arrow <1.5mm> [0.25,0.75] from 2.8 0.8 to 2.2 0.2
\arrow <1.5mm> [0.25,0.75] from 3 1.8 to 3 1.2
\arrow <1.5mm> [0.25,0.75] from 3 2.2 to 3 2.8
\arrow <1.5mm> [0.25,0.75] from 2.2 3 to 2.8 3
\arrow <1.5mm> [0.25,0.75] from 1.2 3 to 1.8 3
\arrow <1.5mm> [0.25,0.75] from 0.8 2.8 to 0.2 2.2
\arrow <1.5mm> [0.25,0.75] from 0 1.8 to 0 1.2
\arrow <1.5mm> [0.25,0.75] from 0 0.2 to 0 0.8
\endpicture}
$$
Here is the quiver:
$$
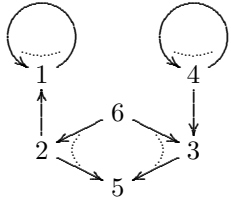
\hbox{\beginpicture
\setcoordinatesystem units <1cm,1cm>
\put{} at 0 2
\put{$2$} at 0 0
\put{$1$} at 0 1
\put{$4$} at 2 1
\put{$3$} at 2 0
\put{$5$} at 1 -0.5
\put{$6$} at 1 0.5
\arrow <1.5mm> [0.25,0.75] from 0 0.2 to 0 0.8
\arrow <1.5mm> [0.25,0.75] from 2 0.8 to 2 0.2
\arrow <1.5mm> [0.25,0.75] from 0.8 0.4 to 0.2 0.1
\arrow <1.5mm> [0.25,0.75] from 0.2 -.1 to 0.8 -.4
\arrow <1.5mm> [0.25,0.75] from 1.2 0.4 to 1.8 0.1
\arrow <1.5mm> [0.25,0.75] from 1.8 -.1 to 1.2 -.4
\circulararc 310 degrees from 0.2 1.1 center at 0 1.5 
\circulararc 310 degrees from 2.2 1.1 center at 2 1.5 
\arrow <1.5mm> [0.25,0.75] from -0.25 1.14 to -0.2 1.1
\arrow <1.5mm> [0.25,0.75] from 1.75 1.14 to 1.8 1.1
\setdots <.5mm>
\setquadratic
\plot -0.25 1.3 0 1.25 .25 1.3 /
\plot 0.5 -0.2  0.4 0  0.5 0.2 /
\plot 1.75 1.3 2 1.25 2.25 1.3 /
\plot 1.5 -0.2  1.6 0  1.5 0.2 /
\endpicture}
$$
with the additional relations $6 \to 2 \to 1 \to 1$ and $5 \leftarrow 3 \leftarrow 4
\leftarrow 4$. The bars are again the arrows $2 \to 1$ and $4 \to 3$.
	               
\bigskip\bigskip\bigskip

\centerline{\large\bf Part II. The module categories}

\section{The  cycle algebras}

We consider the algebras $H = H(\epsilon)$ where 
$\epsilon$ is not constant, so that $H$ is finite dimensional.

\subsection{The Auslander-Reiten quiver}
The structure of the Auslander-Reiten quiver of $H$ is well-known: there is the
preprojective and the preinjective component, the remaining components are regular
tubes. 

The string modules form four components of the Auslander-Reiten quiver, namely 
 the preprojective component, the preinjective component, and two tubes (we call them
{\it string tubes}); the remaining components (the {\it band tubes})
are homogeneous. Note that the string tubes may be
homogeneous or exceptional!

The Auslander-Reiten quiver of $H$ looks as follows:
$$
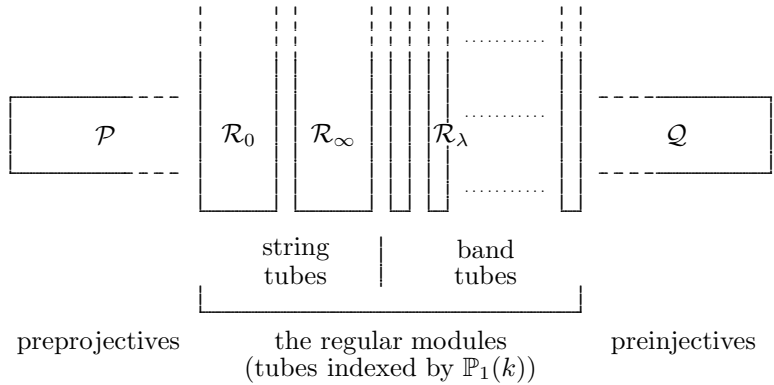
\hbox{\beginpicture
\setcoordinatesystem units <.5cm,.5cm>
\multiput{} at 0 0  20 5.5 /
\plot 3 3  0 3  0 1  3 1  /
\plot 5 4  5 0  7 0  7 4 /
\plot 7.5 4  7.5 0  9.5 0  9.5 4 /
\plot 10 4  10 0  10.5 0  10.5 4 /
\plot 11 4  11 0  11.5 0  11.5 1.7 /
\plot 11.5 2.35  11.5 4 /
\plot 14.5 4  14.5 0  15 0  15 4 /

\plot 17 1  20 1  20 3  17 3 /

\setdashes <1mm>
\plot 3 3  4.5 3 /
\plot 3 1  4.5 1 /
\plot 5 4  5 5.5 /
\plot 7 4  7 5.5 /
\plot 7.5 4  7.5 5.5 /
\plot 9.5 4  9.5 5.5 /
\plot 10 4  10 5.5 /
\plot 10.5 4  10.5 5.5 /
\plot 11 4  11 5.5 /
\plot 11.5 4  11.5 5.5 /
\plot 14.5 4  14.5 5.5 /
\plot 15 4  15 5.5 /
\plot 15.5 1  17 1 /
\plot 15.5 3  17 3 /

\setdots <1mm>
\plot 12 0.5 14 0.5 /
\plot 12 2.5 14 2.5 /
\plot 12 4.5 14 4.5 /
\setsolid
\put{preprojectives} at 2.3 -3.5
\put{preinjectives} at 17.7 -3.5
\plot 5 -2 5 -2.7  15 -2.7 15 -2 /
\put{the regular modules} at 10 -3.5
\put{string} at 7.5 -1
\put{tubes} at 7.5 -1.7
\plot 9.75 -0.7  9.75 -2 /
\put{band} at 12.5 -1
\put{tubes} at 12.5 -1.7
\put{(tubes indexed by $\mathbb P_1(k)$)} at 10 -4.2
\put{$\mathcal P$} at 2.5 2
\put{$\mathcal R_0$} at 6 2 
\put{$\mathcal R_\infty$} at 8.5 2 
\put{$\mathcal R_\lambda$} at 11.6 2 
\put{$\mathcal Q$} at 17.5 2
\endpicture}
$$

One should stress that in this case all the Auslander-Reiten components	
are (considered as simplicial complexes, thus as topological spaces) 
surfaces with boundary; all are homeomorphic to $[0,\infty[\times S^1$.

\subsection{An example} 
We consider $\epsilon = (++++----)$ and depict here, for later reference, the four
components which contain string modules, always we draw two fundamental
domains  inside the universal cover of the component, they are separated by 
dashed lines (horizontal ones for the preprojective and the preinjective component,
vertical ones for the regular components). Note that instead 
of arrows, we show edges, the orientation is from left to right.
$$ 
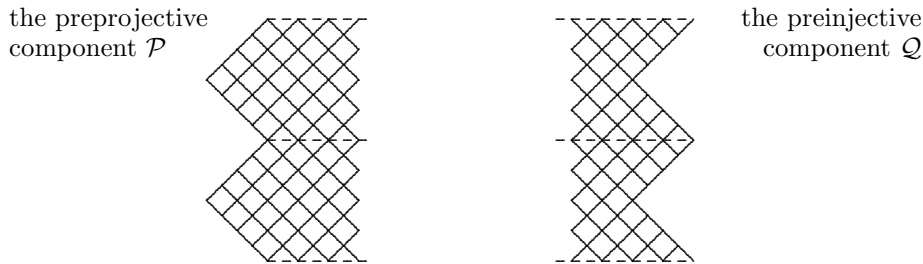
\hbox{\beginpicture
\setcoordinatesystem units <.2cm,.2cm>
\put{the preprojective} [l] at  -15 28
\put{component $\mathcal P$} [l] at -15 26
\put{the preinjective} [r] at 45 28
\put{component $\mathcal Q$} [r] at 45 26

\put{\beginpicture
\multiput{} at 0 0  10 16 /
\plot 0 12 4 16 /
\plot 1 11 6 16 /
\plot 2 10 8 16 /
\plot 3 9  10 16 /
\plot 0 4  10 14 /
\plot 1 3  10 12 /
\plot 2 2  10 10 /
\plot 3 1  10 8 /
\plot 4 0  10 6 /
\plot 6 0  10 4 /
\plot 8 0  10 2 /

\plot 0 4  4 0 /
\plot 1 5  6 0 /
\plot 2 6  8 0 /
\plot 3 7  10 0 /
\plot 0 12 10 2 /
\plot 1 13 10 4 /
\plot 2 14 10 6 /
\plot 3 15 10 8 /
\plot 4 16 10 10 /
\plot 6 16 10 12 /
\plot 8 16 10 14 /
\setdashes <1mm>
\plot 4 0  11 0 /
\plot 4 8  11 8 /
\plot 4 16 11 16 /

\endpicture} at 3 20

\put{\beginpicture
\multiput{} at 0 0  10 16 /
\plot 6 0  7 1 /
\plot 4 0  6 2 /
\plot 2 0  5 3 /
\plot 0 0  8 8 /
\plot 0 2  7 9 /
\plot 0 4  6 10 /
\plot 0 6  5 11 /
\plot 0 8  8 16 /
\plot 0 10  6 16 /
\plot 0 12  4 16 /
\plot 0 14  2 16 /

\plot 6 16  7 15 /
\plot 4 16  6 14 /
\plot 2 16  5 13 /
\plot 0 16  8 8 /
\plot 0 14  7 7 /
\plot 0 12  6 6 /
\plot 0 10  5 5 /

\plot 0 8  8 0 /
\plot 0 6  6 0 /
\plot 0 4  4 0 /
\plot 0 2  2 0 /
\setdashes <1mm>
\plot -1 0 8 0 /
\plot -1 8 8 8 /
\plot -1 16 8 16 /

\endpicture} at 27 20 
\endpicture}
$$
the two exceptional tubes:
$$ 
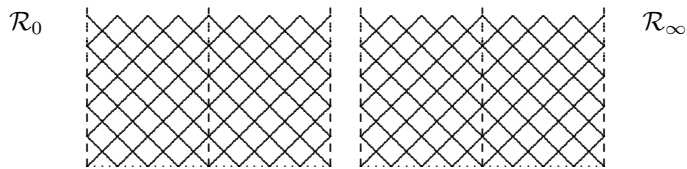
\hbox{\beginpicture
\setcoordinatesystem units <.2cm,.2cm>
\put{$\mathcal R_0$} at -12 5
\put{$\mathcal R_\infty$} at 30 5
\put{\beginpicture
\multiput{} at 0 0  16 9 /
\plot 14 0  16 2 /
\plot 12 0  16 4 /
\plot 10 0  16 6 /
\plot 8 0   16 8 /
\plot 6 0  16 10 /
\plot 4 0  14 10 /
\plot 2 0  12 10 /
\plot 0 0  10 10 /
\plot 0 2  8 10 /
\plot 0 4  6 10 /
\plot 0 6  4 10 /
\plot 0 8  2 10 /

\plot 0 2  2 0 /
\plot 0 4  4 0 /
\plot 0 6  6 0 /
\plot 0 8  8 0 /
\plot 0 10 10 0 /
\plot 2 10 12 0 /
\plot 4 10 14 0 /
\plot 6 10 16 0 /
\plot 8 10 16 2 /
\plot 10 10 16 4 /
\plot 12 10 16 6 /
\plot 14 10 16 8 /

\setdashes <1mm>
\plot 0 0  0 11 /
\plot 8 0  8 11 /
\plot 16 0  16 11 /
\setdots <1mm> 
\plot 0 0  16 0 /

\endpicture} at  0 0
\put{\beginpicture
\multiput{} at 0 0  16 9 /
\plot 14 0  16 2 /
\plot 12 0  16 4 /
\plot 10 0  16 6 /
\plot 8 0   16 8 /
\plot 6 0  16 10 /
\plot 4 0  14 10 /
\plot 2 0  12 10 /
\plot 0 0  10 10 /
\plot 0 2  8 10 /
\plot 0 4  6 10 /
\plot 0 6  4 10 /
\plot 0 8  2 10 /

\plot 0 2  2 0 /
\plot 0 4  4 0 /
\plot 0 6  6 0 /
\plot 0 8  8 0 /
\plot 0 10 10 0 /
\plot 2 10 12 0 /
\plot 4 10 14 0 /
\plot 6 10 16 0 /
\plot 8 10 16 2 /
\plot 10 10 16 4 /
\plot 12 10 16 6 /
\plot 14 10 16 8 /

\setdashes <1mm>
\plot 0 0  0 11 /
\plot 8 0  8 11 /
\plot 16 0  16 11 /

\setdots <1mm> 
\plot 0 0  16 0 / 

\endpicture} at  18 0
\endpicture}
$$

\subsection{The serial modules for a cycle algebra}

Recall that a module is called {\it serial} provided it has a unique composition series.
We consider the serial $H$-modules, where $H$ is a cycle algebra. The following assertions
are  easy to verify:
     
\begin{lem}
 Any serial $H$-module $M$ is projective or regular or injective.

If  $M$ is a serial $H$-module of length at least two, and 
 $M$ is projective, then $M/\soc$  belongs to $\mathcal R_0$ or $\mathcal R_\infty.$

If  $M$ is a serial $H$-module of length at least two, and 
 $M$ is injective, then $\rad M$ belongs to $\mathcal R_0$ or $\mathcal R_\infty.$
\end{lem}

In the barification process, we barify a projective serial module $M(b)$  of length at least $2$
with $M(b)/\soc$ say in $\mathcal R_0$  and an injective serial module $M(b')$ of the same length with $\rad M(b')$ 
in $\mathcal R_\infty$. 
   \medskip

Our convention for distinguishing $\mathcal R_0$ and $\mathcal R_\infty$  
will be the following:
given an indecomposable projective $H$-module $P$ with radical $\rad P = X\oplus X'$, where
$X$ and $X'$ are serial modules, we fix the order $X, X'$ and assume that the module
$P/X'$ as well as the composition factors of $X'/\soc$ are simple regular objects of
$\mathcal R_0$, whereas the module
$P/X$ as well as the composition factors of $X/\soc$ are simple regular objects of
$\mathcal R_\infty$:
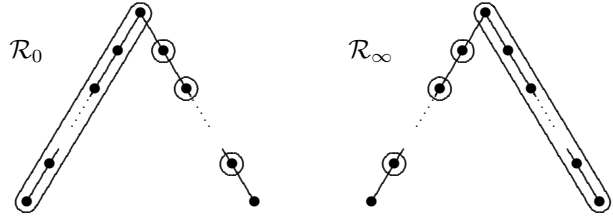
$$
\hbox{\beginpicture
\setcoordinatesystem units <.3cm,.5cm>
\put{\beginpicture
\multiput{$\bullet$} at 0 0 1 1  3 3  4 4  5 5  6 4  7 3    9 1  10 0 /
\plot 0 0  1.4 1.4 /
\plot 2.8 2.8  5 5 7.2 2.8 /
\plot 8.6 1.4  10 0 /
\circulararc 360 degrees from 9 1.3 center at 9 1
\circulararc 360 degrees from 6 4.3 center at 6 4
\circulararc 360 degrees from 7 3.3 center at 7 3
\plot 0.4 -.15  5.4 4.85 /
\plot -.4  .15  4.6 5.15 /
\circulararc -180 degrees from .4 -.15 center at 0 0
\circulararc  180 degrees from 5.4 4.85 center at 5 5
\setdots <1mm>
\plot 2 2 3 3 /
\plot 8 2 7 3 /
\put{$\mathcal R_0$} at 0 4
\endpicture} at 0 0

\put{\beginpicture
\multiput{$\bullet$} at 0 0 1 1  3 3  4 4  5 5  6 4  7 3    9 1  10 0 /
\plot 0 0  1.4 1.4 /
\plot 2.8 2.8  5 5 7.2 2.8 /
\plot 8.6 1.4  10 0 /
\circulararc 360 degrees from 1 1.3 center at 1 1
\circulararc 360 degrees from 3 3.3 center at 3 3
\circulararc 360 degrees from 4 4.3 center at 4 4
\plot 10.4 .15  5.4 5.15 /
\plot 9.6  -.15  4.6 4.85 /
\circulararc  180 degrees from  5.4 5.15 center at 5 5 
\circulararc  180 degrees from 9.6 -.15 center at 10 0
\setdots <1mm>
\plot 2 2 3 3 /
\plot 8 2 7 3 /
\put{$\mathcal R_\infty$} at 0 4

\endpicture} at 15 0
\endpicture}
$$

\section{ The wind wheel algebras.}

Let $H$ be hereditary of type $\widetilde {\mathbb A}$, and $W \subset H$ a corresponding wind wheel
algebra. Let us look at the restriction functor $ \eta:\mo H \to \mo W$.

First, let us recall the shape of the category $\mo H$. 
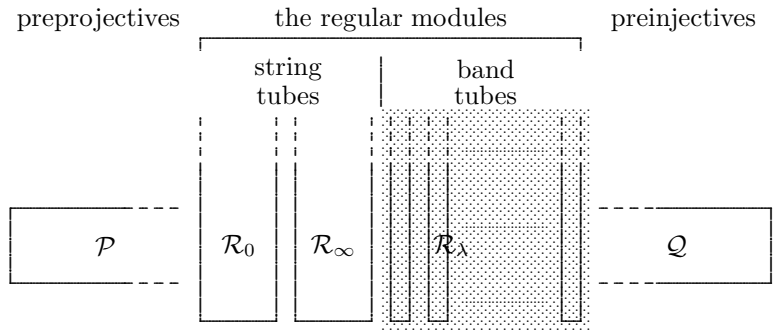
$$
\hbox{\beginpicture
\setcoordinatesystem units <.5cm,.5cm>
\multiput{} at 0 0  20 5.5 /
\plot 3 3  0 3  0 1  3 1  /
\plot 5 4  5 0  7 0  7 4 /
\plot 7.5 4  7.5 0  9.5 0  9.5 4 /
\plot 10 4  10 0  10.5 0  10.5 4 /
\plot 11 4  11 0  11.5 0  11.5 4 /
\plot 14.5 4  14.5 0  15 0  15 4 /

\plot 17 1  20 1  20 3  17 3 /

\setdashes <1mm>
\plot 3 3  4.5 3 /
\plot 3 1  4.5 1 /
\plot 5 4  5 5.5 /
\plot 7 4  7 5.5 /
\plot 7.5 4  7.5 5.5 /
\plot 9.5 4  9.5 5.5 /
\plot 10 4  10 5.5 /
\plot 10.5 4  10.5 5.5 /
\plot 11 4  11 5.5 /
\plot 11.5 4  11.5 5.5 /
\plot 14.5 4  14.5 5.5 /
\plot 15 4  15 5.5 /
\plot 15.5 1  17 1 /
\plot 15.5 3  17 3 /

\setdots <1mm>
\plot 12 0.5 14 0.5 /
\plot 12 2.5 14 2.5 /
\plot 12 4.5 14 4.5 /
\setsolid
\put{preprojectives} at 2.3 8
\put{preinjectives} at 17.7 8
\plot 5 7.2 5 7.5  15 7.5 15 7.2 /
\put{the regular modules} at 10 8
\put{string} at 7.3 6.7
\put{tubes} at 7.3 6
\plot 9.75 5.7  9.75 7 /
\put{band} at 12.5 6.7
\put{tubes} at 12.5 6
\put{$\mathcal P$} at 2.5 2
\put{$\mathcal R_0$} at 6 2 
\put{$\mathcal R_\infty$} at 8.5 2 
\put{$\mathcal R_\lambda$} at 11.6 2 
\put{$\mathcal Q$} at 17.5 2
\setshadegrid span <0.5mm>
\hshade -0.3 9.7 15.3  5.7 9.7 15.3 / 
\endpicture}
$$

We have shaded the homogeneous tubes: they  remain untouched; whereas 
the other four components are cut (between rays or corays) 
into pieces  and these pieces  are embedded (with some overlap) into a 
component  which contain in addition so-called quarters.  This cut-and-paste
process will now be explained. 

\subsection{Example} We consider the wind wheel algebra $W(w)$ for the word
$$
 w = \alpha\beta_1\beta_2\beta_3\gamma^{-1}\beta_3^{-1}\beta_2^{-1}\beta_1^{-1},
$$
thus we start with the  quiver $H = H(\epsilon(w))$
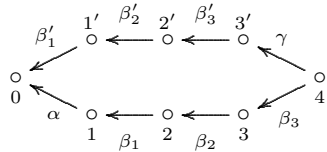
$$ 
\hbox{\beginpicture
\setcoordinatesystem units <1cm,.5cm>
\multiput{$\circ$} at 0 1  1 0  1 2  2 0  2 2  3 0  3 2  4 1 /
\arrow <1.5mm> [0.25,0.75] from 0.8 0.2 to 0.2 0.8
\arrow <1.5mm> [0.25,0.75] from 0.8 1.8 to 0.2 1.2
\arrow <1.5mm> [0.25,0.75] from 1.8 0 to 1.2 0
\arrow <1.5mm> [0.25,0.75] from 1.8 2 to 1.2 2
\arrow <1.5mm> [0.25,0.75] from 2.8 0 to 2.2 0
\arrow <1.5mm> [0.25,0.75] from 2.8 2 to 2.2 2
\arrow <1.5mm> [0.25,0.75] from 3.8 1.2 to 3.2 1.8
\arrow <1.5mm> [0.25,0.75] from 3.8 0.8 to 3.2 0.2
\put{$\ssize 0$} at 0 0.5
\put{$\ssize 1'$} at 1 2.5
\put{$\ssize 2'$} at 2 2.5
\put{$\ssize 3'$} at 3 2.5
\put{$\ssize 1$} at 1 -0.5
\put{$\ssize 2$} at 2 -0.5
\put{$\ssize 3$} at 3 -0.5
\put{$\ssize 4$} at 4 0.5
\put{$\ssize \alpha$} at 0.5 0
\put{$\ssize \beta_1'$} at 0.4 2.1
\put{$\ssize \beta_2'$} at 1.5 2.7
\put{$\ssize \beta_3'$} at 2.5 2.7
\put{$\ssize \beta_1$} at 1.5 -.7
\put{$\ssize \beta_2$} at 2.5 -.7
\put{$\ssize \beta_3$} at 3.6 -.1
\put{$\ssize \gamma$} at 3.5 2
\endpicture}
$$
and barify the subquivers $1 \leftarrow 2 \leftarrow 3 \leftarrow 4$ and $0 \leftarrow 1' \leftarrow 2' \leftarrow 3'$.

We obtain in this way the wind wheel algebra $W = W(w)$
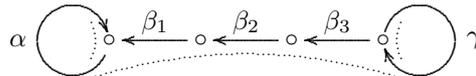
$$
\hbox{\beginpicture
\setcoordinatesystem units <0.6cm,0.6cm>
\put{} at -1 0.8
\put{} at  3 -0.8
\multiput{$\circ$} at 0 0  2 0  4 0  6 0 /
\arrow <1.5mm> [0.25,0.75] from 1.7 0 to 0.3 0
\arrow <1.5mm> [0.25,0.75] from 3.7 0 to 2.3 0
\arrow <1.5mm> [0.25,0.75] from 5.7 0 to 4.3 0
\circulararc 320 degrees from -0.05 0.2 center at -0.8 0 
\arrow <1.5mm> [0.25,0.75] from -0.09 0.3 to -0.05 0.2
\circulararc -330 degrees from  6.05 0.2 center at  6.8 0 
\arrow <1.5mm> [0.25,0.75] from 6.09 -0.3 to 6.05 -0.2
\put{$\alpha$} at -2 0
\put{$\beta_1$} at 1 0.4
\put{$\beta_2$} at 3 0.4
\put{$\beta_3$} at 5 0.4
\put{$\gamma$} at 8 0
\setdots <.7mm>
\setquadratic
\setdots <.7mm>
\setquadratic
\plot -0.4 -0.4 -0.3 0 -0.4 0.4 /
\plot 6.4 -0.4  6.3 0  6.4 0.4 /
\plot -0.3 -.8  3 -.3  6.3  -.8 / 
\endpicture}
$$
with $\alpha^2 = \gamma^2 = \alpha\beta_1\beta_2\beta_3\gamma = 0$.
     
Recall that  we know:
{\it There is precisely one  non-periodic (but biperiodic)
$\mathbb Z$-word, namely 
$$ r(b) = \ \ 
   {}^\infty (w'') \cdot \alpha^{-1} 
      \beta_1\beta_2\beta_3\gamma^{-1}\cdot (w')^{\infty}
$$
where $w' = 
\beta_3^{-1}\beta_2^{-1}\beta_1^{-1}\alpha\beta_1\beta_2\beta_3\gamma^{-1},$
and $w'' = \alpha^{-1}\beta_1\beta_2\beta_3\gamma\beta_3^{-1}\beta_2^{-1}\beta_1^{-1}.$
}

$$
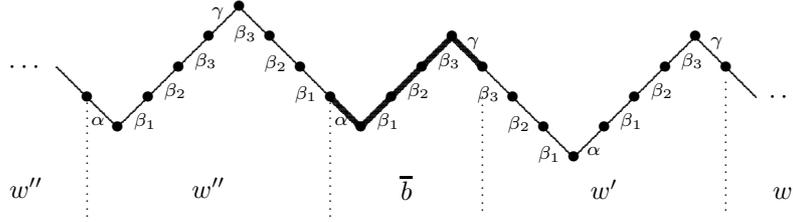
\hbox{\beginpicture
\setcoordinatesystem units <0.4cm,0.4cm>
\put{} at 0 -2.5
\multiput{$\bullet$} at 0 1  1 0  2 1  3 2  4 3  5 4  6 3  7 2  8 1  9 0
   10 1  11 2  12 3  13 2  14 1  15 0  16 -1  17 0  18 1  19 2  20 3  21 2 /
\plot -1 2  1 0  5 4  9 0  12 3  16 -1  20 3  22 1 /
\plot 8 1.05  9 0.05  12 3.05  13 2.05 /
\plot 8 1.1  9 0.1  12 3.1  13 2.1 /
\plot 8 .95  9 -.05  12 2.95  13 1.95 /
\plot 8 .9  9 -.1  12 2.9  13 1.9 /
\put{$\cdots$} at -2 2
\put{$\cdots$} at 23 1
\setdots <1mm>
\plot 0 .8  0 -3 /
\plot 8 .8  8 -3 /
\plot 13 1  13 -3 /
\plot 21 1.5  21 -3 /
\multiput{$w''$} at 4 -2  -2 -2 /
\put{$\overline b$} at 10.5 -2
\multiput{$w'$} at 17 -2.05  23 -2.05 /
\multiput{$\ssize \alpha$} at 0.35 0.2  8.35 0.2  16.65 -.8 /
\multiput{$\ssize \gamma$} at 4.3 3.8  12.7 2.8  20.7 2.8 /
\multiput{$\ssize \beta_1$} at 1.9 0.2  7.2 1.1  9.9 0.2  15.2 -.9  17.9 .2 /
\multiput{$\ssize \beta_2$} at 2.9 1.2  6.2 2.1  10.9 1.2  14.2 0.1  18.9 1.2 /
\multiput{$\ssize \beta_3$} at 3.9 2.2  5.2 3.1  11.9 2.2  13.2 1.1  19.9 2.2 /
\endpicture}
$$
with $b = \beta_1\beta_2\beta_3$ and $\overline b = \gamma^{-1}\beta_1\beta_2\beta_3\alpha^{-1}$
(the word $w'$ is obtained from $w$ by rotation, the word $w''$ by rotation and inversion). 

     \medskip

Here we see the reason why we call these algebras the wind wheels:
We consider the word $r(b)$ as a pair of opposite ''rotor blades''.

\subsection{Proposition} {\it The restriction functor 
$$
 \eta:\mo H \to \mo W
$$ 
has the following properties: }

(1) {\it Indecomposable modules are sent to indecomposable modules.}

(2) {\it Corresponding modules on the two $\mathbb A_4$-quivers which yield the bar become isomorphic,
otherwise non-isomorphy is preserved.}

(3) {\it The indecomposable $W$-modules which are not in
the image of the functor are the string modules for words which contain 
$\alpha^{-1}\beta_1\beta_2\beta_3\gamma^{-1}$ as a subword.} 
															        \medskip

Note that $\overline b = \alpha^{-1}\beta_1\beta_2\beta_3\gamma^{-1}$ is the closure
of the bar $b = \beta_1\beta_2\beta_3$, as defined in section 6.
   \medskip

In order to outline the cut-and-paste process, we start with the 
Auslander-Reiten components containing string modules, as shown above. 
We assume that $\mathcal R_0$ is the tube which contains at the boundary the simple $H$-modules
$1, 2, 3$ as well as the serial module with composition factors $0,1',2',3',4$, whereas 
$\mathcal R_\infty$ is the tube which contains at the boundary the simple $H$-modules
$1', 2', 3'$ and the serial module with composition factors $0,1,2,3,4$.

Let us look at the full subcategories of $\mo H$ with modules with support in $\{0,1,2,3\}$ on the one
hand (see the left picture) as well
as those with support in $\{1',2',3',4\}$ on the other hand (the right picture) and describe the role
of the various modules inside $\mo H$: 
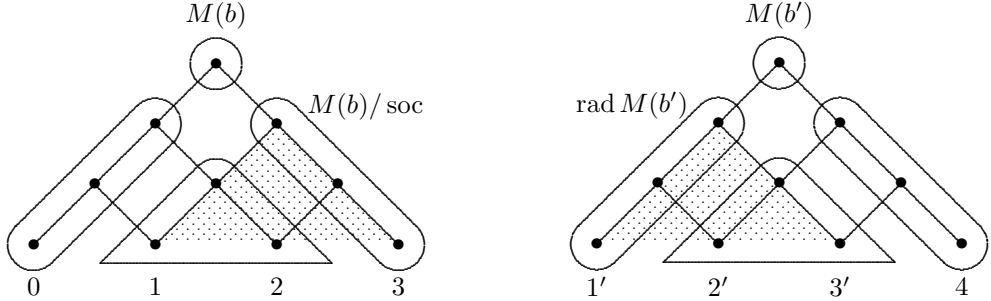
$$
\hbox{\beginpicture
\setcoordinatesystem units <.8cm,.8cm>
\put{\beginpicture
\multiput{$\bullet$} at 0 0  1 1  2 0  2 2  3 1  3 3  4 0  4 2  5 1  6 0 /
\plot 0 0  3 3  6 0 /
\plot 1 1  2 0  4 2 /
\plot 2 2  4 0  5 1 /
\setshadegrid span <0.7mm>
\hshade 0 2 6  2 4 4 /
\plot 0.3 -0.3  2.3 1.7 /
\plot -.3 0.3  1.7 2.3 /
\circulararc 180 degrees from 2.3 1.7 center at 2 2 
\circulararc 180 degrees from -.3 0.3 center at 0 0 

\plot 3.7 1.7  5.7  -.3 /
\plot 4.3 2.3  6.3 .3 /
\circulararc 180 degrees from 4.3 2.3 center at 4 2 
\circulararc 180 degrees from 5.7 -.3 center at 6 0 

\circulararc 360 degrees from 3.3 3.3 center at 3 3 

\plot 3.3 1.3  4.9 -.3  1.1 -.3  2.7 1.3 /
\circulararc  90 degrees from 3.3 1.3 center at 3 1
\put{$0$} at 0 -.7
\put{$1$} at 2 -.7
\put{$2$} at 4 -.7
\put{$3$}  at 6 -.7
\put{$M(b)$} at 3 3.8
\put{$M(b)/\soc$} [l] at 4.5 2.3

\endpicture} at 0 0
\put{\beginpicture
\multiput{$\bullet$} at 0 0  1 1  2 0  2 2  3 1  3 3  4 0  4 2  5 1  6 0 /
\plot 0 0  3 3  6 0 /
\plot 1 1  2 0  4 2 /
\plot 2 2  4 0  5 1 /
\setshadegrid span <0.7mm>
\hshade 0 0 4  2 2 2 /
\plot 0.3 -0.3  2.3 1.7 /
\plot -.3 0.3  1.7 2.3 /
\circulararc 180 degrees from 2.3 1.7 center at 2 2 
\circulararc 180 degrees from -.3 0.3 center at 0 0 

\plot 3.7 1.7  5.7  -.3 /
\plot 4.3 2.3  6.3 .3 /
\circulararc 180 degrees from 4.3 2.3 center at 4 2 
\circulararc 180 degrees from 5.7 -.3 center at 6 0 

\circulararc 360 degrees from 3.3 3.3 center at 3 3 

\plot 3.3 1.3  4.9 -.3  1.1 -.3  2.7 1.3 /
\circulararc  90 degrees from 3.3 1.3 center at 3 1
\put{$1'$} at 0 -.7
\put{$2'$} at 2 -.7
\put{$3'$} at 4 -.7
\put{$4$}  at 6 -.7
\put{$M(b')$} at 3 3.8
\put{$\rad M(b')$} [r] at 1.5 2.3

\endpicture} at 9 0
\endpicture}
$$
In the left picture, the shaded area marks those modules which belong to $\mathcal R_0$, whereas the
remaining modules (those which form the left boundary) are projective, thus in $\mathcal P.$
In the right picture,  the shaded area marks those modules which belong to $\mathcal R_\infty$, 
the remaining ones are injective, thus in $\mathcal Q.$
According to property (2), any module of the left triangle is identified under $\eta$ with the corresponding
module of the right triangle.

Let us look at the various components of $\mo H$ which contain string modules.
We will add bullets $\bullet$ in order to 
mark the position of the indecomposables with support contained either in $\{0,1,2,3\}$
or else in $\{1',2',3',4\}.$ We are going to  cut these components into suitable pieces:
 these are the dashed areas seen in the pictures. 

First, we exhibit the
preprojective component (left) and the preinjective component (right):
$$ 
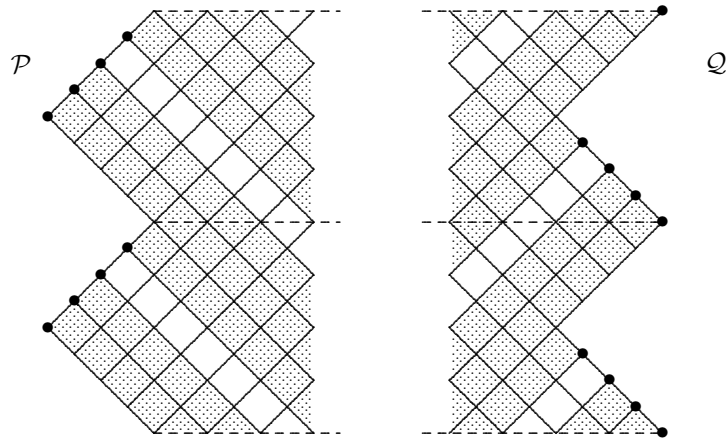
\hbox{\beginpicture
\setcoordinatesystem units <.35cm,.35cm>
\put{\beginpicture
\put{$\mathcal P$} at -1 14
\multiput{} at 0 0  10 16 /
\multiput{$\bullet$} at 0 4  1 5  2 6  3 7  0 12  1 13  2 14  3 15 /
\plot 0 12 4 16 /
\plot 1 11 6 16 /
\plot 2 10 8 16 /
\plot 3 9  10 16 /
\plot 0 4  10 14 /
\plot 1 3  10 12 /
\plot 2 2  10 10 /
\plot 3 1  10 8 /
\plot 4 0  10 6 /
\plot 6 0  10 4 /
\plot 8 0  10 2 /

\plot 0 4  4 0 /
\plot 1 5  6 0 /
\plot 2 6  8 0 /
\plot 3 7  10 0 /
\plot 0 12 10 2 /
\plot 1 13 10 4 /
\plot 2 14 10 6 /
\plot 3 15 10 8 /
\plot 4 16 10 10 /
\plot 6 16 10 12 /
\plot 8 16 10 14 /
\setdashes <1mm>
\plot 4 0  11 0 /
\plot 4 8  11 8 /
\plot 4 16 11 16 /
\setshadegrid span <0.5mm>
\hshade 0 4 8 <,,,z> 4 0 4 <,,z,> 6 2 2 /
\hshade 0 10 10 <,,,z> 6 4  10 <,,z,z> 7 3 9 <,,z,z> 8 4 8 <,,z,z> 12 0 4 <,,z,> 14 2 2 /
\hshade 8 10 10 <,,,z>  14 4 10 <,,z,z> 15  3 9 <,,z,> 16 4 8 /

\endpicture} at 0 20

\put{\beginpicture
\multiput{} at 0 0  10 16 /
\put{$\mathcal Q$} at 10 14
\multiput{$\bullet$} at 5 3  6 2  7 1  8 0  5 11  6 10  7 9  8 8  8 16  /
\plot 6 0  7 1 /
\plot 4 0  6 2 /
\plot 2 0  5 3 /
\plot 0 0  8 8 /
\plot 0 2  7 9 /
\plot 0 4  6 10 /
\plot 0 6  5 11 /
\plot 0 8  8 16 /
\plot 0 10  6 16 /
\plot 0 12  4 16 /
\plot 0 14  2 16 /

\plot 6 16  7 15 /
\plot 4 16  6 14 /
\plot 2 16  5 13 /
\plot 0 16  8 8 /
\plot 0 14  7 7 /
\plot 0 12  6 6 /
\plot 0 10  5 5 /

\plot 0 8  8 0 /
\plot 0 6  6 0 /
\plot 0 4  4 0 /
\plot 0 2  2 0 /
\setdashes <1mm>
\plot -1 0 8 0 /
\plot -1 8 8 8 /
\plot -1 16 8 16 /
\setshadegrid span <0.5mm>
\hshade 0 4 8  2 6 6 /
\hshade 0 0 2 <,,,z> 3 0 5 <,,z,z> 4 0 4 <,,z,z> 8 4 8 <,,z,> 10 6 6 /
\hshade 6 0 0 <,,,z> 11 0 5 <,,z,z> 12 0 4 <,,z,> 16 4 8 /
\hshade 14 0 0  16 0 2 /
\endpicture} at 16 20 
\endpicture}
$$
Next, the two regular components (the boundary of any of the two components contains three
simple modules; they are labeled):
$$ 
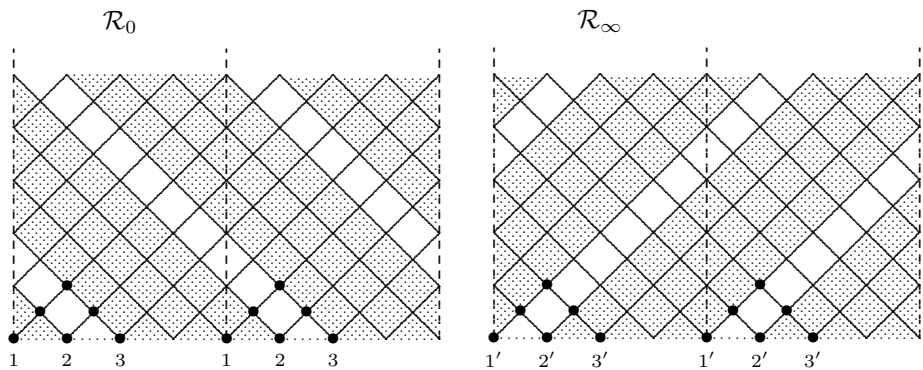
\hbox{\beginpicture
\setcoordinatesystem units <.35cm,.35cm>
\put{\beginpicture
\multiput{} at 0 0  16 11 /
\put{$\mathcal R_0$} at 4  12

\multiput{$\bullet$} at 0 0  1 1  2 0  2 2  3 1  4 0  8 0  9 1  10 0  10 2  11 1  12 0 /
\plot 14 0  16 2 /
\plot 12 0  16 4 /
\plot 10 0  16 6 /
\plot 8 0   16 8 /
\plot 6 0  16 10 /
\plot 4 0  14 10 /
\plot 2 0  12 10 /
\plot 0 0  10 10 /
\plot 0 2  8 10 /
\plot 0 4  6 10 /
\plot 0 6  4 10 /
\plot 0 8  2 10 /

\plot 0 2  2 0 /
\plot 0 4  4 0 /
\plot 0 6  6 0 /
\plot 0 8  8 0 /
\plot 0 10 10 0 /
\plot 2 10 12 0 /
\plot 4 10 14 0 /
\plot 6 10 16 0 /
\plot 8 10 16 2 /
\plot 10 10 16 4 /
\plot 12 10 16 6 /
\plot 14 10 16 8 /

\setdashes <1mm>
\plot 0 0  0 11 /
\plot 8 0  8 11 /
\plot 16 0  16 11 /
\setshadegrid span <0.5mm>
\hshade 0 0 2  2 0 0 /
\hshade 0 4 10 <,,,z> 4 0 6 <,,z,> 10 0 0 /
\hshade 0 12 16 <,,,z> 2 10 16 <,,z,z> 10 2 8 /
\hshade 4 16 16 10 10 16 /
\put{$\ssize 1$} at 0 -0.8
\put{$\ssize 2$} at 2 -0.8
\put{$\ssize 3$} at 4 -0.8
\put{$\ssize 1$} at 8 -0.8
\put{$\ssize 2$} at 10 -0.8
\put{$\ssize 3$} at 12 -0.8
\setdots <1mm> 
\plot 0 0  16 0 /

\endpicture} at  0 0
\put{\beginpicture
\multiput{} at 0 0  16 11 /
\put{$\mathcal R_\infty$} at 4  12
\multiput{$\bullet$} at 0 0  1 1  2 0  2 2  3 1  4 0  8 0  9 1  10 0  10 2  11 1  12 0 /
\plot 14 0  16 2 /
\plot 12 0  16 4 /
\plot 10 0  16 6 /
\plot 8 0   16 8 /
\plot 6 0  16 10 /
\plot 4 0  14 10 /
\plot 2 0  12 10 /
\plot 0 0  10 10 /
\plot 0 2  8 10 /
\plot 0 4  6 10 /
\plot 0 6  4 10 /
\plot 0 8  2 10 /

\plot 0 2  2 0 /
\plot 0 4  4 0 /
\plot 0 6  6 0 /
\plot 0 8  8 0 /
\plot 0 10 10 0 /
\plot 2 10 12 0 /
\plot 4 10 14 0 /
\plot 6 10 16 0 /
\plot 8 10 16 2 /
\plot 10 10 16 4 /
\plot 12 10 16 6 /
\plot 14 10 16 8 /

\setdashes <1mm>
\plot 0 0  0 11 /
\plot 8 0  8 11 /
\plot 16 0  16 11 /
\setshadegrid span <0.5mm>
\hshade 8 0 0  10 0 2 /
\hshade 0 0 0  <,,,z> 6 0 6 <,,z,> 10 4 10 /
\hshade 0 2 8  <,,,z> 8 10 16  <,,z,> 10 12 16 /
\hshade 0 10 16  6 16 16 /
\put{$\ssize 1'$} at 0 -0.8
\put{$\ssize 2'$} at 2 -0.8
\put{$\ssize 3'$} at 4 -0.8
\put{$\ssize 1'$} at 8 -0.8
\put{$\ssize 2'$} at 10 -0.8
\put{$\ssize 3'$} at 12 -0.8

\setdots <1mm> 
\plot 0 0  16 0 / 

\endpicture} at  18 0
\endpicture}
$$

As we have mentioned, under the restriction functor 
$ \eta:\mo H \to \mo W$, some serial $H$-modules become isomorphic,
namely, the  ten $H$-modules with support in in the subquiver with
vertices $\{0,1,2,3\}$ are identified with the corresponding ten $H$-modules 
 with support in in the subquiver with
vertices $\{1',2',3',4\}$. In a first step, we make the identification
of the nine pairs consisting of modules of length at most 3. 
We obtain the following partial translation quiver:
$$ 
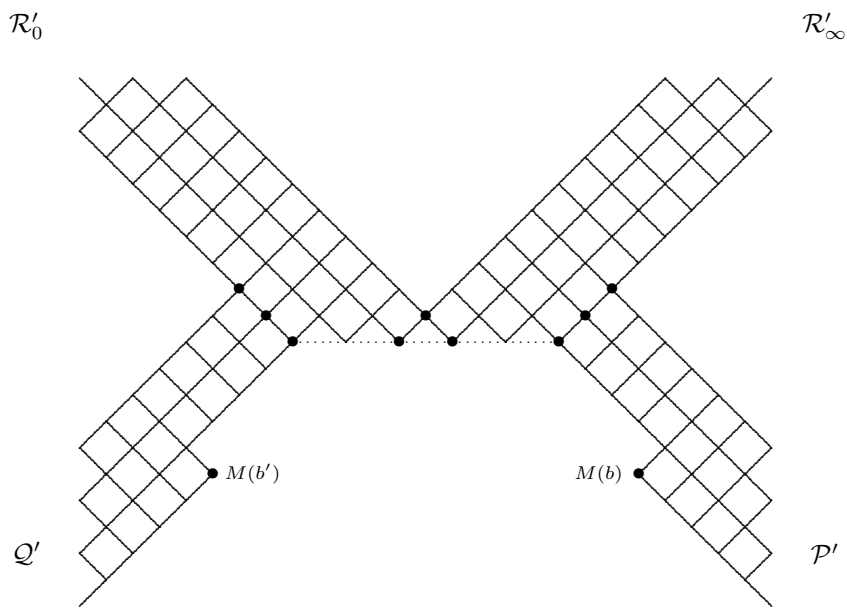
\hbox{\beginpicture
\setcoordinatesystem units <.35cm,.35cm>
\multiput{} at 0 0  26 20 /
\put{$\mathcal P'$} at 28 2
\put{$\mathcal Q'$} at -2  2
\put{$\mathcal R'_0$} at -2  22
\put{$\mathcal R'_\infty$} at 28  22

\multiput{$\bullet$} at 6 12  7 11  8 10  12 10  13 11  14 10  18 10  19 11  20 12  
   5 5  21 5 /
\plot 0 0  5 5  /
\plot 0 2  11 13 /
\plot 0 4  10 14 /
\plot 0 6  9 15 /

\plot 0 2  1 1 /
\plot 0 4  2 2 /
\plot 0 6  3 3 /
\plot 1 7  4 4 /
\plot 2 8  5 5  /
\plot 3 9  5 7 /
\plot 4 10  6 8 /
\plot 5 11  7 9 /
\plot 0 18  8 10 /
\plot 0 20  10 10 /
\plot 2 20  12 10 /
\plot 4 20  14 10 /
\plot 10 10  12 12 /
\plot 12 10  22 20 /
\plot 14 10  24 20 /
\plot 16 10  26 20 /
\plot 18 10 26 18 /

\plot 5 13  8 16 /
\plot 4 14  7 17 /
\plot 3 15  6 18 /
\plot 2 16  5 19 /
\plot 1 17  4 20 /
\plot 0 18  2 20 /
\plot 14 12  16 10 /
\plot 15 13  26 2 /
\plot 16 14  26 4 /
\plot 17 15  26 6 /
\plot 18 16  21 13 /
\plot 19 17  22 14 /
\plot 20 18  23 15 /
\plot 21 19  24 16 /
\plot 22 20  25 17 /
\plot 24 20  26 18 /
\plot 21 5  26 0 /
\plot 19 9  21 11 /
\plot 20 8  22 10 /
\plot 21 7  23 9 /
\plot 21 5  24 8 /
\plot 22 4  25 7 /
\plot 23 3  26 6 /
\plot 24 2  26 4 /
\plot 25 1  26 2 /
\setdots <1mm>
\plot 8 10  18 10 /
\put{$\ssize M(b')$} at 6.5 5
\put{$\ssize M(b)$} at 19.5 5
\endpicture}
$$
Here we denote by $\mathcal P'$ the rays coming from $\mathcal P$, and so on.
Note that we did not yet identify the points labeled $M(b)$ and $M(b')$,
these are $H$-modules of length 4 which are identified under
the restriction functor. 

Now let us make this last identification, and insert the $W$-modules which are
not in the image of $\eta:\mo H \to \mo W.$
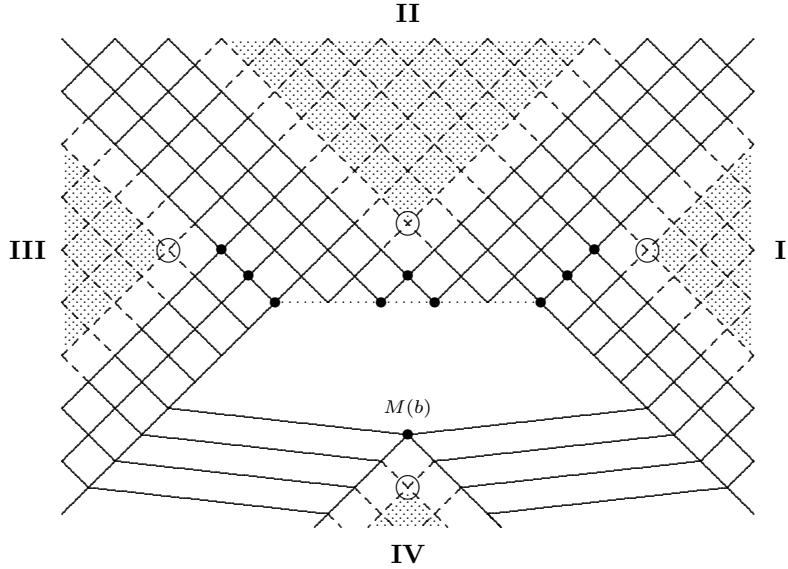
$$ 
\hbox{\beginpicture
\setcoordinatesystem units <.35cm,.35cm>
\multiput{} at 0 0  26 20 /
\multiput{$\bullet$} at 6 12  7 11  8 10  12 10  13 11  14 10  18 10  19 11  20 12  
   13 5   /
\put{$\ssize M(b)$} at 13 6
\plot 9.5 1.5   13 5  /
\plot 13 5  16 2 /

\plot 0 2  11 13 /
\plot 0 4  10 14 /
\plot 0 6  9 15 /

\plot 0 4  1 3  10 2 /

\plot 0 6  2 4  11 3 /
\plot 1 7  3 5  12 4 /
\plot 2 8  4 6  13 5 /
\plot 3 9  5 7 /
\plot 4 10  6 8 /
\plot 5 11  7 9 /
\plot 0 18  8 10 /
\plot 0 20  10 10 /
\plot 2 20  12 10 /
\plot 4 20  14 10 /
\plot 10 10  12 12 /
\plot 12 10  22 20 /
\plot 14 10  24 20 /
\plot 16 10  26 20 /
\plot 18 10 26 18 /

\plot 5 13  8 16 /
\plot 4 14  7 17 /
\plot 3 15  6 18 /
\plot 2 16  5 19 /
\plot 1 17  4 20 /
\plot 0 18  2 20 /
\plot 14 12  16 10 /
\plot 15 13  26 2 /
\plot 16 14  26 4 /
\plot 17 15  26 6 /
\plot 18 16  21 13 /
\plot 19 17  22 14 /
\plot 20 18  23 15 /
\plot 21 19  24 16 /
\plot 22 20  25 17 /
\plot 24 20  26 18 /

\plot 19 9  21 11 /
\plot 20 8  22 10 /
\plot 21 7  23 9 /
\plot 13 5  22 6  24 8 /
\plot 14 4  23 5  25 7 /
\plot 15 3  24 4  26 6 /
\plot 16.5 1.5 16 2  25 3  26 4 /
\setdots <1mm>
\plot 8 10  18 10 /

\setdashes <1mm>
\plot 11 3  12.5 1.5 /
\plot 10 2  10.5 1.5 /
\plot 16 2  16.5 1.5 /

\plot 12 4  14.5 1.5 /
\plot 14 4  11.5 1.5 /
\plot 15 3  13.5 1.5 /

\plot 0 8  5 13 /
\plot 0 10  4 14 /
\plot 0 12  3 15 /
\plot 0 14  2 16 /
\plot 0 16  1 17 /

\plot 0 16  5 11 /
\plot 0 14  4 10 /
\plot 0 12  3 9 /
\plot 0 10  2 8 /
\plot 0 8  1 7 /

\plot 21 11  26 16 /
\plot 22 10  26 14 /
\plot 23 9  26 12 /
\plot 24 8  26 10 /
\plot 25 7  26 8 /

\plot 21 13 26 8 /
\plot 22 14 26 10 /
\plot 23 15 26 12 /
\plot 24 16 26 14 /
\plot 25 17 26 16 /

\plot 6 20  14 12 /
\plot 8 20  15 13 /
\plot 10 20  16 14 /
\plot 12 20  17 15 /
\plot 14 20  18 16 /
\plot 16 20  19 17 /
\plot 18 20  20 18 /
\plot 20 20  21 19 /

\plot 6 20  5 19  /
\plot 8 20  6 18  /
\plot 10 20  7 17  /
\plot 12 20  8 16  /
\plot 14 20  9 15  /
\plot 16 20  10 14  /
\plot 18 20  11 13  /
\plot 20 20  12 12  /
\multiput{$\bigcirc$} at 13 3  4 12  22 12  13 13 /
\setshadegrid span <0.5mm>
\hshade  1.5 11.5  14.5  3 13 13 /
\hshade  13 13 13  20 6 20 /
\vshade 0 8 16  4 12 12 /
\vshade 22 12 12  26 8 16 /

\put{\bf I} at  27 12
\put{\bf II} at  13 21
\put{\bf III} at -1.3 12
\put{\bf IV} at 13 0.5
\endpicture}
$$

What are the additional modules? These are the string modules for
words which contain a completed bar as a subword. These modules
form quarters (as introduced in \cite{Rinf}), namely the four shaded
areas in the picture above. The four quarters can
be rearranged in order to be parts of a tile, similar to those exhibited in \cite{Rinf},
p 54f, this will be explained in the next section. 
Of special interest seem to be  the four encircled module, the corner modules for the
quarters.

\subsection{The corner modules}
It may be worthwhile to identify explicitly the four corner modules
(in the presentation of this component given above, we have encircled these modules):
$$
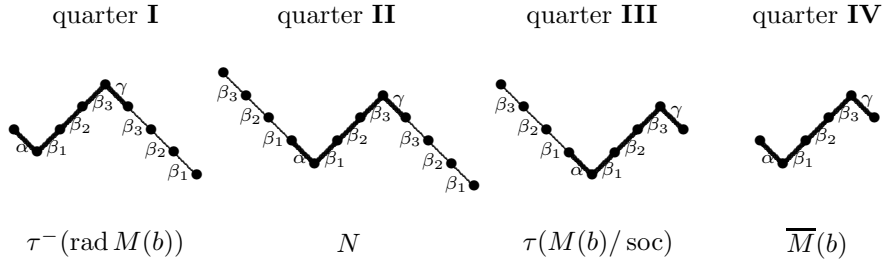
\hbox{\beginpicture
\setcoordinatesystem units <.8cm,1cm>
\put{quarter {\bf I}} at 0 0 
\put{quarter {\bf II}} at 3.8 0 
\put{quarter {\bf III}} at 8 0 
\put{quarter {\bf IV}} at 11.7 0 

\put{\beginpicture
\setcoordinatesystem units <0.3cm,0.3cm>
\multiput{$\bullet$} at    8 1  9 0
   10 1  11 2  12 3  13 2  14 1  15 0  16 -1  /
\plot   8 1   9 0  12 3  16 -1   /
\plot 8 1.05  9 0.05  12 3.05  13 2.05 /
\plot 8 1.1  9 0.1  12 3.1  13 2.1 /
\plot 8 .95  9 -.05  12 2.95  13 1.95 /
\plot 8 .9  9 -.1  12 2.9  13 1.9 /
\multiput{$\ssize \alpha$} at  8.35 0.2  /
\multiput{$\ssize \gamma$} at  12.7 2.8  /
\multiput{$\ssize \beta_1$} at   9.9 0.2  15.2 -.9   /
\multiput{$\ssize \beta_2$} at   10.9 1.2  14.2 0.1   /
\multiput{$\ssize \beta_3$} at  11.9 2.2  13.2 1.1   /
\endpicture} at 0 -1.5

\put{\beginpicture
\setcoordinatesystem units <0.3cm,0.3cm>
\multiput{$\bullet$} at   5 4  6 3  7 2  8 1  9 0
   10 1  11 2  12 3  13 2  14 1  15 0  16 -1  /
\plot   5 4  9 0  12 3  16 -1   /
\plot 8 1.05  9 0.05  12 3.05  13 2.05 /
\plot 8 1.1  9 0.1  12 3.1  13 2.1 /
\plot 8 .95  9 -.05  12 2.95  13 1.95 /
\plot 8 .9  9 -.1  12 2.9  13 1.9 /
\multiput{$\ssize \alpha$} at  8.35 0.2  /
\multiput{$\ssize \gamma$} at  12.7 2.8  /
\multiput{$\ssize \beta_1$} at  7.2 1.1  9.9 0.2  15.2 -.9   /
\multiput{$\ssize \beta_2$} at  6.2 2.1  10.9 1.2  14.2 0.1   /
\multiput{$\ssize \beta_3$} at  5.2 3.1  11.9 2.2  13.2 1.1   /
\endpicture} at 4 -1.5

\put{\beginpicture
\setcoordinatesystem units <0.3cm,0.3cm>
\multiput{$\bullet$} at   5 4  6 3  7 2  8 1  9 0
   10 1  11 2  12 3  13 2   /
\plot   5 4  9 0  12 3  13 2   /
\plot 8 1.05  9 0.05  12 3.05  13 2.05 /
\plot 8 1.1  9 0.1  12 3.1  13 2.1 /
\plot 8 .95  9 -.05  12 2.95  13 1.95 /
\plot 8 .9  9 -.1  12 2.9  13 1.9 /
\multiput{$\ssize \alpha$} at  8.35 0.2  /
\multiput{$\ssize \gamma$} at  12.7 2.8  /
\multiput{$\ssize \beta_1$} at  7.2 1.1  9.9 0.2    /
\multiput{$\ssize \beta_2$} at  6.2 2.1  10.9 1.2    /
\multiput{$\ssize \beta_3$} at  5.2 3.1  11.9 2.2    /
\endpicture} at 8 -1.5

\put{\beginpicture
\setcoordinatesystem units <0.3cm,0.3cm>
\multiput{$\bullet$} at    8 1  9 0
   10 1  11 2  12 3  13 2   /
\plot   8 1  9 0  12 3  13 2   /
\plot 8 1.05  9 0.05  12 3.05  13 2.05 /
\plot 8 1.1  9 0.1  12 3.1  13 2.1 /
\plot 8 .95  9 -.05  12 2.95  13 1.95 /
\plot 8 .9  9 -.1  12 2.9  13 1.9 /
\multiput{$\ssize \alpha$} at  8.35 0.2  /
\multiput{$\ssize \gamma$} at  12.7 2.8  /
\multiput{$\ssize \beta_1$} at    9.9 0.2    /
\multiput{$\ssize \beta_2$} at   10.9 1.2    /
\multiput{$\ssize \beta_3$} at   11.9 2.2    /
\endpicture} at 11.7 -1.5

\put{$\tau^-(\rad M(b))$} at 0 -3
\put{$N$} at 4 -3
\put{$\tau(M(b)/\soc)$} at 8.1 -3
\put{$\overline M(b)$} at 11.7 -3
\endpicture}
$$
In our example both the socle and the top of all corner modules are of length 2.
In general, the corner modules for the quarters {\bf I} and {\bf IV} may have a socle of
length 3, and dually, the  corner modules for the quarter {\bf III} and {\bf IV} may have a top of
length 3, as the following description shows:

For the quarter {\bf I}, the corner module
$\tau^-(\rad M(b))$ is obtained from $\rad M(b)$ by adding
hooks on the left and on the right. 

Dually, for the quarter {\bf III},
the corner module $\tau(M(b)/\soc)$ is obtained from $M(b)/\soc$ 
by adding cohooks on the left and on the right. 

For the quarter
{\bf IV} we obtain the corner module 
$\overline M(b)$ by adding  to $M(b)$ a hook on the left, a cohook
on the right. In our case we have $\overline M(b) = M(\overline b)$, where
$\overline b$ is the completion of $b$. 

The corner module for the quarter {\bf II} has been denoted 
here by $N = N(b)$, in our example, we start with
$\rad M(b)/\soc$ and add a cohook on the left and a hook on the
right, in order to obtain $N$ --- however, this rule makes sense only in case 
$\rad M(b)/\soc$ is non-zero, thus in case the bar module  $M(b)$ is
of length at least $3$. In general, let  $N_0$ be the boundary module
in $\mathcal R_0$ which has the same socle as $M(b)$ and $N_\infty$ the boundary
module in $\mathcal R_\infty$ which has the same top as $M(b)$.
Then $N$ has a filtration 
$
 0 \subset N'' \subseteq N' \subset N 
$
with 
$$
 N'' = \eta(N_0),\quad \quad  N'/N_0 =  \rad M(b)/\soc, \quad N/N' = \eta(N_\infty).
$$ 
In case $M(b)$ is of length $2$, say $b = \beta$ 
where $\beta$ is an arrow, 
then we deal with an exact sequence
$$
 0 \to \eta(N_0) \to N \to \eta(N_\infty) \to 0.
$$
This is one of the
Auslander-Reiten sequences involving string modules and having an indecomposable
middle term, namely that corresponding to the arrow$\beta$, see \cite{BR}.
       \medskip

This description of the corner modules shows that all of them are
related to the following Auslander-Reiten sequence 
$$
 0 \to \rad M(b) \to M(b)\oplus \rad M(b)/\soc \to M(b)/\soc \to 0
$$
for $W/I$, where $I$ is the annihilator of $M(b)$. Recall that we have
used the Auslander-Reiten quiver of $W/I$ as our gluing device, let us 
mark the Auslander-Reiten sequence in question:
$$
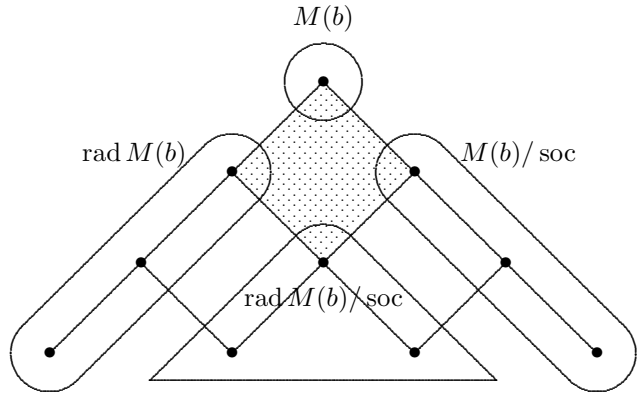
\hbox{\beginpicture
\setcoordinatesystem units <1.2cm,1.2cm>
\multiput{$\bullet$} at 0 0  1 1  2 0  2 2  3 1  3 3  4 0  4 2  5 1  6 0 /
\plot 0 0  3 3  6 0 /
\plot 1 1  2 0  4 2 /
\plot 2 2  4 0  5 1 /
\setshadegrid span <0.7mm>
\hshade 1 3 3 <,,z,> 2 2 4 <,,,z> 3 3 3 /
\plot 0.3 -0.3  2.3 1.7 /
\plot -.3 0.3  1.7 2.3 /
\circulararc 180 degrees from 2.3 1.7 center at 2 2 
\circulararc 180 degrees from -.3 0.3 center at 0 0 

\plot 3.7 1.7  5.7  -.3 /
\plot 4.3 2.3  6.3 .3 /
\circulararc 180 degrees from 4.3 2.3 center at 4 2 
\circulararc 180 degrees from 5.7 -.3 center at 6 0 

\circulararc 360 degrees from 3.3 3.3 center at 3 3 

\plot 3.3 1.3  4.9 -.3  1.1 -.3  2.7 1.3 /
\circulararc  90 degrees from 3.3 1.3 center at 3 1
 
\put{$M(b)$} at 3 3.7
\put{$\rad M(b)$} [r] at 1.5 2.2
\put{$M(b)/\soc$} [l] at 4.5 2.2
\put{$\rad M(b)/\soc$} at 3 .6
\put{} at 0 -.3
\endpicture}
$$
In case $M(b)$ is of length at least $3$ we deal with a square, if it is
of length 2, then with a triangle. In section 13.2 we will see in which 
way this square or triangle is enlarged in $\mo W$.

\subsection{Wind wheels with several bars} 
Now we consider the general case of $t$ bars. Any of the components $\mathcal P, \mathcal R_0,
\mathcal R_\infty, \mathcal Q$ will be cut into $t$ pieces, and always the pieces will be
indexed the the set $\mathcal B$ of the direct bars.

Write $w = w_1\cdots w_t$ where all the words $w_i$
start with a direct letter and end with an inverse letter, and such that any $w_i$
ends with an inverse  bar, say $(b_i)^{-1}$. 
We denote by $\lambda:\mathcal B \to \mathcal B$ the
cyclic permutation with $\lambda(b_i) = b_{i+1}.$

Similar to the case $t=1$, we remove arrows from the 
preprojective component,  but now we want to retain $t$ connected pieces $\mathcal P(b)$
with $b\in \mathcal B$. 
The piece $\mathcal P(b)$
is supposed to contain the projective modules starting with $\rad M(b)$ and
ending with $M(\lambda b)$. Here is a picture of $\mathcal P(b)$:

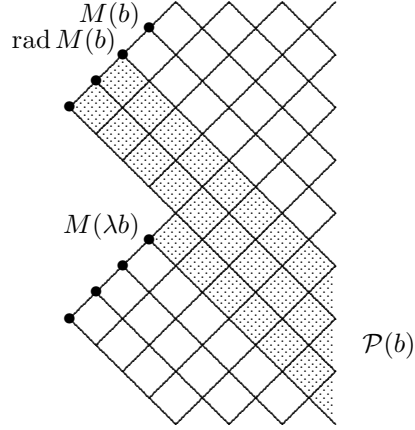
$$ 
\hbox{\beginpicture
\setcoordinatesystem units <.35cm,.35cm>
\multiput{} at 0 0  10 16 /
\multiput{$\bullet$} at 0 4  1 5  2 6  3 7  0 12  1 13  2 14  3 15 /
\put{$M(b)$} at 1.5 15.5
\put{$\rad M(b)$} at -.2 14.5
\put{$M(\lambda b)$} at 1.2 7.5
\plot 0 12 4 16 /
\plot 1 11 6 16 /
\plot 2 10 8 16 /
\plot 3 9  10 16 /
\plot 0 4  10 14 /
\plot 1 3  10 12 /
\plot 2 2  10 10 /
\plot 3 1  10 8 /
\plot 4 0  10 6 /
\plot 6 0  10 4 /
\plot 8 0  10 2 /

\plot 0 4  4 0 /
\plot 1 5  6 0 /
\plot 2 6  8 0 /
\plot 3 7  10 0 /
\plot 0 12 10 2 /
\plot 1 13 10 4 /
\plot 2 14 10 6 /
\plot 3 15 10 8 /
\plot 4 16 10 10 /
\plot 6 16 10 12 /
\plot 8 16 10 14 /
\setdashes <1mm>
\setshadegrid span <0.5mm>
\hshade 0 10 10 <,,,z> 6 4  10 <,,z,z> 7 3 9 <,,z,z> 8 4 8 <,,z,z> 12 0 4 <,,z,> 14 2 2 /
\put{$\mathcal P(b)$} at 12 3
\endpicture} 
$$

Next, consider the regular component $\mathcal R_0$; according to our convention, this is 
the component 
which contains the simple modules $T(b) = \topp M(b)$.
Again, we remove arrows in order to obtain $t$ pieces consisting of full corays;
the piece $\mathcal R(b)$ with index $b$ 
shall contain the modules $T(\lambda b), \tau^{-1}T(\lambda b), \tau^{-2}T(\lambda b),\dots $ up to 
$\tau T(b)$.
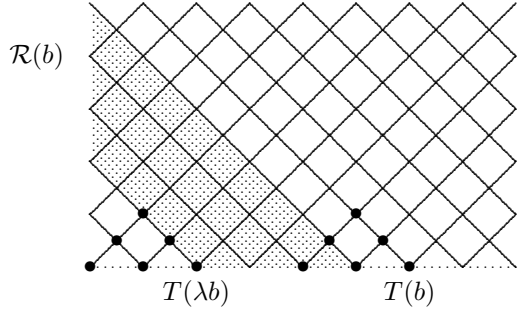
$$ 
\hbox{\beginpicture
\setcoordinatesystem units <.35cm,.35cm>
\multiput{} at 0 0  16 11 /
\multiput{$\bullet$} at 0 0  1 1  2 0  2 2  3 1  4 0  8 0  9 1  10 0  10 2  11 1  12 0 /
\plot 14 0  16 2 /
\plot 12 0  16 4 /
\plot 10 0  16 6 /
\plot 8 0   16 8 /
\plot 6 0  16 10 /
\plot 4 0  14 10 /
\plot 2 0  12 10 /
\plot 0 0  10 10 /
\plot 0 2  8 10 /
\plot 0 4  6 10 /
\plot 0 6  4 10 /
\plot 0 8  2 10 /

\plot 0 2  2 0 /
\plot 0 4  4 0 /
\plot 0 6  6 0 /
\plot 0 8  8 0 /
\plot 0 10 10 0 /
\plot 2 10 12 0 /
\plot 4 10 14 0 /
\plot 6 10 16 0 /
\plot 8 10 16 2 /
\plot 10 10 16 4 /
\plot 12 10 16 6 /
\plot 14 10 16 8 /

\setdashes <1mm>
\setshadegrid span <0.5mm>
\hshade 0 4 10 <,,,z> 4 0 6 <,,z,> 10 0 0 /
\put{$T(\lambda b)$} at 4 -1
\put{$T(b)$} at 12 -1
\setdots <1mm> 
\plot 0 0  16 0 / 
\put{$\mathcal R(b)$} at -2 8
\endpicture}
$$

Similarly, we consider a word $w'$ obtained from $w$ by cyclic rotation, such that
$w' = w_1'\cdots w_t'$ where any $w_i'$ starts with an inverse letter and ends with a direct
letter, and ends in a (direct)  bar, say the bar $b_{\sigma(i)}$ and we denote by 
$\rho: \mathcal B \to \mathcal B$ the permutation which sends $b_{\sigma(i)}$ to
$b_{\sigma(i+1)}.$ 

Now we cut the preinjective component in order to get pieces made up of corays.
the 
piece $\mathcal Q(b)$ has to contain $M(b)/\soc$ up to $M(\rho b)$.
Here is a picture of $\mathcal Q(b)$.

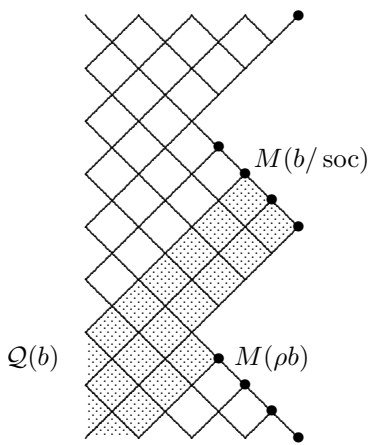
$$ 
\hbox{\beginpicture
\setcoordinatesystem units <.35cm,.35cm>

\multiput{} at 0 0  10 16 /
\multiput{$\bullet$} at 5 3  6 2  7 1  8 0  5 11  6 10  7 9  8 8  8 16  /
\plot 6 0  7 1 /
\plot 4 0  6 2 /
\plot 2 0  5 3 /
\plot 0 0  8 8 /
\plot 0 2  7 9 /
\plot 0 4  6 10 /
\plot 0 6  5 11 /
\plot 0 8  8 16 /
\plot 0 10  6 16 /
\plot 0 12  4 16 /
\plot 0 14  2 16 /

\plot 6 16  7 15 /
\plot 4 16  6 14 /
\plot 2 16  5 13 /
\plot 0 16  8 8 /
\plot 0 14  7 7 /
\plot 0 12  6 6 /
\plot 0 10  5 5 /

\plot 0 8  8 0 /
\plot 0 6  6 0 /
\plot 0 4  4 0 /
\plot 0 2  2 0 /
\setdashes <1mm>
\setshadegrid span <0.5mm>
\hshade 0 0 2 <,,,z> 3 0 5 <,,z,z> 4 0 4 <,,z,z> 8 4 8 <,,z,> 10 6 6 /
\put{$M(b/\soc)$} at 8.5 10.5
\put{$M(\rho b)$} at 7 3 
\put{$\mathcal Q(b)$} at -2 3

\endpicture}
$$

Finally, we consider the regular component which contains the simple modules 
$S(b) = \soc M(b)$.
Again, we remove arrows in order to obtain $t$ pieces consisting now of full rays.
The piece $\mathcal R_\infty(b)$ indexed by $b$ contains the rays 
starting at
$ \tau^{-1}S(b),\tau^{-2}S(b), \dots,$ up to $S(\rho b).$
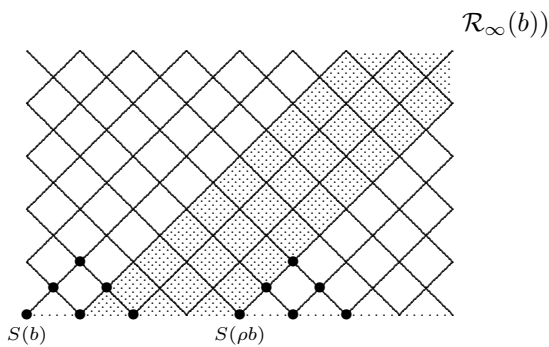
$$ 
\hbox{\beginpicture
\setcoordinatesystem units <.35cm,.35cm>
\multiput{} at 0 0  16 11 /
\multiput{$\bullet$} at 0 0  1 1  2 0  2 2  3 1  4 0  8 0  9 1  10 0  10 2  11 1  12 0 /
\put{$\mathcal R_\infty(b))$} at 18 11
\plot 14 0  16 2 /
\plot 12 0  16 4 /
\plot 10 0  16 6 /
\plot 8 0   16 8 /
\plot 6 0  16 10 /
\plot 4 0  14 10 /
\plot 2 0  12 10 /
\plot 0 0  10 10 /
\plot 0 2  8 10 /
\plot 0 4  6 10 /
\plot 0 6  4 10 /
\plot 0 8  2 10 /

\plot 0 2  2 0 /
\plot 0 4  4 0 /
\plot 0 6  6 0 /
\plot 0 8  8 0 /
\plot 0 10 10 0 /
\plot 2 10 12 0 /
\plot 4 10 14 0 /
\plot 6 10 16 0 /
\plot 8 10 16 2 /
\plot 10 10 16 4 /
\plot 12 10 16 6 /
\plot 14 10 16 8 /

\setdashes <1mm>
\setshadegrid span <0.5mm>
\hshade 0 2 8  <,,,z> 8 10 16  <,,z,> 10 12 16 /
\put{$\ssize  S(\rho b)$} at 8 -0.8                               
\put{$\ssize S(b)$} at 0 -0.8                          

\setdots <1mm> 
\plot 0 0  16 0 / 

\endpicture}
$$

Altogether we have cut the four string components of $\mo H$  into flat pieces: any
such component yields $t$ pieces. The gluing of these pieces is done by identifying
$H$-modules which become isomorphic under the restriction functor 
$$
 \mo H \to \mo W
$$
(and finally we will have to add various quarters). 

As in the case $t=1$, we first will look at the proper subfactors of the bars. 
Identifying
the corresponding $H$-modules, we obtain partial translation quivers which are planar:
Any of the components $\mathcal P, \mathcal Q,
\mathcal R_0, \mathcal R_\infty$ has been cut into $t$ pieces, and the identification process
will use one piece of each kind, in order to obtain $t$ planar 
partial translation quivers of the following form ($**$):
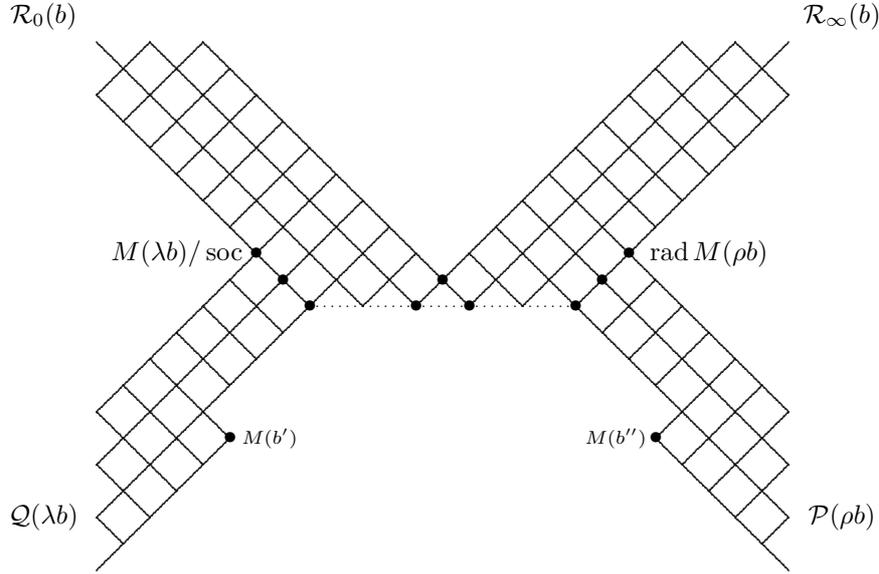
$$ 
\hbox{\beginpicture
\setcoordinatesystem units <.35cm,.35cm>
\multiput{} at 0 0  26 20 /
\put{$\mathcal P(\rho b)$} at 28 2
\put{$\mathcal Q(\lambda b)$} at -2  2
\put{$\mathcal R_0(b)$} at -2  21
\put{$\mathcal R_\infty(b)$} at 28  21

\put{$\rad M(\rho b)$} at 23 12
\put{$M(\lambda b)/\soc$} at 3 12

\multiput{$\bullet$} at 6 12  7 11  8 10  12 10  13 11  14 10  18 10  19 11  20 12  
   5 5  21 5 /
\plot 0 0  5 5  /
\plot 0 2  11 13 /
\plot 0 4  10 14 /
\plot 0 6  9 15 /

\plot 0 2  1 1 /
\plot 0 4  2 2 /
\plot 0 6  3 3 /
\plot 1 7  4 4 /
\plot 2 8  5 5  /
\plot 3 9  5 7 /
\plot 4 10  6 8 /
\plot 5 11  7 9 /
\plot 0 18  8 10 /
\plot 0 20  10 10 /
\plot 2 20  12 10 /
\plot 4 20  14 10 /
\plot 10 10  12 12 /
\plot 12 10  22 20 /
\plot 14 10  24 20 /
\plot 16 10  26 20 /
\plot 18 10 26 18 /

\plot 5 13  8 16 /
\plot 4 14  7 17 /
\plot 3 15  6 18 /
\plot 2 16  5 19 /
\plot 1 17  4 20 /
\plot 0 18  2 20 /
\plot 14 12  16 10 /
\plot 15 13  26 2 /
\plot 16 14  26 4 /
\plot 17 15  26 6 /
\plot 18 16  21 13 /
\plot 19 17  22 14 /
\plot 20 18  23 15 /
\plot 21 19  24 16 /
\plot 22 20  25 17 /
\plot 24 20  26 18 /
\plot 21 5  26 0 /
\plot 19 9  21 11 /
\plot 20 8  22 10 /
\plot 21 7  23 9 /
\plot 21 5  24 8 /
\plot 22 4  25 7 /
\plot 23 3  26 6 /
\plot 24 2  26 4 /
\plot 25 1  26 2 /
\setdots <1mm>
\plot 8 10  18 10 /
\put{$\ssize M(b')$} at 6.5 5
\put{$\ssize M(b'')$} at 19.5 5
\endpicture}
$$

It is important to observe that in contrast to the case $t=1$, the 
modules labeled $M(b')$ and $M(b'')$ (corresponding to bars $b'$ and $b''$)
now may be different!

We obtain in this way a permutation $\pi$ of the bars
such that $b'' = \pi(b')$. But we know that $b' = \rho\lambda (b)$
and $b'' = \lambda\rho(b).$ Now  $b' = \rho\lambda(b)$ means that 
 $b = \lambda^{-1}\rho^{-1}(b')$, thus 
$$
 \pi(b') = b'' = \lambda\rho(b) = \lambda\rho \lambda^{-1}\rho^{-1}(b') = 
[\lambda,\rho](b').
$$
This shows: 
$$
 \pi =  [\lambda,\rho].
$$
     \medskip

Of course, as we have mentioned already, we still have to add the indecomposable
$W$-modules  which do not belong to the
image of $\eta$. We know that there are precisely $t$ non-periodic (but biperiodic)
$\mathbb Z$-words; they give rise to $4\cdot t$ quarters  which have to be inserted 
as in  the case $t=1$. Before making the final identifcations, let us attach 
the quarters of type {\bf IV} to the pieces of the
form $\mathcal Q(b)$.
In this way, we obtain $t$ partial translation quivers of the form
$$ 
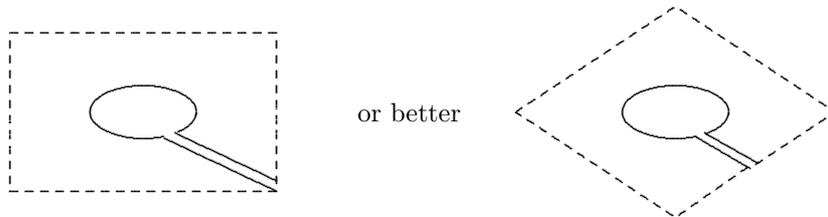
\hbox{\beginpicture
\setcoordinatesystem units <.7cm,.7cm>

\put{\beginpicture
\multiput{} at 0 0  5 3 /

\setdashes <1mm>
\plot 0 0 5 0  5 3  0 3  0 0 /
\setsolid
\ellipticalarc axes ratio 2:1 293 degrees from 3.5 1.5 center at 2.5 1.5 
\ellipticalarc axes ratio 2:1 -50 degrees from 3.5 1.5 center at 2.5 1.5 
\plot 5 0  2.9 1 /
\plot 5 0.2 3.1 1.1 /
\endpicture} at 0 0
\put{or better} at 5 0

\put{\beginpicture
\multiput{} at -.5 -.5  5.5 3.5 /

\setdashes <1mm>
\plot -.5 1.5  2.5 3.5  5.5 1.5  4 .5 /
\plot 3.85  0.45  2.5 -0.5  -.5 1.5 /

\setsolid
\ellipticalarc axes ratio 2:1 293 degrees from 3.5 1.5 center at 2.5 1.5 
\ellipticalarc axes ratio 2:1 -55 degrees from 3.5 1.5 center at 2.5 1.5 
\plot 3.9 0.43  2.88 1.03 /
\plot 4.05 0.53 3.05 1.1 /
\endpicture} at 10 0

\endpicture}
$$
(The visualization on the right hand side takes into account the embedding of
these Auslander-Reiten components into the corresponding ``Auslander-Reiten quilt''
which we will discuss in the next section.)
      \medskip

The partial translation quivers 
are sewn together in the same way as one constructs the Riemann surfaces of the $n$-the root functions
in complex analysis (taking into account the permutation $\pi$). For example,  we may obtain
a $3$-ramified component which roughly will have the following shape:
$$
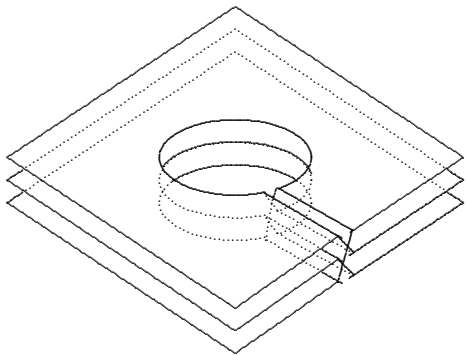
\hbox{\beginpicture
\setcoordinatesystem units <1cm,1cm>
\put{\beginpicture
\multiput{} at -.5 -.5  5.5 3.5 /

\setdots <.2mm>
\plot -.5 1.5  2.5 3.5  5.5 1.5  4 .5 /
\plot 3.85  0.45  2.5 -0.5  -.5 1.5 /

\setsolid
\ellipticalarc axes ratio 2:1 293 degrees from 3.5 1.5 center at 2.5 1.5 
\ellipticalarc axes ratio 2:1 -55 degrees from 3.5 1.5 center at 2.5 1.5 
\plot 3.9 0.43  2.88 1.03 /
\plot 4.05 0.53 3.05 1.1 /

\endpicture} at 0 0
\put{\beginpicture
\multiput{} at -.5 -.5  5.5 3.5 /

\setdots <.6mm>
\plot -.5 1.5  2.5 3.5  5.5 1.5  4 .5 /
\plot 3.9 0.43  2.88 1.03 /
\plot 4.05 0.53 3.05 1.1 /

\ellipticalarc axes ratio 2:1 293 degrees from 3.5 1.5 center at 2.5 1.5 
\ellipticalarc axes ratio 2:1 -55 degrees from 3.5 1.5 center at 2.5 1.5 

\plot 3.05 1.1  2.9  1.3 /

\setdots <.2mm>

\plot 3 1.25  3.04 1.4 /

\ellipticalarc axes ratio 2:1 150 degrees from 3.5 1.59 center at 2.5 1.5 

\plot 3.9 0.43  3.65  .57 /

\plot 3.85  0.45  2.5 -0.5  -.5 1.5  -.3 1.63 /
\plot 4 .5  5.5 1.5  5.3  1.63 / 
\setsolid

\plot 4.05 0.55  3.85 .75 / 


\endpicture} at 0 -0.3

\put{\beginpicture
\multiput{} at -.5 -.5  5.5 3.5 /

\setdots <.6mm>
\plot -.5 1.5  2.5 3.5  5.5 1.5  4 .5 /
\plot 3.9 0.43  2.88 1.03 /
\plot 4.05 0.53 3.05 1.1 /

\ellipticalarc axes ratio 2:1 293 degrees from 3.5 1.5 center at 2.5 1.5 
\ellipticalarc axes ratio 2:1 -55 degrees from 3.5 1.5 center at 2.5 1.5 

\plot 3.05 1.1  2.9  1.3 /

\plot 2.9 1   3 1.6 /

\setdots <.2mm>


\ellipticalarc axes ratio 2:1 110 degrees from 3.3 1.8 center at 2.5 1.5 

\plot 3.9 0.43  3.65  .57 /

\plot 3.85  0.45  2.5 -0.5  -.5 1.5  -.3 1.63 /
\plot 4 .5  5.5 1.5  5.3  1.63 / 
\setsolid

\plot 4.05 0.55  3.85 .75 / 

\plot 3.85 0.45  4.04 1.14  /

\endpicture} at 0 -0.6

\endpicture}
$$
We will call such a component with $r$ leaves an {\it $r$-ramified component of type
$\mathbb A_\infty^\infty$.}

\begin{prop} Let $W$ be a wind wheel with $t$ bars. Then:
Any non-regular Auslander-Reiten component is an $r$-ramified component of
type $\mathbb A_\infty^\infty$ with $1\le r \le t.$
If $\mathcal C_1,\dots, \mathcal C_c$ are the
non-regular Auslander-Reiten components of $W$ and $\mathcal C_i$ is $r_i$-ramified,
for $1\le i\le c$, then $\sum_{i=1}^c r_i = t.$
\end{prop}

We may assume that $r_1\ge r_2 \ge \dots \ge r_c$, thus we deal with a partition
and we call this partition $(r_1,r_2,\cdots,r_c)$ the {\it ramification sequence} of $W$.

\subsection{Wind wheels with arbitrarily many non-regular components} We are going 
to present a wind wheel with $t$ bars which has $t$ non-regular Auslander-Reiten
components (all being necessarily $1$-ramified: the ramification sequence is $(1,1,\dots,1)$).
Here  is the quiver:
$$
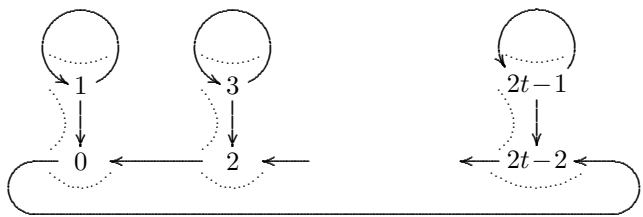
\hbox{\beginpicture
\setcoordinatesystem units <1cm,1cm>
\put{} at 0 -1.6
\put{\beginpicture
\put{} at 0 2
\put{$0$} at 0 0
\put{$1$} at 0 1
\arrow <1.5mm> [0.25,0.75] from 0 0.8 to 0 0.2
\circulararc 155 degrees from 0 2 center at 0 1.5
\circulararc -155 degrees from 0 2 center at 0 1.5
\arrow <1.5mm> [0.25,0.75] from -0.22 1.05 to -0.17 1.02
\setdots <.7mm>
\setquadratic
\plot -.4 1.4 0 1.3 .4 1.4 /

\endpicture} at 0 0 
\put{\beginpicture
\put{} at 0 2
\put{$2$} at 0 0
\put{$3$} at 0 1
\arrow <1.5mm> [0.25,0.75] from 0 0.8 to 0 0.2

\circulararc 155 degrees from 0 2 center at 0 1.5
\circulararc -155 degrees from 0 2 center at 0 1.5
\arrow <1.5mm> [0.25,0.75] from -0.22 1.05 to -0.17 1.02
\setdots <.7mm>
\setquadratic
\plot -.4 1.4 0 1.3 .4 1.4 /

\endpicture} at 2 0

\put{\beginpicture
\put{} at 0 2
\put{$2t\!-\!2$} at 0 0
\put{$2t\!-\!1$} at 0 1
\arrow <1.5mm> [0.25,0.75] from 0 0.8 to 0 0.2

\circulararc 123 degrees from 0 2 center at 0 1.5
\circulararc -125 degrees from 0 2 center at 0 1.5
\arrow <1.5mm> [0.25,0.75] from -0.47 1.33 to -0.42 1.25
\setdots <.7mm>
\setquadratic
\plot -.4 1.4 0 1.3 .4 1.4 /

\endpicture} at 6 0 

\plot -0.3 -.95  -.6 -.95 /

\arrow <1.5mm> [0.25,0.75] from 1.6 -.95 to 0.4 -.95
\arrow <1.5mm> [0.25,0.75] from 3 -.95 to 2.4 -.95
\arrow <1.5mm> [0.25,0.75] from 5.5 -.95 to 5 -.95

\arrow <1.5mm> [0.25,0.75] from 7 -.95 to 6.5 -.95

\circulararc 180 degrees from -.6 -.95 center at -.6 -1.3
\circulararc -180 degrees from 7 -.95 center at 7 -1.3
\plot -.6 -1.65  7 -1.65 /

\setdots <.7mm>
\setquadratic
\plot -.4 -1.1  0 -1.3 .4 -1.1 /
\plot 1.6 -1.1  2 -1.3 2.4 -1.1 /
\plot 5.4 -1.1  6 -1.3 6.6 -1.1 /

\plot -0.4 0  -0.2 -.45 -.4 -.8 /
\plot  1.6 0  1.8 -.45  1.6 -.8 /
\plot  5.5 0  5.7 -.45  5.5 -.8 /
\endpicture}
$$

For example, for $t=5$, the primitive cyclic word is
$$ 
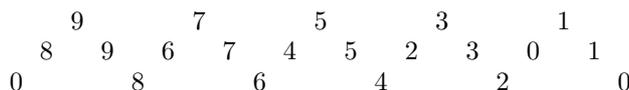
\hbox{\beginpicture
\setcoordinatesystem units <.4cm,.4cm>
\put{$0$} at 0 0
\put{$1$} at -1 1
\put{$1$} at -2 2
\put{$0$} at -3 1
\put{$2$} at -4 0
\put{$3$} at -5 1
\put{$3$} at -6 2
\put{$2$} at -7 1
\put{$4$} at -8 0
\put{$5$} at -9 1
\put{$5$} at -10 2
\put{$4$} at -11 1
\put{$6$} at -12 0
\put{$7$} at -13 1
\put{$7$} at -14 2

\put{$6$} at -15 1
\put{$8$} at -16 0
\put{$9$} at -17 1
\put{$9$} at -18 2
\put{$8$} at -19 1
\put{$0$} at -20 0
\endpicture}
$$

All the non-regular Auslander-Reiten components of this algebra look similar,
here is one of these components:
$$ 
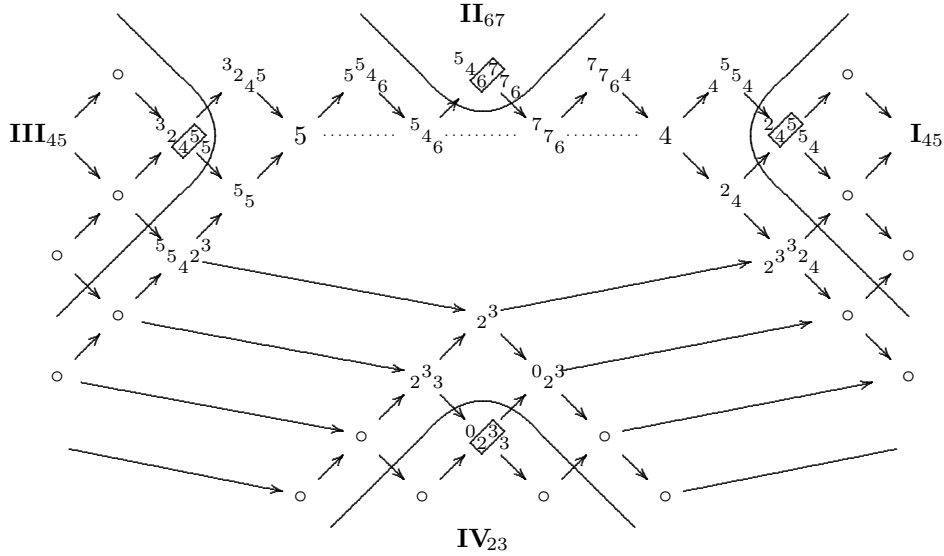
\hbox{\beginpicture
\setcoordinatesystem units <.8cm,.8cm>
\multiput{} at 0 -1  14 7 /
\put{
 \beginpicture
 \setcoordinatesystem units <.15cm,.15cm>
 \put{$\ssize 3$} at 4 2
 \put{$\ssize 2$} at 3 1
 \put{$\ssize 4$} at 2 0
 \put{$\ssize 5$} at 1 1
 \put{$\ssize 5$} at 0 2 
\endpicture} at 2 4

\put{
 \beginpicture
 \setcoordinatesystem units <.15cm,.15cm>
 \put{$\ssize 5$} at 4 1
 \put{$\ssize 5$} at 3 2 
\endpicture} at 3 5

\put{$5$} at 4 6

\put{$\circ$} at 1 7

\put{
 \beginpicture
 \setcoordinatesystem units <.15cm,.15cm>
 \put{$\ssize 5$} at 7 1
 \put{$\ssize 5$} at 6 2
 \put{$\ssize 4$} at 5 1
 \put{$\ssize 2$} at 4 2
 \put{$\ssize 3$} at 3 3
\plot 4 1  5 0  7 2  6 3  4 1 /

 \endpicture} at  2 6

\plot 2 7  1 8 /
\plot 2 5  0 3 /
\setquadratic
\plot 2 7  2.6  6  2 5 /
\setlinear

\put{
 \beginpicture
 \setcoordinatesystem units <.15cm,.15cm>
 \put{$\ssize 5$} at 7 2
 \put{$\ssize 4$} at 6 1
 \put{$\ssize 2$} at 5 2
 \put{$\ssize 3$} at 4 3
 
\endpicture} at  3 7

\put{
 \beginpicture
 \setcoordinatesystem units <.15cm,.15cm>
 \put{$\ssize 6$} at 3 0
 \put{$\ssize 4$} at 2 1
 \put{$\ssize 5$} at 1 2
 \put{$\ssize 5$} at 0 1  
\endpicture} at  5 7

\put{
 \beginpicture
 \setcoordinatesystem units <.15cm,.15cm>
 \put{$\ssize 6$} at 5 0
 \put{$\ssize 4$} at 4 1
 \put{$\ssize 5$} at 3 2
\endpicture} at  6 6

\put{
 \beginpicture
 \setcoordinatesystem units <.15cm,.15cm>
 \put{$\ssize 6$} at 5 0
 \put{$\ssize 7$} at 4 1
 \put{$\ssize 7$} at 3 2
\endpicture} at  8 6

\put{$\circ$} at 1 5
\put{$\circ$} at 0 4

\put{$\circ$} at 13 5
\put{$\circ$} at 14 4

\put{$\circ$} at 1 3
\put{$\circ$} at 13 3
\put{$\circ$} at 0 2

\put{\bf III$_{45}$} at -0.3 6

\put{$\circ$} at 14 2

\put{\bf I$_{45}$} at 14.3 6

\put{\bf II$_{67}$} at 7 8
\put{\bf IV\!$_{23}$} at 7 -.7

\put{ 
\beginpicture
 \setcoordinatesystem units <.15cm,.15cm>
 \put{$\ssize 6$} at 5 0
 \put{$\ssize 7$} at 4 1
 \put{$\ssize 7$} at 3 2
 \put{$\ssize 6$} at 2 1
 \put{$\ssize 4$} at 1 2
 \put{$\ssize 5$} at 0 3
\plot 1 1  2 0  4 2  3 3  1 1 /

\endpicture} at  7 7

\plot 6 7  5 8 /
\plot 8 7  9 8 /
\setquadratic
\plot 6 7  7 6.4  8 7 /
\setlinear

\put{
 \beginpicture
 \setcoordinatesystem units <.15cm,.15cm>
 \put{$\ssize 4$} at 5 1
 \put{$\ssize 6$} at 4 0
 \put{$\ssize 7$} at 3 1
 \put{$\ssize 7$} at 2 2
\endpicture} at  9 7

\put{
 \beginpicture
 \setcoordinatesystem units <.15cm,.15cm>
 \put{$\ssize 4$} at 5 0
 \put{$\ssize 5$} at 4 1
 \put{$\ssize 5$} at 3 2
 \put{$\ssize 4$} at 2 0
\endpicture} at  11 7

\put{
 \beginpicture
 \setcoordinatesystem units <.15cm,.15cm>
 \put{$\ssize 4$} at 6 0
 \put{$\ssize 5$} at 5 1
 \put{$\ssize 5$} at 4 2
 \put{$\ssize 4$} at 3 1
 \put{$\ssize 2$} at 2 2
\plot 2 1  3 0  5 2  4 3  2 1 /
\endpicture} at  12 6

\plot 12 7  13 8 /
\plot 12 5  14 3 /
\setquadratic
\plot 12 7  11.4  6  12 5 /
\setlinear

\put{$\circ$} at 13 7

\put{$4$} at 10 6 

\put{
 \beginpicture
 \setcoordinatesystem units <.15cm,.15cm>
 \put{$\ssize 4$} at 6 0
 \put{$\ssize 2$} at 5 1
\endpicture} at  11 5

\put{
 \beginpicture
 \setcoordinatesystem units <.15cm,.15cm>
 \put{$\ssize 4$} at 4 0
 \put{$\ssize 2$} at 3 1
 \put{$\ssize 3$} at 2 2
 \put{$\ssize 3$} at 1 1
 \put{$\ssize 2$} at 0 0
\endpicture} at  12 4

\put{
 \beginpicture
 \setcoordinatesystem units <.15cm,.15cm>
 \put{$\ssize 3$} at 1 1
 \put{$\ssize 2$} at 0 0
\endpicture} at  7 3

\put{
 \beginpicture
 \setcoordinatesystem units <.15cm,.15cm>
 \put{$\ssize 3$} at 2 0
 \put{$\ssize 3$} at 1 1
 \put{$\ssize 2$} at 0 0
\endpicture} at  6 2

\put{
 \beginpicture
 \setcoordinatesystem units <.15cm,.15cm>
 \put{$\ssize 3$} at 2 1
 \put{$\ssize 2$} at 1 0
 \put{$\ssize 0$} at 0 1
\endpicture} at  8 2

\put{$\circ$} at 5 1

\put{
 \beginpicture
 \setcoordinatesystem units <.15cm,.15cm>
 \put{$\ssize 3$} at 3 0
 \put{$\ssize 3$} at 2 1
 \put{$\ssize 2$} at 1 0
 \put{$\ssize 0$} at 0 1
\plot 0 0  1 -1  3 1  2 2  0 0 /

\endpicture} at  7 1

\plot 6 1  4.5 -.5 /
\plot 8 1  9.5 -.5 /
\setquadratic
\plot 6 1  7 1.6  8 1 /
\setlinear

\put{$\circ$} at 9 1
\put{$\circ$} at 4 0
\put{$\circ$} at 6 0
\put{$\circ$} at 8 0 
\put{$\circ$} at 10 0

\arrow <1.5mm> [0.25,0.75] from 0.3 4.3 to 0.7 4.7
\arrow <1.5mm> [0.25,0.75] from 2.3 4.3 to 2.7 4.7
\arrow <1.5mm> [0.25,0.75] from 1.3 5.3 to 1.7 5.7
\arrow <1.5mm> [0.25,0.75] from 3.3 5.3 to 3.7 5.7
\arrow <1.5mm> [0.25,0.75] from 4.3 6.3 to 4.7 6.7
\arrow <1.5mm> [0.25,0.75] from 6.3 6.3 to 6.6 6.6
\arrow <1.5mm> [0.25,0.75] from 8.3 6.3 to 8.7 6.7
\arrow <1.5mm> [0.25,0.75] from 5.3 6.7 to 5.7 6.3
\arrow <1.5mm> [0.25,0.75] from 7.3 6.7 to 7.7 6.3
\arrow <1.5mm> [0.25,0.75] from 9.3 6.7 to 9.7 6.3
\arrow <1.5mm> [0.25,0.75] from 10.3 5.7 to 10.7 5.3
\arrow <1.5mm> [0.25,0.75] from 11.3 4.7 to 11.7 4.3
\arrow <1.5mm> [0.25,0.75] from 12.3 5.7 to 12.7 5.3
\arrow <1.5mm> [0.25,0.75] from 13.3 4.7 to 13.7 4.3

\arrow <1.5mm> [0.25,0.75] from 4.3 0.3 to 4.7 0.7
\arrow <1.5mm> [0.25,0.75] from 6.3 0.3 to 6.7 0.7
\arrow <1.5mm> [0.25,0.75] from 8.3 0.3 to 8.7 0.7
\arrow <1.5mm> [0.25,0.75] from 5.3 0.7 to 5.7 0.3
\arrow <1.5mm> [0.25,0.75] from 7.3 0.7 to 7.7 0.3
\arrow <1.5mm> [0.25,0.75] from 9.3 0.7 to 9.7 0.3

\arrow <1.5mm> [0.25,0.75] from 5.3 1.3 to 5.7 1.7
\arrow <1.5mm> [0.25,0.75] from 7.3 1.3 to 7.7 1.7
\arrow <1.5mm> [0.25,0.75] from 6.3 1.7 to 6.7 1.3
\arrow <1.5mm> [0.25,0.75] from 8.3 1.7 to 8.7 1.3

\arrow <1.5mm> [0.25,0.75] from 6.3 2.3 to 6.7 2.7
\arrow <1.5mm> [0.25,0.75] from 7.3 2.7 to 7.7 2.3

\arrow <1.5mm> [0.25,0.75] from 0.3 6.3 to 0.7 6.7
\arrow <1.5mm> [0.25,0.75] from 2.3 6.3 to 2.7 6.7
\arrow <1.5mm> [0.25,0.75] from 10.3 6.3 to 10.7 6.7
\arrow <1.5mm> [0.25,0.75] from 12.3 6.3 to 12.7 6.7
\arrow <1.5mm> [0.25,0.75] from 12.3 4.3 to 12.7 4.7
\arrow <1.5mm> [0.25,0.75] from 0.3 2.3 to 0.7 2.7
\arrow <1.5mm> [0.25,0.75] from 11.3 5.3 to 11.7 5.7
\arrow <1.5mm> [0.25,0.75] from 13.3 5.3 to 13.7 5.7
\arrow <1.5mm> [0.25,0.75] from 13.3 3.3 to 13.7 3.7
\arrow <1.5mm> [0.25,0.75] from 1.3 3.3 to 1.7 3.7

\arrow <1.5mm> [0.25,0.75] from 0.3 3.7 to 0.7 3.3
\arrow <1.5mm> [0.25,0.75] from 0.3 5.7 to 0.7 5.3
\arrow <1.5mm> [0.25,0.75] from 1.3 4.7 to 1.7 4.3
\arrow <1.5mm> [0.25,0.75] from 1.3 6.7 to 1.7 6.3

\arrow <1.5mm> [0.25,0.75] from 2.3 5.7 to 2.7 5.3
\arrow <1.5mm> [0.25,0.75] from 3.3 6.7 to 3.7 6.3

\arrow <1.5mm> [0.25,0.75] from 11.3 6.7 to 11.7 6.3
\arrow <1.5mm> [0.25,0.75] from 13.3 6.7 to 13.7 6.3

\arrow <1.5mm> [0.25,0.75] from 12.3 3.7 to 12.7 3.3
\arrow <1.5mm> [0.25,0.75] from 13.3 2.7 to 13.7 2.3

\arrow <1.5mm> [0.25,0.75] from 2.4 3.9 to 6.7 3.1
\arrow <1.5mm> [0.25,0.75] from 1.4 2.9 to 5.7 2.1
\arrow <1.5mm> [0.25,0.75] from 0.4 1.9 to 4.7 1.1
\arrow <1.5mm> [0.25,0.75] from 0.2 0.8 to 3.7 0.1

\arrow <1.5mm> [0.25,0.75] from 7.3 3.1 to 11.4 3.9
\arrow <1.5mm> [0.25,0.75] from 8.3 2.1 to 12.4 2.9
\arrow <1.5mm> [0.25,0.75] from 9.3 1.1 to 13.4 1.9
\plot 10.3 0.1 13.8 0.8 /

\setdots<1mm>
\plot 4.4 6  5.6 6 /
\plot 6.4 6  7.6 6 /
\plot 8.4 6  9.6 6 /

\setshadegrid span <0.7mm>

\endpicture}
$$
The inserted quarters are labeled {\bf I, II, III, IV,} 
with a bar $b$ as an index: all the modules in such a quarter
are of the form $M(v)$, where $v$ is a word which contains $\overline b$
as a subword (such a quarter will later be seen as part of the tile
$\mathcal T(b)$). 

\subsection{Wind wheels with  non-regular Auslander-Reiten components with
arbitrary ramification}

First, let us present an example with a $3$-ramified component. Here is 
the quiver with the zero relations of length 2 (in addition
all paths of length 3 are zero relations):
$$
\hbox{\beginpicture
\setcoordinatesystem units <1cm,1cm>
\put{} at 0 -2
\put{\beginpicture
\put{} at 0 2
\put{$0$} at 0 0
\put{$1$} at 0 1
\arrow <1.5mm> [0.25,0.75] from 0 0.8 to 0 0.2
\circulararc 155 degrees from 0 2 center at 0 1.5
\circulararc -155 degrees from 0 2 center at 0 1.5
\arrow <1.5mm> [0.25,0.75] from -0.22 1.05 to -0.17 1.02
\setdots <.7mm>
\setquadratic
\plot -.4 1.4 0 1.3 .4 1.4 /

\endpicture} at 0 0 
\put{\beginpicture
\put{} at 0 2
\put{$2$} at 0 0
\put{$3$} at 0 1
\arrow <1.5mm> [0.25,0.75] from 0 0.8 to 0 0.2

\circulararc 155 degrees from 0 2 center at 0 1.5
\circulararc -155 degrees from 0 2 center at 0 1.5
\arrow <1.5mm> [0.25,0.75] from -0.22 1.05 to -0.17 1.02
\setdots <.7mm>
\setquadratic
\plot -.4 1.4 0 1.3 .4 1.4 /

\endpicture} at 2 0 

\put{\beginpicture
\put{} at 0 2
\put{$4$} at 0 0
\put{$5$} at 0 1
\arrow <1.5mm> [0.25,0.75] from 0 0.8 to 0 0.2


\endpicture} at 4 0 
\put{\beginpicture
\put{} at 0 2
\put{$6$} at 0 0
\put{$7$} at 0 1
\arrow <1.5mm> [0.25,0.75] from 0 0.8 to 0 0.2

\circulararc 155 degrees from 0 -1 center at 0 -.5
\circulararc -155 degrees from 0 -1 center at 0 -.5
\arrow <1.5mm> [0.25,0.75] from -0.22 -.05 to -0.17 .02
\setdots <.7mm>
\setquadratic
\plot -.4 -.4 0 -.3 .4 -.4 /

\endpicture} at 6 0

\setquadratic
\plot 0  -1.3   2 -2  4 -1.3 / 

\arrow <1.5mm> [0.25,0.75] from 1.6 -.95 to 0.4 -.95
\arrow <1.5mm> [0.25,0.75] from 3.6 -.95 to 2.4 -.95

\setquadratic
\plot 5.6 0.4   5 0.6    4.4 0.4    /
\plot 5.6 0.0   5 -.2    4.4 0.0    /

\arrow <1.5mm> [0.25,0.75] from 4.45 0.43 to 4.4 .4
\arrow <1.5mm> [0.25,0.75] from 5.55 .0 to 5.6 .03

\setdots <.7mm>
\setquadratic
\plot .6 -1.5  0.2 -1.2 .6 -1.1 /
\plot 1.6 -1.1  2 -1.3 2.4 -1.1 /
\plot 3.4 -1.5  3.8 -1.2 3.4 -1.1 /

\plot 4.7 0.4  4.4 0.2  4.7 0.0 /
\plot 5.3 0.4  5.6 0.2  5.3 0.0 /
\endpicture}
$$

The primitive cyclic word is
$$ 
\hbox{\beginpicture
\setcoordinatesystem units <.4cm,.4cm>
\put{$4$} at 0 0
\put{$0$} at 1 1
\put{$1$} at 2 2
\put{$1$} at 3 1
\put{$0$} at 4 0
\put{$2$} at 5 1
\put{$3$} at 6 2
\put{$3$} at 7 1
\put{$2$} at 8 0
\put{$4$} at 9 1
\put{$5$} at 10 2
\put{$7$} at 11 1
\put{$6$} at 12 0
\put{$6$} at 13 1
\put{$7$} at 14 2
\put{$5$} at 15 1
\put{$4$} at 16 0
\endpicture}
$$
Let us exhibit one part of the $3$-ramified 
Auslander-Reiten component (as  before, the inserted quarters 
are labeled {\bf I, II, III, IV,} with a bar $b$ as an index):

$$ 
\hbox{\beginpicture
\setcoordinatesystem units <.8cm,.8cm>
\multiput{} at 0 -1  14 7 /
\put{
 \beginpicture
 \setcoordinatesystem units <.15cm,.15cm>
 \put{$\ssize 3$} at 4 2
 \put{$\ssize 2$} at 3 1
 \put{$\ssize 0$} at 2 0
 \put{$\ssize 1$} at 1 1
 \put{$\ssize 1$} at 0 2 
\endpicture} at 2 4

\put{
 \beginpicture
 \setcoordinatesystem units <.15cm,.15cm>
 \put{$\ssize 1$} at 4 1
 \put{$\ssize 1$} at 3 2 
\endpicture} at 3 5

\put{$1$} at 4 6

\put{$\circ$} at 1 7

\put{
 \beginpicture
 \setcoordinatesystem units <.15cm,.15cm>
 \put{$\ssize 1$} at 7 1
 \put{$\ssize 1$} at 6 2
 \put{$\ssize 0$} at 5 1
 \put{$\ssize 2$} at 4 2
 \put{$\ssize 3$} at 3 3
\plot 4 1  5 0  7 2  6 3  4 1 /

 \endpicture} at  2 6

\plot 2 7  1 8 /
\plot 2 5  0 3 /
\setquadratic
\plot 2 7  2.6  6  2 5 /
\setlinear

\put{
 \beginpicture
 \setcoordinatesystem units <.15cm,.15cm>
 \put{$\ssize 1$} at 7 2
 \put{$\ssize 0$} at 6 1
 \put{$\ssize 2$} at 5 2
 \put{$\ssize 3$} at 4 3
 
\endpicture} at  3 7

\put{
 \beginpicture
 \setcoordinatesystem units <.15cm,.15cm>
 \put{$\ssize 4$} at 3 0
 \put{$\ssize 0$} at 2 1
 \put{$\ssize 1$} at 1 2
 \put{$\ssize 1$} at 0 1  
\endpicture} at  5 7

\put{
 \beginpicture
 \setcoordinatesystem units <.15cm,.15cm>
 \put{$\ssize 4$} at 5 0
 \put{$\ssize 0$} at 4 1
 \put{$\ssize 1$} at 3 2
\endpicture} at  6 6

\put{
 \beginpicture
 \setcoordinatesystem units <.15cm,.15cm>
 \put{$\ssize 6$} at 5 0
 \put{$\ssize 7$} at 4 1
 \put{$\ssize 5$} at 3 2
\endpicture} at  8 6

\put{$\circ$} at 1 5
\put{$\circ$} at 0 4

\put{$\circ$} at 13 5
\put{$\circ$} at 14 4

\put{$\circ$} at 1 3
\put{$\circ$} at 13 3
\put{$\circ$} at 0 2

\put{\bf III$_{01}$} at -0.3 6

\put{$\circ$} at 14 2

\put{\bf I$_{67}$} at 14.3 6

\put{\bf II$_{45}$} at 7 8
\put{\bf IV\!$_{23}$} at 3 -.7

\put{ 
\beginpicture
 \setcoordinatesystem units <.15cm,.15cm>
 \put{$\ssize 6$} at 5 0
 \put{$\ssize 7$} at 4 1
 \put{$\ssize 5$} at 3 2
 \put{$\ssize 4$} at 2 1
 \put{$\ssize 0$} at 1 2
 \put{$\ssize 1$} at 0 3
\plot 1 1  2 0  4 2  3 3  1 1 /

\endpicture} at  7 7

\plot 6 7  5 8 /
\plot 8 7  9 8 /
\setquadratic
\plot 6 7  7 6.4  8 7 /
\setlinear

\put{
 \beginpicture
 \setcoordinatesystem units <.15cm,.15cm>
 \put{$\ssize 6$} at 5 1
 \put{$\ssize 6$} at 4 0
 \put{$\ssize 7$} at 3 1
 \put{$\ssize 5$} at 2 2
\endpicture} at  9 7

\put{ \beginpicture
 \setcoordinatesystem units <.15cm,.15cm>
 \put{$\ssize 4$} at 5 0
 \put{$\ssize 5$} at 4 1
 \put{$\ssize 7$} at 3 2
 \put{$\ssize 6$} at 2 0
\endpicture} at  11 7

\put{
 \beginpicture
 \setcoordinatesystem units <.15cm,.15cm>
 \put{$\ssize 4$} at 6 0
 \put{$\ssize 5$} at 5 1
 \put{$\ssize 7$} at 4 2
 \put{$\ssize 6$} at 3 1
 \put{$\ssize 6$} at 2 2
\plot 2 1  3 0  5 2  4 3  2 1 /
\endpicture} at  12 6

\plot 12 7  13 8 /
\plot 12 5  14 3 /
\setquadratic
\plot 12 7  11.4  6  12 5 /
\setlinear

\put{$\circ$} at 13 7

\put{$6$} at 10 6 

\put{
 \beginpicture
 \setcoordinatesystem units <.15cm,.15cm>
 \put{$\ssize 6$} at 6 0
 \put{$\ssize 6$} at 5 1
\endpicture} at  11 5

\put{
 \beginpicture
 \setcoordinatesystem units <.15cm,.15cm>
 \put{$\ssize 6$} at 4 0
 \put{$\ssize 6$} at 3 1
 \put{$\ssize 7$} at 2 2
 \put{$\ssize 5$} at 1 1
 \put{$\ssize 4$} at 0 0
\endpicture} at  12 4

\put{
 \beginpicture
 \setcoordinatesystem units <.15cm,.15cm>
 \put{$\ssize 3$} at 1 1
 \put{$\ssize 2$} at 0 0
\endpicture} at  3 3

\put{
 \beginpicture
 \setcoordinatesystem units <.15cm,.15cm>
 \put{$\ssize 3$} at 2 0
 \put{$\ssize 3$} at 1 1
 \put{$\ssize 2$} at 0 0
\endpicture} at  2 2

\put{
 \beginpicture
 \setcoordinatesystem units <.15cm,.15cm>
 \put{$\ssize 3$} at 2 1
 \put{$\ssize 2$} at 1 0
 \put{$\ssize 4$} at 0 1
\endpicture} at  4 2

\put{$\circ$} at 1 1

\put{
 \beginpicture
 \setcoordinatesystem units <.15cm,.15cm>
 \put{$\ssize 3$} at 3 0
 \put{$\ssize 3$} at 2 1
 \put{$\ssize 2$} at 1 0
 \put{$\ssize 4$} at 0 1
\plot 0 0  1 -1  3 1  2 2  0 0 /

\endpicture} at  3 1

\plot 2 1  .5 -.5 /
\plot 4 1  5.5 -.5 /
\setquadratic
\plot 2 1  3 1.6  4 1 /
\setlinear

\put{$\circ$} at 5 1
\put{$\circ$} at 0 0
\put{$\circ$} at 2 0
\put{$\circ$} at 4 0 
\put{$\circ$} at 6 0

\put{
 \beginpicture
 \setcoordinatesystem units <.15cm,.15cm>
 \put{$\ssize 5$} at 1 1
 \put{$\ssize 4$} at 0 0
\endpicture} at  11 3

\put{
 \beginpicture
 \setcoordinatesystem units <.15cm,.15cm>

 \put{$\ssize 0$} at -1 1
 \put{$\ssize 5$} at 1 1
 \put{$\ssize 4$} at 0 0

\endpicture} at  12 2

\put{$\circ$} at 13 1
\put{$\circ$} at 14 0

\arrow <1.5mm> [0.25,0.75] from 11.3 3.3 to 11.6 3.6
\arrow <1.5mm> [0.25,0.75] from 12.3 2.3 to 12.6 2.6
\arrow <1.5mm> [0.25,0.75] from 13.3 1.3 to 13.7 1.7
\arrow <1.5mm> [0.25,0.75] from 14.3  .3 to 14.7  .7

\arrow <1.5mm> [0.25,0.75] from 11.3 2.7 to 11.7 2.3
\arrow <1.5mm> [0.25,0.75] from 12.3 1.7 to 12.7 1.3
\arrow <1.5mm> [0.25,0.75] from 13.3  .7 to 13.7  .3

\arrow <1.5mm> [0.25,0.75] from 0.3 4.3 to 0.7 4.7
\arrow <1.5mm> [0.25,0.75] from 2.3 4.3 to 2.7 4.7
\arrow <1.5mm> [0.25,0.75] from 1.3 5.3 to 1.7 5.7
\arrow <1.5mm> [0.25,0.75] from 3.3 5.3 to 3.7 5.7
\arrow <1.5mm> [0.25,0.75] from 4.3 6.3 to 4.7 6.7
\arrow <1.5mm> [0.25,0.75] from 6.3 6.3 to 6.6 6.6
\arrow <1.5mm> [0.25,0.75] from 8.3 6.3 to 8.7 6.7
\arrow <1.5mm> [0.25,0.75] from 5.3 6.7 to 5.7 6.3
\arrow <1.5mm> [0.25,0.75] from 7.3 6.7 to 7.7 6.3
\arrow <1.5mm> [0.25,0.75] from 9.3 6.7 to 9.7 6.3
\arrow <1.5mm> [0.25,0.75] from 10.3 5.7 to 10.7 5.3
\arrow <1.5mm> [0.25,0.75] from 11.3 4.7 to 11.7 4.3
\arrow <1.5mm> [0.25,0.75] from 12.3 5.7 to 12.7 5.3
\arrow <1.5mm> [0.25,0.75] from 13.3 4.7 to 13.7 4.3

\arrow <1.5mm> [0.25,0.75] from 0.3 0.3 to 0.7 0.7
\arrow <1.5mm> [0.25,0.75] from 2.3 0.3 to 2.7 0.7
\arrow <1.5mm> [0.25,0.75] from 4.3 0.3 to 4.7 0.7
\arrow <1.5mm> [0.25,0.75] from 1.3 0.7 to 1.7 0.3
\arrow <1.5mm> [0.25,0.75] from 3.3 0.7 to 3.7 0.3
\arrow <1.5mm> [0.25,0.75] from 5.3 0.7 to 5.7 0.3

\arrow <1.5mm> [0.25,0.75] from 1.3 1.3 to 1.7 1.7
\arrow <1.5mm> [0.25,0.75] from 3.3 1.3 to 3.7 1.7
\arrow <1.5mm> [0.25,0.75] from 2.3 1.7 to 2.7 1.3
\arrow <1.5mm> [0.25,0.75] from 4.3 1.7 to 4.7 1.3
\arrow <1.5mm> [0.25,0.75] from 2.3 2.3 to 2.7 2.7
\arrow <1.5mm> [0.25,0.75] from 3.3 2.7 to 3.7 2.3

\arrow <1.5mm> [0.25,0.75] from 0.3 6.3 to 0.7 6.7
\arrow <1.5mm> [0.25,0.75] from 2.3 6.3 to 2.7 6.7
\arrow <1.5mm> [0.25,0.75] from 10.3 6.3 to 10.7 6.7
\arrow <1.5mm> [0.25,0.75] from 12.3 6.3 to 12.7 6.7
\arrow <1.5mm> [0.25,0.75] from 12.3 4.3 to 12.7 4.7
\arrow <1.5mm> [0.25,0.75] from 0.3 2.3 to 0.7 2.7
\arrow <1.5mm> [0.25,0.75] from 11.3 5.3 to 11.7 5.7
\arrow <1.5mm> [0.25,0.75] from 13.3 5.3 to 13.7 5.7
\arrow <1.5mm> [0.25,0.75] from 13.3 3.3 to 13.7 3.7
\arrow <1.5mm> [0.25,0.75] from 1.3 3.3 to 1.7 3.7

\arrow <1.5mm> [0.25,0.75] from 0.3 3.7 to 0.7 3.3
\arrow <1.5mm> [0.25,0.75] from 0.3 5.7 to 0.7 5.3
\arrow <1.5mm> [0.25,0.75] from 1.3 4.7 to 1.7 4.3
\arrow <1.5mm> [0.25,0.75] from 1.3 6.7 to 1.7 6.3

\arrow <1.5mm> [0.25,0.75] from 2.3 5.7 to 2.7 5.3
\arrow <1.5mm> [0.25,0.75] from 3.3 6.7 to 3.7 6.3

\arrow <1.5mm> [0.25,0.75] from 11.3 6.7 to 11.7 6.3
\arrow <1.5mm> [0.25,0.75] from 13.3 6.7 to 13.7 6.3

\arrow <1.5mm> [0.25,0.75] from 12.3 3.7 to 12.7 3.3
\arrow <1.5mm> [0.25,0.75] from 13.3 2.7 to 13.7 2.3

\arrow <1.5mm> [0.25,0.75] from 2.4 3.6 to 2.7 3.3
\arrow <1.5mm> [0.25,0.75] from 1.4 2.6 to 1.7 2.3
\arrow <1.5mm> [0.25,0.75] from 0.4 1.6 to 0.7 1.3

\setdots<1mm>
\plot 4.4 6  5.6 6 /
\plot 6.4 6  7.6 6 /
\plot 8.4 6  9.6 6 /

\setshadegrid span <0.7mm>

\endpicture}
$$

This concerns the part of the component containing 
the module $I(0) = 
 \beginpicture
 \setcoordinatesystem units <.15cm,.15cm>
 \put{$\ssize 3$} at 4 2
 \put{$\ssize 2$} at 3 1
 \put{$\ssize 0$} at 2 0
 \put{$\ssize 1$} at 1 1
 \put{$\ssize 1$} at 0 2 
\endpicture.
$
The leaves containing the modules $I(6)$ and $I(2)$ look
similar. These three leaves together form a component,
namely a 3-ramified component of type $\Bbb A_\infty^\infty$.
$$ 
\hbox{\beginpicture
\setcoordinatesystem units <.8cm,.55cm>
\put{\beginpicture
\put{$1$} at 0 0

\put{
 \beginpicture
 \setcoordinatesystem units <.15cm,.15cm>
 \put{$\ssize 1$} at 0 2
 \put{$\ssize 0$} at 1 1
 \put{$\ssize 4$} at 2 0
\endpicture} at  2 0

\put{
 \beginpicture
 \setcoordinatesystem units <.15cm,.15cm>
 \put{$\ssize 5$} at 0 2
 \put{$\ssize 7$} at 1 1
 \put{$\ssize 6$} at 2 0
\endpicture} at  4 0

\put{$6$} at 6 0 

\put{
 \beginpicture
 \setcoordinatesystem units <.15cm,.15cm>
 \put{$\ssize 6$} at 1 1
 \put{$\ssize 6$} at 2 0
\endpicture} at  7 -1

\put{
 \beginpicture
 \setcoordinatesystem units <.15cm,.15cm>
 \put{$\ssize 4$} at 0 0
 \put{$\ssize 5$} at 1 1
 \put{$\ssize 7$} at 2 2
 \put{$\ssize 6$} at 3 1
 \put{$\ssize 6$} at 4 0
\endpicture} at  8 -2

\put{
 \beginpicture
 \setcoordinatesystem units <.15cm,.15cm>
 \put{$\ssize 4$} at 1 1
 \put{$\ssize 5$} at 2 2
\endpicture} at 7 -3

\put{
 \beginpicture
 \setcoordinatesystem units <.15cm,.15cm>
 \put{$\ssize 1$} at 1 1
 \put{$\ssize 1$} at 2 0
\endpicture} at  -1 -1

\put{
 \beginpicture
 \setcoordinatesystem units <.15cm,.15cm>
 \put{$\ssize 1$} at 0 2
 \put{$\ssize 1$} at 1 1
 \put{$\ssize 0$} at 2 0
 \put{$\ssize 2$} at 3 1
 \put{$\ssize 3$} at 4 2
\endpicture} at  -2 -2

\put{
 \beginpicture
 \setcoordinatesystem units <.15cm,.15cm>
 \put{$\ssize 2$} at 1 1
 \put{$\ssize 3$} at 2 2
\endpicture} at  -1 -3

\setdots<1mm>
\plot 0.4 0  1.6 0 /
\plot 2.4 0  3.6 0 /
\plot 4.4 0  5.6 0 /

\setsolid
\arrow <1.5mm> [0.25,0.75] from -.7 -0.7 to -.3 -.3
\arrow <1.5mm> [0.25,0.75] from -1.6 -1.6 to -1.3 -1.3
\arrow <1.5mm> [0.25,0.75] from -1.6 -2.4 to -1.3 -2.7

\arrow <1.5mm> [0.25,0.75] from 6.4 -0.4 to 6.7 -.7
\arrow <1.5mm> [0.25,0.75] from 7.4 -1.4 to 7.7 -1.7
\arrow <1.5mm> [0.25,0.75] from 7.4 -2.7 to 7.7 -2.4

\setshadegrid span <.7mm>
\vshade -3 -4 1  <z,z,,>  -2 -4 1 /
\vshade -2 -4 -2   <z,z,,>  -1 -4 -3 /
\vshade -2 -2 1  <z,z,,>   -0 0 1  <z,z,,>   6 0 1  <z,z,,>   7.99 -2 1  <z,z,,>   8  -4 1  <z,z,,>    9 -4 1 / 
\vshade 7 -4 -3   <z,z,,>   8 -4 -2 /

\multiput{} at 0 -4  0 1 /  
\endpicture} at 0 0
\put{\beginpicture
\put{$5$} at 0 0

\put{
 \beginpicture
 \setcoordinatesystem units <.15cm,.15cm>
 \put{$\ssize 7$} at 0 2
 \put{$\ssize 6$} at 1 1
 \put{$\ssize 6$} at 2 0
\endpicture} at  2 0

\put{
 \beginpicture
 \setcoordinatesystem units <.15cm,.15cm>
 \put{$\ssize 7$} at 0 2
 \put{$\ssize 5$} at 1 1
 \put{$\ssize 4$} at 2 0
\endpicture} at  4 0

\put{$0$} at 6 0 

\put{
 \beginpicture
 \setcoordinatesystem units <.15cm,.15cm>
 \put{$\ssize 2$} at 1 1
 \put{$\ssize 0$} at 2 0
\endpicture} at  7 -1

\put{
 \beginpicture
 \setcoordinatesystem units <.15cm,.15cm>
 \put{$\ssize 2$} at 0 0
 \put{$\ssize 3$} at 1 1
 \put{$\ssize 3$} at 2 2
 \put{$\ssize 2$} at 3 1
 \put{$\ssize 0$} at 4 0
\endpicture} at  8 -2

\put{
 \beginpicture
 \setcoordinatesystem units <.15cm,.15cm>
 \put{$\ssize 2$} at 1 1
 \put{$\ssize 3$} at 2 2
\endpicture} at 7 -3

\put{
 \beginpicture
 \setcoordinatesystem units <.15cm,.15cm>
 \put{$\ssize 5$} at 1 1
 \put{$\ssize 7$} at 2 0
\endpicture} at  -1 -1

\put{
 \beginpicture
 \setcoordinatesystem units <.15cm,.15cm>
 \put{$\ssize 5$} at 0 2
 \put{$\ssize 7$} at 1 1
 \put{$\ssize 6$} at 2 0
 \put{$\ssize 6$} at 3 1
 \put{$\ssize 7$} at 4 2
\endpicture} at  -2 -2

\put{
 \beginpicture
 \setcoordinatesystem units <.15cm,.15cm>
 \put{$\ssize 6$} at 1 1
 \put{$\ssize 7$} at 2 2
\endpicture} at  -1 -3

\setdots<1mm>
\plot 0.4 0  1.6 0 /
\plot 2.4 0  3.6 0 /
\plot 4.4 0  5.6 0 /

\setsolid
\arrow <1.5mm> [0.25,0.75] from -.7 -0.7 to -.3 -.3
\arrow <1.5mm> [0.25,0.75] from -1.6 -1.6 to -1.3 -1.3
\arrow <1.5mm> [0.25,0.75] from -1.6 -2.4 to -1.3 -2.7

\arrow <1.5mm> [0.25,0.75] from 6.4 -0.4 to 6.7 -.7
\arrow <1.5mm> [0.25,0.75] from 7.4 -1.4 to 7.7 -1.7
\arrow <1.5mm> [0.25,0.75] from 7.4 -2.7 to 7.7 -2.4

\setshadegrid span <.7mm>
\vshade -3 -4 1  <z,z,,>  -2 -4 1 /
\vshade -2 -4 -2   <z,z,,>  -1 -4 -3 /
\vshade -2 -2 1  <z,z,,>   -0 0 1  <z,z,,>   6 0 1  <z,z,,>   7.99 -2 1  <z,z,,>   8  -4 1  <z,z,,>    9 -4 1 / 
\vshade 7 -4 -3   <z,z,,>   8 -4 -2 /

\multiput{} at 0 -4  0 1 /  
\endpicture} at 0 -6
\put{\beginpicture
\put{$3$} at 0 0

\put{
 \beginpicture
 \setcoordinatesystem units <.15cm,.15cm>
 \put{$\ssize 3$} at 0 2
 \put{$\ssize 2$} at 1 1
 \put{$\ssize 0$} at 2 0
\endpicture} at  2 0

\put{
 \beginpicture
 \setcoordinatesystem units <.15cm,.15cm>
 \put{$\ssize 1$} at 0 2
 \put{$\ssize 1$} at 1 1
 \put{$\ssize 0$} at 2 0
\endpicture} at  4 0

\put{$2$} at 6 0 

\put{
 \beginpicture
 \setcoordinatesystem units <.15cm,.15cm>
 \put{$\ssize 4$} at 1 1
 \put{$\ssize 2$} at 2 0
\endpicture} at  7 -1

\put{
 \beginpicture
 \setcoordinatesystem units <.15cm,.15cm>
 \put{$\ssize 6$} at 0 0
 \put{$\ssize 7$} at 1 1
 \put{$\ssize 5$} at 2 2
 \put{$\ssize 4$} at 3 1
 \put{$\ssize 2$} at 4 0
\endpicture} at  8 -2

\put{
 \beginpicture
 \setcoordinatesystem units <.15cm,.15cm>
 \put{$\ssize 6$} at 1 1
 \put{$\ssize 7$} at 2 2
\endpicture} at 7 -3

\put{
 \beginpicture
 \setcoordinatesystem units <.15cm,.15cm>
 \put{$\ssize 3$} at 1 1
 \put{$\ssize 3$} at 2 0
\endpicture} at  -1 -1

\put{
 \beginpicture
 \setcoordinatesystem units <.15cm,.15cm>
 \put{$\ssize 3$} at 0 2
 \put{$\ssize 3$} at 1 1
 \put{$\ssize 2$} at 2 0
 \put{$\ssize 4$} at 3 1
 \put{$\ssize 5$} at 4 2
\endpicture} at  -2 -2

\put{
 \beginpicture
 \setcoordinatesystem units <.15cm,.15cm>
 \put{$\ssize 4$} at 1 1
 \put{$\ssize 5$} at 2 2
\endpicture} at  -1 -3

\setdots<1mm>
\plot 0.4 0  1.6 0 /
\plot 2.4 0  3.6 0 /
\plot 4.4 0  5.6 0 /

\setsolid
\arrow <1.5mm> [0.25,0.75] from -.7 -0.7 to -.3 -.3
\arrow <1.5mm> [0.25,0.75] from -1.6 -1.6 to -1.3 -1.3
\arrow <1.5mm> [0.25,0.75] from -1.6 -2.4 to -1.3 -2.7

\arrow <1.5mm> [0.25,0.75] from 6.4 -0.4 to 6.7 -.7
\arrow <1.5mm> [0.25,0.75] from 7.4 -1.4 to 7.7 -1.7
\arrow <1.5mm> [0.25,0.75] from 7.4 -2.7 to 7.7 -2.4

\setshadegrid span <.7mm>
\vshade -3 -4 1  <z,z,,>  -2 -4 1 /
\vshade -2 -4 -2   <z,z,,>  -1 -4 -3 /
\vshade -2 -2 1  <z,z,,>   -0 0 1  <z,z,,>   6 0 1  <z,z,,>   7.99 -2 1  <z,z,,>   8  -4 1  <z,z,,>    9 -4 1 / 
\vshade 7 -4 -3   <z,z,,>   8 -4 -2 /

\multiput{} at 0 -4  0 1 /  
\endpicture} at 0 -12

\endpicture}
$$

In addition there is a second non-regular Auslander-Reiten component, 
namely the component containing the module $I(4)$; 
it is $1$-ramified. The boundary looks as follows:
$$ 
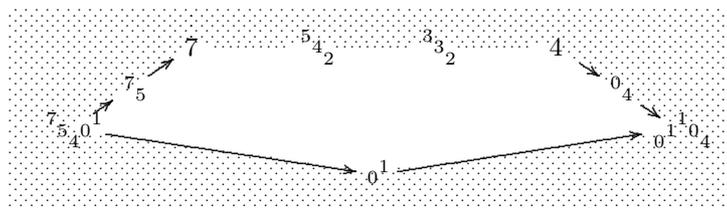
\hbox{\beginpicture
\setcoordinatesystem units <.8cm,.55cm>
\put{$7$} at 0 0

\put{
 \beginpicture
 \setcoordinatesystem units <.15cm,.15cm>
 \put{$\ssize 5$} at 0 2
 \put{$\ssize 4$} at 1 1
 \put{$\ssize 2$} at 2 0
\endpicture} at  2 0

\put{
 \beginpicture
 \setcoordinatesystem units <.15cm,.15cm>
 \put{$\ssize 3$} at 0 2
 \put{$\ssize 3$} at 1 1
 \put{$\ssize 2$} at 2 0
\endpicture} at  4 0

\put{$4$} at 6 0 

\put{
 \beginpicture
 \setcoordinatesystem units <.15cm,.15cm>
 \put{$\ssize 0$} at 1 1
 \put{$\ssize 4$} at 2 0
\endpicture} at  7 -1

\put{
 \beginpicture
 \setcoordinatesystem units <.15cm,.15cm>
 \put{$\ssize 0$} at 0 0
 \put{$\ssize 1$} at 1 1
 \put{$\ssize 1$} at 2 2
 \put{$\ssize 0$} at 3 1
 \put{$\ssize 4$} at 4 0
\endpicture} at  8 -2

\put{
 \beginpicture
 \setcoordinatesystem units <.15cm,.15cm>
 \put{$\ssize 7$} at 1 1
 \put{$\ssize 5$} at 2 0
\endpicture} at  -1 -1

\put{
 \beginpicture
 \setcoordinatesystem units <.15cm,.15cm>
 \put{$\ssize 7$} at 0 2
 \put{$\ssize 5$} at 1 1
 \put{$\ssize 4$} at 2 0
 \put{$\ssize 0$} at 3 1
 \put{$\ssize 1$} at 4 2
\endpicture} at  -2 -2

\put{
 \beginpicture
 \setcoordinatesystem units <.15cm,.15cm>
 \put{$\ssize 0$} at 1 1
 \put{$\ssize 1$} at 2 2
\endpicture} at  3 -3

\setdots<1mm>
\plot 0.4 0  1.6 0 /
\plot 2.4 0  3.6 0 /
\plot 4.4 0  5.6 0 /

\setsolid
\arrow <1.5mm> [0.25,0.75] from -.7 -0.7 to -.3 -.3
\arrow <1.5mm> [0.25,0.75] from -1.6 -1.6 to -1.3 -1.3
\arrow <1.5mm> [0.25,0.75] from -1.4 -2.1 to 2.7 -3

\arrow <1.5mm> [0.25,0.75] from 6.4 -0.4 to 6.7 -.7
\arrow <1.5mm> [0.25,0.75] from 7.4 -1.4 to 7.7 -1.7
\arrow <1.5mm> [0.25,0.75] from 3.4 -3 to 7.4 -2.1

\setshadegrid span <.7mm>
\vshade -3 -4 1  <z,z,,>  -2 -4 1 /
\vshade -2 -4 -2   <z,z,,>  3 -4 -3 /
\vshade -2 -2 1  <z,z,,>   -0 0 1  <z,z,,>   6 0 1  <z,z,,>   7.99 -2 1  <z,z,,>   8  -4 1  <z,z,,>    9 -4 1 / 
\vshade 3 -4 -3   <z,z,,>   8 -4 -2 /

\multiput{} at 0 -4  0 1 /  
\endpicture}
$$

We use Galois coverings of this wind wheel in order to exhibit
wind wheels with arbitrary ramification. Let us 
consider the Galois covering obtained by an $s$-covering of the
cycle of length $2$, thus we deal with a wind wheel of the following shape
(for $s = 3$, we have to identify the upper left hand arrow with the upper
right hand arrow in order to have an arrow $7' \to {}'5$):
$$
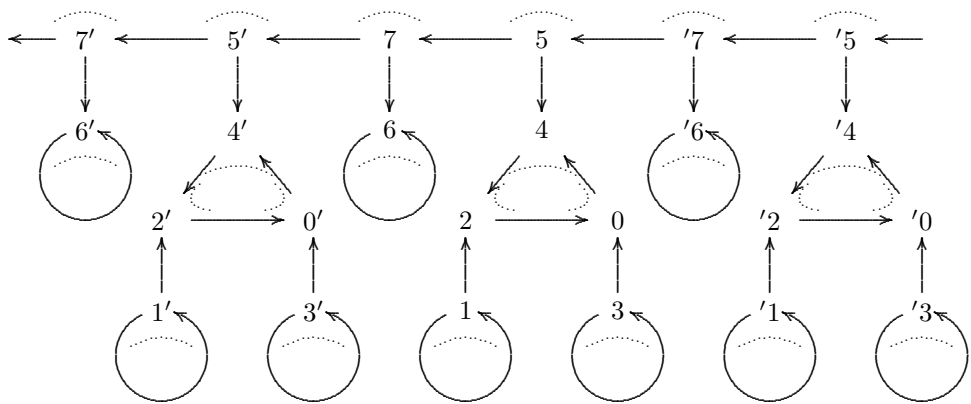
\hbox{\beginpicture
\setcoordinatesystem units <1cm,1.2cm>
\multiput{} at 0 -4.5  0 0.2 /
\put{$7'$} at 0 0
\put{$5'$} at 2 0
\put{$7$} at 4 0
\put{$5$} at 6 0
\put{$'7$} at 8 0
\put{$'5$} at 10 0

\put{$6'$} at 0 -1
\put{$4'$} at 2 -1
\put{$6$} at 4 -1
\put{$4$} at 6 -1 
\put{$'6$} at 8 -1
\put{$'4$} at 10 -1

\put{$2'$} at 1 -2
\put{$0'$} at 3 -2
\put{$2$} at 5 -2
\put{$0$} at 7 -2
\put{$'2$} at 9 -2
\put{$'0$} at 11 -2

\put{$1'$} at 1 -3
\put{$3'$} at 3 -3
\put{$1$} at 5 -3
\put{$3$} at 7 -3
\put{$'1$} at 9 -3
\put{$'3$} at 11 -3

\arrow <1.5mm> [0.25,0.75] from 0 -.2 to 0 -.8
\arrow <1.5mm> [0.25,0.75] from 2 -.2 to 2 -.8
\arrow <1.5mm> [0.25,0.75] from 4 -.2 to 4 -.8
\arrow <1.5mm> [0.25,0.75] from 6 -.2 to 6 -.8
\arrow <1.5mm> [0.25,0.75] from 8 -.2 to 8 -.8
\arrow <1.5mm> [0.25,0.75] from 10 -.2 to 10 -.8

\arrow <1.5mm> [0.25,0.75] from 1 -2.8 to 1 -2.2 
\arrow <1.5mm> [0.25,0.75] from 3 -2.8 to 3 -2.2 
\arrow <1.5mm> [0.25,0.75] from 5 -2.8 to 5 -2.2 
\arrow <1.5mm> [0.25,0.75] from 7 -2.8 to 7 -2.2 
\arrow <1.5mm> [0.25,0.75] from 9 -2.8 to 9 -2.2 
\arrow <1.5mm> [0.25,0.75] from 11 -2.8 to 11 -2.2

\circulararc -155 degrees from 0 -2 center at 0 -1.5
\circulararc  155 degrees from 0 -2 center at 0 -1.5
\arrow <1.5mm> [0.25,0.75] from 0.22 -1.04 to 0.17 -1.02
\circulararc -155 degrees from 4 -2 center at 4 -1.5
\circulararc  155 degrees from 4 -2 center at 4 -1.5
\arrow <1.5mm> [0.25,0.75] from 4.22 -1.04 to 4.17 -1.02
\circulararc -155 degrees from 8 -2 center at 8 -1.5
\circulararc  155 degrees from 8 -2 center at 8 -1.5
\arrow <1.5mm> [0.25,0.75] from 8.22 -1.04 to 8.17 -1.02
\setdots <.7mm>
\setquadratic
\plot -.4 -1.4 0 -1.3 .4 -1.4 /
\plot 3.6 -1.4 4 -1.3 4.4 -1.4 /
\plot 7.6 -1.4 8 -1.3 8.4 -1.4 /

\setsolid
\circulararc -155 degrees from 1 -4 center at 1 -3.5
\circulararc  155 degrees from 1 -4 center at 1 -3.5
\arrow <1.5mm> [0.25,0.75] from 1.22 -3.04 to 1.17 -3.02
\circulararc -155 degrees from 3 -4 center at 3 -3.5
\circulararc  155 degrees from 3 -4 center at 3 -3.5
\arrow <1.5mm> [0.25,0.75] from 3.22 -3.04 to 3.17 -3.02
\circulararc -155 degrees from 5 -4 center at 5 -3.5
\circulararc  155 degrees from 5 -4 center at 5 -3.5
\arrow <1.5mm> [0.25,0.75] from 5.22 -3.04 to 5.17 -3.02
\circulararc -155 degrees from 7 -4 center at 7 -3.5
\circulararc  155 degrees from 7 -4 center at 7 -3.5
\arrow <1.5mm> [0.25,0.75] from 7.22 -3.04 to 7.17 -3.02
\circulararc -155 degrees from 9 -4 center at 9 -3.5
\circulararc  155 degrees from 9 -4 center at 9 -3.5
\arrow <1.5mm> [0.25,0.75] from 9.22 -3.04 to 9.17 -3.02
\circulararc -155 degrees from 11 -4 center at 11 -3.5
\circulararc  155 degrees from 11 -4 center at 11 -3.5
\arrow <1.5mm> [0.25,0.75] from 11.22 -3.04 to 11.17 -3.02
\setdots <.7mm>
\setquadratic
\plot .6 -3.4  1 -3.3 1.4 -3.4 /
\plot 2.6 -3.4  3 -3.3 3.4 -3.4 /
\plot 4.6 -3.4  5 -3.3 5.4 -3.4 /
\plot 6.6 -3.4  7 -3.3 7.4 -3.4 /
\plot 8.6 -3.4  9 -3.3 9.4 -3.4 /
\plot 10.6 -3.4  11 -3.3 11.4 -3.4 /

\setsolid
\arrow <1.5mm> [0.25,0.75] from 1.7 -1.3 to 1.3 -1.7
\arrow <1.5mm> [0.25,0.75] from 2.7 -1.7 to 2.3 -1.3
\arrow <1.5mm> [0.25,0.75] from 1.4 -2 to 2.6 -2 

\arrow <1.5mm> [0.25,0.75] from 5.7 -1.3 to 5.3 -1.7
\arrow <1.5mm> [0.25,0.75] from 6.7 -1.7 to 6.3 -1.3
\arrow <1.5mm> [0.25,0.75] from 5.4 -2 to 6.6 -2 

\arrow <1.5mm> [0.25,0.75] from 9.7 -1.3 to 9.3 -1.7
\arrow <1.5mm> [0.25,0.75] from 10.7 -1.7 to 10.3 -1.3
\arrow <1.5mm> [0.25,0.75] from 9.4 -2 to 10.6 -2 

\arrow <1.5mm> [0.25,0.75] from -.4 0 to -1 0
\arrow <1.5mm> [0.25,0.75] from 1.6 0 to .4 0
\arrow <1.5mm> [0.25,0.75] from 3.6 0 to 2.4 0
\arrow <1.5mm> [0.25,0.75] from 5.6 0 to 4.4 0
\arrow <1.5mm> [0.25,0.75] from 7.6 0 to 6.4 0
\arrow <1.5mm> [0.25,0.75] from 9.6 0 to 8.4 0
\arrow <1.5mm> [0.25,0.75] from 11 0 to 10.4 0

\setdots <.7mm>
\plot 2.45 -1.5  2 -1.4  1.55 -1.5 /
\plot 6.45 -1.5  6 -1.4  5.55 -1.5 /
\plot 10.45 -1.5  10 -1.4  9.55 -1.5 /

\plot 1.5 -1.6  1.4 -1.8  1.7 -1.9 /
\plot 2.5 -1.6  2.6 -1.8  2.3 -1.9 /

\plot 5.5 -1.6  5.4 -1.8  5.7 -1.9 /
\plot 6.5 -1.6  6.6 -1.8  6.3 -1.9 /

\plot 9.5 -1.6  9.4 -1.8  9.7 -1.9 /
\plot 10.5 -1.6  10.6 -1.8  10.3 -1.9 /

\plot -.4 0.2  0 0.3  .4 .2 /
\plot 1.6 0.2  2 0.3  2.4 .2 /
\plot 3.6 0.2  4 0.3  4.4 .2 /
\plot 5.6 0.2  6 0.3  6.4 .2 /
\plot 7.6 0.2  8 0.3  8.4 .2 /
\plot 9.6 0.2  10 0.3  10.4 .2 /

\endpicture}
$$

The corresponding primitive cyclic word is
$$ 
\hbox{\beginpicture
\setcoordinatesystem units <.35cm,.4cm>
\put{$'2$} at -8 0
\put{$'4$} at -7 1
\put{$'5$} at -6 2
\put{$'7$} at -5 1
\put{$'6$} at -4 0
\put{$'6$} at -3 1
\put{$'7$} at -2 2
\put{$5$} at -1 1
\put{$4$} at 0 0
\put{$0$} at 1 1
\put{$1$} at 2 2
\put{$1$} at 3 1
\put{$0$} at 4 0
\put{$2$} at 5 1
\put{$3$} at 6 2
\put{$3$} at 7 1
\put{$2$} at 8 0
\put{$4$} at 9 1
\put{$5$} at 10 2
\put{$7$} at 11 1
\put{$6$} at 12 0
\put{$6$} at 13 1
\put{$7$} at 14 2
\put{$5'$} at 15 1
\put{$4'$} at 16 0
\put{$0'$} at 17 1
\put{$1'$} at 18 2
\put{$1'$} at 19 1
\put{$0'$} at 20 0
\put{$2'$} at 21 1
\put{$3'$} at 22 2
\put{$3'$} at 23 1
\put{$2'$} at 24 0
\put{$4'$} at 25 1
\put{$5'$} at 26 2
\put{$7'$} at 27 1
\put{$6'$} at 28 0
\multiput{$\cdots$} at 29 1  -9 1 /
\endpicture}
$$

Let us show the leaves which contain
the modules $I(0), I(6), I(2)$, they are quite similar to those seen above
--- the only difference occurs on the right hand side of the upper leaf
and in the upper row of the middle leaf:
$$ 
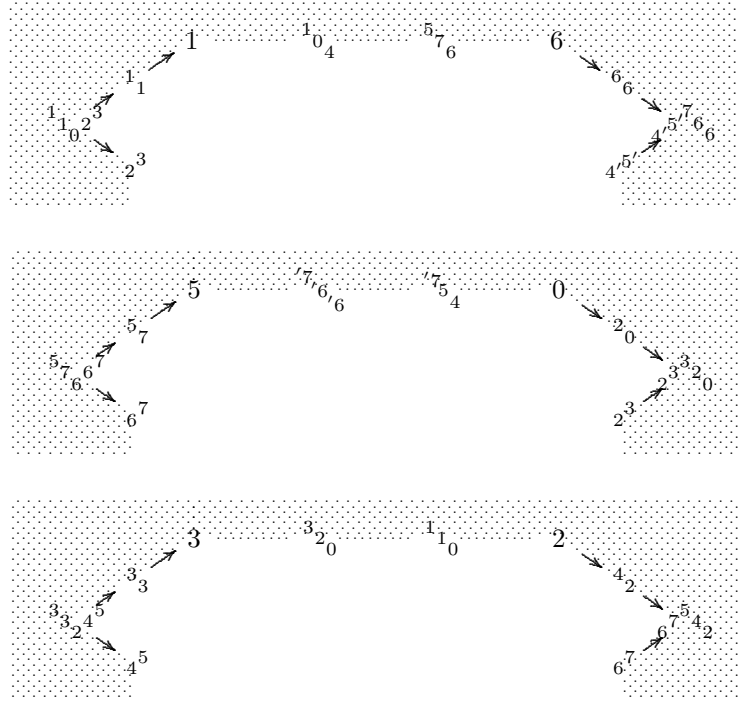
\hbox{\beginpicture
\setcoordinatesystem units <.8cm,.55cm>
\put{\beginpicture
\put{$1$} at 0 0

\put{
 \beginpicture
 \setcoordinatesystem units <.15cm,.15cm>
 \put{$\ssize 1$} at 0 2
 \put{$\ssize 0$} at 1 1
 \put{$\ssize 4$} at 2 0
\endpicture} at  2 0

\put{
 \beginpicture
 \setcoordinatesystem units <.15cm,.15cm>
 \put{$\ssize 5$} at 0 2
 \put{$\ssize 7$} at 1 1
 \put{$\ssize 6$} at 2 0
\endpicture} at  4 0

\put{$6$} at 6 0 

\put{
 \beginpicture
 \setcoordinatesystem units <.15cm,.15cm>
 \put{$\ssize 6$} at 1 1
 \put{$\ssize 6$} at 2 0
\endpicture} at  7 -1

\put{
 \beginpicture
 \setcoordinatesystem units <.15cm,.15cm>
 \put{$\ssize 4'$} at -.4 0
 \put{$\ssize 5'$} at 1 1
 \put{$\ssize 7$} at 2 2
 \put{$\ssize 6$} at 3 1
 \put{$\ssize 6$} at 4 0
\endpicture} at  8 -2

\put{
 \beginpicture
 \setcoordinatesystem units <.15cm,.15cm>
 \put{$\ssize 4'$} at .6 1
 \put{$\ssize 5'$} at 2 2
\endpicture} at 7 -3

\put{
 \beginpicture
 \setcoordinatesystem units <.15cm,.15cm>
 \put{$\ssize 1$} at 1 1
 \put{$\ssize 1$} at 2 0
\endpicture} at  -1 -1

\put{
 \beginpicture
 \setcoordinatesystem units <.15cm,.15cm>
 \put{$\ssize 1$} at 0 2
 \put{$\ssize 1$} at 1 1
 \put{$\ssize 0$} at 2 0
 \put{$\ssize 2$} at 3 1
 \put{$\ssize 3$} at 4 2
\endpicture} at  -2 -2

\put{
 \beginpicture
 \setcoordinatesystem units <.15cm,.15cm>
 \put{$\ssize 2$} at 1 1
 \put{$\ssize 3$} at 2 2
\endpicture} at  -1 -3

\setdots<1mm>
\plot 0.4 0  1.6 0 /
\plot 2.4 0  3.6 0 /
\plot 4.4 0  5.6 0 /

\setsolid
\arrow <1.5mm> [0.25,0.75] from -.7 -0.7 to -.3 -.3
\arrow <1.5mm> [0.25,0.75] from -1.6 -1.6 to -1.3 -1.3
\arrow <1.5mm> [0.25,0.75] from -1.6 -2.4 to -1.3 -2.7

\arrow <1.5mm> [0.25,0.75] from 6.4 -0.4 to 6.7 -.7
\arrow <1.5mm> [0.25,0.75] from 7.4 -1.4 to 7.7 -1.7
\arrow <1.5mm> [0.25,0.75] from 7.4 -2.7 to 7.7 -2.4

\setshadegrid span <.7mm>
\vshade -3 -4 1  <z,z,,>  -2 -4 1 /
\vshade -2 -4 -2   <z,z,,>  -1 -4 -3 /
\vshade -2 -2 1  <z,z,,>   -0 0 1  <z,z,,>   6 0 1  <z,z,,>   7.99 -2 1  <z,z,,>   8  -4 1  <z,z,,>    9 -4 1 / 
\vshade 7 -4 -3   <z,z,,>   8 -4 -2 /

\multiput{} at 0 -4  0 1 /  
\endpicture} at 0 0
\put{\beginpicture
\put{$5$} at 0 0

\put{
 \beginpicture
 \setcoordinatesystem units <.2cm,.18cm>
 \put{$\ssize {}'7$} at 0 2
 \put{$\ssize {}'6$} at 1 1
 \put{$\ssize {}'6$} at 2.1 0
\endpicture} at  2 0

\put{
 \beginpicture
 \setcoordinatesystem units <.15cm,.15cm>
 \put{$\ssize {}'7$} at 0 2
 \put{$\ssize 5$} at 1 1
 \put{$\ssize 4$} at 2 0
\endpicture} at  4 0

\put{$0$} at 6 0 

\put{
 \beginpicture
 \setcoordinatesystem units <.15cm,.15cm>
 \put{$\ssize 2$} at 1 1
 \put{$\ssize 0$} at 2 0
\endpicture} at  7 -1

\put{
 \beginpicture
 \setcoordinatesystem units <.15cm,.15cm>
 \put{$\ssize 2$} at 0 0
 \put{$\ssize 3$} at 1 1
 \put{$\ssize 3$} at 2 2
 \put{$\ssize 2$} at 3 1
 \put{$\ssize 0$} at 4 0
\endpicture} at  8 -2

\put{
 \beginpicture
 \setcoordinatesystem units <.15cm,.15cm>
 \put{$\ssize 2$} at 1 1
 \put{$\ssize 3$} at 2 2
\endpicture} at 7 -3

\put{
 \beginpicture
 \setcoordinatesystem units <.15cm,.15cm>
 \put{$\ssize 5$} at 1 1
 \put{$\ssize 7$} at 2 0
\endpicture} at  -1 -1

\put{
 \beginpicture
 \setcoordinatesystem units <.15cm,.15cm>
 \put{$\ssize 5$} at 0 2
 \put{$\ssize 7$} at 1 1
 \put{$\ssize 6$} at 2 0
 \put{$\ssize 6$} at 3 1
 \put{$\ssize 7$} at 4 2
\endpicture} at  -2 -2

\put{
 \beginpicture
 \setcoordinatesystem units <.15cm,.15cm>
 \put{$\ssize 6$} at 1 1
 \put{$\ssize 7$} at 2 2
\endpicture} at  -1 -3

\setdots<1mm>
\plot 0.4 0  1.6 0 /
\plot 2.4 0  3.6 0 /
\plot 4.4 0  5.6 0 /

\setsolid
\arrow <1.5mm> [0.25,0.75] from -.7 -0.7 to -.3 -.3
\arrow <1.5mm> [0.25,0.75] from -1.6 -1.6 to -1.3 -1.3
\arrow <1.5mm> [0.25,0.75] from -1.6 -2.4 to -1.3 -2.7

\arrow <1.5mm> [0.25,0.75] from 6.4 -0.4 to 6.7 -.7
\arrow <1.5mm> [0.25,0.75] from 7.4 -1.4 to 7.7 -1.7
\arrow <1.5mm> [0.25,0.75] from 7.4 -2.7 to 7.7 -2.4

\setshadegrid span <.7mm>
\vshade -3 -4 1  <z,z,,>  -2 -4 1 /
\vshade -2 -4 -2   <z,z,,>  -1 -4 -3 /
\vshade -2 -2 1  <z,z,,>   -0 0 1  <z,z,,>   6 0 1  <z,z,,>   7.99 -2 1  <z,z,,>   8  -4 1  <z,z,,>    9 -4 1 / 
\vshade 7 -4 -3   <z,z,,>   8 -4 -2 /

\multiput{} at 0 -4  0 1 /  
\endpicture} at 0 -6
\put{\beginpicture
\put{$3$} at 0 0

\put{
 \beginpicture
 \setcoordinatesystem units <.15cm,.15cm>
 \put{$\ssize 3$} at 0 2
 \put{$\ssize 2$} at 1 1
 \put{$\ssize 0$} at 2 0
\endpicture} at  2 0

\put{
 \beginpicture
 \setcoordinatesystem units <.15cm,.15cm>
 \put{$\ssize 1$} at 0 2
 \put{$\ssize 1$} at 1 1
 \put{$\ssize 0$} at 2 0
\endpicture} at  4 0

\put{$2$} at 6 0

\put{
 \beginpicture
 \setcoordinatesystem units <.15cm,.15cm>
 \put{$\ssize 4$} at 1 1
 \put{$\ssize 2$} at 2 0
\endpicture} at  7 -1

\put{
 \beginpicture
 \setcoordinatesystem units <.15cm,.15cm>
 \put{$\ssize 6$} at 0 0
 \put{$\ssize 7$} at 1 1
 \put{$\ssize 5$} at 2 2
 \put{$\ssize 4$} at 3 1
 \put{$\ssize 2$} at 4 0
\endpicture} at  8 -2

\put{
 \beginpicture
 \setcoordinatesystem units <.15cm,.15cm>
 \put{$\ssize 6$} at 1 1
 \put{$\ssize 7$} at 2 2
\endpicture} at 7 -3

\put{
 \beginpicture
 \setcoordinatesystem units <.15cm,.15cm>
 \put{$\ssize 3$} at 1 1
 \put{$\ssize 3$} at 2 0
\endpicture} at  -1 -1

\put{
 \beginpicture
 \setcoordinatesystem units <.15cm,.15cm>
 \put{$\ssize 3$} at 0 2
 \put{$\ssize 3$} at 1 1
 \put{$\ssize 2$} at 2 0
 \put{$\ssize 4$} at 3 1
 \put{$\ssize 5$} at 4 2
\endpicture} at  -2 -2

\put{
 \beginpicture
 \setcoordinatesystem units <.15cm,.15cm>
 \put{$\ssize 4$} at 1 1
 \put{$\ssize 5$} at 2 2
\endpicture} at  -1 -3

\setdots<1mm>
\plot 0.4 0  1.6 0 /
\plot 2.4 0  3.6 0 /
\plot 4.4 0  5.6 0 /

\setsolid
\arrow <1.5mm> [0.25,0.75] from -.7 -0.7 to -.3 -.3
\arrow <1.5mm> [0.25,0.75] from -1.6 -1.6 to -1.3 -1.3
\arrow <1.5mm> [0.25,0.75] from -1.6 -2.4 to -1.3 -2.7

\arrow <1.5mm> [0.25,0.75] from 6.4 -0.4 to 6.7 -.7
\arrow <1.5mm> [0.25,0.75] from 7.4 -1.4 to 7.7 -1.7
\arrow <1.5mm> [0.25,0.75] from 7.4 -2.7 to 7.7 -2.4

\setshadegrid span <.7mm>
\vshade -3 -4 1  <z,z,,>  -2 -4 1 /
\vshade -2 -4 -2   <z,z,,>  -1 -4 -3 /
\vshade -2 -2 1  <z,z,,>   -0 0 1  <z,z,,>   6 0 1  <z,z,,>   7.99 -2 1  <z,z,,>   8  -4 1  <z,z,,>    9 -4 1 / 
\vshade 7 -4 -3   <z,z,,>   8 -4 -2 /

\multiput{} at 0 -4  0 1 /  
\endpicture} at 0 -12

\endpicture}
$$

As before, the upper leaf and the middle leaf are sewn together
(both contain the module $M(23)$), similarly, the middle leaf and the lower
leaf are sewn together (both contain the module $M(67)$).
But the change of the right hand side of the upper leaf is important,
since it means that the upper leave and the lower leave no longer
are sewn together (the lower leaf contains the module $M(45)$,
the upper one the shifted module $M(4'5')$). It follows that for the 
$s$-fold covering $3s$ leaves are sewn together and form a
$3s$-ramified component (this is the Auslander-Reiten component 
which contains the modules $I(0), I(6), I(4)$ and their shifts
under the Galois group). 
      \medskip

What happens with the remaining non-regular component (the $1$-ramified one)?
Thus, let us start to calculate the Auslander-Reiten component which contains the 
module $P(1)$. Here is the relevant part of the boundary:
$$ 
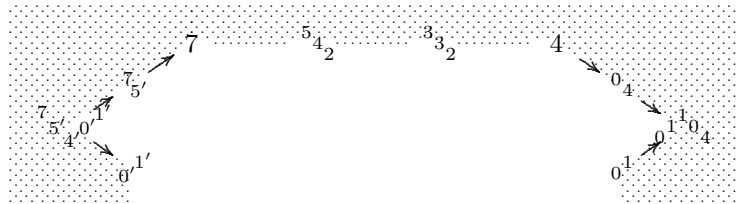
\hbox{\beginpicture
\setcoordinatesystem units <.8cm,.55cm>
\put{$7$} at 0 0

\put{
 \beginpicture
 \setcoordinatesystem units <.15cm,.15cm>
 \put{$\ssize 5$} at 0 2
 \put{$\ssize 4$} at 1 1
 \put{$\ssize 2$} at 2 0
\endpicture} at  2 0

\put{
 \beginpicture
 \setcoordinatesystem units <.15cm,.15cm>
 \put{$\ssize 3$} at 0 2
 \put{$\ssize 3$} at 1 1
 \put{$\ssize 2$} at 2 0
\endpicture} at  4 0

\put{$4$} at 6 0 

\put{
 \beginpicture
 \setcoordinatesystem units <.15cm,.15cm>
 \put{$\ssize 0$} at 1 1
 \put{$\ssize 4$} at 2 0
\endpicture} at  7 -1

\put{
 \beginpicture
 \setcoordinatesystem units <.15cm,.15cm>
 \put{$\ssize 0$} at 0 0
 \put{$\ssize 1$} at 1 1
 \put{$\ssize 1$} at 2 2
 \put{$\ssize 0$} at 3 1
 \put{$\ssize 4$} at 4 0
\endpicture} at  8 -2

\put{
 \beginpicture
 \setcoordinatesystem units <.15cm,.15cm>
 \put{$\ssize 7$} at 1 1
 \put{$\ssize 5'$} at 2 0
\endpicture} at  -1 -1

\put{
 \beginpicture
 \setcoordinatesystem units <.2cm,.18cm>
 \put{$\ssize 7$} at 0 2
 \put{$\ssize 5'$} at 1 1
 \put{$\ssize 4'$} at 2 0
 \put{$\ssize 0'$} at 3 1
 \put{$\ssize 1'$} at 4 2
\endpicture} at  -2 -2

\setdots<1mm>
\plot 0.4 0  1.6 0 /
\plot 2.4 0  3.6 0 /
\plot 4.4 0  5.6 0 /

\setsolid
\arrow <1.5mm> [0.25,0.75] from -.7 -0.7 to -.3 -.3
\arrow <1.5mm> [0.25,0.75] from -1.6 -1.6 to -1.3 -1.3

\arrow <1.5mm> [0.25,0.75] from 6.4 -0.4 to 6.7 -.7
\arrow <1.5mm> [0.25,0.75] from 7.4 -1.4 to 7.7 -1.7


\setshadegrid span <.7mm>
\vshade -3 -4 1  <z,z,,>  -2 -4 1 /
\vshade -2 -4 -2   <z,z,,>  -1 -4 -3 /
\vshade -2 -2 1  <z,z,,>   -0 0 1  <z,z,,>   6 0 1  <z,z,,>   7.99 -2 1  <z,z,,>   8  -4 1  <z,z,,>    9 -4 1 / 
\vshade 7 -4 -3   <z,z,,>   8 -4 -2 /

\put{
 \beginpicture
 \setcoordinatesystem units <.15cm,.15cm>
 \put{$\ssize 0$} at 1 1
 \put{$\ssize 1$} at 2 2
\endpicture} at 7 -3

\put{
 \beginpicture
 \setcoordinatesystem units <.2cm,.18cm>
 \put{$\ssize 0'$} at 1 1
 \put{$\ssize 1'$} at 2 2
\endpicture} at  -1 -3

\arrow <1.5mm> [0.25,0.75] from -1.6 -2.4 to -1.3 -2.7
\arrow <1.5mm> [0.25,0.75] from 7.4 -2.7 to 7.7 -2.4

\multiput{} at 0 -4  0 1 /  
\endpicture}
$$
It follows that the Galois shifts of the module $P(1)$ all lie
in one component, and this is a component of
type $\mathbb A_\infty^\infty$ which is  $s$-ramified.
This shows that any ramification does occur.


\subsection{The ramification sequence of a wind wheel} 
We have seen above, that the sewing of the leaves is accomplished via the
permutation $\pi =  [\lambda,\rho].$
	    \medskip

Thus, we see: {\it All the non-regular components are $1$-ramified if and only if the
permutations $\lambda$ and $\rho$ commute} (in particular, this
will be the case if these permutations coincide, as in example 12.5).
      \medskip

In general, we see that we do not get all the possible permutations for $\pi$.
The mathematics behind it, is as follows: In the symmetric group $\Sigma_t$, we fix one
$t$-cycle as $\rho$ and form for any $t$-cycle $\lambda$ the commutator 
$\pi= [\lambda,\rho]$: {\it these are the permutations which arise for the sewing procedure.}

\begin{prop} A partition of $t$ is the ramifcation sequence of a wind wheel if and only if
it is the cycle partition for the commutator of two $t$-cycles.
\end{prop}

(By definition, the cycle partition of a permutation has as parts the lengths of the cycles when
written as a product of disjoint cycles.)

For $t=2$, the group $\Sigma_t$ is commutative, thus we get as $\pi$ only the identity.
This means: For $t = 2$, we always get two non-regular components, both being $1$-ramified.

For $t=3$, the group  $\Sigma_t$ is no longer commutative, however the $3$-cycles
commute, thus again the only commutator $\pi = [\sigma,\rho]$ is the identity,
thus again we see that we only get $1$-ramified components.

The first case where one can obtain an $r$-ramified component with $r > 1$ is $t = 4;$
an explicit example has been discussed in 12.6.
     
For $t=4$ one checks easily that the possible ramification sequences are $(3,1)$ and
$(1,1,1,1).$. For $t=5$, they are $(5)$, $(3,1,1)$ and $(1,1,1,1,1)$  (for example,
the commutator of the permutations $(12345)$ and $(12354)$ has the cycle partition
$(3,1,1)$, that of $(12345)$ and $(12453)$ has the cycle partition $(5)$). In particular,
we see that  for $t\le 5$, there are no $2$-ramified components. For $t=6$, the
commutator of $(123456)$ and $(124653)$ has the cycle partition $(4,2)$.

It seems that a description of the structure of the commutators $[\lambda,\rho],$
where $\lambda$ and $\rho$ are $t$-cycles in $\Sigma_t$ is not known (but see the related
investigations \cite{Hu, Be}).


\section{The Auslander-Reiten quilt of a wind wheel}

Auslander-Reiten quilts 
have been considered until now only for suitable special biserial
algebras $\Lambda$. A general definition can be given in case $\Lambda$ is a 1-domestic
special biserial algebra, see \cite{Rinf}:
The vertices are (finite or infinite) words,
and there are arrows,  meshes, but also a convergence relation. The Auslander-Reiten quilt 
considers not only the indecomposable $\Lambda$-modules of finite length, but also related 
indecomposable $\Lambda$-modules of infinite length which are algebraically compact. 
The main objective is to sew together Auslander-Reiten components which contain string modules,
using $\mathbb N$-words and $\mathbb Z$-words, the $\mathbb Z$-words yield ''tiles''.

\subsection{Tiles and quarters} We recall from \cite{Rinf} some considerations
concerning the Auslander-Reiten quilt of a special biserial algebra. 

The poset $\Sigma $ is the ordered sum of $\mathbb N$ and $-\mathbb N$; inserting a limit
point $\omega$ in the middle, we obtain the completion $\overline \Sigma $
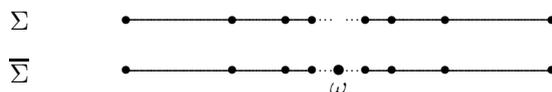
$$
\hbox{\beginpicture
\setcoordinatesystem units <0.7cm,0.5cm>
\put{\beginpicture
\setcoordinatesystem units <0.7cm,0.5cm>
\put{$\ssize\bullet$} at 0 0
\put{$\ssize\bullet$} at 2 0
\put{$\ssize\bullet$} at 3 0
\put{$\ssize\bullet$} at 3.5 0
\put{$\ssize\bullet$} at 4.5 0
\put{$\ssize\bullet$} at 5 0
\put{$\ssize\bullet$} at 6 0
\put{$\ssize\bullet$} at 8 0
\plot 0 0    3.5 0 /
\plot 4.5 0  8 0 /
\setdots<2pt>
\plot 3.65 0 3.9 0 /
\plot 4.15 0 4.4 0 /
\put{$\Sigma $} at -2 0
\endpicture} at 0 1.5
\put{\beginpicture
\setcoordinatesystem units <0.7cm,0.5cm>
\put{$\ssize\bullet$} at 0 0
\put{$\ssize\bullet$} at 2 0
\put{$\ssize\bullet$} at 3 0
\put{$\ssize\bullet$} at 3.5 0
\put{$\bullet$} at 4 -0.01
\put{$\omega$} at 4 -0.5
\put{$\ssize\bullet$} at 4.5 0
\put{$\ssize\bullet$} at 5 0
\put{$\ssize\bullet$} at 6 0
\put{$\ssize\bullet$} at 8 0
\plot 0 0    3.5 0 /
\plot 4.5 0  8 0 /
\setdots<2pt>
\plot 3.65 0 3.9 0 /
\plot 4.15 0 4.4 0 /
\put{$\overline\Sigma $} at -2 0
\endpicture} at 0 0
\endpicture}
$$ 

We may consider $\overline \Sigma $ as a topological space, namely
as a closed interval; correspondingly, we will consider 
$\mathcal T = \overline\Sigma 
\times \overline\Sigma $ as a square or better as
a lozenge. The center $\blacksquare$ is the vertex $(\omega,\omega)$, 
the vertices on the diagonals (some are marked by $\bullet$) are the
pairs $(x,\omega)$ and $(\omega, x)$ with $x\in \overline\Sigma $.
In a rather obvious way, we can consider $\mathcal T$ also
as a translation quiver and as in \cite{Rinf} we will call it a {\it tile},
see the following picture on the left:
$$ 
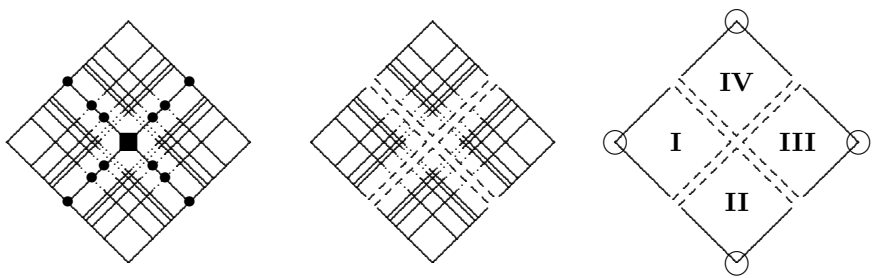
\hbox{\beginpicture
\setcoordinatesystem units <0.8cm,0.8cm>
\put{\beginpicture
\multiput{$\bullet$} at 1 1  3 3  1 3  3 1  1.4 1.4  1.6 1.6 
2.4 2.4  2.6 2.6  1.4 2.6  1.6 2.4
   2.4 1.6  2.6 1.4 /
\plot 0 2  2 0  4 2  2 4  0 2 /
\plot 1 1  3 3 /
\plot 3 1  1 3 /
\put{} at 0 0
\put{} at 0 4
\put{} at 4 0
\plot 0.4 1.6    1.2 2.4 /
\plot 0.6 1.4    1.4 2.2 /
\plot 0.7 1.3    1.5 2.1 /
\plot 0.75 1.25  1.55 2.05 /
\plot 1.6  0.4   2.4  1.2  /
\plot 1.4  0.6   2.2  1.4  /
\plot 1.3  0.7   2.1  1.5  /
\plot 1.25 0.75  2.05 1.55 /
\plot 2.4  0.4   1.6  1.2  /
\plot 2.6  0.6   1.8  1.4  /
\plot 2.7  0.7   1.9  1.5  /
\plot 2.75 0.75  1.95 1.55 /
\plot 0.4  2.4   1.2  1.6  /
\plot 0.6  2.6   1.4  1.8  /
\plot 0.7  2.7   1.5  1.9  /
\plot 0.75 2.75  1.55 1.95 /
\plot 1.6  3.6   2.4  2.8  /
\plot 1.4  3.4   2.2  2.6  /
\plot 1.3  3.3   2.1  2.5  /
\plot 1.25 3.25  2.05 2.45 /
\plot 2.4  3.6   1.6  2.8  /
\plot 2.6  3.4   1.8  2.6  /
\plot 2.7  3.3   1.9  2.5  /
\plot 2.75 3.25  1.95 2.45 /
\plot 3.6  2.4   2.8  1.6  /
\plot 3.4  2.6   2.6  1.8  /
\plot 3.3  2.7   2.5  1.9  /
\plot 3.25 2.75  2.45 1.95 /
\plot 3.6  1.6   2.8  2.4  /
\plot 3.4  1.4   2.6  2.2  /
\plot 3.3  1.3   2.5  2.1  /
\plot 3.25 1.25  2.45 2.05 /
\setdots<2pt>

\plot 1.2  2.4  1.6 2.8 /
\plot 1.4  2.2  1.8 2.6 /
\plot 1.5  2.1  1.9 2.5 /
\plot 1.55 2.05 1.95 2.45 /
\plot 1.6  1.2  1.2  1.6 /
\plot 1.8  1.4  1.4  1.8 /
\plot 1.9  1.5  1.5  1.9 /
\plot 1.95 1.55 1.55 1.95 /
\plot 2.8  1.6  2.4  1.2 /
\plot 2.6  1.8  2.2  1.4 /
\plot 2.5  1.9  2.1  1.5 /
\plot 2.45 1.95 2.05 1.55 /
\plot 2.8  2.4  2.4  2.8 /
\plot 2.6  2.2  2.2  2.6 /
\plot 2.5  2.1  2.1  2.5  /
\plot 2.45 2.05 2.05 2.45 /
\put{$\blacksquare$} at 2 2
\endpicture} at 0 0

\put{\beginpicture
\setcoordinatesystem units <.8cm,.8cm>
\plot 0.9 1.1  0 2  0.9 2.9 /
\plot 1.1 3.1  2 4  2.9 3.1 /
\plot 3.1 2.9  4 2  3.1 1.1 /
\plot 2.9 0.9  2 0  1.1 0.9 /
\setdashes <1mm>
\plot 0.95 1.05  1.9 2  0.95 2.95 /
\plot 1.05 0.95  2 1.9  2.95 0.95 /
\plot 1.05 3.05  2 2.1  2.95 3.05 /
\plot 3.05 2.95  2.1 2  3.05 1.05 /
\setsolid
\put{} at 0 0
\put{} at 0 4
\put{} at 4 0
\plot 0.4 1.6    1.2 2.4 /
\plot 0.6 1.4    1.4 2.2 /
\plot 0.7 1.3    1.5 2.1 /
\plot 0.75 1.25  1.55 2.05 /
\plot 1.6  0.4   2.4  1.2  /
\plot 1.4  0.6   2.2  1.4  /
\plot 1.3  0.7   2.1  1.5  /
\plot 1.25 0.75  2.05 1.55 /
\plot 2.4  0.4   1.6  1.2  /
\plot 2.6  0.6   1.8  1.4  /
\plot 2.7  0.7   1.9  1.5  /
\plot 2.75 0.75  1.95 1.55 /
\plot 0.4  2.4   1.2  1.6  /
\plot 0.6  2.6   1.4  1.8  /
\plot 0.7  2.7   1.5  1.9  /
\plot 0.75 2.75  1.55 1.95 /
\plot 1.6  3.6   2.4  2.8  /
\plot 1.4  3.4   2.2  2.6  /
\plot 1.3  3.3   2.1  2.5  /
\plot 1.25 3.25  2.05 2.45 /
\plot 2.4  3.6   1.6  2.8  /
\plot 2.6  3.4   1.8  2.6  /
\plot 2.7  3.3   1.9  2.5  /
\plot 2.75 3.25  1.95 2.45 /
\plot 3.6  2.4   2.8  1.6  /
\plot 3.4  2.6   2.6  1.8  /
\plot 3.3  2.7   2.5  1.9  /
\plot 3.25 2.75  2.45 1.95 /
\plot 3.6  1.6   2.8  2.4  /
\plot 3.4  1.4   2.6  2.2  /
\plot 3.3  1.3   2.5  2.1  /
\plot 3.25 1.25  2.45 2.05 /
\setdots<2pt>
\plot 1.2  2.4  1.35 2.55 /
\plot           1.45 2.65 1.6 2.8 /
\plot 1.4  2.2  1.55 2.35 /
\plot           1.65 2.45  1.8 2.6 /
\plot 1.5  2.1  1.65 2.25 /
\plot           1.75 2.35  1.9 2.5 /
\plot 1.55 2.05 1.7  2.2 /
\plot           1.8  2.3   1.95 2.45 /
\plot 1.6  1.2  1.45 1.35 /
\plot           1.35 1.45  1.2  1.6 /
\plot 1.8  1.4  1.65 1.55 /
\plot           1.55 1.65  1.4  1.8 /
\plot 1.9  1.5  1.75 1.65 /
\plot           1.65 1.75  1.5  1.9 /
\plot 1.95 1.55 1.8  1.7 /
\plot           1.7  1.8   1.55 1.95 /
\plot 2.8  1.6  2.65 1.45 /
\plot           2.55 1.35  2.4  1.2 /
\plot 2.6  1.8  2.45 1.65 /
\plot           2.35 1.55  2.2  1.4 /
\plot 2.5  1.9  2.35 1.75 /
\plot           2.25 1.65  2.1  1.5 /
\plot 2.45 1.95 2.3  1.8 /
\plot           2.2  1.7   2.05 1.55 /
\plot 2.8  2.4  2.65 2.55 /
\plot           2.55 2.65  2.4  2.8 /
\plot 2.6  2.2  2.45 2.35 /
\plot           2.35 2.45  2.2  2.6 /
\plot 2.5  2.1  2.35 2.25 /
\plot           2.25 2.35  2.1  2.5  /
\plot 2.45 2.05 2.3  2.2 /
\plot           2.2  2.3   2.05 2.45 /
\endpicture}  at 5 0
\put{\beginpicture
\setcoordinatesystem units <0.8cm,0.8cm>
\plot 0.9 1.1  0 2  0.9 2.9 /
\plot 1.1 3.1  2 4  2.9 3.1 /
\plot 3.1 2.9  4 2  3.1 1.1 /
\plot 2.9 0.9  2 0  1.1 0.9 /
\put{\bf I} at 1 2
\put{\bf II} at 2 1
\put{\bf III} at 3 2
\put{\bf IV} at 2 3

\multiput{$\bigcirc$} at 2 0  0 2  2 4  4 2 /
\setdashes <1mm>
\plot 0.95 1.05  1.9 2  0.95 2.95 /
\plot 1.05 0.95  2 1.9  2.95 0.95 /
\plot 1.05 3.05  2 2.1  2.95 3.05 /
\plot 3.05 2.95  2.1 2  3.05 1.05 /
\put{} at 0 0 
\put{} at 4 4 
\endpicture} at 10 0 

\endpicture}
$$
The middle picture presents the translation subquiver $\Sigma \times\Sigma $
obtained from the left picture by deleting the diagonals through the center.
It may be considered as the disjoint union of four parts, the {\it quarters} 
{\bf I, II, III, IV,}
see the right picture. Here, we also have marked  the four {\it corners} using circles. 

Given any bar $b$, we obtain such a tile by considering all the (finite or infinite)
words containing the completed bar $\overline b$, thus by considering all the
(finite or infinite) subwords of $r(b)$. The word $r(b)$ itself will be the center
of the tile, the infinite words form the diagonals through the center. If we look
only at the finite words, we obtain in this way the four quarters. 

As we have seen in \cite{Rinf}, tiles can occur as hammocks (the bridges in \cite{Rinf}
yield such a hammock). Here we encounter a different situation where tiles
appear. The tiles
which occur when dealing with minimal representation-finite special biserial algebras
have the special property that all the irreducible maps in the quarter {\bf I} are
monomorphisms, whereas those in the quarter {\bf III} are epimorphisms. 

\subsection{Return to the example 12.1} There is one bar $b$, thus one tile $\mathcal T(b)$.
The finite-dimensional modules in the tile are those which do not lie in the image of the
restriction functor $\eta:\mo H \to \mo W,$ thus the string modules $M(v)$ where $v$
is a finite word which contains the completed bar $\overline b$. 

Any tile yields four quarters, and we have seen above that the 
corresponfing corner modules are related to the following Auslander-Reiten
sequence of $W/I$, where $I$ is the annihilator of $M(b)$ (in case $M(b)$ 
is of length $2$ we deal with a triangle, otherwise with a square):
$$
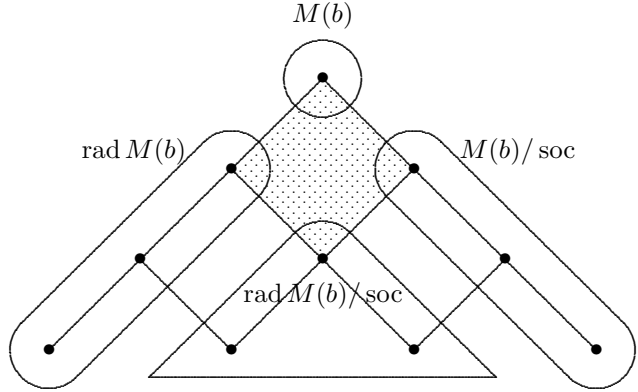
\hbox{\beginpicture
\setcoordinatesystem units <1.2cm,1.2cm>
\multiput{$\bullet$} at 0 0  1 1  2 0  2 2  3 1  3 3  4 0  4 2  5 1  6 0 /
\plot 0 0  3 3  6 0 /
\plot 1 1  2 0  4 2 /
\plot 2 2  4 0  5 1 /
\setshadegrid span <0.7mm>
\hshade 1 3 3 <,,z,> 2 2 4 <,,,z> 3 3 3 /
\plot 0.3 -0.3  2.3 1.7 /
\plot -.3 0.3  1.7 2.3 /
\circulararc 180 degrees from 2.3 1.7 center at 2 2 
\circulararc 180 degrees from -.3 0.3 center at 0 0 

\plot 3.7 1.7  5.7  -.3 /
\plot 4.3 2.3  6.3 .3 /
\circulararc 180 degrees from 4.3 2.3 center at 4 2 
\circulararc 180 degrees from 5.7 -.3 center at 6 0 

\circulararc 360 degrees from 3.3 3.3 center at 3 3 

\plot 3.3 1.3  4.9 -.3  1.1 -.3  2.7 1.3 /
\circulararc  90 degrees from 3.3 1.3 center at 3 1
 
\put{$M(b)$} at 3 3.7
\put{$\rad M(b)$} [r] at 1.5 2.2
\put{$M(b)/\soc$} [l] at 4.5 2.2
\put{$\rad M(b)/\soc$} at 3 .6
\put{} at 0 -.3
\endpicture}
$$
We are going to outline in which 
way this square or triangle is enlarged in $\mo W$. It is the 
tile $\mathcal T$ which may be considered as being inserted into 
this square --- alternatively, we may say that we add a border to the tile:
$$
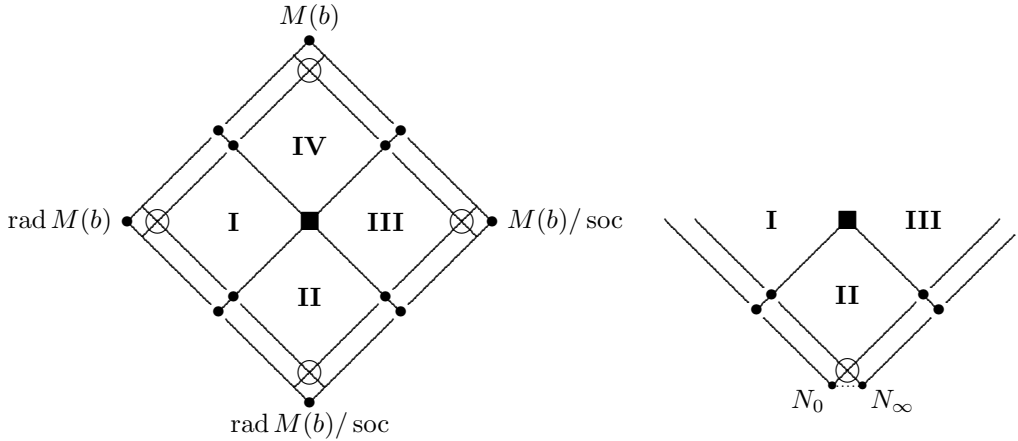
\hbox{\beginpicture
\setcoordinatesystem units <1cm,1cm>
\put{\beginpicture
\multiput{$\bullet$} at -.4 2  4.4 2  2 4.4  2 -.4 /
\plot .7 3.1  -.4 2  .7 0.9 /
\plot 3.3 3.1  4.4 2  3.3 0.9 /
\plot .9 0.7  2 -.4  3.1 .7 /
\plot .9 3.3  2 4.4  3.1 3.3 /
\put{$\rad M(b)$} [r] at -.6 2
\put{$M(b)/\soc$} [l] at 4.6 2
\put{$M(b)$} at 2 4.7
\put{$\rad M(b)/\soc$}  at 2 -.7

\plot 1.8 -.2  2 0  2.2 -.2 /
\plot 1.8 4.2  2 4  2.2 4.2 /
\plot -.2 1.8  0 2 -.2 2.2 /
\plot 4.2 2.2  4 2 4.2 1.8 /

\plot 0.9 1.1  0 2  0.9 2.9 /
\plot 1.1 3.1  2 4  2.9 3.1 /
\plot 3.1 2.9  4 2  3.1 1.1 /
\plot 2.9 0.9  2 0  1.1 0.9 /
\put{\bf I} at 1 2
\put{\bf II} at 2 1
\put{\bf III} at 3 2
\put{\bf IV} at 2 3

\multiput{$\bigcirc$} at 2 0  0 2  2 4  4 2 /
\multiput{$\bullet$} at 1 1 1 3 3 1 3 3 0.8 0.8 0.8 3.2 3.2 0.8 3.2 3.2 /
\multiput{$\blacksquare$} at 2 2 /
\plot 0.8 0.8  3.2 3.2 /
\plot 0.8 3.2  3.2 0.8 /

\put{} at 0 0 
\put{} at 4 4 
\endpicture} at 0 0

\put{\beginpicture
\plot  -.4 2  .7 0.9 /
\plot .9 0.7  1.8 -.2   /
\plot 2.2 -.2 3.1 0.7 /
\plot 3.3 0.9 4.2 1.8 /

\plot 1.8 -.2  2 0  2.2 -.2 /

\plot 0.9 1.1  0 2  /
\plot   4 2  3.1 1.1 /
\plot 2.9 0.9  2 0  1.1 0.9 /
\put{\bf I} at 1 2
\put{\bf II} at 2 1
\put{\bf III} at 3 2

\put{$\bigcirc$} at   2 0 

\multiput{$\bullet$} at 1 1  3 1  0.8 0.8  3.2 0.8  /
\multiput{$\blacksquare$} at 2 2 /
\plot 0.8 0.8  2 2 /
\plot 2 2  3.2 0.8 /

\setdots <.7mm>
\plot 1.8 -.2  2.2 -.2 / 
\put{} at 0 0 
\put{} at 4 4 
\put{$N_0$} [r] at 1.7 -.4
\put{$N_\infty$} [l] at 2.3 -.4
\multiput{$\ssize\bullet$} at 1.8 -.2  2.2 -.2 /  
\endpicture} at 7 -.25
\endpicture}
$$
As we have noted above, dealing with the corner of the quarter {\bf II}
we have to distinguish whether $\rad M(b)/\soc$ is non-zero (as in the left picture)
or zero (the right picture shows the changes); in the zero case 
we have marked the modules $N_0$ and $N_\infty$, where $N_0$ is the boundary module
in $\mathcal R_0$ which has the same 
socle as $M(b),$ and $N_\infty$ is the boundary
module in $\mathcal R_\infty$ which has the same top as $M(b)$.
\vfill\eject

\subsection{One tile in detail}
We have  exhibited in 12.5  a wind wheel  with  5 bars, thus  there are 5 tiles.
Let us present at least one of the tiles, say $\mathcal T(45)$  in more detail:
$$
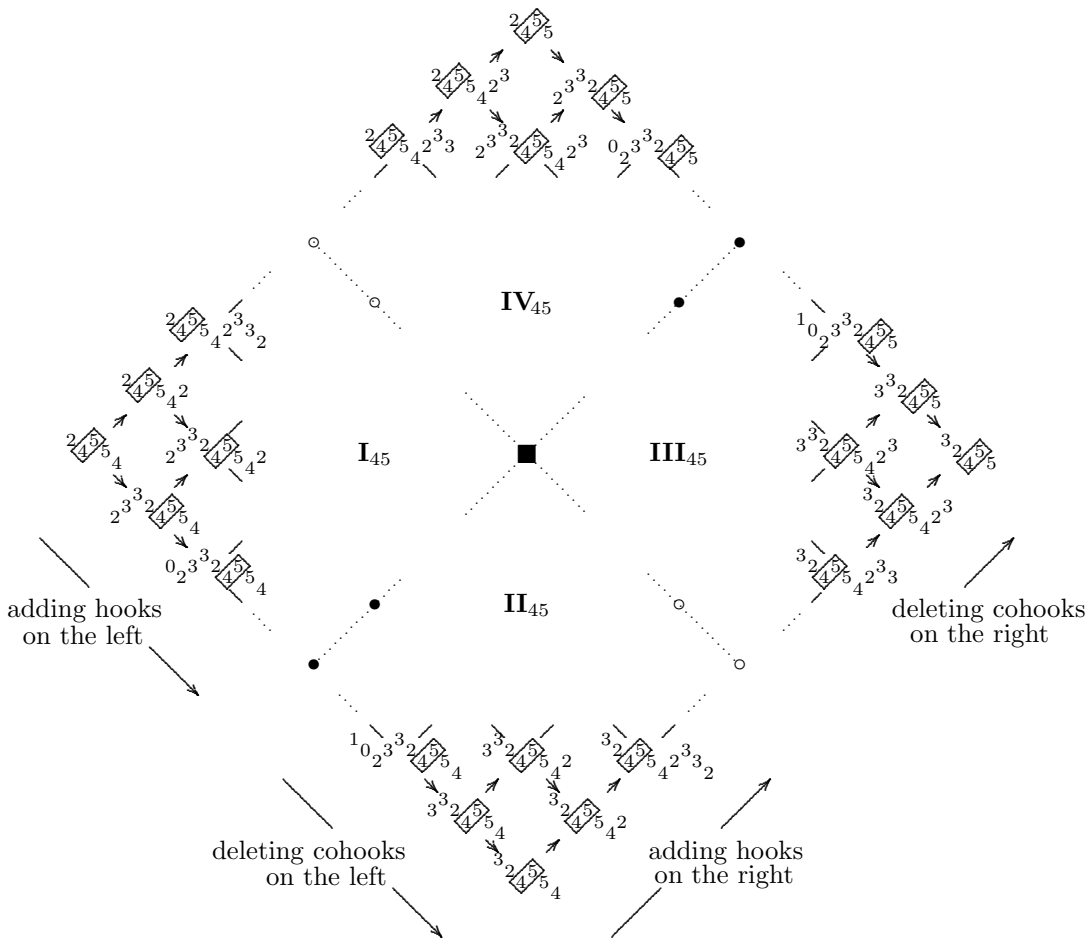
\hbox{\beginpicture
\setcoordinatesystem units <.8cm,.8cm>
\put{} at 0 -8.5
\put{\beginpicture
\multiput{} at 0 2  2 -2 /
\put{
 \beginpicture
 \setcoordinatesystem units <.15cm,.15cm>
 \put{$\ssize 2$} at 0 2
 \put{$\ssize 4$} at 1 1
 \put{$\ssize 5$} at 2 2
 \put{$\ssize 5$} at 3 1
 \put{$\ssize 4$} at 4 0
 \plot 0 1  1 0  3 2  2 3  0 1 /
\endpicture} at  0 0

\put{
 \beginpicture
 \setcoordinatesystem units <.15cm,.15cm>
 \put{$\ssize 2$} at 0 2
 \put{$\ssize 4$} at 1 1
 \put{$\ssize 5$} at 2 2
 \put{$\ssize 5$} at 3 1
 \put{$\ssize 4$} at 4 0
 \put{$\ssize 2$} at 5 1
 \plot 0 1  1 0  3 2  2 3  0 1 /
\endpicture} at  1 1

\put{
 \beginpicture
 \setcoordinatesystem units <.15cm,.15cm>
 \put{$\ssize 2$} at 0 2
 \put{$\ssize 4$} at 1 1
 \put{$\ssize 5$} at 2 2
 \put{$\ssize 5$} at 3 1
 \put{$\ssize 4$} at 4 0
 \put{$\ssize 2$} at 5 1
 \put{$\ssize 3$} at 6 2
 \put{$\ssize 3$} at 7 1
 \put{$\ssize 2$} at 8 0
 \plot 0 1  1 0  3 2  2 3  0 1 /
\endpicture} at  2 2

\put{
 \beginpicture
 \setcoordinatesystem units <.15cm,.15cm>
 \put{$\ssize 2$} at -3 1
 \put{$\ssize 3$} at -2 2
 \put{$\ssize 3$} at -1 3
 \put{$\ssize 2$} at 0 2
 \put{$\ssize 4$} at 1 1
 \put{$\ssize 5$} at 2 2
 \put{$\ssize 5$} at 3 1
 \put{$\ssize 4$} at 4 0
 \put{$\ssize 2$} at 5 1
 \plot 0 1  1 0  3 2  2 3  0 1 /
\endpicture} at  2 0

\put{
 \beginpicture
 \setcoordinatesystem units <.15cm,.15cm>
 \put{$\ssize 2$} at -3 1
 \put{$\ssize 3$} at -2 2
 \put{$\ssize 3$} at -1 3
 \put{$\ssize 2$} at 0 2
 \put{$\ssize 4$} at 1 1
 \put{$\ssize 5$} at 2 2
 \put{$\ssize 5$} at 3 1
 \put{$\ssize 4$} at 4 0
 \plot 0 1  1 0  3 2  2 3  0 1 /
\endpicture} at  1 -1

\put{
 \beginpicture
 \setcoordinatesystem units <.15cm,.15cm>
 \put{$\ssize 0$} at -4 2
 \put{$\ssize 2$} at -3 1
 \put{$\ssize 3$} at -2 2
 \put{$\ssize 3$} at -1 3
 \put{$\ssize 2$} at 0 2
 \put{$\ssize 4$} at 1 1
 \put{$\ssize 5$} at 2 2
 \put{$\ssize 5$} at 3 1
 \put{$\ssize 4$} at 4 0
 \plot 0 1  1 0  3 2  2 3  0 1 /
\endpicture} at  2 -2

\arrow <1.5mm> [0.25,0.75] from 0.4 0.4  to 0.6 0.6 
\arrow <1.5mm> [0.25,0.75] from 1.4 1.4  to 1.6 1.6 
\arrow <1.5mm> [0.25,0.75] from 1.4 0.6  to 1.6 0.4 
\arrow <1.5mm> [0.25,0.75] from 0.4 -.4  to 0.6 -.6 
\arrow <1.5mm> [0.25,0.75] from 1.4 -1.4  to 1.6 -1.6 
\arrow <1.5mm> [0.25,0.75] from 1.4 -.6  to 1.6 -.4 
\plot 2.3 2.3  2.5 2.5 /
\plot 2.3 -2.3  2.5 -2.5 /
\plot 2.3 1.7  2.5 1.5 /
\plot 2.3 -1.7  2.5 -1.5 /

\plot 2.3 0.3  2.5 0.5 /
\plot 2.3 -.3  2.5 -.5 /

\setdots <1mm>

\plot 2.7 2.7  3 3 /
\plot 2.7 -2.7  3 -3 /
\endpicture} at -6 0
\put{\beginpicture
\multiput{} at 0 2  -2 -2 /
\put{
 \beginpicture
 \setcoordinatesystem units <.15cm,.15cm>
 \put{$\ssize 3$} at -1 3
 \put{$\ssize 2$} at 0 2
 \put{$\ssize 4$} at 1 1
 \put{$\ssize 5$} at 2 2
 \put{$\ssize 5$} at 3 1
 \plot 0 1  1 0  3 2  2 3  0 1 /
\endpicture} at  0 0

\put{
 \beginpicture
 \setcoordinatesystem units <.15cm,.15cm>
 \put{$\ssize 3$} at -2 2
 \put{$\ssize 3$} at -1 3
 \put{$\ssize 2$} at 0 2
 \put{$\ssize 4$} at 1 1
 \put{$\ssize 5$} at 2 2
 \put{$\ssize 5$} at 3 1
 \plot 0 1  1 0  3 2  2 3  0 1 /
\endpicture} at  -1 1

\put{
 \beginpicture
 \setcoordinatesystem units <.15cm,.15cm>
 \put{$\ssize 1$} at -5 3
 \put{$\ssize 0$} at -4 2
 \put{$\ssize 2$} at -3 1
 \put{$\ssize 3$} at -2 2
 \put{$\ssize 3$} at -1 3
 \put{$\ssize 2$} at 0 2
 \put{$\ssize 4$} at 1 1
 \put{$\ssize 5$} at 2 2
 \put{$\ssize 5$} at 3 1
 \plot 0 1  1 0  3 2  2 3  0 1 /
\endpicture} at  -2 2

\put{
 \beginpicture
 \setcoordinatesystem units <.15cm,.15cm>
 \put{$\ssize 3$} at -2 2
 \put{$\ssize 3$} at -1 3
 \put{$\ssize 2$} at 0 2
 \put{$\ssize 4$} at 1 1
 \put{$\ssize 5$} at 2 2
 \put{$\ssize 5$} at 3 1
 \put{$\ssize 4$} at 4 0
 \put{$\ssize 2$} at 5 1
 \put{$\ssize 3$} at 6 2
 \plot 0 1  1 0  3 2  2 3  0 1 /
\endpicture} at  -2 0

\put{
 \beginpicture
 \setcoordinatesystem units <.15cm,.15cm>
 \put{$\ssize 3$} at -1 3
 \put{$\ssize 2$} at 0 2
 \put{$\ssize 4$} at 1 1
 \put{$\ssize 5$} at 2 2
 \put{$\ssize 5$} at 3 1
 \put{$\ssize 4$} at 4 0
 \put{$\ssize 2$} at 5 1
 \put{$\ssize 3$} at 6 2
 \plot 0 1  1 0  3 2  2 3  0 1 /
\endpicture} at  -1 -1

\put{
 \beginpicture
 \setcoordinatesystem units <.15cm,.15cm>
 \put{$\ssize 3$} at -1 3
 \put{$\ssize 2$} at 0 2
 \put{$\ssize 4$} at 1 1
 \put{$\ssize 5$} at 2 2
 \put{$\ssize 5$} at 3 1
 \put{$\ssize 4$} at 4 0
 \put{$\ssize 2$} at 5 1
 \put{$\ssize 3$} at 6 2
 \put{$\ssize 3$} at 7 1
 \plot 0 1  1 0  3 2  2 3  0 1 /
\endpicture} at  -2 -2

\arrow <1.5mm> [0.25,0.75] from  -0.6 0.6 to -0.4 0.4
\arrow <1.5mm> [0.25,0.75] from -1.6 1.6 to -1.4 1.4  
\arrow <1.5mm> [0.25,0.75] from -1.6 0.4 to -1.4 0.6  
\arrow <1.5mm> [0.25,0.75] from -0.6 -.6 to -0.4 -.4  
\arrow <1.5mm> [0.25,0.75] from  -1.6 -1.6 to -1.4 -1.4
\arrow <1.5mm> [0.25,0.75] from  -1.6 -.4 to -1.4 -.6  
\plot -2.3 2.3  -2.5 2.5 /
\plot -2.35 -2.35  -2.5 -2.5 /
\plot -2.3 1.7  -2.5 1.5 /
\plot -2.3 -1.7  -2.5 -1.5 /

\plot -2.3 0.3  -2.5 0.5 /
\plot -2.3 -.3  -2.5 -.5 /

\setdots <1mm>

\plot -2.7 2.7  -3 3 /
\plot -2.7 -2.7  -3 -3 /
\endpicture} at 6 0

\put{\beginpicture

\put{
 \beginpicture
 \setcoordinatesystem units <.15cm,.15cm>
 \put{$\ssize 2$} at 0 2
 \put{$\ssize 4$} at 1 1
 \put{$\ssize 5$} at 2 2
 \put{$\ssize 5$} at 3 1
 \plot 0 1  1 0  3 2  2 3  0 1 /
\endpicture} at  0 2

\put{
 \beginpicture
 \setcoordinatesystem units <.15cm,.15cm>
 \put{$\ssize 2$} at 0 2
 \put{$\ssize 4$} at 1 1
 \put{$\ssize 5$} at 2 2
 \put{$\ssize 5$} at 3 1
 \put{$\ssize 4$} at 4 0
 \put{$\ssize 2$} at 5 1
 \put{$\ssize 3$} at 6 2
 \plot 0 1  1 0  3 2  2 3  0 1 /
\endpicture} at  -1 1

\put{
 \beginpicture
 \setcoordinatesystem units <.15cm,.15cm>
 \put{$\ssize 2$} at 0 2
 \put{$\ssize 4$} at 1 1
 \put{$\ssize 5$} at 2 2
 \put{$\ssize 5$} at 3 1
 \put{$\ssize 4$} at 4 0
 \put{$\ssize 2$} at 5 1
 \put{$\ssize 3$} at 6 2
 \put{$\ssize 3$} at 7 1
 \plot 0 1  1 0  3 2  2 3  0 1 /
\endpicture} at  -2 0

\put{
 \beginpicture
 \setcoordinatesystem units <.15cm,.15cm>
 \put{$\ssize 2$} at -3 1
 \put{$\ssize 3$} at -2 2
 \put{$\ssize 3$} at -1 3
 \put{$\ssize 2$} at 0 2
 \put{$\ssize 4$} at 1 1
 \put{$\ssize 5$} at 2 2
 \put{$\ssize 5$} at 3 1
 \plot 0 1  1 0  3 2  2 3  0 1 /
\endpicture} at  1 1

\put{
 \beginpicture
 \setcoordinatesystem units <.15cm,.15cm>
 \put{$\ssize 2$} at -3 1
 \put{$\ssize 3$} at -2 2
 \put{$\ssize 3$} at -1 3
 \put{$\ssize 2$} at 0 2
 \put{$\ssize 4$} at 1 1
 \put{$\ssize 5$} at 2 2
 \put{$\ssize 5$} at 3 1
 \put{$\ssize 4$} at 4 0
 \put{$\ssize 2$} at 5 1
 \put{$\ssize 3$} at 6 2
 \plot 0 1  1 0  3 2  2 3  0 1 /
\endpicture} at  0 0

\put{
 \beginpicture
 \setcoordinatesystem units <.15cm,.15cm>
 \put{$\ssize 0$} at -4 2
 \put{$\ssize 2$} at -3 1
 \put{$\ssize 3$} at -2 2
 \put{$\ssize 3$} at -1 3
 \put{$\ssize 2$} at 0 2
 \put{$\ssize 4$} at 1 1
 \put{$\ssize 5$} at 2 2
 \put{$\ssize 5$} at 3 1
 \plot 0 1  1 0  3 2  2 3  0 1 /
\endpicture} at  2 0

\arrow <1.5mm> [0.25,0.75] from  -1.6 0.4 to  -1.4 0.6
\arrow <1.5mm> [0.25,0.75] from  -.6 1.4 to    -.4 1.6
\arrow <1.5mm> [0.25,0.75] from  -.6 0.6 to   -.4 .4
\arrow <1.5mm> [0.25,0.75] from  0.4 1.6 to  0.6 1.4
\arrow <1.5mm> [0.25,0.75] from  1.4 0.6 to  1.6 0.4
\arrow <1.5mm> [0.25,0.75] from  0.4 0.4 to  0.6 0.6
\plot -2.5 -0.5  -2.3 -0.3 /
\plot -1.7 -0.3  -1.5 -.5 /
\plot -0.5 -0.5  -.3 -.3  /
\plot 2.5 -0.5  2.35 -0.35 /
\plot 1.7 -0.3  1.5 -.5 /
\plot 0.5 -0.5  .3 -.3  /

\setdots <1mm>

\plot  -3 -1  -2.7 -.7 /
\plot   3 -1  2.7 -.7 /

\endpicture} at  0 6

\put{\beginpicture
\multiput{} at -2 0  2 -2 /

\put{
 \beginpicture
 \setcoordinatesystem units <.15cm,.15cm>
 \put{$\ssize 3$} at -1 3
 \put{$\ssize 2$} at 0 2
 \put{$\ssize 4$} at 1 1
 \put{$\ssize 5$} at 2 2
 \put{$\ssize 5$} at 3 1
 \put{$\ssize 4$} at 4 0
 \plot 0 1  1 0  3 2  2 3  0 1 /
\endpicture} at  0 -2

\put{
 \beginpicture
 \setcoordinatesystem units <.15cm,.15cm>
 \put{$\ssize 3$} at -2 2
 \put{$\ssize 3$} at -1 3
 \put{$\ssize 2$} at 0 2
 \put{$\ssize 4$} at 1 1
 \put{$\ssize 5$} at 2 2
 \put{$\ssize 5$} at 3 1
 \put{$\ssize 4$} at 4 0
 \plot 0 1  1 0  3 2  2 3  0 1 /
\endpicture} at  -1 -1

\put{
 \beginpicture
 \setcoordinatesystem units <.15cm,.15cm>
 \put{$\ssize 1$} at -5 3
 \put{$\ssize 0$} at -4 2
 \put{$\ssize 2$} at -3 1
 \put{$\ssize 3$} at -2 2
 \put{$\ssize 3$} at -1 3
 \put{$\ssize 2$} at 0 2
 \put{$\ssize 4$} at 1 1
 \put{$\ssize 5$} at 2 2
 \put{$\ssize 5$} at 3 1
 \put{$\ssize 4$} at 4 0
 \plot 0 1  1 0  3 2  2 3  0 1 /
\endpicture} at  -2 0

\put{
 \beginpicture
 \setcoordinatesystem units <.15cm,.15cm>
 \put{$\ssize 3$} at -2 2
 \put{$\ssize 3$} at -1 3
 \put{$\ssize 2$} at 0 2
 \put{$\ssize 4$} at 1 1
 \put{$\ssize 5$} at 2 2
 \put{$\ssize 5$} at 3 1
 \put{$\ssize 4$} at 4 0
 \put{$\ssize 2$} at 5 1
 \plot 0 1  1 0  3 2  2 3  0 1 /
\endpicture} at  0 0 

\put{
 \beginpicture
 \setcoordinatesystem units <.15cm,.15cm>
 \put{$\ssize 3$} at -1 3
 \put{$\ssize 2$} at 0 2
 \put{$\ssize 4$} at 1 1
 \put{$\ssize 5$} at 2 2
 \put{$\ssize 5$} at 3 1
 \put{$\ssize 4$} at 4 0
 \put{$\ssize 2$} at 5 1
 \plot 0 1  1 0  3 2  2 3  0 1 /
\endpicture} at   1 -1 

\put{
 \beginpicture
 \setcoordinatesystem units <.15cm,.15cm>
 \put{} at -3 3
 \put{$\ssize 3$} at -1 3
 \put{$\ssize 2$} at 0 2
 \put{$\ssize 4$} at 1 1
 \put{$\ssize 5$} at 2 2
 \put{$\ssize 5$} at 3 1
 \put{$\ssize 4$} at 4 0
 \put{$\ssize 2$} at 5 1
 \put{$\ssize 3$} at 6 2
 \put{$\ssize 3$} at 7 1
 \put{$\ssize 2$} at 8 0
 \plot 0 1  1 0  3 2  2 3  0 1 /
\endpicture} at   2 0 
\arrow <1.5mm> [0.25,0.75] from  -1.6 -0.4 to  -1.4 -0.6
\arrow <1.5mm> [0.25,0.75] from  -.6 -1.4 to    -.4 -1.6
\arrow <1.5mm> [0.25,0.75] from  -.6 -0.6 to   -.4 -.4
\arrow <1.5mm> [0.25,0.75] from  0.4 -1.6 to  0.6 -1.4
\arrow <1.5mm> [0.25,0.75] from  1.4 -0.6 to  1.6 -0.4
\arrow <1.5mm> [0.25,0.75] from  0.4 -0.4 to  0.6 -0.6
\plot -2.5 0.5  -2.3 0.3 /
\plot -1.7 0.3  -1.5 .5 /
\plot -0.5 0.5  -.3 .3  /
\plot 2.5 0.5  2.3 0.3 /
\plot 1.7 0.3  1.5 .5 /
\plot 0.5 0.5  .3 .3  /

\setdots <1mm>

\plot  -3 1  -2.7 .7 /
\plot   3 1  2.7 .7 /

\endpicture} at  0 -6

\put{\bf I$_{45}$} at -2.5 0
\put{\bf II$_{45}$} at 0 -2.5
\put{\bf III$_{45}$} at 2.5 0
\put{\bf IV\!$_{45}$} at  0 2.5

\setdots <1mm>
\plot -3.5 3.5 -2 2 /
\plot -1 1  1 -1 /
\plot  2 -2 3.5 -3.5 /

\plot -3.5 -3.5 -2 -2 /
\plot -1 -1  1 1 /
\plot  2 2 3.5 3.5 /

\put{$\blacksquare$} at 0 0 
\multiput{$\bullet$} at -3.5 -3.5  -2.5 -2.5  
  3.5 3.5  2.5 2.5  /
\multiput{$\circ$} at 
  3.5 -3.5  2.5 -2.5 
  -3.5 3.5  -2.5 2.5  /

\put{adding hooks} [r] at -6 -2.6
\put{on the left} [r] at -6.3 -3
\put{deleting cohooks} [r] at -2 -6.6
\put{on the left} [r] at -2.3 -7

\put{adding hooks} [l] at 2 -6.6
\put{on the right} [l] at 2.1 -7
\put{deleting cohooks} [l] at 6 -2.6
\put{on the right} [l] at 6.3 -3

\setsolid
\plot -8 -1.4  -7.2 -2.2 /
\arrow <1.5mm> [0.25,0.75] from  -6.2 -3.2 to  -5.4 -4
\plot -4 -5.4  -3.2 -6.2 /
\arrow <1.5mm> [0.25,0.75] from  -2.2 -7.2 to  -1.4 -8
\arrow <1.5mm> [0.25,0.75] from  7.2 -2.2 to  8 -1.4
\plot 1.4 -8  2.2 -7.2 /
\arrow <1.5mm> [0.25,0.75] from  3.2 -6.2 to  4 -5.4

\endpicture}
$$

The vertices of such a tile $\mathcal T(b)$, with $b$ a bar, 
are all the (finite or infinite) words which contain the 
completed bar $\overline b$ as a subword. Thus, in our case
we deal with the words which contain the word \ 
 \beginpicture
  \setcoordinatesystem units <.15cm,.15cm>
   \put{$\ssize 2$} at 0 2
    \put{$\ssize 4$} at 1 1
     \put{$\ssize 5$} at 2 2
      \put{$\ssize 5$} at 3 1
       \plot 0 1  1 0  3 2  2 3  0 1 /
       \endpicture \ \
        as a subword.  
There is precisely one $\mathbb  Z$-word of this form, namely $r(b)$,
it lies in the center and is marked by the black square $\blacksquare$. There are
many $\mathbb  N$-words, they lie on the two diagonals through the center.
Note that the $\mathbb  Z$-word as well as all the $\mathbb  N$-words are not periodic,
since any word in $\mathcal T(b)$ contains $\overline b$ only once as a subword. 
There are two kinds of $\mathbb  N$-words: On the northwest-southeast diagonal they
are marked by a circle $\circ$, these are the words of the form
$$
  x = v\overline b(w')^\infty
$$
with $v$ a finite word. The maximal periodic subword of $x$ starts with the bar $b$; one easily
checks that $x$ is contracting in the sense of \cite{Ralgcom} . Those on the northeast-southwest
diagonal are marked by a bullet $\bullet$, these are the words of the form
$$
  y = v\overline b^{-1}(w'')^{-\infty}
$$
again with $v$ a finite word. Here the maximal periodic subword starts with $b^{-1}$ and
$y$ turns out to be expanding.

As we know, the corner modules for the quarters are related to 
the exact sequence
$$
 0 \to \rad M(b) \to M(b)\oplus \rad M(b)/\soc \to M(b)/\soc \to 0,
 $$
 we obtain a border for $\mathcal T(b)$:
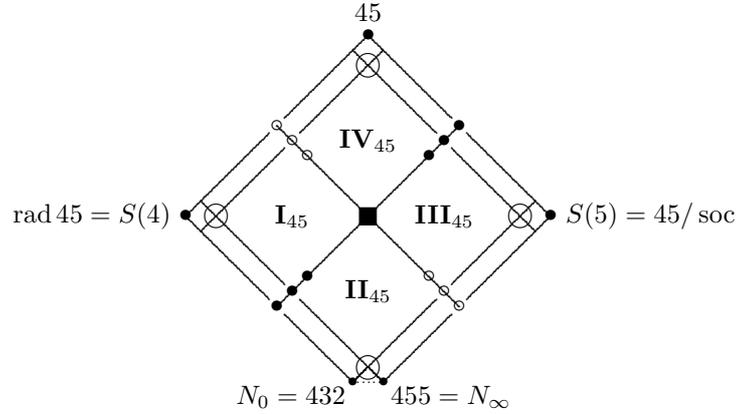
 $$
 \hbox{\beginpicture
 \setcoordinatesystem units <1cm,1cm>

 \multiput{$\bullet$} at -.4 2  4.4 2  2 4.4  /

 \plot .7 3.1  -.4 2  .7 0.9 /
 \plot 3.3 3.1  4.4 2  3.3 0.9 /

 \plot .9 3.3  2 4.4  3.1 3.3 /

 \plot .9 0.7  1.8 -.2   /
 \plot 2.2 -.2 3.1 0.7 /

 \plot 1.8 -.2  2 0  2.2 -.2 /

 \setdots <.7mm>
 \plot 1.8 -.2  2.2 -.2 /
 \put{$N_0 = 432$} [r] at 1.7 -.4
 \put{$455 = N_\infty$} [l] at 2.3 -.4
 \multiput{$\ssize\bullet$} at 1.8 -.2  2.2 -.2 /

 \put{$\rad 45= S(4)$} [r] at -.6 2
 \put{$S(5) = 45/\soc$} [l] at 4.6 2
 \put{$45$} at 2 4.7
 \setsolid
 \plot 1.8 -.2  2 0  2.2 -.2 /
 \plot 1.8 4.2  2 4  2.2 4.2 /
 \plot -.2 1.8  0 2 -.2 2.2 /
 \plot 4.2 2.2  4 2 4.2 1.8 /

 \plot 0.9 1.1  0 2  0.9 2.9 /
 \plot 1.1 3.1  2 4  2.9 3.1 /
 \plot 3.1 2.9  4 2  3.1 1.1 /
 \plot 2.9 0.9  2 0  1.1 0.9 /
 \put{\bf I$_{45}$} at 1 2
 \put{\bf II$_{45}$} at 2 1
 \put{\bf III$_{45}$} at 3 2
 \put{\bf IV$_{45}$} at 2 3

 \multiput{$\bigcirc$} at 2 0  0 2  2 4  4 2 /
 \multiput{$\bullet$} at 1 1  3 3 0.8 0.8   3.2 3.2  1.2 1.2  2.8 2.8 /
 \multiput{$\circ$} at 1 3  3 1  0.8 3.2  3.2 0.8  1.2 2.8  2.8 1.2 /
 
\multiput{$\blacksquare$} at 2 2 /
\plot 0.8 0.8  3.2 3.2 /
\plot 0.8 3.2  3.2 0.8 /

\put{} at 0 0
\put{} at 4 4
\endpicture}
$$
Here,
$N_0$ is the boundary module
in $\mathcal R_0$ which has the same socle as $M(45)$ and $N_\infty$ is the boundary
module in $\mathcal R_\infty$ which has the same top as $M(45)$.

We can further enlarge the picture by adding rays from $\mathcal P$ and $\mathcal R_\infty$,
as well as corays from $\mathcal R_0$ and $\mathcal Q$. In the case $t=1$, we obtain in this way 
the complete module category (of course, opposite edges have to be identified): 
$$
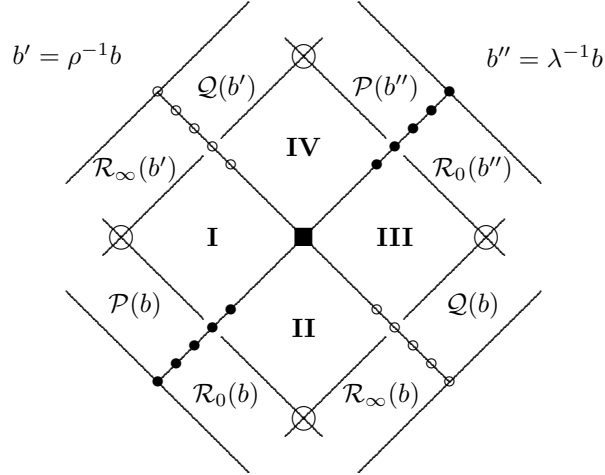
\hbox{\beginpicture
\setcoordinatesystem units <1.2cm,1.2cm>
\multiput{} at 0 4.6  0 -.6 /
\multiput{$\mathcal P(b)$} at .15 1.25  /
\multiput{$\mathcal R_0(b)$} at 1.15 0.25  /
\multiput{$\mathcal Q(b')$} at   1.15 3.65 /
\multiput{$\mathcal R_\infty(b')$} at  .15 2.75 /

\multiput{$\mathcal Q(b)$} at 3.85 1.25   /
\multiput{$\mathcal R_\infty(b)$} at 2.85 0.25  /
\multiput{$\mathcal P(b'')$} at  2.9 3.65 /
\multiput{$\mathcal R_0(b'')$} at  3.85 2.75 /

\put{$b' = \rho^{-1} b$} [r] at  0 4
\put{$b'' = \lambda^{-1} b$} [l] at 4 4

\plot -.6 1.4  1.4 -.6 /
\plot 2.6 4.6  4.6  2.6 /

\plot -.6 2.6  1.4 4.6 /
\plot 2.6 -0.6  4.6 1.4 /

\setsolid
\plot 1.8 -.2  2 0  2.2 -.2 /
\plot 1.8 4.2  2 4  2.2 4.2 /
\plot -.2 1.8  0 2 -.2 2.2 /
\plot 4.2 2.2  4 2 4.2 1.8 /

\plot 0.9 1.1  0 2  0.9 2.9 /
\plot 1.1 3.1  2 4  2.9 3.1 /
\plot 3.1 2.9  4 2  3.1 1.1 /
\plot 2.9 0.9  2 0  1.1 0.9 /
\put{\bf I} at 1 2
\put{\bf II} at 2 1
\put{\bf III} at 3 2
\put{\bf IV} at 2 3

\multiput{$\bigcirc$} at 2 0  0 2  2 4  4 2 /
\multiput{$\bullet$} at 1 1  3 3  0.8 0.8  0.6 0.6  3.2 3.2  3.4 3.4  1.2 1.2  2.8 2.8 
 3.6 3.6 0.4 0.4 /
\multiput{$\circ$} at 1 3  3 1  0.8 3.2  3.2 0.8  0.6 3.4  3.4 0.6  1.2 2.8  2.8 1.2 
 3.6 0.4  0.4 3.6  /

\multiput{$\blacksquare$} at 2 2 /
\plot 0.4 0.4  3.6 3.6 /
\plot 0.4 3.6  3.6 0.4 /

\put{} at 0 0
\put{} at 4 4
\endpicture}
$$

In particular, we get the following neighboring relation for tiles:
$$
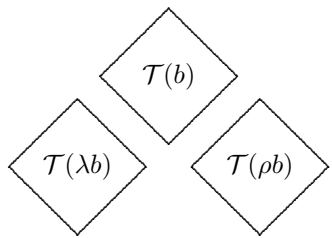
\hbox{\beginpicture
\setcoordinatesystem units <.3cm,.3cm>
\plot 0 0  3 3  6 0  3 -3 0 0 /
\plot 4 4  7 7  10 4  7 1  4 4 /
\plot 8 0  11 3  14 0  11 -3  8 0 /
\put{$\mathcal T(b)$} at 7 4 
\put{$\mathcal T(\lambda b)$} at 3 0
\put{$\mathcal T(\rho b)$} at 11 0 
\multiput{} at 0 -2.5  0  6.5 /
\endpicture}
$$

       \medskip

According to \cite{Ralgcom},  for any almost periodic
$\mathbb  N$-word or biperiodic $\mathbb  Z$-word $x$ there exists an indecomposable algebraically
compact module $C(x)$. In case $x$ is a contracting $\mathbb  N$-word, the module $C(x)$ is 
the usual string module for $x$, whereas for $x$ an expanding $\mathbb  N$-word, one needs to complete
the string module in order to obtain the module $C(x)$.
The $\mathbb  Z$-words $r(b)$ are mixed words (the right hand side is contracting, the left
hand side expanding), thus the module $C(r(b))$ is obtained from the corresponding string module
by a partial completion: the left hand side has to be completed, the right hand side not.

Let us observe that for our enlargement by adding  rays from $\mathcal P, \mathcal R_\infty$
and corays from $\mathcal R_\infty, \mathcal Q$, we also have to invoke the adic modules 
for the component $\mathcal R_0$ and the Pr\"ufer 
modules for the component $\mathcal R_\infty$.
\vfill\eject

\subsection{Example of a quilt}
In our example 12.5 with $t=5$,  five non-regular Auslander-Reiten components have been obtained.
Using infinite words, thus 
infinite dimensional representations, we see that these Auslander-Reiten components have to  be arranged 
as follows:

$$
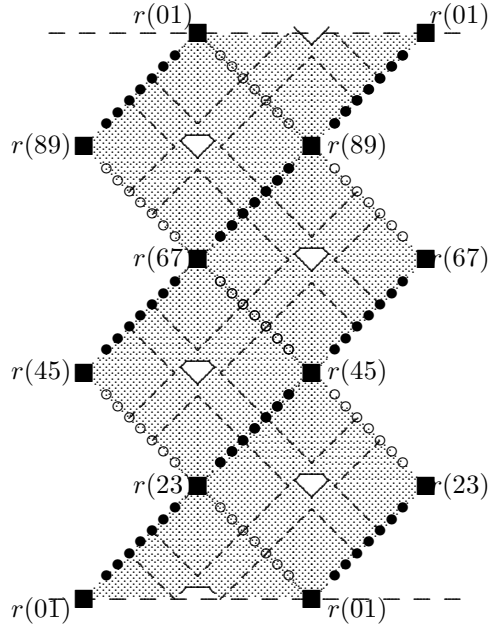
\hbox{\beginpicture
\setcoordinatesystem units <1.5cm,1.5cm>

\multiput{\beginpicture
\plot 0 -.1  -.15  .05 -.1 0.1  .1 .1  .15 .05 0 -.1 /  
\setdashes <1mm>
\plot -.6 -.4  -.2 0  -.6 .4 / 
\plot -.4 .6  0 .2  .4 .6 /

\plot .6 -.4  .2 0  .6 .4 / 
\plot .4 -.6  0 -.2  -.4 -.6 /
\setdots <1mm> 
\plot 0 -1  -1 0  0 1  1 0  0 -1 / 
\multiput{$\circ$} at 0.2 0.8  0.3 0.7  0.4 0.6  0.5 0.5  0.6 0.4  0.7 0.3  0.8 0.2 /
\multiput{$\circ$} at -0.2 -0.8  -0.3 -0.7  -0.4 -0.6  -0.5 -0.5  -0.6 -0.4  -0.7 -0.3  -0.8 -0.2 /
\multiput{$\bullet$} at 0.2 -0.8  0.3 -0.7  0.4 -0.6  0.5 -0.5  0.6 -0.4  0.7 -0.3  0.8 -0.2 /
\multiput{$\bullet$} at -0.2 0.8  -0.3 0.7  -0.4 0.6  -0.5 0.5  -0.6 0.4  -0.7 0.3  -0.8 0.2 /

\setshadegrid  span <.4mm>
\vshade -1 -0 -0 <,z,,>
  -0.15 -0.85 0.85 <z,z,,>
  -0.16 0.1 0.9 <z,z,,> 
  0 0.1 1  <z,z,,> .14 0.1 0.9  <z,z,,>
  .15 -.85 .85  <z,,,> 
  1 0 0 /
\vshade -0.16 -.9 0 <z,z,,> 
  0 -1 -.1  <z,z,,> 
  .15 -.9 0  /

\endpicture} at 0 0  1 -1  0 -2  1 -3 /

\put{\beginpicture
\multiput{} at -1 -1  1 1 /
\plot .15 .05 0 -.1 -.15  .05 /  
\setdashes <1mm>
\plot -.6 -.4  -.2 0   / 

\plot .6 -.4  .2 0  /
\plot .4 -.6  0 -.2  -.4 -.6 /
\setdots <1mm> 
\plot 1 0  0 -1  -1 0 /
\plot -1 0  0 -1 / 
\multiput{$\bullet$} at 0.2 -0.8  0.3 -0.7  0.4 -0.6  0.5 -0.5  0.6 -0.4  0.7 -0.3  0.8 -0.2 /

\multiput{$\bullet$} at 0.2 -0.8  0.3 -0.7  0.4 -0.6  0.5 -0.5  0.6 -0.4  0.7 -0.3  0.8 -0.2 /
\setshadegrid  span <.4mm>
\vshade -1 -0 -0 <,z,,>
  -0.15 -0.85 0 <z,z,,>
  0 -1 -.1  <z,z,,> 
  .15 -.85 0  <z,,,> 
  1 0 0 /

\endpicture} at 1 1 

\put{\beginpicture
\multiput{} at -1 -1  1 1 /
\plot  -.15  .05 -.1 0.1  .1 .1 .15 .05 /  
\setdashes <1mm>
\plot  -.2 0  -.6 .4 / 
\plot -.4 .6  0 .2  .4 .6 /

\plot  .2 0  .6 .4 / 
\setdots <1mm> 
\plot  -1 0  0 1  1 0   / 
\multiput{$\bullet$} at -0.2 0.8  -0.3 0.7  -0.4 0.6  -0.5 0.5  -0.6 0.4  -0.7 0.3  -0.8 0.2 /

\setshadegrid  span <.4mm>
\vshade -1 -0 -0 <,z,,>
  -0.15 0 0.85 <z,z,,>
  -0.16 0.1 0.9 <z,z,,> 
  0 0.1 1  <z,z,,> .14 0.1 0.9  <z,z,,>
  .15 0 .85  <z,,,> 
  1 0 0 /

\endpicture} at 0 -4

\setdashes <2mm>
\plot -1.3 1  2.3 1 /
\plot -1.3 -4  2.3 -4 /
\multiput{$r(01)$} at -.3 1.15  2.3 1.15 /
\multiput{$r(89)$} at 1.4 -0  -1.37 -0 /
\multiput{$r(67)$} at -.3 -1   2.3 -1 /
\multiput{$r(45)$} at 1.4 -2   -1.37 -2 /
\multiput{$r(23)$} at -.3 -3   2.3 -3 /
\multiput{$r(01)$} at 1.4 -4.1   -1.37 -4.1 /
\multiput{$\blacksquare$} at 0 1  1 0  0 -1  1 -2  0 -3  1 -4 
   2 1  -1 0  2 -1  2 -3    -1 -2  -1 -4  /
\endpicture}
$$
Here,  the left boundary has to be identified with the right boundary, and the
lower dashed line with the (slightly rotated) upper dashed line. The quilt which we obtain in this
way is a torus with 5 holes. 

Let us summarize: The picture above presents the quilt of our wind wheel, it exhibits on a surface
(finite and infinite) words which give rise to relevant indecomposable algebraically compact
modules. The black squares $\blacksquare$ mark the non-periodic $\mathbb Z$-words, the circles $\circ$
and the bullets $\bullet$ mark the $\Bbb N$-words. If we delete the infinite words,
we obtain the 5 Auslander-Reiten components which contain string modules, 
all being shown here as squares with a hole in the middle.
On the other hand, for any bar $b$, we also spot easily the tile $\mathcal T(b)$, it is a square with
center $r(b)$.
       
\subsection{The indecomposable algebraically compact modules} Let us stress that almost all, but {\bf not all}
indecomposable algebraically compact modules are used in the construction of the Auslander-Reiten quilt
of a wind wheel. Here is the list of the additional modules:

(1) The generic module,

(2) the Pr\"ufer modules for the tube $\mathcal R_0,$  and

(3) the adic modules for the tube $\mathcal R_\infty.$

\subsection{The Euler characteristic}
There is the following general observation:

\begin{prop}
Let $\Lambda$ be a wind wheel with $t$ bars. The Auslander-Reiten
quilt $\Gamma$ of $\Lambda$ is a
connected surface with boundary, its Euler characteristic is $\chi(\Gamma) = -t.$
\end{prop}

This result can be interpreted as follows: Let $h$ be the number of non-regular Auslander-Reiten components
of $\Lambda$, thus $h$ is the number of components of the boundary of $\Gamma$ and we have $h\le t$.
There is a (connected) compact Riemann surface $\Gamma'$ without boundary  and with
Euler characteristic $\chi(\Gamma') = -t+h$ such that $\Gamma$ is obtained from $\Gamma'$
by inserting $h$ holes.  For $h = t$ (as in the examples 12.8),
the surface $\Gamma'$ has Euler characteristic $0$, thus it is a torus. 

\begin{proof}
Our cut-and-paste procedure presents $\Gamma$ as being obtained from $4t$ pieces of
the following form
$$
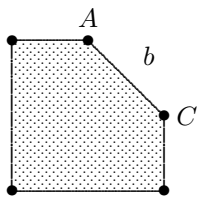
\hbox{\beginpicture
\setcoordinatesystem units <1cm,1cm>
\multiput{$\bullet$} at 0 0  0 2  2 0  1 2  2 1 /
\plot  0 0  2 0  2 1  1 2  0 2  0 0 /
\put{$A$} at 1 2.3
\put{$C$} at 2.3 1
\put{$b$} at 1.8 1.8
\setshadegrid span <0.5mm>
\vshade 0 0 2  <,z,,> 1 0 2 <z,,,> 2 0 1 / 
\endpicture}
$$
where the edge $b$ will be part of the boundary, whereas the remaining edges have to be
identified in pairs. Looking at the vertices, we have to distinguish the  endpoints $A$ and $C$ of 
the boundary edge $b$ and the remaining ones: the endpoints of the boundary edges are 
identified pairwise, whereas always four of the remaining ones yield a vertex of $\Gamma$.
Let $v,e,f$ be the number of vertices, edges and faces  respectively. 
There are $f = 4t$ faces, there are $e = 4t + 4t\cdot 4/2$ edges and $v = 4t\cdot 2/2 + 
4t \cdot 3/4$ vertices, thus $f-e+v = -t.$
\end{proof}

\subsection{Orientability} 

\begin{prop}
The Auslander-Reiten
quilt $\Gamma$ of any wind wheel is connected and orientable.
\end{prop}

\begin{proof}
Compact surfaces are often presented by a polygon with an even number of
edges and an identification rule for pairs of the edges, this rule
is shown by a word using the edges (and their inverses) as letters, this
takes into account the orientation in which the edges are identified.

Let $\overline Gamma$ be obtained from $\Gamma$ by filling the holes. 
Clearly, we obtain $\overline\Gamma$ 
by taking $t$ squares (corresponding to the tiles) and identifying pairs of edges. 
The neighboring relation for the tiles shows in which way we have to identify the edges.

First, we use the permutation $\lambda$ in order to obtain the following 
rectangle:
$$
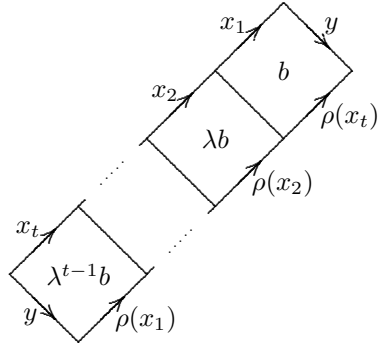
\hbox{\beginpicture
\setcoordinatesystem units <.3cm,.3cm>
\plot 3.3 3.3  0 0  3 -3  6.3 0.3 /
\plot 5.7 5.7 12 12  15 9  8.7 2.7 /
\plot 3 3   6 0 /
\plot  6 6  9 3 /
\plot 9 9 12 6 /
\setdots <1mm>
\plot 4 4  5 5 /
\plot 7 1  8 2 /
\put{$b$} at 12 9 
\put{$\lambda b$} at 9 6
\put{$\lambda^{t-1}b$} at 3 0
\put{$y$} at .9 -1.9
\put{$y$} at 14.1 11 
\put{$x_t$} at .8 2
\put{$\rho(x_t)$} [l] at 13.6 7
\put{$x_2$} at 6.8 8
\put{$\rho(x_2)$} [l] at 10.6 4
\put{$x_1$} at 9.8 11 
\put{$\rho(x_1)$} [l] at 4.6 -2
\setsolid
\arrow <2mm>  [0.25,0.75 ] from 1 1 to 2 2 
\arrow <2mm>  [0.25,0.75 ] from 7 7 to 8 8
\arrow <2mm>  [0.25,0.75 ] from 10 10 to 11 11 
\arrow <2mm>  [0.25,0.75 ] from 13 11 to 14 10
\arrow <2mm>  [0.25,0.75 ] from 1 -1 to 2 -2
\arrow <2mm>  [0.25,0.75 ] from 4 -2 to 5 -1
\arrow <2mm>  [0.25,0.75 ] from 10 4 to 11 5
\arrow <2mm>  [0.25,0.75 ] from 13 7  to 14 8

\endpicture}
$$
This already shows the connectedness. But we also see that the further identifications are
achieved by the word
$$
 y\rho(x_t)\cdots \rho(x_2)\rho(x_1)y^{-1}x_1^{-1}x_2^{-2}\cdots x_t^{-t}.
$$
It is well-known (and easy to see) that we can change the word to a product
of commutators, but this means that $\overline \Gamma$ is orientable.

\end{proof}

\subsection{Warning concerning the orientability}
 As we have seen, the 
Auslander-Reiten quilt of any wind wheel is orientable. For example, in section 13.4, we have
exhibited a the quilt of a wind wheel $W$ which is a torus with holes.

But the category $\mo W$ contains as a full subcategory the module category $\mo L$ of the following
algebra $L$, and the Auslander-Reiten quiver of $L$ is obviously (homeomorphic to) a M\"obius strip:
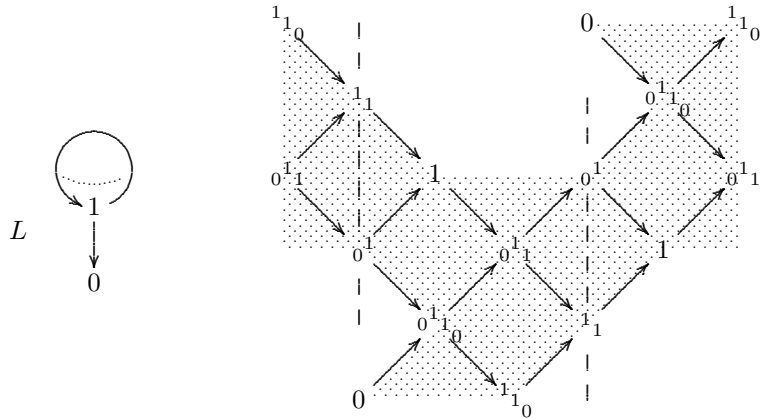
$$
\hbox{\beginpicture
\setcoordinatesystem units <1cm,1cm>
\put{\beginpicture

\put{} at 0 2
\put{$0$} at 0 0
\put{$1$} at 0 1
\put{$L$} at -1 0.7 
\arrow <1.5mm> [0.25,0.75] from 0 0.8 to 0 0.2
\circulararc 155 degrees from 0 2 center at 0 1.5
\circulararc -155 degrees from 0 2 center at 0 1.5
\arrow <1.5mm> [0.25,0.75] from -0.22 1.05 to -0.17 1.02
\setdots <.7mm>
\setquadratic
\plot -.4 1.4 0 1.3 .4 1.4 /

\endpicture} at -1 0 
\put{\beginpicture

\put{ \beginpicture
 \setcoordinatesystem units <.15cm,.15cm>
 \put{$\ssize 1$} at 0 2
 \put{$\ssize 1$} at 1 1
 \put{$\ssize 0$} at 2 0
\endpicture} at -1 5

\put{ \beginpicture
 \setcoordinatesystem units <.15cm,.15cm>
\put{$\ssize 0$} at 0 0
 \put{$\ssize 1$} at 1 1
 \put{$\ssize 1$} at 2 0
\endpicture} at -1 3

\put{$0$}  at 0 0

\put{ \beginpicture
 \setcoordinatesystem units <.15cm,.15cm>
 \put{$\ssize 1$} at 1 1
 \put{$\ssize 0$} at 0 0
\endpicture} at 0 2

\put{ \beginpicture
 \setcoordinatesystem units <.15cm,.15cm>
 \put{$\ssize 1$} at 1 0
 \put{$\ssize 1$} at 0 1
\endpicture} at 0 4

\put{ \beginpicture
 \setcoordinatesystem units <.15cm,.15cm>
 \put{$\ssize 0$} at 0 1
 \put{$\ssize 1$} at 1 2
 \put{$\ssize 1$} at 2 1
 \put{$\ssize 0$} at 3 0
\endpicture} at 1 1

\put{$1$} at  1 3

\put{ \beginpicture
 \setcoordinatesystem units <.15cm,.15cm>
 \put{$\ssize 1$} at 0 2
 \put{$\ssize 1$} at 1 1
 \put{$\ssize 0$} at 2 0
\endpicture} at 2 0

\put{ \beginpicture
 \setcoordinatesystem units <.15cm,.15cm>
 \put{$\ssize 0$} at 0 0
 \put{$\ssize 1$} at 1 1
 \put{$\ssize 1$} at 2 0
\endpicture} at 2 2

\put{ \beginpicture
 \setcoordinatesystem units <.15cm,.15cm>
 \put{$\ssize 1$} at 0 1
 \put{$\ssize 1$} at 1 0
\endpicture} at 3 1

\put{ \beginpicture
 \setcoordinatesystem units <.15cm,.15cm>
 \put{$\ssize 0$} at 0 0
 \put{$\ssize 1$} at 1 1
\endpicture} at 3 3

\put{$0$} at  3 5

\put{$1$} at 4 2

\put{ \beginpicture
 \setcoordinatesystem units <.15cm,.15cm>
 \put{$\ssize 0$} at 0 1
 \put{$\ssize 1$} at 1 2
 \put{$\ssize 1$} at 2 1
 \put{$\ssize 0$} at 3 0 
\endpicture} at 4 4

\put{ \beginpicture
 \setcoordinatesystem units <.15cm,.15cm>
 \put{$\ssize 1$} at 0 2
 \put{$\ssize 1$} at 1 1
 \put{$\ssize 0$} at 2 0
\endpicture} at 5 5

\put{ \beginpicture
 \setcoordinatesystem units <.15cm,.15cm>
\put{$\ssize 0$} at 0 0
 \put{$\ssize 1$} at 1 1
 \put{$\ssize 1$} at 2 0
\endpicture} at 5 3

\arrow <1.5mm> [0.25,0.75] from -.8 4.8  to -.2 4.2 
\arrow <1.5mm> [0.25,0.75] from -.8 3.2  to -.2 3.8 
\arrow <1.5mm> [0.25,0.75] from -.8 2.8  to -.2 2.2 

\arrow <1.5mm> [0.25,0.75] from .2 3.8   to .8 3.2
\arrow <1.5mm> [0.25,0.75] from .2 2.2   to .8 2.8
\arrow <1.5mm> [0.25,0.75] from .2 1.8   to .8 1.2
\arrow <1.5mm> [0.25,0.75] from .2   .2   to .8  .8

\arrow <1.5mm> [0.25,0.75] from 1.2 2.8   to 1.8 2.2
\arrow <1.5mm> [0.25,0.75] from 1.2 1.2   to 1.8 1.8
\arrow <1.5mm> [0.25,0.75] from 1.2   .8   to 1.8   .2

\arrow <1.5mm> [0.25,0.75] from 2.2 2.2   to 2.8 2.8
\arrow <1.5mm> [0.25,0.75] from 2.2 1.8   to 2.8 1.2
\arrow <1.5mm> [0.25,0.75] from 2.2   .2   to 2.8  .8

\arrow <1.5mm> [0.25,0.75] from 3.2 4.8   to 3.8 4.2
\arrow <1.5mm> [0.25,0.75] from 3.2 3.2   to 3.8 3.8
\arrow <1.5mm> [0.25,0.75] from 3.2 2.8   to 3.8 2.2
\arrow <1.5mm> [0.25,0.75] from 3.2 1.2   to 3.8 1.8

\arrow <1.5mm> [0.25,0.75] from 4.2 4.2   to 4.8 4.8
\arrow <1.5mm> [0.25,0.75] from 4.2 3.8   to 4.8 3.2
\arrow <1.5mm> [0.25,0.75] from 4.2 2.2   to 4.8 2.8

\setdashes <2mm>
\plot 0 5  0 4.3 /
\plot 0 2.3  0 3.7 /
\plot 0 1  0 1.7 /

\plot 3 4     3 3.3 /
\plot 3 1.3  3 2.7 /
\plot 3 0     3  .7 /

\setshadegrid span <.7mm>
\vshade -1 2  5  <z,z,,>  0 2 4 <z,z,,> .99 1 3 <z,z,,>  1 0 3 <z,z,,> 2 0 3 <z,z,,> 3 1 3 <z,z,,> 3.99 2 4  <z,z,,> 4 2 5  <z,z,,> 5 2 5 /
\vshade 0 0 0  <z,z,,> 1 0 1 /
\vshade 3 5 5 <z,z,,> 4 4 5 /

\endpicture} at 5 0
\endpicture}
$$
Here, the vertical dashes lines mark a fundamental domain.

Instead of looking at the wind wheel considered in section 12.5,
let $W$ now be  the smallest possible wind wheel, with two vertices $0,1$, a loop $\alpha$ at
the vertex $0$, a loop $\gamma$ at the vertex $1$ and an arrow $\beta:1 \to 0$
(and the relations $\alpha^2 = \gamma^2 = \alpha\beta\gamma = 0$). 
We want to analyze the embedding of $\mo L$ into $\mo W.$

Some parts of the Auslander-Reiten quiver of $L$ can be identified in the quilt of
$W$ without problems:
$$
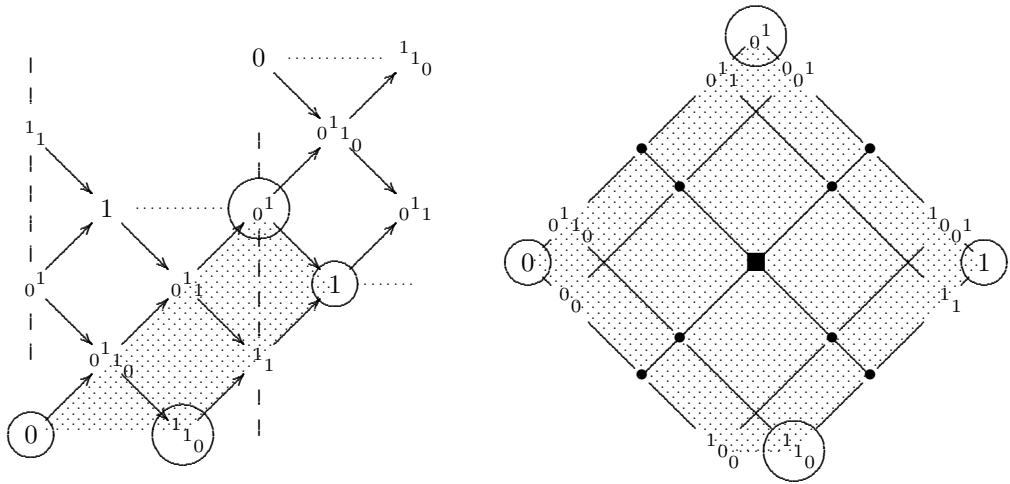
\hbox{\beginpicture
\setcoordinatesystem units <1cm,1cm>
 
\put{\beginpicture

\put{$0$}  at 0 0

\put{ \beginpicture
 \setcoordinatesystem units <.15cm,.15cm>
 \put{$\ssize 1$} at 1 1
 \put{$\ssize 0$} at 0 0
\endpicture} at 0 2

\put{ \beginpicture
 \setcoordinatesystem units <.15cm,.15cm>
 \put{$\ssize 1$} at 1 0
 \put{$\ssize 1$} at 0 1
\endpicture} at 0 4

\put{ \beginpicture
 \setcoordinatesystem units <.15cm,.15cm>
 \put{$\ssize 0$} at 0 1
 \put{$\ssize 1$} at 1 2
 \put{$\ssize 1$} at 2 1
 \put{$\ssize 0$} at 3 0
\endpicture} at 1 1

\put{$1$} at  1 3

\put{ \beginpicture
 \setcoordinatesystem units <.15cm,.15cm>
 \put{$\ssize 1$} at 0 2
 \put{$\ssize 1$} at 1 1
 \put{$\ssize 0$} at 2 0
\endpicture} at 2 0

\put{ \beginpicture
 \setcoordinatesystem units <.15cm,.15cm>
 \put{$\ssize 0$} at 0 0
 \put{$\ssize 1$} at 1 1
 \put{$\ssize 1$} at 2 0
\endpicture} at 2 2

\put{ \beginpicture
 \setcoordinatesystem units <.15cm,.15cm>
 \put{$\ssize 1$} at 0 1
 \put{$\ssize 1$} at 1 0
\endpicture} at 3 1

\put{ \beginpicture
 \setcoordinatesystem units <.15cm,.15cm>
 \put{$\ssize 0$} at 0 0
 \put{$\ssize 1$} at 1 1
\endpicture} at 3 3

\put{$0$} at  3 5

\put{$1$} at 4 2

\put{ \beginpicture
 \setcoordinatesystem units <.15cm,.15cm>
 \put{$\ssize 0$} at 0 1
 \put{$\ssize 1$} at 1 2
 \put{$\ssize 1$} at 2 1
 \put{$\ssize 0$} at 3 0 
\endpicture} at 4 4

\put{ \beginpicture
 \setcoordinatesystem units <.15cm,.15cm>
 \put{$\ssize 1$} at 0 2
 \put{$\ssize 1$} at 1 1
 \put{$\ssize 0$} at 2 0
\endpicture} at 5 5

\put{ \beginpicture
 \setcoordinatesystem units <.15cm,.15cm>
\put{$\ssize 0$} at 0 0
 \put{$\ssize 1$} at 1 1
 \put{$\ssize 1$} at 2 0
\endpicture} at 5 3

\arrow <1.5mm> [0.25,0.75] from .2 3.8   to .8 3.2
\arrow <1.5mm> [0.25,0.75] from .2 2.2   to .8 2.8
\arrow <1.5mm> [0.25,0.75] from .2 1.8   to .8 1.2
\arrow <1.5mm> [0.25,0.75] from .2   .2   to .8  .8

\arrow <1.5mm> [0.25,0.75] from 1.2 2.8   to 1.8 2.2
\arrow <1.5mm> [0.25,0.75] from 1.2 1.2   to 1.8 1.8
\arrow <1.5mm> [0.25,0.75] from 1.2   .8   to 1.8   .2

\arrow <1.5mm> [0.25,0.75] from 2.2 2.2   to 2.8 2.8
\arrow <1.5mm> [0.25,0.75] from 2.2 1.8   to 2.8 1.2
\arrow <1.5mm> [0.25,0.75] from 2.2   .2   to 2.8  .8

\arrow <1.5mm> [0.25,0.75] from 3.2 4.8   to 3.8 4.2
\arrow <1.5mm> [0.25,0.75] from 3.2 3.2   to 3.8 3.8
\arrow <1.5mm> [0.25,0.75] from 3.2 2.8   to 3.8 2.2
\arrow <1.5mm> [0.25,0.75] from 3.2 1.2   to 3.8 1.8

\arrow <1.5mm> [0.25,0.75] from 4.2 4.2   to 4.8 4.8
\arrow <1.5mm> [0.25,0.75] from 4.2 3.8   to 4.8 3.2
\arrow <1.5mm> [0.25,0.75] from 4.2 2.2   to 4.8 2.8

\setdashes <2mm>
\plot 0 5  0 4.3 /
\plot 0 2.3  0 3.7 /
\plot 0 1  0 1.7 /

\plot 3 4     3 3.3 /
\plot 3 1.3  3 2.7 /
\plot 3 0     3  .7 /

\setdots <1mm>
\plot 1.4 3  2.6 3 /
\plot 4.4 2  5 2 /
\plot 3.4 5  4.6 5 /

\setshadegrid span <.7mm>
\vshade 0 0 0  <z,z,,>  2 0 2 <z,z,,> 3 1 3 <z,z,,> 4 2 2  /

\setsolid
\circulararc 360 degrees from 2.4 0 center at  2 0
\circulararc 360 degrees from 0.3 0 center at 0 0
\circulararc 360 degrees from 4.3 2  center at 4 2
\circulararc 360 degrees from 3.4 3 center at 3 3

\endpicture} at 0 0
\put{\beginpicture
\setcoordinatesystem units <1cm,1cm>

\put{$0$}  at -1 2

\put{ \beginpicture
 \setcoordinatesystem units <.15cm,.15cm>
 \put{$\ssize 0$} at 0 1
 \put{$\ssize 1$} at 1 2
 \put{$\ssize 1$} at 2 1
 \put{$\ssize 0$} at 3 0
\endpicture} at -.5 2.5

\put{ \beginpicture
 \setcoordinatesystem units <.15cm,.15cm>
 \put{$\ssize 0$} at 0 1
 \put{$\ssize 0$} at 1 0
\endpicture} at -.5 1.5

\put{$1$} at   5 2

\put{ \beginpicture
 \setcoordinatesystem units <.15cm,.15cm>
 \put{$\ssize 1$} at 0 2
 \put{$\ssize 0$} at 1 1
 \put{$\ssize 0$} at 2 0
\endpicture} at 1.5 -0.5

\put{ \beginpicture
 \setcoordinatesystem units <.15cm,.15cm>
 \put{$\ssize 1$} at 0 2
 \put{$\ssize 1$} at 1 1
 \put{$\ssize 0$} at 2 0
\endpicture} at 2.5 -0.5

\put{ \beginpicture
 \setcoordinatesystem units <.15cm,.15cm>
 \put{$\ssize 0$} at 0 0
 \put{$\ssize 1$} at 1 1
 \put{$\ssize 1$} at 2 0
\endpicture} at 1.5 4.5

\put{ \beginpicture
 \setcoordinatesystem units <.15cm,.15cm>
 \put{$\ssize 0$} at 0 1
 \put{$\ssize 0$} at 1 0
 \put{$\ssize 1$} at 2 1
\endpicture} at 2.5 4.5

\put{ \beginpicture
 \setcoordinatesystem units <.15cm,.15cm>
 \put{$\ssize 1$} at 0 1
 \put{$\ssize 1$} at 1 0
\endpicture} at 4.5 1.5

\put{ \beginpicture
 \setcoordinatesystem units <.15cm,.15cm>
 \put{$\ssize 0$} at 0 0
 \put{$\ssize 1$} at 1 1
\endpicture} at 2 5

\put{ \beginpicture
 \setcoordinatesystem units <.15cm,.15cm>
 \put{$\ssize 1$} at 0 2
 \put{$\ssize 0$} at 1 1
 \put{$\ssize 0$} at 2 0
 \put{$\ssize 1$} at 3 1
\endpicture} at 4.5 2.5

\setsolid
\plot 1.8 -.2  2 0  2.2 -.2 /
\plot 1.8 4.2  2 4  2.2 4.2 /
\plot -.2 1.8  0 2 -.2 2.2 /
\plot 4.2 2.2  4 2 4.2 1.8 /

\plot 0.9 1.1  0 2  0.9 2.9 /
\plot 1.1 3.1  2 4  2.9 3.1 /
\plot 3.1 2.9  4 2  3.1 1.1 /
\plot 2.9 0.9  2 0  1.1 0.9 /

\multiput{$\bullet$} at 1 1 1 3 3 1 3 3 
    0.5 0.5 0.5 3.5 3.5 0.5 3.5 3.5 /
\multiput{$\blacksquare$} at 2 2 /
\plot 0.5 0.5  3.5 3.5 /
\plot 0.5 3.5  3.5 0.5 /

\plot -.2 1.2  0.4 0.6 /
\plot 0.6 0.4 1.2 -.2 /
\plot -.8 1.8  -.7 1.7 /

\plot -.2 2.8  0.4 3.4 /
\plot 0.6 3.6 1.2 4.2 /
\plot -.8 2.2   -.7  2.3 /

\plot 4.2 1.2  3.6 0.6 /
\plot 3.4 0.4  2.8 -.2 /

\plot 4.2 2.8  3.6 3.4 /
\plot 3.4 3.6  2.8 4.2 /

\plot 4.8 1.8  4.7 1.7 /
\plot 4.85 2.15   4.75  2.25 /

\plot 1.7 4.7  1.85 4.85 /
\plot 2.3 4.7  2.15 4.85 /

\setdots <1mm>
\plot 1.7 -.5  2.3 -.5 /

\setshadegrid span <.7mm>
\hshade -0.5 1.5 2.5  <,,z,z> 2 -1 5 <,,z,z> 5 2 2 /

\setsolid
\circulararc 360 degrees from 2.9 -.5 center at 2.5 -.5
\circulararc 360 degrees from -1.3 2 center at -1 2
\circulararc 360 degrees from 2.4 5  center at 2 5
\circulararc 360 degrees from 5.3 2 center at 5 2

\endpicture} at 7 0

\endpicture}
$$
All the maps in the shaded part of $\mo L$ (see the left picture)
are nicely factorized in $\mo W$, see the right picture, note that
on the right we see the bordered tile for $W$.

Of course, this concerns also the following shaded part on the left
(actually, we deal with the same part of $\mo L$):
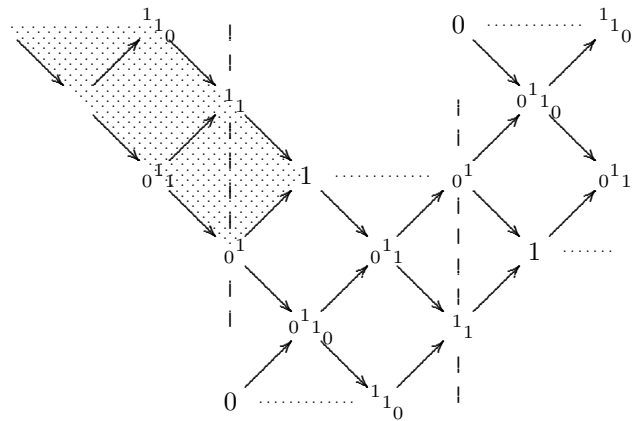
$$
\hbox{\beginpicture
\setcoordinatesystem units <1cm,1cm>
 
\put{\beginpicture

\put{$0$}  at 0 0

\put{ \beginpicture
 \setcoordinatesystem units <.15cm,.15cm>
 \put{$\ssize 1$} at 1 1
 \put{$\ssize 0$} at 0 0
\endpicture} at 0 2

\put{ \beginpicture
 \setcoordinatesystem units <.15cm,.15cm>
 \put{$\ssize 1$} at 1 0
 \put{$\ssize 1$} at 0 1
\endpicture} at 0 4

\put{ \beginpicture
 \setcoordinatesystem units <.15cm,.15cm>
 \put{$\ssize 0$} at 0 1
 \put{$\ssize 1$} at 1 2
 \put{$\ssize 1$} at 2 1
 \put{$\ssize 0$} at 3 0
\endpicture} at 1 1

\put{$1$} at  1 3

\put{ \beginpicture
 \setcoordinatesystem units <.15cm,.15cm>
 \put{$\ssize 1$} at 0 2
 \put{$\ssize 1$} at 1 1
 \put{$\ssize 0$} at 2 0
\endpicture} at 2 0

\put{ \beginpicture
 \setcoordinatesystem units <.15cm,.15cm>
 \put{$\ssize 0$} at 0 0
 \put{$\ssize 1$} at 1 1
 \put{$\ssize 1$} at 2 0
\endpicture} at 2 2

\put{ \beginpicture
 \setcoordinatesystem units <.15cm,.15cm>
 \put{$\ssize 1$} at 0 1
 \put{$\ssize 1$} at 1 0
\endpicture} at 3 1

\put{ \beginpicture
 \setcoordinatesystem units <.15cm,.15cm>
 \put{$\ssize 0$} at 0 0
 \put{$\ssize 1$} at 1 1
\endpicture} at 3 3

\put{$0$} at  3 5

\put{$1$} at 4 2

\put{ \beginpicture
 \setcoordinatesystem units <.15cm,.15cm>
 \put{$\ssize 0$} at 0 1
 \put{$\ssize 1$} at 1 2
 \put{$\ssize 1$} at 2 1
 \put{$\ssize 0$} at 3 0 
\endpicture} at 4 4

\put{ \beginpicture
 \setcoordinatesystem units <.15cm,.15cm>
 \put{$\ssize 1$} at 0 2
 \put{$\ssize 1$} at 1 1
 \put{$\ssize 0$} at 2 0
\endpicture} at 5 5

\put{ \beginpicture
 \setcoordinatesystem units <.15cm,.15cm>
\put{$\ssize 0$} at 0 0
 \put{$\ssize 1$} at 1 1
 \put{$\ssize 1$} at 2 0
\endpicture} at 5 3

\arrow <1.5mm> [0.25,0.75] from .2 3.8   to .8 3.2
\arrow <1.5mm> [0.25,0.75] from .2 2.2   to .8 2.8
\arrow <1.5mm> [0.25,0.75] from .2 1.8   to .8 1.2
\arrow <1.5mm> [0.25,0.75] from .2   .2   to .8  .8

\arrow <1.5mm> [0.25,0.75] from 1.2 2.8   to 1.8 2.2
\arrow <1.5mm> [0.25,0.75] from 1.2 1.2   to 1.8 1.8
\arrow <1.5mm> [0.25,0.75] from 1.2   .8   to 1.8   .2

\arrow <1.5mm> [0.25,0.75] from 2.2 2.2   to 2.8 2.8
\arrow <1.5mm> [0.25,0.75] from 2.2 1.8   to 2.8 1.2
\arrow <1.5mm> [0.25,0.75] from 2.2   .2   to 2.8  .8

\arrow <1.5mm> [0.25,0.75] from 3.2 4.8   to 3.8 4.2
\arrow <1.5mm> [0.25,0.75] from 3.2 3.2   to 3.8 3.8
\arrow <1.5mm> [0.25,0.75] from 3.2 2.8   to 3.8 2.2
\arrow <1.5mm> [0.25,0.75] from 3.2 1.2   to 3.8 1.8

\arrow <1.5mm> [0.25,0.75] from 4.2 4.2   to 4.8 4.8
\arrow <1.5mm> [0.25,0.75] from 4.2 3.8   to 4.8 3.2
\arrow <1.5mm> [0.25,0.75] from 4.2 2.2   to 4.8 2.8

\setdashes <2mm>
\plot 0 5  0 4.3 /
\plot 0 2.3  0 3.7 /
\plot 0 1  0 1.7 /

\plot 3 4     3 3.3 /
\plot 3 1.3  3 2.7 /
\plot 3 0     3  .7 /

\setdots <1mm>
\plot 1.4 3  2.6 3 /
\plot 4.4 2  5 2 /
\plot 3.4 5  4.6 5 /
\plot 0.4 0  1.6 0  /

\setshadegrid span <.7mm>
\hshade 2 0 0  <,,z,z>  3 -1 1 <,,z,z> 5 -3 -1  /

\setsolid

\put{ \beginpicture
 \setcoordinatesystem units <.15cm,.15cm>
 \put{$\ssize 1$} at 0 2
 \put{$\ssize 1$} at 1 1
 \put{$\ssize 0$} at 2 0
\endpicture} at -1 5

\put{ \beginpicture
 \setcoordinatesystem units <.15cm,.15cm>
\put{$\ssize 0$} at 0 0
 \put{$\ssize 1$} at 1 1
 \put{$\ssize 1$} at 2 0
\endpicture} at -1 3

\arrow <1.5mm> [0.25,0.75] from -.8 4.8   to -.2 4.2
\arrow <1.5mm> [0.25,0.75] from -.8 3.2   to -.2 3.8
\arrow <1.5mm> [0.25,0.75] from -.8 2.8   to -.2 2.2

\arrow <1.5mm> [0.25,0.75] from -1.8 4.2   to -1.2 4.8
\arrow <1.5mm> [0.25,0.75] from -1.8 3.8   to -1.2 3.2

\arrow <1.5mm> [0.25,0.75] from -2.8 4.8   to -2.2 4.2

\endpicture} at 0 0
\endpicture}
$$

It is the following square in $\mo L$ which is difficult to recover
in the Auslander-Reiten quilt of $W$:
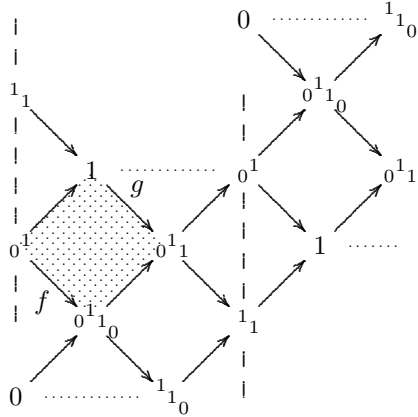
$$
\hbox{\beginpicture
\setcoordinatesystem units <1cm,1cm>
 
\put{\beginpicture

\put{$0$}  at 0 0

\put{ \beginpicture
 \setcoordinatesystem units <.15cm,.15cm>
 \put{$\ssize 1$} at 1 1
 \put{$\ssize 0$} at 0 0
\endpicture} at 0 2

\put{ \beginpicture
 \setcoordinatesystem units <.15cm,.15cm>
 \put{$\ssize 1$} at 1 0
 \put{$\ssize 1$} at 0 1
\endpicture} at 0 4

\put{ \beginpicture
 \setcoordinatesystem units <.15cm,.15cm>
 \put{$\ssize 0$} at 0 1
 \put{$\ssize 1$} at 1 2
 \put{$\ssize 1$} at 2 1
 \put{$\ssize 0$} at 3 0
\endpicture} at 1 1

\put{$1$} at  1 3

\put{ \beginpicture
 \setcoordinatesystem units <.15cm,.15cm>
 \put{$\ssize 1$} at 0 2
 \put{$\ssize 1$} at 1 1
 \put{$\ssize 0$} at 2 0
\endpicture} at 2 0

\put{ \beginpicture
 \setcoordinatesystem units <.15cm,.15cm>
 \put{$\ssize 0$} at 0 0
 \put{$\ssize 1$} at 1 1
 \put{$\ssize 1$} at 2 0
\endpicture} at 2 2

\put{ \beginpicture
 \setcoordinatesystem units <.15cm,.15cm>
 \put{$\ssize 1$} at 0 1
 \put{$\ssize 1$} at 1 0
\endpicture} at 3 1

\put{ \beginpicture
 \setcoordinatesystem units <.15cm,.15cm>
 \put{$\ssize 0$} at 0 0
 \put{$\ssize 1$} at 1 1
\endpicture} at 3 3

\put{$0$} at  3 5

\put{$1$} at 4 2

\put{ \beginpicture
 \setcoordinatesystem units <.15cm,.15cm>
 \put{$\ssize 0$} at 0 1
 \put{$\ssize 1$} at 1 2
 \put{$\ssize 1$} at 2 1
 \put{$\ssize 0$} at 3 0 
\endpicture} at 4 4

\put{ \beginpicture
 \setcoordinatesystem units <.15cm,.15cm>
 \put{$\ssize 1$} at 0 2
 \put{$\ssize 1$} at 1 1
 \put{$\ssize 0$} at 2 0
\endpicture} at 5 5

\put{ \beginpicture
 \setcoordinatesystem units <.15cm,.15cm>
\put{$\ssize 0$} at 0 0
 \put{$\ssize 1$} at 1 1
 \put{$\ssize 1$} at 2 0
\endpicture} at 5 3

\arrow <1.5mm> [0.25,0.75] from .2 3.8   to .8 3.2
\arrow <1.5mm> [0.25,0.75] from .2 2.2   to .8 2.8
\arrow <1.5mm> [0.25,0.75] from .2 1.8   to .8 1.2
\arrow <1.5mm> [0.25,0.75] from .2   .2   to .8  .8

\arrow <1.5mm> [0.25,0.75] from 1.2 2.8   to 1.8 2.2
\arrow <1.5mm> [0.25,0.75] from 1.2 1.2   to 1.8 1.8
\arrow <1.5mm> [0.25,0.75] from 1.2   .8   to 1.8   .2

\arrow <1.5mm> [0.25,0.75] from 2.2 2.2   to 2.8 2.8
\arrow <1.5mm> [0.25,0.75] from 2.2 1.8   to 2.8 1.2
\arrow <1.5mm> [0.25,0.75] from 2.2   .2   to 2.8  .8

\arrow <1.5mm> [0.25,0.75] from 3.2 4.8   to 3.8 4.2
\arrow <1.5mm> [0.25,0.75] from 3.2 3.2   to 3.8 3.8
\arrow <1.5mm> [0.25,0.75] from 3.2 2.8   to 3.8 2.2
\arrow <1.5mm> [0.25,0.75] from 3.2 1.2   to 3.8 1.8

\arrow <1.5mm> [0.25,0.75] from 4.2 4.2   to 4.8 4.8
\arrow <1.5mm> [0.25,0.75] from 4.2 3.8   to 4.8 3.2
\arrow <1.5mm> [0.25,0.75] from 4.2 2.2   to 4.8 2.8

\setdashes <2mm>
\plot 0 5  0 4.3 /
\plot 0 2.3  0 3.7 /
\plot 0 1  0 1.7 /

\plot 3 4     3 3.3 /
\plot 3 1.3  3 2.7 /
\plot 3 0     3  .7 /

\setdots <1mm>
\plot 1.4 3  2.6 3 /
\plot 4.4 2  5 2 /
\plot 3.4 5  4.6 5 /
\plot 0.4 0 1.6 0  /

\setshadegrid span <.7mm>
\vshade 0 2 2  <z,z,,>  1 1 3 <z,z,,> 2 2 2  /

\setsolid
\put{$f$} at 0.35 1.25
\put{$g$} at 1.6 2.75 

\endpicture} at 0 0
\endpicture}
$$
To be more precise: we should distinguish between the maps
pointing upwards and those pointing downwards, 
the maps pointing upwards have already been discussed
since they are part of the bordered tile, thus let us
concentrate on the maps pointing downwards, they are labeled
$f$ and $g$ in the picture.
We want to see in which way these maps can be factorized in the
category $\mo W$.

First, consider the map   $f:
 \beginpicture
 \setcoordinatesystem units <.15cm,.15cm>
 \put{$\ssize 0$} at 0 0
 \put{$\ssize 1$} at 1 1
\endpicture \longrightarrow 
 \beginpicture
 \setcoordinatesystem units <.15cm,.15cm>
 \put{$\ssize 0$} at 0 1
 \put{$\ssize 1$} at 1 2
 \put{$\ssize 1$} at 2 1
 \put{$\ssize 0$} at 3 0
\endpicture =
 \beginpicture
 \setcoordinatesystem units <.15cm,.15cm>
 \put{$\ssize 0$} at 0 0
 \put{$\ssize 1$} at 1 1
 \put{$\ssize 1$} at 2 2
 \put{$\ssize 0$} at 3 1
\endpicture.
$
We can factor is as follows (always additions or deletions on the right):
$$
\hbox{\beginpicture
 \setcoordinatesystem units <.8cm,.7cm>

\put{\beginpicture
 \setcoordinatesystem units <.1cm,.15cm>
 \put{$\ssize 0$} at 0 0
 \put{$\ssize 1$} at 1 1
\endpicture} at 0.2 0

\put{\beginpicture
 \setcoordinatesystem units <.1cm,.15cm>
 \put{$\ssize 0$} at 0 0
 \put{$\ssize 1$} at 1 1
 \put{$\ssize 1$} at 2 2
 \put{$\ssize 0$} at 3 1
 \put{$\ssize 0$} at 4 0 
\endpicture} at 2 0

\put{\beginpicture
 \setcoordinatesystem units <.1cm,.15cm>
 \put{$\ssize 0$} at 0 0
 \put{$\ssize 1$} at 1 1
 \put{$\ssize 1$} at 2 2
 \put{$\ssize 0$} at 3 1
 \put{$\ssize 0$} at 4 0 
 \put{$\ssize 1$} at 5 1
\endpicture} at 4 0

\put{\beginpicture
 \setcoordinatesystem units <.1cm,.15cm>
 \put{$\ssize 0$} at 0 0
 \put{$\ssize 1$} at 1 1
 \put{$\ssize 1$} at 2 2
 \put{$\ssize 0$} at 3 1
 \put{$\ssize 0$} at 4 0 
 \put{$\ssize 1$} at 5 1
 \put{$\ssize 1$} at 6 2
 \put{$\ssize 0$} at 7 1
 \put{$\ssize 0$} at 8 0 
\endpicture} at 6 0

\arrow <1.5mm> [0.25,0.75] from 0.5 0   to 1.5 0
\arrow <1.5mm> [0.25,0.75] from 2.5 0   to 3.5 0
\arrow <1.5mm> [0.25,0.75] from 4.5 0   to 5.5 0
\arrow <1.5mm> [0.25,0.75] from 6.5 0   to 7.5 0
\put{$f_1$} at 1 0.5
\put{$f_2$} at 3 0.5
\put{$f_3$} at 5 0.5
\put{$f_4$} at 7 0.5

\put{$\cdots$} at 8 0
\put{$\cdots$} at 10 0

\put{\beginpicture
 \setcoordinatesystem units <.1cm,.15cm>
 \put{$\ssize 0$} at 0 0
 \put{$\ssize 1$} at 1 1
 \put{$\ssize 1$} at 2 2
 \put{$\ssize 0$} at 3 1
\endpicture} at 15.8 0

\put{\beginpicture
 \setcoordinatesystem units <.1cm,.15cm>
 \put{$\ssize 0$} at 0 0
 \put{$\ssize 1$} at 1 1
 \put{$\ssize 1$} at 2 2
 \put{$\ssize 0$} at 3 1
 \put{$\ssize 0$} at 4 0 
 \put{$\ssize 1$} at 5 1
 \put{$\ssize 1$} at 6 2
\endpicture} at 14 0

\put{\beginpicture
 \setcoordinatesystem units <.1cm,.15cm>
 \put{$\ssize 0$} at 0 0
 \put{$\ssize 1$} at 1 1
 \put{$\ssize 1$} at 2 2
 \put{$\ssize 0$} at 3 1
 \put{$\ssize 0$} at 4 0 
 \put{$\ssize 1$} at 5 1
 \put{$\ssize 1$} at 6 2
 \put{$\ssize 0$} at 7 1
\endpicture} at 12 0

\arrow <1.5mm> [0.25,0.75] from 10.5 0   to 11.5 0
\arrow <1.5mm> [0.25,0.75] from 12.5 0   to 13.5 0
\arrow <1.5mm> [0.25,0.75] from 14.5 0   to 15.5 0
\put{${}_3f$} at 11 0.5
\put{${}_2f$} at 13 0.5
\put{${}_1f$} at 15 0.5

\endpicture}
$$
Similarly, we look at the map 
$g: 1 
\longrightarrow 
 \beginpicture
 \setcoordinatesystem units <.15cm,.15cm>
 \put{$\ssize 0$} at 0 0
 \put{$\ssize 1$} at 1 1
 \put{$\ssize 1$} at 2 0
\endpicture =
 \beginpicture
 \setcoordinatesystem units <.15cm,.15cm>
 \put{$\ssize 1$} at 0 0
 \put{$\ssize 1$} at 1 1
 \put{$\ssize 0$} at 2 0
\endpicture
$
and factor it  (again always additions or deletions on the right):
$$
\hbox{\beginpicture
 \setcoordinatesystem units <.8cm,.7cm>

\put{\beginpicture
 \setcoordinatesystem units <.1cm,.15cm>
 \put{$\ssize 1$} at 1 1
\endpicture} at 0.2 0

\put{\beginpicture
 \setcoordinatesystem units <.1cm,.15cm>
 \put{$\ssize 1$} at 1 1
 \put{$\ssize 1$} at 2 2
 \put{$\ssize 0$} at 3 1
 \put{$\ssize 0$} at 4 0 
\endpicture} at 2 0

\put{\beginpicture
 \setcoordinatesystem units <.1cm,.15cm>
 \put{$\ssize 1$} at 1 1
 \put{$\ssize 1$} at 2 2
 \put{$\ssize 0$} at 3 1
 \put{$\ssize 0$} at 4 0 
 \put{$\ssize 1$} at 5 1
\endpicture} at 4 0

\put{\beginpicture
 \setcoordinatesystem units <.1cm,.15cm>

 \put{$\ssize 1$} at 1 1
 \put{$\ssize 1$} at 2 2
 \put{$\ssize 0$} at 3 1
 \put{$\ssize 0$} at 4 0 
 \put{$\ssize 1$} at 5 1
 \put{$\ssize 1$} at 6 2
 \put{$\ssize 0$} at 7 1
 \put{$\ssize 0$} at 8 0 
\endpicture} at 6 0

\arrow <1.5mm> [0.25,0.75] from 0.5 0   to 1.5 0
\arrow <1.5mm> [0.25,0.75] from 2.5 0   to 3.5 0
\arrow <1.5mm> [0.25,0.75] from 4.5 0   to 5.5 0
\arrow <1.5mm> [0.25,0.75] from 6.5 0   to 7.5 0
\put{$g_1$} at 1 0.5
\put{$g_2$} at 3 0.5
\put{$g_3$} at 5 0.5
\put{$g_4$} at 7 0.5

\put{$\cdots$} at 8 0
\put{$\cdots$} at 10 0

\put{\beginpicture
 \setcoordinatesystem units <.1cm,.15cm>
 \put{$\ssize 1$} at 1 1
 \put{$\ssize 1$} at 2 2
 \put{$\ssize 0$} at 3 1
\endpicture} at 15.8 0

\put{\beginpicture
 \setcoordinatesystem units <.1cm,.15cm>
 \put{$\ssize 1$} at 1 1
 \put{$\ssize 1$} at 2 2
 \put{$\ssize 0$} at 3 1
 \put{$\ssize 0$} at 4 0 
 \put{$\ssize 1$} at 5 1
 \put{$\ssize 1$} at 6 2
\endpicture} at 14 0

\put{\beginpicture
 \setcoordinatesystem units <.1cm,.15cm>
 \put{$\ssize 1$} at 1 1
 \put{$\ssize 1$} at 2 2
 \put{$\ssize 0$} at 3 1
 \put{$\ssize 0$} at 4 0 
 \put{$\ssize 1$} at 5 1
 \put{$\ssize 1$} at 6 2
 \put{$\ssize 0$} at 7 1
\endpicture} at 12 0

\arrow <1.5mm> [0.25,0.75] from 10.5 0   to 11.5 0
\arrow <1.5mm> [0.25,0.75] from 12.5 0   to 13.5 0
\arrow <1.5mm> [0.25,0.75] from 14.5 0   to 15.5 0
\put{${}_3g$} at 11 0.5
\put{${}_2g$} at 13 0.5
\put{${}_1g$} at 15 0.5

\endpicture}
$$
Now if $i$ is odd, then the maps $f_i$ and $g_i$ both are obtained by  the addition of a hook, thus 
they are irreducible, and both ${}_if$ and ${}_ig$ are obtained by the deletion of a cohook, 
thus they also are irreducible. But for $i$ even, all 
the maps $f_i,\ {}_if,\  g_i,\  {}_ig$ belong to the infinite radical $\rad^\omega$
(and actually not to $(\rad^\omega)^2$). By definition (see for example \cite{Sch}), the infinite
radical $\rad^\omega$ is the intersection of the powers $\rad^d$ with $d\in \mathbb N.$ 
It follows that the maps $f,\ g$  belong to all powers of the infinite radical, thus
to $\rad^{\omega 2}$.

The sequences of maps displayed here show that $f$ factors through the
adic module given by the (expanding) word $(\beta\gamma\beta^{-1}\alpha^{-1})^\infty$, 
whereas $g$ factors through the Pr\"ufer module given by the word (contracting) word
$\gamma\beta^{-1}\alpha^{-1}(\beta)^\infty$. Note that both these modules are 
indecomposable algebraically compact modules which are not used in the quilt.

Let us try to following the factorization in the quilt. First, we consider again $f$.
Looking at the maps $f_i$ we have noted already that those with odd index are irreducible,
they belong to $\mathcal P'$, whereas those with even index factor through an upwards path
in the tile $\mathcal T$. Similarly, the maps ${}_if$ with odd index are irreducible maps
inside $\mathcal R'0$, whereas those with even index factor throgh a downwards path
in the tile $\mathcal T$:
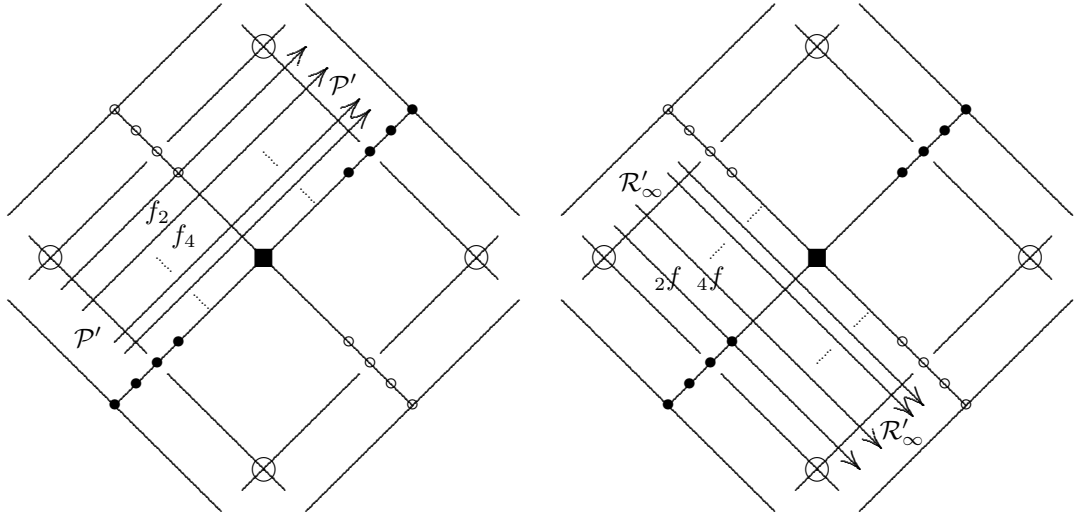
$$
\hbox{\beginpicture
\setcoordinatesystem units <1.4cm,1.4cm>
\put{\beginpicture
\multiput{} at 0 4.4  0 -.4 /
\multiput{$\mathcal P'$} at .35 1.25  2.75 3.65 /
\plot -.4 1.6  1.6 -.4 /
\plot 2.4 4.4  4.4  2.4 /

\plot -.4 2.4  1.6 4.4 /
\plot 2.4 -0.4  4.4 1.6 /

\setsolid
\plot 1.8 -.2  2 0  2.2 -.2 /
\plot 1.8 4.2  2 4  2.2 4.2 /
\plot -.2 1.8  0 2 -.2 2.2 /
\plot 4.2 2.2  4 2 4.2 1.8 /

\plot 0.9 1.1  0 2  0.9 2.9 /
\plot 1.1 3.1  2 4  2.9 3.1 /
\plot 3.1 2.9  4 2  3.1 1.1 /
\plot 2.9 0.9  2 0  1.1 0.9 /

\multiput{$\bigcirc$} at 2 0  0 2  2 4  4 2 /
\multiput{$\bullet$} at 1 1  3 3  0.8 0.8  0.6 0.6  3.2 3.2  3.4 3.4  1.2 1.2  2.8 2.8 /
\multiput{$\circ$} at 1 3  3 1  0.8 3.2  3.2 0.8  0.6 3.4  3.4 0.6  1.2 2.8  2.8 1.2  /

\multiput{$\blacksquare$} at 2 2 /
\plot 0.6 0.6  3.4 3.4 /
\plot 0.6 3.4  3.4 0.6 /

\put{} at 0 0
\put{} at 4 4 

\arrow <2.5mm> [0.25,0.75] from 0.1 1.7 to 2.4 4  
\arrow <2.5mm> [0.25,0.75] from 0.3 1.5 to 2.6 3.8  
\arrow <2.5mm> [0.25,0.75] from 0.6 1.2 to 2.9 3.5
\arrow <2.5mm> [0.25,0.75] from 0.7 1.1 to 3.0 3.4  
\put{$f_2$} at 1.0 2.43
\put{$f_4$} at 1.25 2.2
\setdots <.5mm>
\plot 2 3  2.15 2.85 /
\plot 2.35 2.65  2.5 2.5 /
\plot 1.0 2.0 1.15 1.85 /
\plot 1.35 1.65  1.5 1.5 /
\endpicture} at 0 0
\put{\beginpicture
\multiput{} at 0 4.4  0 -.4 /
\multiput{$\mathcal R'_\infty$} at 2.8 0.35 .35 2.7 /
\plot -.4 1.6  1.6 -.4 /
\plot 2.4 4.4  4.4  2.4 /

\plot -.4 2.4  1.6 4.4 /
\plot 2.4 -0.4  4.4 1.6 /

\setsolid
\plot 1.8 -.2  2 0  2.2 -.2 /
\plot 1.8 4.2  2 4  2.2 4.2 /
\plot -.2 1.8  0 2 -.2 2.2 /
\plot 4.2 2.2  4 2 4.2 1.8 /

\plot 0.9 1.1  0 2  0.9 2.9 /
\plot 1.1 3.1  2 4  2.9 3.1 /
\plot 3.1 2.9  4 2  3.1 1.1 /
\plot 2.9 0.9  2 0  1.1 0.9 /

\multiput{$\bigcirc$} at 2 0  0 2  2 4  4 2 /
\multiput{$\bullet$} at 1 1  3 3  0.8 0.8  0.6 0.6  3.2 3.2  3.4 3.4  1.2 1.2  2.8 2.8 /
\multiput{$\circ$} at 1 3  3 1  0.8 3.2  3.2 0.8  0.6 3.4  3.4 0.6  1.2 2.8  2.8 1.2  /

\multiput{$\blacksquare$} at 2 2 /
\plot 0.6 0.6  3.4 3.4 /
\plot 0.6 3.4  3.4 0.6 /

\put{} at 0 0
\put{} at 4 4 

\arrow <2.5mm> [0.25,0.75] from 0.1 2.3 to 2.4 0  
\arrow <2.5mm> [0.25,0.75] from 0.3 2.5 to 2.6 0.2  
\arrow <2.5mm> [0.25,0.75] from 0.6 2.8 to 2.9 .5
\arrow <2.5mm> [0.25,0.75] from 0.7 2.9 to 3.0 .6  
\put{${}_2f$} at .6 1.8
\put{${}_4f$} at 1 1.8

\setdots <.5mm>
\plot 2 1  2.15 1.15 /
\plot 2.35 1.35  2.5 1.5 /
\plot 1.0 2.0 1.15 2.15 /
\plot 1.35 2.35  1.5 2.5 /

\endpicture} at 5.2 0
\endpicture}
$$
What is of importance is the change of the direction which we encounter:
as long as we deal with the maps $f_i$ we work with maps pointing upwards,
but after we have passed the adic module (which is hidden) we deal with the
maps ${}_if$ and they point downwards. 

There is the similar feature for the map $g$: 
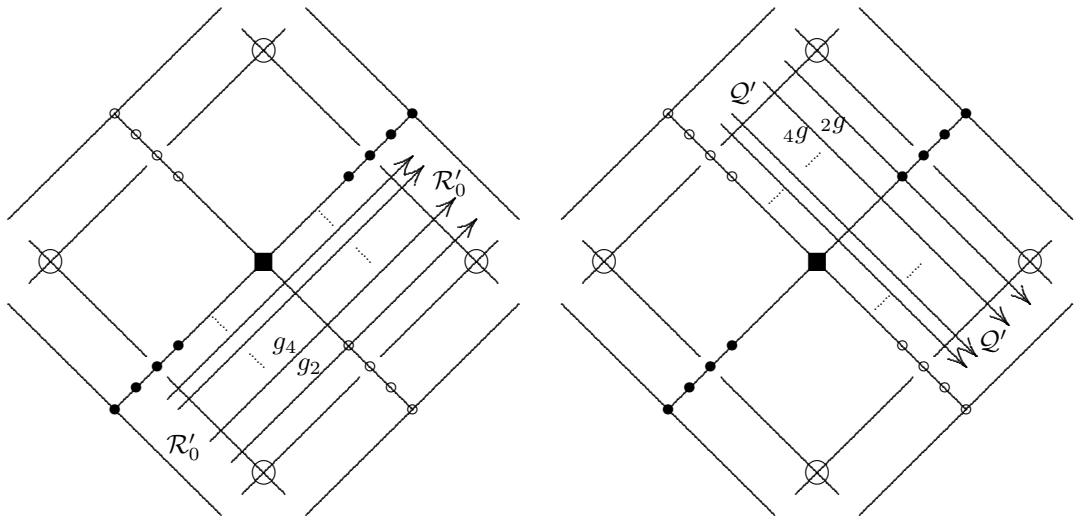
$$
\hbox{\beginpicture
\setcoordinatesystem units <1.4cm,1.4cm>
\put{\beginpicture
\multiput{} at 0 4.4  0 -.4 /
\multiput{$\mathcal R'_0$} at 1.25 0.25 3.75 2.75 /
\plot -.4 1.6  1.6 -.4 /
\plot 2.4 4.4  4.4  2.4 /

\plot -.4 2.4  1.6 4.4 /
\plot 2.4 -0.4  4.4 1.6 /

\setsolid
\plot 1.8 -.2  2 0  2.2 -.2 /
\plot 1.8 4.2  2 4  2.2 4.2 /
\plot -.2 1.8  0 2 -.2 2.2 /
\plot 4.2 2.2  4 2 4.2 1.8 /

\plot 0.9 1.1  0 2  0.9 2.9 /
\plot 1.1 3.1  2 4  2.9 3.1 /
\plot 3.1 2.9  4 2  3.1 1.1 /
\plot 2.9 0.9  2 0  1.1 0.9 /

\multiput{$\bigcirc$} at 2 0  0 2  2 4  4 2 /
\multiput{$\bullet$} at 1 1  3 3  0.8 0.8  0.6 0.6  3.2 3.2  3.4 3.4  1.2 1.2  2.8 2.8 /
\multiput{$\circ$} at 1 3  3 1  0.8 3.2  3.2 0.8  0.6 3.4  3.4 0.6  1.2 2.8  2.8 1.2  /

\multiput{$\blacksquare$} at 2 2 /
\plot 0.6 0.6  3.4 3.4 /
\plot 0.6 3.4  3.4 0.6 /

\put{} at 0 0
\put{} at 4 4

\arrow <2.5mm> [0.25,0.75] from 1.7 .1 to  4 2.4   
\arrow <2.5mm> [0.25,0.75] from 1.5 .3 to 3.8 2.6  
\arrow <2.5mm> [0.25,0.75] from 1.2  .6 to 3.5 2.9 
\arrow <2.5mm> [0.25,0.75] from 1.1  .7 to 3.4 3.0  

\put{$g_2$} at  2.43 1.0
\put{$g_4$} at  2.2 1.2
\setdots <.5mm>
\plot 3 2  2.85 2.15  /
\plot 2.65 2.35  2.5 2.5 /
\plot 2.0 1.0 1.85 1.15 /
\plot 1.65 1.35  1.5 1.5 /
\endpicture} at 0 0

\put{\beginpicture
\multiput{} at 0 4.4  0 -.4 /
\multiput{$\mathcal Q'$} at 3.65 1.25  1.3 3.6 /
\plot -.4 1.6  1.6 -.4 /
\plot 2.4 4.4  4.4  2.4 /

\plot -.4 2.4  1.6 4.4 /
\plot 2.4 -0.4  4.4 1.6 /

\setsolid
\plot 1.8 -.2  2 0  2.2 -.2 /
\plot 1.8 4.2  2 4  2.2 4.2 /
\plot -.2 1.8  0 2 -.2 2.2 /
\plot 4.2 2.2  4 2 4.2 1.8 /

\plot 0.9 1.1  0 2  0.9 2.9 /
\plot 1.1 3.1  2 4  2.9 3.1 /
\plot 3.1 2.9  4 2  3.1 1.1 /
\plot 2.9 0.9  2 0  1.1 0.9 /

\multiput{$\bigcirc$} at 2 0  0 2  2 4  4 2 /
\multiput{$\bullet$} at 1 1  3 3  0.8 0.8  0.6 0.6  3.2 3.2  3.4 3.4  1.2 1.2  2.8 2.8 /
\multiput{$\circ$} at 1 3  3 1  0.8 3.2  3.2 0.8  0.6 3.4  3.4 0.6  1.2 2.8  2.8 1.2  /

\multiput{$\blacksquare$} at 2 2 /
\plot 0.6 0.6  3.4 3.4 /
\plot 0.6 3.4  3.4 0.6 /

\put{} at 0 0
\put{} at 4 4 

\arrow <2.5mm> [0.25,0.75] from 1.7 3.9 to 4. 1.6   
\arrow <2.5mm> [0.25,0.75] from 1.5 3.7 to 3.8 1.4   
\arrow <2.5mm> [0.25,0.75] from 1.2 3.4 to 3.5 1.1   
\arrow <2.5mm> [0.25,0.75] from 1.1 3.3 to 3.4 1.   

\put{${}_2g$} at 2.15 3.3
\put{${}_4g$} at 1.8 3.2

\setdots <.5mm>
\plot 2.85 1.85  3 2 /
\plot 2.55 1.55  2.7 1.7 /
\plot 1.9 2.9 2.05 3.05 /
\plot 1.55 2.55  1.7 2.7 /

\endpicture} at 5.2 0

\endpicture}
$$
Here we deal first with the maps $g_i$  pointing wards,
and after we have passed the hidden Pr\"ufer module we deal with the
maps ${}_ig$ and they point again downwards. 
     \medskip

Altogether we may say that the maps $f$ and $g$ are embedded into the
quilt of $W$ with a kind of crossing, so that the shaded parts are
connected by a square which is folded over:
$$
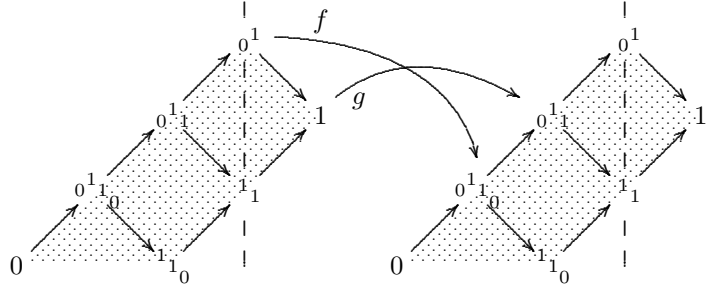
\hbox{\beginpicture
\setcoordinatesystem units <1cm,1cm>
 
\multiput{\beginpicture

\put{$0$}  at 0 0

\put{ \beginpicture
 \setcoordinatesystem units <.15cm,.15cm>
 \put{$\ssize 0$} at 0 1
 \put{$\ssize 1$} at 1 2
 \put{$\ssize 1$} at 2 1
 \put{$\ssize 0$} at 3 0
\endpicture} at 1 1

\put{ \beginpicture
 \setcoordinatesystem units <.15cm,.15cm>
 \put{$\ssize 1$} at 0 2
 \put{$\ssize 1$} at 1 1
 \put{$\ssize 0$} at 2 0
\endpicture} at 2 0

\put{ \beginpicture
 \setcoordinatesystem units <.15cm,.15cm>
 \put{$\ssize 0$} at 0 0
 \put{$\ssize 1$} at 1 1
 \put{$\ssize 1$} at 2 0
\endpicture} at 2 2

\put{ \beginpicture
 \setcoordinatesystem units <.15cm,.15cm>
 \put{$\ssize 1$} at 0 1
 \put{$\ssize 1$} at 1 0
\endpicture} at 3 1

\put{ \beginpicture
 \setcoordinatesystem units <.15cm,.15cm>
 \put{$\ssize 0$} at 0 0
 \put{$\ssize 1$} at 1 1
\endpicture} at 3 3

\put{$1$} at 4 2

\arrow <1.5mm> [0.25,0.75] from .2   .2   to .8  .8
\arrow <1.5mm> [0.25,0.75] from 1.2   1.2   to 1.8  1.8
\arrow <1.5mm> [0.25,0.75] from 2.2   2.2   to 2.8  2.8

\arrow <1.5mm> [0.25,0.75] from 1.2   .8   to 1.8   .2

\arrow <1.5mm> [0.25,0.75] from 2.2 1.8   to 2.8 1.2
\arrow <1.5mm> [0.25,0.75] from 2.2   .2   to 2.8  .8

\arrow <1.5mm> [0.25,0.75] from 3.2 2.8   to 3.8 2.2
\arrow <1.5mm> [0.25,0.75] from 3.2 1.2   to 3.8 1.8

\setdashes <2mm>

\plot 3 3.5     3 3.3 /
\plot 3 1.3  3 2.7 /
\plot 3 0     3  .7 /

\setshadegrid span <.7mm>
\vshade 0 0 0  <z,z,,>  2 0 2 <z,z,,> 3 1 3 <z,z,,> 4 2 2  /
\put{} at 0 3.5
\endpicture} at 0 0  5 0 /
\setquadratic
\plot 2.2 .6  3.3 1  4.6 0.6 /
\plot 1.4 1.4 3.1 1  4. 0 /
\arrow <1.5mm> [0.25,0.75] from 4.55 0.63  to 4.61 0.6
\arrow <1.5mm> [0.25,0.75] from 4 0  to 4.05 -.2
\put{$f$} at 2 1.6
\put{$g$} at 2.5 .55

\endpicture}
$$

\section{ The Auslander-Reiten quiver of a barbell.}

\begin{proposition} Barbell algebras are of non-polynomial growth.
\end{proposition}

\begin{proof} Given an arrow $\alpha$, we denote by $\mathcal N(\alpha)$ 
the set of cyclic words starting with $\alpha$ and ending in an inverse
letter (for all the words in $\mathcal N(\alpha)$, the last letter is a fixed one,
namely the inverse of the only arrow 
different from $\alpha$ which has the same end point as $\alpha$).
Clearly, $\mathcal N(\alpha)$ is a semigroup.
Note that the given algebra is non-domestic (and then even of non-polynomial growth)
if and only if there exists an arrow
$\alpha$ such that 
$\mathcal N(\alpha)$ is non-empty and not cyclic
([R1], Proposition 2 and its proof). 

Here we take $\alpha = \alpha_1.$ We assume that the length of $\epsilon, \eta, \epsilon'$ 
is $r,s,t$, respectively. Let $u = \alpha_1^{\epsilon(1)}\cdots \alpha_r^{\epsilon(r)}$,
$v = \alpha_{r+1}^{\epsilon(r+1)}\cdots \alpha_{r+s}^{\epsilon(r+s)}$ and
$w = \alpha_{r+s+1}^{\epsilon(r+s+1)}\cdots \alpha_{r+s+t}^{\epsilon(r+s+t)}$.
Then both $uvwv^{-1}$ and $uvw{-1}v^{-1}$ are elements of $\mathcal N(\alpha)$.
This shows that $B(\epsilon,\eta,\epsilon')$ is of non-polynomial growth.
 \end{proof}
 \medskip

We consider the algebra given in Example 2. The non-regular component looks as follows:
$$ 
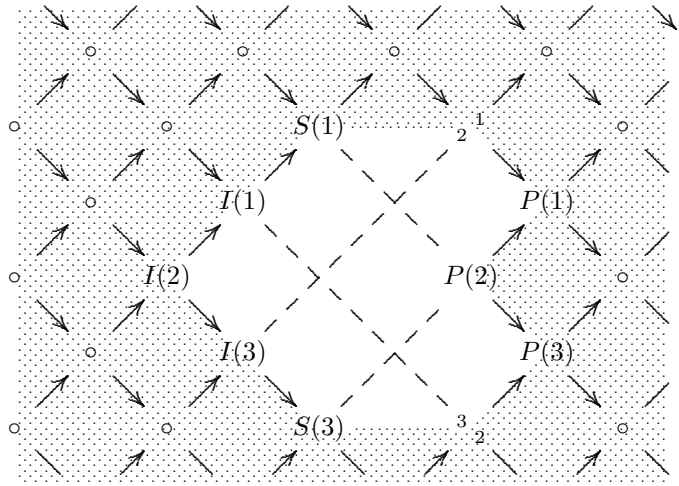
\hbox{\beginpicture
\setcoordinatesystem units <1cm,1cm>
\put{} at 0 5.6
\put{} at 0 -0.7 
\put{$\circ$} at 0 0
\put{$\circ$} at 2 0
\put{$S(3)$} at 4 0
\put{$\smallmatrix 3 & \cr
                     & 2 \endsmallmatrix$} at 6 0
\put{$\circ$} at 8 0

\put{$\circ$} at 1 1
\put{$I(3)$} at 3 1
\put{$P(3)$} at 7 1

\put{$\circ$} at 0 2
\put{$I(2)$} at 2 2
\put{$P(2)$} at 6 2
\put{$\circ$} at 8 2

\put{$\circ$} at 1 3
\put{$I(1)$} at 3 3
\put{$P(1)$} at 7 3

\put{$\circ$} at 0 4
\put{$\circ$} at 2 4
\put{$S(1)$} at 4 4
\put{$\smallmatrix  & 1\cr
                  2 &  \endsmallmatrix$} at 6 4
\put{$\circ$} at 8 4

\put{$\circ$} at 1 5
\put{$\circ$} at 3 5
\put{$\circ$} at 5 5
\put{$\circ$} at 7 5

\arrow <2mm> [0.25,0.75] from 0.3 0.3 to 0.7 0.7
\arrow <2mm> [0.25,0.75] from 2.3 0.3 to 2.7 0.7
\arrow <2mm> [0.25,0.75] from 6.3 0.3 to 6.7 0.7

\arrow <2mm> [0.25,0.75] from 0.3 2.3 to 0.7 2.7
\arrow <2mm> [0.25,0.75] from 2.3 2.3 to 2.7 2.7

\arrow <2mm> [0.25,0.75] from 0.3 4.3 to 0.7 4.7
\arrow <2mm> [0.25,0.75] from 2.3 4.3 to 2.7 4.7
\arrow <2mm> [0.25,0.75] from 4.3 4.3 to 4.7 4.7
\arrow <2mm> [0.25,0.75] from 6.3 4.3 to 6.7 4.7

\arrow <2mm> [0.25,0.75] from 6.3 2.3 to 6.7 2.7

\arrow <2mm> [0.25,0.75] from 3.3 3.3 to 3.7 3.7
\arrow <2mm> [0.25,0.75] from 1.3 1.3 to 1.7 1.7
\arrow <2mm> [0.25,0.75] from 1.3 3.3 to 1.7 3.7

\arrow <2mm> [0.25,0.75] from 7.3 1.3 to 7.7 1.7  

\arrow <2mm> [0.25,0.75] from 7.3 3.3 to 7.7 3.7

\arrow <2mm> [0.25,0.75] from 1.4 -0.6 to 1.7 -0.3
\arrow <2mm> [0.25,0.75] from 3.4 -0.6 to 3.7 -0.3
\arrow <2mm> [0.25,0.75] from 5.4 -0.6 to 5.7 -0.3
\arrow <2mm> [0.25,0.75] from 7.4 -0.6 to 7.7 -0.3

\arrow <2mm> [0.25,0.75] from 1.3 0.7 to 1.7 0.3
\arrow <2mm> [0.25,0.75] from 3.3 0.7 to 3.7 0.3
\arrow <2mm> [0.25,0.75] from 7.3 0.7 to 7.7 0.3

\arrow <2mm> [0.25,0.75] from 1.3 2.7 to 1.7 2.3
\arrow <2mm> [0.25,0.75] from 7.3 2.7 to 7.7 2.3

\arrow <2mm> [0.25,0.75] from 7.3 4.7 to 7.7 4.3
\arrow <2mm> [0.25,0.75] from 5.3 4.7 to 5.7 4.3
\arrow <2mm> [0.25,0.75] from 3.3 4.7 to 3.7 4.3
\arrow <2mm> [0.25,0.75] from 1.3 4.7 to 1.7 4.3

\arrow <2mm> [0.25,0.75] from 8.4 5.6 to 8.7 5.3

\arrow <2mm> [0.25,0.75] from 6.3 3.7 to 6.7 3.3
\arrow <2mm> [0.25,0.75] from 6.3 1.7 to 6.7 1.3

\arrow <2mm> [0.25,0.75] from 0.3 1.7 to 0.7 1.3
\arrow <2mm> [0.25,0.75] from 2.3 1.7 to 2.7 1.3

\arrow <2mm> [0.25,0.75] from 0.3 3.7 to 0.7 3.3
\arrow <2mm> [0.25,0.75] from 2.3 3.7 to 2.7 3.3

\arrow <2mm> [0.25,0.75] from 0.4 5.6 to 0.7 5.3
\arrow <2mm> [0.25,0.75] from 2.4 5.6 to 2.7 5.3
\arrow <2mm> [0.25,0.75] from 4.4 5.6 to 4.7 5.3
\arrow <2mm> [0.25,0.75] from 6.4 5.6 to 6.7 5.3

\plot 1.3 5.3  1.6 5.6 /
\plot 3.3 5.3  3.6 5.6 /
\plot 5.3 5.3  5.6 5.6 /
\plot 7.3 5.3  7.6 5.6 /
\plot 8.3 4.3  8.6 4.6 /
\plot 8.3 2.3  8.6 2.6 /
\plot 8.3 0.3  8.6 0.6 /
\plot 8.3 3.7  8.6 3.4 /
\plot 8.3 1.7  8.6 1.4 /
\plot 8.3 -0.3  8.6 -0.6 /
\plot 6.3 -0.3  6.6 -0.6 /
\plot 4.3 -0.3  4.6 -0.6 /
\plot 2.3 -0.3  2.6 -0.6 /
\plot 0.3 -0.3  0.6 -0.6 /

\setdots<1mm>
\plot 4.5 0  5.7 0 /
\plot 4.45 4  5.8 4 /
\setshadegrid span <0.7mm>
\vshade 0 -0.8 5.6 <,z,,> 2  -0.8 5.6 /
\vshade 2 2 5.6  <z,z,,> 4 4 5.6 <z,z,,>
    6  4 5.6 <z,z,,> 7 3 5.6 /   
\vshade 2 -0.8 2 <z,z,,> 4 -0.8 0 <z,z,,> 6 -0.8 0 <z,z,,>
    7 -0.8 1 /
\vshade 7 -0.8 5.6 <z,z,,> 8.6 -0.8 5.6 /
\vshade 6 2 2 <z,z,,> 7 1 3 /

\setdashes <2mm>
\plot 3.3 2.7  5.7 0.3 /
\plot 4.3 3.7  5.7 2.3 /
\plot 3.3 1.3  5.7 3.7 /
\plot 4.3 0.3  5.7 1.7 /

\endpicture}
$$
The component contains 10 of the 12 string modules which are  boundary
modules; the remaining two string modules which are boundary modules
are the serial string modules of length 3
(with composition factors going up $2,1,1$ and $2,3,3$, respectively).
They form the boundary of a stable tube of rank 2; the boundary meshes are
those provided by the arrows $1\to 2$ and $3\to 2$ respectively:
$$ 
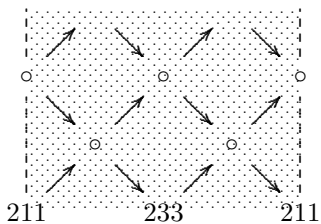
\hbox{\beginpicture
\setcoordinatesystem units <.9cm,.9cm>
\put{} at 0 3
\multiput{$\circ$} at 0 2  2 2  4 2 /

\put{$\circ$} at 1 1 
\put{$\circ$} at 3 1 

\put{$211$} at 0 0
\put{$233$} at 2 0 
\put{$211$} at 4 0 
\arrow <1.5mm> [0.25,0.75] from 0.3 0.3 to 0.7 0.7
\arrow <1.5mm> [0.25,0.75] from 2.3 0.3 to 2.7 0.7

\arrow <1.5mm> [0.25,0.75] from 1.3 0.7 to 1.7 0.3
\arrow <1.5mm> [0.25,0.75] from 3.3 0.7 to 3.7 0.3

\arrow <1.5mm> [0.25,0.75] from 1.3 2.7 to 1.7 2.3
\arrow <1.5mm> [0.25,0.75] from 3.3 2.7 to 3.7 2.3

\arrow <1.5mm> [0.25,0.75] from 1.3 1.3 to 1.7 1.7
\arrow <1.5mm> [0.25,0.75] from 3.3 1.3 to 3.7 1.7

\arrow <1.5mm> [0.25,0.75] from 0.3 2.3 to 0.7 2.7
\arrow <1.5mm> [0.25,0.75] from 2.3 2.3 to 2.7 2.7
\arrow <1.5mm> [0.25,0.75] from 0.3 1.7 to 0.7 1.3
\arrow <1.5mm> [0.25,0.75] from 2.3 1.7 to 2.7 1.3
\setdashes <1mm>
\plot 0 0.3  0 1.7 /
\plot 4 0.3  4 1.7 /

\plot 0 2.3  0 3 /
\plot 4 2.3  4 3 /

\setshadegrid span <0.7mm>
\vshade 0  0 3  4 0 3 /

\endpicture}
$$
	\medskip

Let us have another look at the non-regular component. 
The picture shows nicely a phenomenon which has attracted a lot of attention
lately, in some other context, namely when dealing with cluster tilted algebras. 
Let us recall the relevant facts:
Given a cluster tilted algebra $\Lambda$,
the category $\mo \Lambda$ is obtained from the corresponding cluster category
by factoring out a cluster tilting object \cite{BMR}.
Looking at a vertex $a$ of the quiver
of $\Lambda$, the corresponding indecomposable projective $\Lambda$-module
$P(a)$ and the  corresponding indecomposable injective $\Lambda$-module $I(a)$
satisfy  
$$
\tau^2 P(a) =  I(a),
$$ 
where $\tau$ is the Auslander-Reiten translation in the cluster category (if we denote
by $\tau_\Lambda$ the Auslander-Reiten translation in the category $\mo \Lambda$,
then $\tau_\Lambda M = \tau M$ for any indecomposable non-projective $\Lambda$-module,
whereas, of course, $\tau_\Lambda M = 0$ for $M$ indecomposable projective).

As we see in the picture, the non-regular component is a translation quiver which
can be considered as part of a regular translation quiver $\Xi$ 
obtained by adding
a new vertex $p'$ for every  projective vertex $p$, such that the translate of $p$ is $p'$ 
and   the translate of $p'$ is an injective vertex. Let us define a function $f$ on the set
of vertices of $\Xi$ as follows: if $x$ is an old vertex, let $f(x)$ be the length of the corresponding module,
if $x = p'$ is a new vertex, let $f(x) = -1.$ Then $f$ satisfies the following property:
$$
 f(z) + f(\tau z) = \sum f(y), \quad\text{where we sum over all arrows $y \to z$ with $f(y) > 0$,}
$$
for all vertices $z$ of $\Xi$ (one may say that such a function with values in $\mathbb Z$ is 
''cluster-additive'', see \cite{Rca}).
							        \medskip

Other similarities with cluster tilted algebras (see \cite{KR}) should be mentioned:

\begin{proposition} The barbell algebras are Gorenstein algebras of Gorenstein dimension $1$
und the stable category of Cohen-Macaulay modules is Calabi-Yau of CY-dimension $3$.
\end{proposition}    

For our example 2, here are the minimal injective resolutions of the indecomposable projective modules:
\begin{align}
 0 \longrightarrow P(1) \longrightarrow I(2)\oplus I(2) \longrightarrow I(1)\oplus I(3)\oplus I(3) \longrightarrow 0 \nonumber
\\
 0 \longrightarrow P(2) \longrightarrow I(2) \longrightarrow I(1)\oplus I(3) \longrightarrow 0 \qquad\qquad \nonumber
 \\
 0 \longrightarrow P(3) \longrightarrow I(2)\oplus I(2) \longrightarrow I(1)\oplus I(1)\oplus I(3) \longrightarrow 0 \nonumber
\end{align}

Let $\mathcal  L$ be the full subcategory of all torsionless modules (by definition,
a module is torsionless if it can be embedded into a projective module) 
and $\mathcal  P$ the full subcategory of all projective modules. We have to
calculate the factor category $\underline{\mathcal L} = \mathcal  L/\mathcal  P.$ Since we deal with a $1$-Gorenstein
algebra, $\underline{\mathcal L}$ is a triangulated category with Auslander-Reiten translation.

It is not difficult to check that the only indecomposable modules which are
torsionless and not projective are the two serial modules of length 2 with 
socle $P(2)$, we denote them by 
$L(1)=\smallmatrix 1\cr 2 \endsmallmatrix$ 
and $L(3)=\smallmatrix 3\cr 2 \endsmallmatrix.$ 
Thus, $\mathcal  L$ has the following Auslander-Reiten quiver:
$$
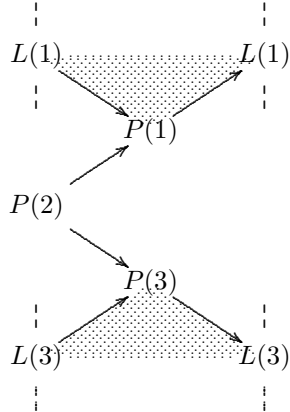
\hbox{\beginpicture
\setcoordinatesystem units <1.5cm,1cm>
\multiput{} at 0 -.5  0 4.5 /
\put{$P(1)$} at 1 3
\put{$P(2)$} at 0 2
\put{$P(3)$} at 1 1
\multiput{$L(1)$} at 0 4  2 4 /
\multiput{$L(3)$} at 0 0  2 0 /
\arrow <1.5mm> [0.25,0.75] from 0.2 0.2 to 0.8 0.8 
\arrow <1.5mm> [0.25,0.75] from 1.2 0.8 to 1.8 0.2
\arrow <1.5mm> [0.25,0.75] from 0.3 1.7 to 0.8 1.2
\arrow <1.5mm> [0.25,0.75] from 0.3 2.3 to 0.8 2.8
\arrow <1.5mm> [0.25,0.75] from 0.2 3.8 to 0.8 3.2
\arrow <1.5mm> [0.25,0.75] from 1.2 3.2 to 1.8 3.8
\setdots <1mm> 
\plot 0.3 0  1.7 0 /
\plot 0.3 4  1.7 4 /
\setshadegrid span <0.5mm>
\hshade 0 0 2  1 1 1 /
\hshade 3 1 1  4 0 2 / 
\setdashes <1mm>
\plot 0 0.4  0 0.7 /
\plot 0 -.4  0 -.7 /
\plot 2 0.4  2 0.7 /
\plot 2 -.4  2 -.7 /
\plot 0 4.4  0 4.7 /
\plot 0 3.6  0 3.3 /
\plot 2 4.4  2 4.7 /
\plot 2 3.6  2 3.3 /

\endpicture}
$$
The dashed line indicate that we have to identify vertices of the triangles
exhibited: Note that both serial modules $L(1)$ 
and $L(3)$ are shown twice, these are the vertices
which have to be identified.

It follows that the (triangulated) category $\underline{\mathcal L}$ has just two indecomposable
objects, both being fixed under the suspension functor as well as under the
Auslander-Reiten translation functor (so that $\underline{\mathcal L}$
is the product of two copies 
of the stable module category of the algebra $k[\epsilon] = k[T]/\langle T^2\rangle$
of dual numbers), the Auslander-Reiten quiver of $\underline{\mathcal L}$ looks as follows:
$$
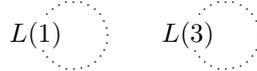
\hbox{\beginpicture
\setcoordinatesystem units <1cm,1cm>
\multiput{} at 0 0.5  0 -.5 /
\put{$L(1)$} at 0 0
\put{$L(3)$} at 2 0
\setdots <1mm>
\circulararc 300 degrees from 0.1 -0.2 center at 0.5 0 
\circulararc 300 degrees from 2.1 -0.2 center at 2.5 0 
\endpicture}
$$

Thus, we deal with a triangulated category 
for which both the suspension functor as well as the
Auslander-Reiten translation functor are the identity functor. This means that
$\underline{\mathcal L}$ is $3$-Calabi-Yau, and indeed $n$-Calabi-Yau for any $n$.
      \medskip

Since the module category of a barbell algebra shares so many properties with
the module category of a cluster tilted algebra, one may wonder whether 
also for a barbell algebra $\Lambda$ the module category $\mo \Lambda$  is obtained from a 
triangulated category $\mathcal C$  by forming
 $\mathcal C/\langle T\rangle$ for some object $T$ in $\mathcal  C$.
As Idun Reiten has pointed out, this is indeed the case if
we deal with a barbell algebra $\Lambda$ with two loops (as in our running example 2) provided
we assume that the characteristic of $k$ is different from $3$: such an algebra 
 is 2-CY-tilted (this means: the endomorphism ring of some cluster tilting object
of a 2-Calabi-Yau category, \cite{Reiten}).
Namely, if  $\Lambda$ is 
a barbell algebra with two loops $\alpha, \delta$ in its quiver $Q$
and  if the characteristic of $k$ is different from $3$,
then $\Lambda$ 
is the Jacobian algebra $J(Q,W) = kQ/\langle 3\alpha^2,3\delta^2\rangle$,
where  $W$ is the potential
$W = \alpha^3+\delta^3$, see \cite{DWZ}, 
thus one can apply theorem 3.6 of Amiot \cite{A}.

\section{Sectional paths}

Recall that a (finite or infinite) path $(\cdots \to X_i \to X_{i+1} \to \cdots)$ in the Auslander-Reiten
quiver of a finite dimensional algebra is called {\it sectional} provided $\tau X_{i+1}$ is not isomorphic
to $X_{i-1}$ for all possible $i$. Such a path will be called {\it maximal} provided it is not
a proper subpath of some sectional path. An infinite sectional path involving only monomorphisms
will be called a {\it mono ray,} an infinite path involving only epimorphisms will be called an
{\it epi coray;} of course, mono rays start with some module, epi corays end in a module.

Note that for Auslander-Reiten components of the form $\mathbb Z\mathbb A_\infty$ as well as 
for stable tubes, all maximal sectional paths are mono rays and epi corays. 

\begin{theorem} 
Let $\Lambda$ be a $k$-algebra which is minimal representation-infinite and special biserial.
Then any maximal sectional path is a mono ray, an epi coray or the concatenation of an
epi coray with a mono ray.
\end{theorem}

\begin{coro}
Assume that $\Lambda$ is minimal representation-infinite
and special biserial. Let $X,Y,Z$ be indecomposable $\Lambda$-modules with an
irreducible monomorphism $X \to Y$ and an irreducible epimorphism $Y \to Z$.
Then $X  = \tau Z.$
\end{coro}

\begin{proof} We may assume that there are no nodes: Namely, if all the sectional
paths of $\nn(\Lambda)$ are as mentioned, the same has to be true for $\Lambda$: the
only maximal sectional paths for $\Lambda$ to be looked at are those passing
through the node. Resolving the node we will obtain sectional paths which are not
double infinite paths, thus by the assumption on $\nn(\Lambda)$, we will deal with 
an epi coray  ending in the node and a mono ray starting in the node, thus
with a concatenation as listed. 

If $\Lambda$ is hereditary of type $\widetilde{\mathbb A}_n$, thus a cycle algebra, then
 any maximal sectional path is a mono ray, an epi coray. We only have to look
at the preprojective component and the preinjective component. But if $f:X\to Y$
is a non-zero map between indecomposable preprojective modules, then $f$
has to be always a monomorphism: otherwise, the kernel of $f$ would have
negative defect, and since the defect of $X$ is $-1$, it would follow that the
image of $f$ is a non-zero submodule of $Y$ with non-negative defect, 
a contradiction. The dual argument shows that the maximal sectional paths
in the preinjective component are epi corays.

Next, let us look at the wind wheels: again, only the non-regular components
have to be considered. But we know how to construct these components:
we use mono rays from the preprojective component and  the tube
$\mathcal R_\infty$ as well as epi corays from the tube $\mathcal R_0$
and from the preinjective component, and in addition rays and corays in the tiles.
But all the maximal sectional paths in the four quarters of a tile are
mono rays and epi corays (in the quarter {\bf I} we have only mono rays,
in {\bf III} only epi corays, whereas {\bf II} and {\bf IV} have both mono
rays and epi corays.

Finally, let us look at the barbells. There is only one non-regular component 
which has to be treated separately. The band modules lie in homogeneous
tubes and there will be an additional regular tube containing string modules.
What really is of interest are the remaining components $\mathcal C$, they are of the form
$\mathbb Z\mathbb A_\infty^\infty$. 
Let us look at the example  2 (the general case
is similar). Let $M = M(v)$ be the Gei\ss{} module (\cite{G})
for $\mathcal C$ (it is the unique
module in $\mathcal C$ of minimal length) and one easily observes that $v$
is a word of the form $1 \to 2 \cdots 2 \leftarrow 3$. It is easy to see that
all the modules $\tau^{-t-1}M$ for $t\ge 0$ are obtained from $\tau^{-t}M$ by
adding hooks both on the left and on the right; similarly, 
all the modules $\tau^{t+1}M$ for $t\ge 0$ are obtained from $\tau^{t}M$ by
adding cohooks both on the left and on the right. But this implies that all the
maximal sectional paths in $\mathcal C$ are concatenation of an
epi coray with a mono ray.
\end{proof}

It may be helpful to call an indecomposable $\Lambda$-module a {\it valley} module
if it is the concatenation vertex for a  sectional path which is the concatenation of an
epi coray with a mono ray,  and to exhibit corresponding pictures: always we encircle
the ''valleys''. First, we present a non-regular component of a wind wheel:
$$ 
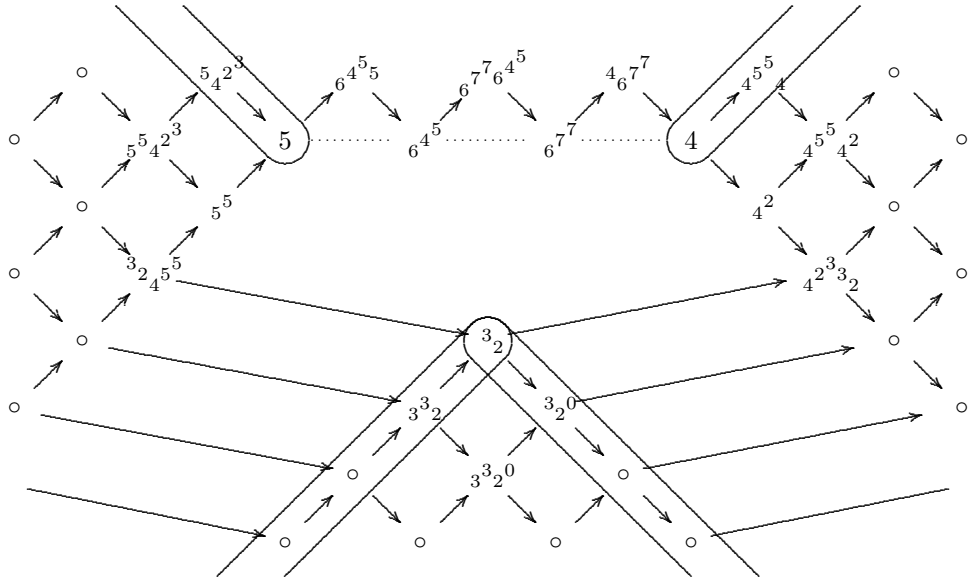
\hbox{\beginpicture
\setcoordinatesystem units <.89cm,.89cm>
\multiput{} at 0 -1  14 8 /

\plot 1.5 8  3.75 5.75 /
\plot 2.5 8  4.25 6.25 /
\circulararc 180 degrees from 3.75 5.75 center at 4 6 
\plot 12.5 8  10.25 5.75 /
\plot 11.5 8  9.75 6.25 /
\circulararc -180 degrees from 10.25 5.75 center at 10 6 

\plot 3 -.5  6.75  3.25 /
\plot 4 -.5  7.25 2.75 /
\circulararc -180 degrees from 6.75  3.25 center at 7 3 
\plot 11 -.5  7.25  3.25 /
\plot 10 -.5  6.75 2.75 /
\circulararc  180 degrees from 7.25  3.25 center at 7 3

\put{
 \beginpicture
 \setcoordinatesystem units <.15cm,.15cm>
 \put{$\ssize 3$} at 0 2
 \put{$\ssize 2$} at 1 1
 \put{$\ssize 4$} at 2 0
 \put{$\ssize 5$} at 3 1
 \put{$\ssize 5$} at 4 2 
\endpicture} at 2 4

\put{
 \beginpicture
 \setcoordinatesystem units <.15cm,.15cm>
 \put{$\ssize 5$} at 3 1
 \put{$\ssize 5$} at 4 2 
\endpicture} at 3 5

\put{$5$} at 4 6

\put{$\circ$} at 1 7

\put{
 \beginpicture
 \setcoordinatesystem units <.15cm,.15cm>
 \put{$\ssize 5$} at 3 1
 \put{$\ssize 5$} at 4 2
 \put{$\ssize 4$} at 5 1
 \put{$\ssize 2$} at 6 2
 \put{$\ssize 3$} at 7 3
 \endpicture} at  2 6

\put{
 \beginpicture
 \setcoordinatesystem units <.15cm,.15cm>
 \put{$\ssize 5$} at 4 2
 \put{$\ssize 4$} at 5 1
 \put{$\ssize 2$} at 6 2
 \put{$\ssize 3$} at 7 3
 
\endpicture} at  3 7

\put{
 \beginpicture
 \setcoordinatesystem units <.15cm,.15cm>
 \put{$\ssize 6$} at 0 0
 \put{$\ssize 4$} at 1 1
 \put{$\ssize 5$} at 2 2
 \put{$\ssize 5$} at 3 1  
\endpicture} at  5 7

\put{
 \beginpicture
 \setcoordinatesystem units <.15cm,.15cm>
 \put{$\ssize 6$} at 3 0
 \put{$\ssize 4$} at 4 1
 \put{$\ssize 5$} at 5 2
\endpicture} at  6 6

\put{
 \beginpicture
 \setcoordinatesystem units <.15cm,.15cm>
 \put{$\ssize 6$} at 3 0
 \put{$\ssize 7$} at 4 1
 \put{$\ssize 7$} at 5 2
\endpicture} at  8 6

\put{$\circ$} at 1 5
\put{$\circ$} at 0 4

\put{$\circ$} at 13 5
\put{$\circ$} at 14 4

\put{$\circ$} at 1 3
\put{$\circ$} at 13 3
\put{$\circ$} at 0 2
\put{$\circ$} at 0 6
\put{$\circ$} at 14 2
\put{$\circ$} at 14 6

\put{
 \beginpicture
 \setcoordinatesystem units <.15cm,.15cm>
 \put{$\ssize 6$} at 0 0
 \put{$\ssize 7$} at 1 1
 \put{$\ssize 7$} at 2 2
 \put{$\ssize 6$} at 3 1
 \put{$\ssize 4$} at 4 2
 \put{$\ssize 5$} at 5 3
\endpicture} at  7 7

\put{
 \beginpicture
 \setcoordinatesystem units <.15cm,.15cm>
 \put{$\ssize 4$} at 2 1
 \put{$\ssize 6$} at 3 0
 \put{$\ssize 7$} at 4 1
 \put{$\ssize 7$} at 5 2
\endpicture} at  9 7

\put{
 \beginpicture
 \setcoordinatesystem units <.15cm,.15cm>
 \put{$\ssize 4$} at 2 0
 \put{$\ssize 5$} at 3 1
 \put{$\ssize 5$} at 4 2
 \put{$\ssize 4$} at 5 0
\endpicture} at  11 7

\put{
 \beginpicture
 \setcoordinatesystem units <.15cm,.15cm>
 \put{$\ssize 4$} at 2 0
 \put{$\ssize 5$} at 3 1
 \put{$\ssize 5$} at 4 2
 \put{$\ssize 4$} at 5 0
 \put{$\ssize 2$} at 6 1
\endpicture} at  12 6

\put{$\circ$} at 13 7

\put{$4$} at 10 6 

\put{
 \beginpicture
 \setcoordinatesystem units <.15cm,.15cm>
 \put{$\ssize 4$} at 5 0
 \put{$\ssize 2$} at 6 1
\endpicture} at  11 5

\put{
 \beginpicture
 \setcoordinatesystem units <.15cm,.15cm>
 \put{$\ssize 4$} at 0 0
 \put{$\ssize 2$} at 1 1
 \put{$\ssize 3$} at 2 2
 \put{$\ssize 3$} at 3 1
 \put{$\ssize 2$} at 4 0
\endpicture} at  12 4

\put{
 \beginpicture
 \setcoordinatesystem units <.15cm,.15cm>
 \put{$\ssize 3$} at 0 1
 \put{$\ssize 2$} at 1 0
\endpicture} at  7 3

\put{
 \beginpicture
 \setcoordinatesystem units <.15cm,.15cm>
 \put{$\ssize 3$} at 0 0
 \put{$\ssize 3$} at 1 1
 \put{$\ssize 2$} at 2 0
\endpicture} at  6 2

\put{
 \beginpicture
 \setcoordinatesystem units <.15cm,.15cm>
 \put{$\ssize 3$} at 0 1
 \put{$\ssize 2$} at 1 0
 \put{$\ssize 0$} at 2 1
\endpicture} at  8 2

\put{$\circ$} at 5 1

\put{
 \beginpicture
 \setcoordinatesystem units <.15cm,.15cm>
 \put{$\ssize 3$} at 0 0
 \put{$\ssize 3$} at 1 1
 \put{$\ssize 2$} at 2 0
 \put{$\ssize 0$} at 3 1
\endpicture} at  7 1

\put{$\circ$} at 9 1
\put{$\circ$} at 4 0
\put{$\circ$} at 6 0
\put{$\circ$} at 8 0 
\put{$\circ$} at 10 0

\arrow <1.5mm> [0.25,0.75] from 0.3 4.3 to 0.7 4.7
\arrow <1.5mm> [0.25,0.75] from 2.3 4.3 to 2.7 4.7
\arrow <1.5mm> [0.25,0.75] from 1.3 5.3 to 1.7 5.7
\arrow <1.5mm> [0.25,0.75] from 3.3 5.3 to 3.7 5.7
\arrow <1.5mm> [0.25,0.75] from 4.3 6.3 to 4.7 6.7
\arrow <1.5mm> [0.25,0.75] from 6.3 6.3 to 6.6 6.6
\arrow <1.5mm> [0.25,0.75] from 8.3 6.3 to 8.7 6.7
\arrow <1.5mm> [0.25,0.75] from 5.3 6.7 to 5.7 6.3
\arrow <1.5mm> [0.25,0.75] from 7.3 6.7 to 7.7 6.3
\arrow <1.5mm> [0.25,0.75] from 9.3 6.7 to 9.7 6.3
\arrow <1.5mm> [0.25,0.75] from 10.3 5.7 to 10.7 5.3
\arrow <1.5mm> [0.25,0.75] from 11.3 4.7 to 11.7 4.3
\arrow <1.5mm> [0.25,0.75] from 12.3 5.7 to 12.7 5.3
\arrow <1.5mm> [0.25,0.75] from 13.3 4.7 to 13.7 4.3

\arrow <1.5mm> [0.25,0.75] from 4.3 0.3 to 4.7 0.7
\arrow <1.5mm> [0.25,0.75] from 6.3 0.3 to 6.7 0.7
\arrow <1.5mm> [0.25,0.75] from 8.3 0.3 to 8.7 0.7
\arrow <1.5mm> [0.25,0.75] from 5.3 0.7 to 5.7 0.3
\arrow <1.5mm> [0.25,0.75] from 7.3 0.7 to 7.7 0.3
\arrow <1.5mm> [0.25,0.75] from 9.3 0.7 to 9.7 0.3

\arrow <1.5mm> [0.25,0.75] from 5.3 1.3 to 5.7 1.7
\arrow <1.5mm> [0.25,0.75] from 7.3 1.3 to 7.7 1.7
\arrow <1.5mm> [0.25,0.75] from 6.3 1.7 to 6.7 1.3
\arrow <1.5mm> [0.25,0.75] from 8.3 1.7 to 8.7 1.3

\arrow <1.5mm> [0.25,0.75] from 6.3 2.3 to 6.7 2.7
\arrow <1.5mm> [0.25,0.75] from 7.3 2.7 to 7.7 2.3

\arrow <1.5mm> [0.25,0.75] from 0.3 6.3 to 0.7 6.7
\arrow <1.5mm> [0.25,0.75] from 2.3 6.3 to 2.7 6.7
\arrow <1.5mm> [0.25,0.75] from 10.3 6.3 to 10.7 6.7
\arrow <1.5mm> [0.25,0.75] from 12.3 6.3 to 12.7 6.7
\arrow <1.5mm> [0.25,0.75] from 12.3 4.3 to 12.7 4.7
\arrow <1.5mm> [0.25,0.75] from 0.3 2.3 to 0.7 2.7
\arrow <1.5mm> [0.25,0.75] from 11.3 5.3 to 11.7 5.7
\arrow <1.5mm> [0.25,0.75] from 13.3 5.3 to 13.7 5.7
\arrow <1.5mm> [0.25,0.75] from 13.3 3.3 to 13.7 3.7
\arrow <1.5mm> [0.25,0.75] from 1.3 3.3 to 1.7 3.7

\arrow <1.5mm> [0.25,0.75] from 0.3 3.7 to 0.7 3.3
\arrow <1.5mm> [0.25,0.75] from 0.3 5.7 to 0.7 5.3
\arrow <1.5mm> [0.25,0.75] from 1.3 4.7 to 1.7 4.3
\arrow <1.5mm> [0.25,0.75] from 1.3 6.7 to 1.7 6.3

\arrow <1.5mm> [0.25,0.75] from 2.3 5.7 to 2.7 5.3
\arrow <1.5mm> [0.25,0.75] from 3.3 6.7 to 3.7 6.3

\arrow <1.5mm> [0.25,0.75] from 11.3 6.7 to 11.7 6.3
\arrow <1.5mm> [0.25,0.75] from 13.3 6.7 to 13.7 6.3

\arrow <1.5mm> [0.25,0.75] from 12.3 3.7 to 12.7 3.3
\arrow <1.5mm> [0.25,0.75] from 13.3 2.7 to 13.7 2.3

\arrow <1.5mm> [0.25,0.75] from 2.4 3.9 to 6.7 3.1
\arrow <1.5mm> [0.25,0.75] from 1.4 2.9 to 5.7 2.1
\arrow <1.5mm> [0.25,0.75] from 0.4 1.9 to 4.7 1.1
\arrow <1.5mm> [0.25,0.75] from 0.2 0.8 to 3.7 0.1

\arrow <1.5mm> [0.25,0.75] from 7.3 3.1 to 11.4 3.9
\arrow <1.5mm> [0.25,0.75] from 8.3 2.1 to 12.4 2.9
\arrow <1.5mm> [0.25,0.75] from 9.3 1.1 to 13.4 1.9
\plot 10.3 0.1 13.8 0.8 /

\setdots<.1cm>
\plot 4.4 6  5.6 6 /
\plot 6.4 6  7.6 6 /
\plot 8.4 6  9.6 6 /

\setshadegrid span <0.7mm>

\endpicture}
$$
The valleys may be considered as the natural places where to cut such a component into
pieces. Of course, in our cut-and-paste process, we followed this rule. 

The second example is the non-regular component of the barbell given as
example 2:

$$ 
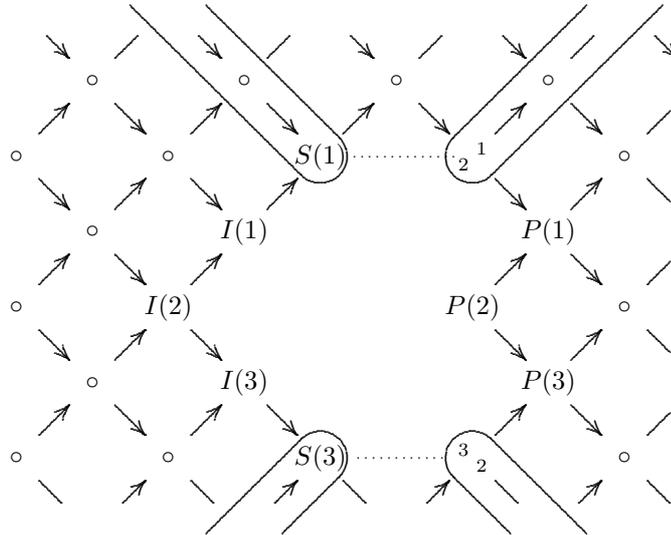
\hbox{\beginpicture
\setcoordinatesystem units <1cm,1cm>
\plot 1.5 6  3.75 3.75 /
\plot 2.5 6  4.25 4.25 /
\circulararc 180 degrees from 3.75 3.75 center at 4 4 
\plot 8.5 6  6.25 3.75 /
\plot 7.5 6  5.75 4.25 /
\circulararc -180 degrees from 6.25 3.75 center at 6 4 

\plot 2.5 -1  3.75  .25 /
\plot 3.5 -1  4.25 -.25 /
\circulararc -180 degrees from 3.75  .25 center at 4 0 
\plot 7.5 -1  6.25  .25 /
\plot 6.5 -1  5.75 -.25 /
\circulararc  180 degrees from 6.25  .25 center at 6 0

\put{} at 0 5.6
\put{} at 0 -1 
\put{$\circ$} at 0 0
\put{$\circ$} at 2 0
\put{$S(3)$} at 4 0
\put{$\smallmatrix 3 & \cr
                     & 2 \endsmallmatrix$} at 6 0
\put{$\circ$} at 8 0

\put{$\circ$} at 1 1
\put{$I(3)$} at 3 1
\put{$P(3)$} at 7 1

\put{$\circ$} at 0 2
\put{$I(2)$} at 2 2
\put{$P(2)$} at 6 2
\put{$\circ$} at 8 2

\put{$\circ$} at 1 3
\put{$I(1)$} at 3 3
\put{$P(1)$} at 7 3

\put{$\circ$} at 0 4
\put{$\circ$} at 2 4
\put{$S(1)$} at 4 4
\put{$\smallmatrix  & 1\cr
                  2 &  \endsmallmatrix$} at 6 4
\put{$\circ$} at 8 4

\put{$\circ$} at 1 5
\put{$\circ$} at 3 5
\put{$\circ$} at 5 5
\put{$\circ$} at 7 5

\arrow <2mm> [0.25,0.75] from 0.3 0.3 to 0.7 0.7
\arrow <2mm> [0.25,0.75] from 2.3 0.3 to 2.7 0.7
\arrow <2mm> [0.25,0.75] from 6.3 0.3 to 6.7 0.7

\arrow <2mm> [0.25,0.75] from 0.3 2.3 to 0.7 2.7
\arrow <2mm> [0.25,0.75] from 2.3 2.3 to 2.7 2.7

\arrow <2mm> [0.25,0.75] from 0.3 4.3 to 0.7 4.7
\arrow <2mm> [0.25,0.75] from 2.3 4.3 to 2.7 4.7
\arrow <2mm> [0.25,0.75] from 4.3 4.3 to 4.7 4.7
\arrow <2mm> [0.25,0.75] from 6.3 4.3 to 6.7 4.7

\arrow <2mm> [0.25,0.75] from 6.3 2.3 to 6.7 2.7

\arrow <2mm> [0.25,0.75] from 3.3 3.3 to 3.7 3.7
\arrow <2mm> [0.25,0.75] from 1.3 1.3 to 1.7 1.7
\arrow <2mm> [0.25,0.75] from 1.3 3.3 to 1.7 3.7

\arrow <2mm> [0.25,0.75] from 7.3 1.3 to 7.7 1.7  

\arrow <2mm> [0.25,0.75] from 7.3 3.3 to 7.7 3.7

\arrow <2mm> [0.25,0.75] from 1.4 -0.6 to 1.7 -0.3
\arrow <2mm> [0.25,0.75] from 3.4 -0.6 to 3.7 -0.3
\arrow <2mm> [0.25,0.75] from 5.4 -0.6 to 5.7 -0.3
\arrow <2mm> [0.25,0.75] from 7.4 -0.6 to 7.7 -0.3

\arrow <2mm> [0.25,0.75] from 1.3 0.7 to 1.7 0.3
\arrow <2mm> [0.25,0.75] from 3.3 0.7 to 3.7 0.3
\arrow <2mm> [0.25,0.75] from 7.3 0.7 to 7.7 0.3

\arrow <2mm> [0.25,0.75] from 1.3 2.7 to 1.7 2.3
\arrow <2mm> [0.25,0.75] from 7.3 2.7 to 7.7 2.3

\arrow <2mm> [0.25,0.75] from 7.3 4.7 to 7.7 4.3
\arrow <2mm> [0.25,0.75] from 5.3 4.7 to 5.7 4.3
\arrow <2mm> [0.25,0.75] from 3.3 4.7 to 3.7 4.3
\arrow <2mm> [0.25,0.75] from 1.3 4.7 to 1.7 4.3

\arrow <2mm> [0.25,0.75] from 8.4 5.6 to 8.7 5.3

\arrow <2mm> [0.25,0.75] from 6.3 3.7 to 6.7 3.3
\arrow <2mm> [0.25,0.75] from 6.3 1.7 to 6.7 1.3

\arrow <2mm> [0.25,0.75] from 0.3 1.7 to 0.7 1.3
\arrow <2mm> [0.25,0.75] from 2.3 1.7 to 2.7 1.3

\arrow <2mm> [0.25,0.75] from 0.3 3.7 to 0.7 3.3
\arrow <2mm> [0.25,0.75] from 2.3 3.7 to 2.7 3.3

\arrow <2mm> [0.25,0.75] from 0.4 5.6 to 0.7 5.3
\arrow <2mm> [0.25,0.75] from 2.4 5.6 to 2.7 5.3
\arrow <2mm> [0.25,0.75] from 4.4 5.6 to 4.7 5.3
\arrow <2mm> [0.25,0.75] from 6.4 5.6 to 6.7 5.3

\plot 1.3 5.3  1.6 5.6 /
\plot 3.3 5.3  3.6 5.6 /
\plot 5.3 5.3  5.6 5.6 /
\plot 7.3 5.3  7.6 5.6 /
\plot 8.3 4.3  8.6 4.6 /
\plot 8.3 2.3  8.6 2.6 /
\plot 8.3 0.3  8.6 0.6 /
\plot 8.3 3.7  8.6 3.4 /
\plot 8.3 1.7  8.6 1.4 /
\plot 8.3 -0.3  8.6 -0.6 /
\plot 6.3 -0.3  6.6 -0.6 /
\plot 4.3 -0.3  4.6 -0.6 /
\plot 2.3 -0.3  2.6 -0.6 /
\plot 0.3 -0.3  0.6 -0.6 /

\setdots<.1cm>
\plot 4.5 0  5.7 0 /
\plot 4.45 4  5.8 4 /

\endpicture}
$$

Of course, when dealing with a barbell and look at a regular component $\mathcal C$
of string modules, say with Gei\ss{}-module $M$, then the valley modules are
precisely those which lie on the sectional paths which contain $M$.

In all these components, the ''valleys'' provide a clear division into regions with
common growth pattern. For example, in the regions on the left, all irreducible maps
are epimorphisms, whereas in the regions on the right, all are monomorphisms. 
    \bigskip\bigskip\bigskip

\centerline{\large\bf Part III. Appendix}
							         \bigskip\medskip

The appendix collects some remarks related to the investigations presented above. First, we show an
example of an algebra which may be considered as a twisted version of a barbell.

\section{Further minimal representation-infinite algebras}

Consider the following algebra:
$$
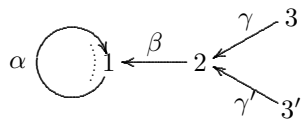
\hbox{\beginpicture
\setcoordinatesystem units <0.6cm,0.6cm>
\put{} at -1 0.8
\put{} at  3 -0.8
\put{$1$} at 0 0
\put{$2$} at 2 0
\put{$3$} at 4 1
\put{$3'$} at 4 -1
\arrow <1.5mm> [0.25,0.75] from 1.7 0 to 0.3 0
\arrow <1.5mm> [0.25,0.75] from 3.7 0.9 to 2.3 0.1
\arrow <1.5mm> [0.25,0.75] from 3.7 -.9 to 2.3 -.1
\circulararc 320 degrees from -0.05 0.2 center at -0.8 0 
\arrow <1.5mm> [0.25,0.75] from -0.09 0.3 to -0.05 0.2
\put{$\alpha$} at -2 0
\put{$\beta$} at 1 0.4
\put{$\gamma$} at 3 0.9
\put{$\gamma'$} at 3 -.9
\setdots <.7mm>
\setquadratic
\setquadratic
\plot -0.4 -0.4 -0.3 0 -0.4 0.4 /

\endpicture}
$$
(or, more generally, the corresponding algebras where 
$\alpha$ and $\beta$ are replaced by longer
paths). Note that the universal covering are the ``dancing girls'' of Brenner-Butler.
	\medskip

This is a Gorenstein algebra of Gorenstein dimension 1, the minimal
injective resolutions of the indecomposable projective modules are as follows:

\begin{align}
 0 \longrightarrow P(1) \longrightarrow I(1) \longrightarrow I(2)\oplus I(2) \longrightarrow 0\quad \, \nonumber
\\
 0 \longrightarrow P(2) \longrightarrow I(1) \longrightarrow I(2)\oplus I(3)\oplus 3' \longrightarrow 0 \nonumber
\\
 0 \longrightarrow P(3) \longrightarrow I(1) \longrightarrow I(2)\oplus I(3') \longrightarrow 0 \quad \nonumber
\\
 0 \longrightarrow P(3') \longrightarrow I(1) \longrightarrow I(2)\oplus I(3) \longrightarrow 0 \quad \nonumber
\end{align}

and here is the central part of the non-regular component:
$$ 
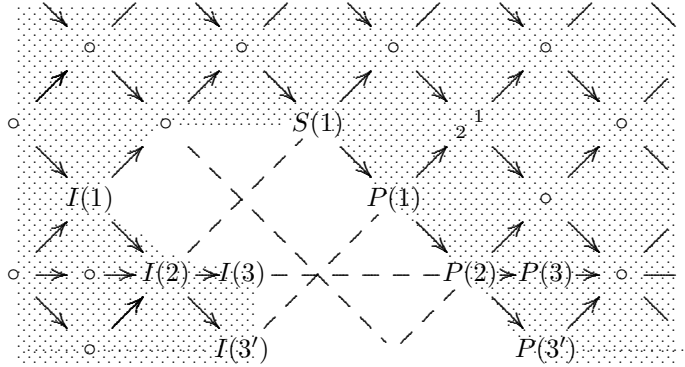
\hbox{\beginpicture
\setcoordinatesystem units <1cm,1cm>
\put{} at 0 5.6
\put{} at 0 -0.5 

\put{$\circ$} at 1 1
\put{$I(3')$} at 3 1
\put{$P(3')$} at 7 1

\put{$I(3)$} at 3 2
\put{$P(3)$} at 7 2

\put{$\circ$} at 0 2
\put{$\circ$} at 1 2

\put{$I(2)$} at 2 2
\put{$P(2)$} at 6 2
\put{$\circ$} at 8 2

\put{$I(1)$} at 1 3
\put{$P(1)$} at 5 3
\put{$\circ$} at 7 3

\put{$\circ$} at 0 4
\put{$\circ$} at 2 4
\put{$S(1)$} at 4 4
\put{$\smallmatrix  & 1\cr
                  2 &  \endsmallmatrix$} at 6 4
\put{$\circ$} at 8 4

\put{$\circ$} at 1 5
\put{$\circ$} at 3 5
\put{$\circ$} at 5 5
\put{$\circ$} at 7 5

\arrow <2mm> [0.25,0.75] from 0.3 4.3 to 0.7 4.7

\arrow <2mm> [0.25,0.75] from 0.3 2.3 to 0.7 2.7

\arrow <2mm> [0.25,0.75] from 0.3 4.3 to 0.7 4.7
\arrow <2mm> [0.25,0.75] from 2.3 4.3 to 2.7 4.7
\arrow <2mm> [0.25,0.75] from 4.3 4.3 to 4.7 4.7
\arrow <2mm> [0.25,0.75] from 6.3 4.3 to 6.7 4.7

\arrow <2mm> [0.25,0.75] from 6.3 2.3 to 6.7 2.7

\arrow <2mm> [0.25,0.75] from 5.3 3.3 to 5.7 3.7
\arrow <2mm> [0.25,0.75] from 1.3 1.3 to 1.7 1.7
\arrow <2mm> [0.25,0.75] from 1.3 3.3 to 1.7 3.7

\arrow <2mm> [0.25,0.75] from 7.3 1.3 to 7.7 1.7

\arrow <2mm> [0.25,0.75] from 7.3 3.3 to 7.7 3.7

\arrow <2mm> [0.25,0.75] from 1.3 2.7 to 1.7 2.3
\arrow <2mm> [0.25,0.75] from 7.3 2.7 to 7.7 2.3

\arrow <2mm> [0.25,0.75] from 7.3 4.7 to 7.7 4.3
\arrow <2mm> [0.25,0.75] from 5.3 4.7 to 5.7 4.3
\arrow <2mm> [0.25,0.75] from 3.3 4.7 to 3.7 4.3
\arrow <2mm> [0.25,0.75] from 1.3 4.7 to 1.7 4.3

\arrow <2mm> [0.25,0.75] from 6.3 3.7 to 6.7 3.3
\arrow <2mm> [0.25,0.75] from 6.3 1.7 to 6.7 1.3

\arrow <2mm> [0.25,0.75] from 0.3 1.7 to 0.7 1.3
\arrow <2mm> [0.25,0.75] from 2.3 1.7 to 2.7 1.3

\arrow <2mm> [0.25,0.75] from 0.3 3.7 to 0.7 3.3
\arrow <2mm> [0.25,0.75] from 4.3 3.7 to 4.7 3.3 
\arrow <2mm> [0.25,0.75] from 5.3 2.7 to 5.7 2.3

\arrow <2mm> [0.25,0.75] from 0.4 5.6 to 0.7 5.3
\arrow <2mm> [0.25,0.75] from 2.4 5.6 to 2.7 5.3
\arrow <2mm> [0.25,0.75] from 4.4 5.6 to 4.7 5.3
\arrow <2mm> [0.25,0.75] from 6.4 5.6 to 6.7 5.3

\arrow <2mm> [0.25,0.75] from 1.3 1.3 to 1.7 1.7

\arrow <2mm> [0.25,0.75] from 0.3 2 to 0.7 2 
\arrow <2mm> [0.25,0.75] from 1.2 2 to 1.6 2 
\arrow <2mm> [0.25,0.75] from 2.4 2 to 2.7 2 

\arrow <2mm> [0.25,0.75] from 6.4 2 to 6.6 2 
\arrow <2mm> [0.25,0.75] from 7.4 2 to 7.7 2 

\plot 8.4 5.6  8.7 5.3 /

\plot 8.3 2  8.7 2 /

\plot 1.3 5.3  1.6 5.6 /
\plot 3.3 5.3  3.6 5.6 /
\plot 5.3 5.3  5.6 5.6 /
\plot 7.3 5.3  7.6 5.6 /
\plot 8.3 4.3  8.6 4.6 /
\plot 8.3 2.3  8.6 2.6 /
\plot 8.3 3.7  8.6 3.4 /
\plot 8.3 1.7  8.6 1.4 /

\setdots<1mm>
\plot 2.2 4  3.5 4 /

\plot 0.2 1  0.8 1 /
\plot 1.2 1  2.5 1 /
\plot 7.5 1  8.8 1 /
\setshadegrid span <0.7mm>
\vshade 0 0.8 5.6 <,z,,> 1 0.8 5.6 /
\vshade 1 3 5.6  <z,z,,> 2 4 5.6 <z,z,,>
    4  4 5.6 <z,z,,> 7 0.8 5.6 <,z,,> 8.7 0.8 5.6 /   
\vshade 1 0.8 3 <z,z,,> 2.2  0.8  2.2 <z,z,,> 3.2 0.8 2  /

\setdashes <2mm>
\plot 2.3 3.7  4.9 1.1 /
\plot 5.1 1.1  5.8 1.8 /
\plot 2.3 2.3  3.8 3.8 /
\plot 3.3 1.3  4.8 2.8 /
\plot 3.4 2    5.7 2 /
\endpicture}
$$

\section{Barification may change the representation type.}
     
Consider the path algebra of the quiver
$$
\hbox{\beginpicture
\setcoordinatesystem units <1cm,1cm>
\multiput{$\circ$} at 0 0  1 0  2 0  3 0  4 0  5 0  /
\arrow <1.5mm> [0.25,0.75] from 0.8 0 to 0.2 0
\arrow <1.5mm> [0.25,0.75] from 1.8 0 to 1.2 0
\arrow <1.5mm> [0.25,0.75] from 2.2 0 to 2.8 0
\arrow <1.5mm> [0.25,0.75] from 3.8 0 to 3.2 0
\arrow <1.5mm> [0.25,0.75] from 4.8 0 to 4.2 0
\put{$\ssize 0$} at 0 0.3
\put{$\ssize 1'$} at 1 0.3
\put{$\ssize 2'$} at 2 0.3
\put{$\ssize 1''$} at 3 0.3
\put{$\ssize 2''$} at 4 0.3
\put{$\ssize 3$} at 5 0.3
\put{$\alpha$} at 0.5 0.3
\put{$\beta'$} at 1.5 0.3
\put{$\gamma$} at 2.5 0.3
\put{$\beta''$} at 3.5 0.3
\put{$\delta$} at 4.5 0.3
\endpicture}
$$
and barify the arrows $\alpha_2$ and $\alpha_4$. The we obtain the quiver
$$
\hbox{\beginpicture
\setcoordinatesystem units <1cm,1cm>
\multiput{$\circ$} at 0 0  1 0  2 0  3 0 /
\arrow <1.5mm> [0.25,0.75] from 0.8 0 to 0.2 0
\arrow <1.5mm> [0.25,0.75] from 2.8 0 to 2.2 0
\arrow <1.5mm> [0.25,0.75] from 1.25 0.125 to 1.2 0.1
\arrow <1.5mm> [0.25,0.75] from 1.25 -.125 to 1.2 -.1
\put{$\ssize 0$} at 0 0.3
\put{$\ssize 1$} at 1 0.3
\put{$\ssize 2$} at 2 0.3
\put{$\ssize 3$} at 3 0.3
\put{$\alpha$} at 0.5 0.3
\put{$\beta$} at 1.5 0.4
\put{$\gamma$} at 1.5 -0.4
\put{$\delta$} at 2.5 0.3
\setquadratic
\plot 1.8 0.1  1.5 0.2  1.2 0.1 /
\plot 1.8 -.1  1.5 -.2  1.2 -.1 /
\setdots <.5mm>
\plot 0.6 -0.1 1 -0.13  1.3 -0.3 /
\plot 2.6 -0.1 2 -0.13  1.7 -0.3 /
\endpicture}
$$
Here, starting with a representation-finite algebra, we obtain a tame one.
Similarly, if we start with the following tame quiver, the barification 
of $b'$ and $b''$ yields
a wild algebra:
$$
\hbox{\beginpicture
\setcoordinatesystem units <1cm,1cm>
\multiput{$\circ$} at 0 0  1 0  2 0  3 0  4 0  5 0  6 -0.5  6 0.5  7 0.5 /
\arrow <1.5mm> [0.25,0.75] from 0.8 0 to 0.2 0
\arrow <1.5mm> [0.25,0.75] from 1.8 0 to 1.2 0
\arrow <1.5mm> [0.25,0.75] from 2.2 0 to 2.8 0
\arrow <1.5mm> [0.25,0.75] from 3.8 0 to 3.2 0
\arrow <1.5mm> [0.25,0.75] from 4.8 0 to 4.2 0
\arrow <1.5mm> [0.25,0.75] from 6.8 0.5 to 6.2 0.5
\put{$\ssize 0$} at 0 0.3
\put{$\ssize 1'$} at 1 0.3
\put{$\ssize 2'$} at 2 0.3
\put{$\ssize 1''$} at 3 0.3
\put{$\ssize 2''$} at 4 0.3
\put{$\ssize 3$} at 5 0.3
\put{$\ssize 4$} at 6 0.7
\put{$\ssize 6$} at 6 -0.3
\put{$\ssize 5$} at 7 0.7
\arrow <1.5mm> [0.25,0.75] from 5.8 0.4 to 5.2 0.1
\arrow <1.5mm> [0.25,0.75] from 5.8 -.4 to 5.2 -.1
\put{$\alpha$} at 0.5 0.3
\put{$\beta'$} at 1.5 0.3
\put{$\gamma$} at 2.5 0.3
\put{$\beta''$} at 3.5 0.3
\put{$\delta$} at 4.5 0.3
\endpicture}
$$
$$
\hbox{\beginpicture
\setcoordinatesystem units <1cm,1cm>

\multiput{$\circ$} at 0 0  1 0  2 0  3 0  4 -0.5  4 0.5  5 0.5 /
\arrow <1.5mm> [0.25,0.75] from 0.8 0 to 0.2 0
\arrow <1.5mm> [0.25,0.75] from 2.8 0 to 2.2 0
\arrow <1.5mm> [0.25,0.75] from 1.25 0.125 to 1.2 0.1
\arrow <1.5mm> [0.25,0.75] from 1.25 -.125 to 1.2 -.1
\put{$\ssize 0$} at 0 0.3
\put{$\ssize 1$} at 1 0.3
\put{$\ssize 2$} at 2 0.3
\put{$\ssize 3$} at 3 0.3
\put{$\alpha$} at 0.5 0.3
\put{$\beta$} at 1.5 0.4
\put{$\gamma$} at 1.5 -0.4
\put{$\delta$} at 2.5 0.3
\setquadratic
\plot 1.8 0.1  1.5 0.2  1.2 0.1 /
\plot 1.8 -.1  1.5 -.2  1.2 -.1 /
\setdots <.5mm>
\plot 0.6 -0.1 1 -0.13  1.3 -0.3 /
\plot 2.6 -0.1 2 -0.13  1.7 -0.3 /

\setsolid
\arrow <1.5mm> [0.25,0.75] from 4.8 0.5 to 4.2 0.5
\put{$\ssize 4$} at 4 0.7
\put{$\ssize 6$} at 4 -0.3
\put{$\ssize 5$} at 5 0.7
\arrow <1.5mm> [0.25,0.75] from 3.8 0.4 to 3.2 0.1
\arrow <1.5mm> [0.25,0.75] from 3.8 -.4 to 3.2 -.1

\endpicture}
$$

\section{Accessible representations}
 
We have mentioned in the introduction that the recent paper \cite{B2} of Bongartz 
has drawn the attention to the minimal representation-infinite
algebras which have a good cover $\widetilde \Lambda$, such that all finite convex subcategories of $\widetilde \Lambda$
are representation-finite. As we show above, these algebras are special biserial and can be completely
classified. The title of the Bongartz paper \cite{B2} indicates that his main concern was 
to proof the following theorem: {\it Let $\Lambda$ be a finite dimensional $k$-algebra where
$k$ is an algebraically closed field. 
If there exists an indecomposable $\Lambda$-module 
of length $n > 1$, then there exists an indecomposable $\Lambda$-module of length $n-1$.}
Unfortunately, the statement does not assert any relationship between 
the modules of length $n$ and 
those of length $n-1$. There is the following open problem:
{\it Given an indecomposable $\Lambda$-module $M$ of length $n\ge 2$. 
Is there an indecomposable submodule or factor module of length $n-1?$}
The three subspace quiver shows that this may not be true in case the field $k$
is not algebraically closed, say if it is finite field with few elements. 

In \cite{RLMS} we slightly modified the arguments of Bongartz  in order to
strengthen his assertion. Using induction, one may define {\it accessible} modules: 
First, the simple modules are accessible.
Second, a module of length $n \ge 2$ is accessible provided it is indecomposable and 
there is a submodule or a factor module of length $n-1$ which is accessible. 
The open problem mentioned above can be reformulated as follows: Are all indecomposable 
representations of a $k$-algebra $\Lambda$, where $k$ is algebraically closed, 
accessible? This is known to hold in case $\Lambda$ is representation-finite and
the aim of \cite{RLMS} was to show that any representation-infinite algebra over
an algebraically closed field 
has at least accessible modules of arbitrarily large length.

In dealing with special biserial algebras, we do not have to worry about the
size of the base field $k$. The following assertion is vaild for $k$-algebras with
$k$ an arbitrary field.

\begin{proposition}Any indecomposable representation of a special biserial algebra is accessible.
\end{proposition}

\begin{proof} It is obvious that string modules are accessible, thus we only have
to consider band modules. It will be sufficient to show the following: {\it any band module
has a maximal submodule which is a string module.} Thus, let $M$ be a band module.

First, let us consider the special case of dealing with the Kronecker algebra, thus
$M = (M_1,M_2;\alpha,\beta)$ with vector spaces $M_1, M_2$ and invertible  linear maps
$\alpha,\beta:M_1\to M_2.$ Let $M'$ be a submodule of $M$ which is a band module
and of smallest possible dimension. Note that $M'$ is uniquely determined and is
contained in any non-zero regular submodule of $M$. Let $0 \neq x \in M'_1$ and
choose a direct complement $U\subset M_1$ for $kx$. Then $N = (U,M_2;\alpha|U,\beta|U)$
is a submodule of $M$, and of course a maximal one. We claim that $N$ is 
a string module. As a submodule of a regular Kronecker module, we can write 
$N = N'\oplus N''$ with $N'$ preprojective and $N''$ regular.
But $N''$ has to be zero, since otherwise $M'\subseteq N''$, thus $x \in N''_1 \subseteq U$, a contradiction.
This shows that $N$ is a direct sum of say $t$ indecomposable preprojective Kronecker
modules. Since $\dim N_1 - \dim N_2 = -1$, it follows that $t = 1.$ This shows that
$N$ is an indecomposable  preprojective Kronecker module and thus a string module.

Now consider an arbitrary special biserial algebra $\Lambda$ with quiver $Q$. There is a primitive
cyclic word $w\in \Omega(\Lambda)$ and an indecomposable  vector space automorphism 
$\phi:V\to V$ such that $M = M(w,\phi).$ Let $w = l_1\cdots l_n$ with letters $l_i$;
we can assume that $l_{n-1}$ is a direct letter, whereas 
$l_n$ is an inverse letter. Denote by $x_{i-1}$ the terminal point of $l_i$,
for $1\le i \le n.$ Then $M$ is given by $t$ copies $V_i$ of $V$, indexed by $0 \le i \le n-1$,
such that  the arrows of $Q$
operate as follows: if $l_i = \alpha$ is a direct letter (thus an arrow), then $\alpha$
is the identity map $V_i \to V_{i-1}$, if  $l_i$ is an inverse  letter, say 
 $l_i = \alpha^{-1}$ for some arrow  $\alpha$, then $\alpha$ 
is the identity map $V_{i-1} \to V_{i}$  for $i\neq n$ and the map $\phi:V_{n-1} \to V_{0}$
for $i = n$. Note that $(V,V;1,\phi)$ is a band module for the Kronecker quiver,
thus, as we have seen already,  it has a maximal submodule $ (U,V;1|U,\phi|U)$
which is a string module. We obtain a submodule $N$ of $M = \bigoplus_{i=0}^{n-1} V_i$  by
taking the subspace $N =  \bigoplus_{i=0}^{n-2} V_ i\oplus U$, where 
$U$ is considered as a subspace of $V_{n-1}$. Since  $ (U,V;1|U,\phi|U)$ is a string
module for the Kronecker algebra, it follows that $N$ is
a string $\Lambda$-module.
\end{proof}

\section{Semigroup algebras}

It should be mentioned that algebras defined by a quiver, commutativity relations and zero relations 
can be
considered as factor algebras of a semigroup algebra $k[S]$  modulo a one-dimensional ideal generated
by a central idempotent $e$, thus the paper may be seen as dealing with a class of minimal
representation-infinite semigroups. 

Let $S$ be a semigroup (a set with an associative binary operation). An element $z$ of $S$ is called
a {\it zero} element provided $sz = z = zs$ for all $s\in S$. Of course, if there is a zero element, then it
is uniquely determined. 
Let $S$ be a semigroup with zero element $z$, we want to  consider the semigroup algebra $k[S]$. 
Obviously, the element $z$ considered as an element of $k[S]$ is a central idempotent and the ideal 
$\langle s\rangle$ generated by $z$ is one-dimensional, thus $z$ is a primitive idempotent. With $z$
also $1-z$ is a central idempotent, and we obtain a direct decomposition of $k[S]$ as a product of
$k$-algebras
$$
 k[S] = \langle z\rangle \times \langle 1-z\rangle = kz \times k[S](1-z).
$$
One may call $ k[S](1-z) = k[S]/\langle z\rangle$ the {\it reduced\/} semigroup algebra of $S$.
It follows that the modules for the reduced semigroup algebra of $S$ are precisely the $k[S]$ modules $M$ with
$zM = 0.$ 

The product decomposition of the semigroup algebra $k[S]$ shows that there is a unique simple
(one-dimensional) $k[S]$-module which is not annihilated by $z$, all other indecomposable
$k[S]$-modules are annihilated by $z$ and thus are modules over the reduced semigroup-algebra. 

Given a quiver $Q$, let $S(Q)$ be obtained from the set of all paths (including the paths of length $0$) by
adding an element $z$ (it will become the zero element). As in the definition of the path algebra $kQ$
of a quiver, define the product of two paths to be the concatenation, if it exist, and to be $z$ otherwise.
In this way, $S(Q)$ becomes a semigroup with zero element $z$, and {\it the reduced semigroup algebra
of $S(Q)$ can be identified with the path algebra $k[Q]$ of the quiver $Q$.}

Of course, if we deal with a set $\rho$ of commutativity relations and zero relations, then we may consider
the factor
semigroup $S(Q,\rho) = S(Q)/\langle \rho\rangle$, this is again a semigroup with zero, and its
reduced semigroup algebra is just the algebra defined by the quiver $Q$ and the relations $\rho$.

\frenchspacing

\end{document}